%% file: main.tex
\documentclass[11pt]{article}
\usepackage{ifthen}

\newcommand{\citep}[1]{\cite{#1}}

\input{new_def.tex}

\begin{document}

\title{Universality of Empirical Risk Minimization}

\author{Basil Saeed\thanks{Department of Electrical Engineering, Stanford University}  \and 
Andrea Montanari\thanks{Department of Statistics and Department of Mathematics, 
Stanford University}}

\maketitle

\begin{abstract}
We study
a general class of optimization problems with decision 
variable $\bTheta\in\reals^{p\times \sfk}$ and cost function which is the sum of $n$  terms,
each dependent on $\bTheta$ through the $\sfk$-dimensional projection 
$\bTheta^{\sT}\bx_i$, where $\bx_i$, $i\le n$ are i.i.d. random vectors.

This setting is general enough to include examples of current interest
in statistical physics, high-dimensional statistics, and statistical learning theory. 

We consider the proportional asymptotics $n,p\to\infty$ , with $n/p=\Theta(1)$,
and prove that, whenever there exists a minimizer satisfying a suitable generalization of a ``delocalization'' condition,
 the minimum value is universal. 
Namely, (for subgaussian $\bx_i$) it depends on the distribution of $\bx_i$ only through
its asymptotic mean and covariance.  This delocalization
condition is essentially necessary.
Earlier universality results for such problems were limited to strongly convex loss functions.

We derive applications of our theory  to statistical learning and 
prove general universality results
both for train and (under additional conditions) test error. 
In particular, we establish universality for vectors $\bx_i$ generated by random 1-layer
neural networks (random features models) and first-order Taylor approximations of 2-layer networks
(neural tangent models).
Finally, we establish that the delocalization property holds for 
a class of statistical learning problems under a condition that is easy to verify.
\end{abstract}

\section{Introduction and main result}

\subsection{Problem statement}
Let $(\bx_i:i \le n)$ be a collection of i.i.d. random vectors in
$\reals^p$, and $(\eps_i:i\le n)$ be i.i.d. random variables independent of
the $\bx_i$'s. Given an integer $\sfk\ge 1$, compact sets $\cC_{j,p}\subseteq \reals^p$, $j\le \sfk$,
 a non-negative continuous function $\ell:\reals^{\sf k}\times\reals\to \reals_{\ge 0}$
 and a continuous function $r:\reals^{p\times\sfk}\to \reals_{\ge 0}$, we study the following minimization  problem
\begin{align}\label{eq:OptimizationFirst}
\hR_n^{\star}(\bX,\beps) &: = \min_{\bTheta\in \cC_p^{\sfk}} \hR_n(\bTheta;\bX,\beps) \, ,\\
 \hR_n(\bTheta;\bX,\beps) &:=\frac{1}{n}
\sum_{i=1}^n\ell\big(\bTheta^{\sT}\bx_i,\eps_i\big)+r(\bTheta)\, .
\end{align}
We introduced the notations $\bX\in\reals^{n\times p}$ for the matrix with rows $\bx_i\in\R^p$, $\beps\in \reals^n$ for the vector with entries 
$\eps_i$, and $\cC_p^{\sfk}:=\cC_{1,p}\times \cdots\times \cC_{\sfk, p}$ 
for the set of matrices $\bTheta\in\reals^{p\times\sfk}$ with $j$-th column in $\cC_{j,p}$.
We write $\cC_p=\cC_{j,p}$ whenever this set is the same for all $j\le \sfk$.
We will refer to \eqref{eq:OptimizationFirst} as an \emph{Empirical Risk Minimization} (ERM) problem. 

We consider sequences of such problems indexed by $n$, whereby $p=p(n)$
is such that $p(n)\asymp n$. We want to understand
to what extent the asymptotics of $\hR_n^\star(\bX,\beps)$ are universal
with respect to the distribution of the vectors $\bx_i$
(i.e. dependent on the distribution of $\bx_i$ only via its mean and covariance.)

Random optimization problems of the form \eqref{eq:OptimizationFirst}
are of interest in
(among other fields) 
statistical physics \cite{hopfield1982neural,SpinGlass},
signal processing \cite{TanakaCDMA,Donoho1,CandesFourier}, 
and random geometry \cite{oymak2018universality,montanari2024tractability,huang2024capacity}.
Our main motivation stems from the analysis of techniques in high-dimensional statistics and machine learning, and we will devote Sections~\ref{sec:StatLearning}-\ref{sec:Deloc} to 
applications of our main results to statistical learning. 

It is instructive to begin with
two examples of Eq.~\eqref{eq:OptimizationFirst} in which universality fails.
\begin{example}
\label{ex:non_univ_0}
For each $i\in[n]$, let $\bx_i\simiid\Unif(\{\pm \sqrt{p}\be_1,\dots,\pm\sqrt{p}\be_p\})$ where $\be_j$ is the $j$-th element
of the canonical basis, while $\bg_i\simiid\normal(\bzero,\bI_p)$, so that the first two moments 
of $\bx_i$ and $\bg_i$ are matched.
Let $\cC_p = \{+1/\sqrt{p},-1/\sqrt{p}\}^p$, $\ell(x) = |x|$ and consider the optimization problem:
\begin{equation}
\hR_n^{\star}(\bX):= \min_{\btheta\in\cC_p}  \hR_n^{\star}(\btheta;\bX)
\, ,
\;\;\;\;\;\;\;  \hR_n(\btheta;\bX):=
\frac1n \sum_{i=1}^n \ell\left(\btheta^\sT\bx_i\right)\, ,\label{eq:TwoCosts}
\end{equation}
and similarly for $\hR_n^{\star}(\bG):= \min_{\btheta\in\cC_p}  \hR_n(\btheta;\bG)$
(where $\bG$ is the matrix whose $i$-th row is $\bg_i$.)
We have $\hR_n^{\star}(\bX) = 1$ (indeed   $\hR_n^{\star}(\btheta;\bX)=1$ for all $\btheta\in \cC_p$).
On the other hand, for any fixed $\btheta_0\in\cC_p$, $\hR_n^{\star}(\bG)\le \hR_n(\btheta_0;\bG)=\E_{Z\sim\normal(0,1)}|Z| +o_P(1) = \sqrt{2/\pi} + o_P(1)$,
so with high probability $\hR_n^{\star}(\bX)-\hR_n^{\star}(\bG)\ge c>0$: universality fails to hold.

\end{example}
\begin{example}
\label{ex:non_univ}
Let $\bx_i \sim \Unif(\{+1,-1\}^p)$ and $\bg_i\sim\normal(\bzero,\bI_p)$.
Consider the non-negative, Lipschitz continuous loss function
\begin{equation}
\nonumber
\ell(t) := \begin{cases}
                \big| 1 - |t| \big| & |t| \le 2 \, ,\\
                1 &  |t| >2\,,
           \end{cases}
\end{equation}
the constraint set
$\cC_p:=\{\btheta\in\reals^p:
\|\btheta\|_2=1\}$ and define $\hR_n^{\star}(\bX)$, $\hR_n^{\star}(\bG)$
as in Eq.~\eqref{eq:TwoCosts}.

We clearly have $\hR_n^{\star}(\bX)=0$ for all $n$,
since  $\hR_n^{\star}(\bX)\ge 0$, while $\hR_n^{\star}(\bX)\le 0$
follows by evaluating the cost at
$\btheta = \be_1=(1,0,\dots,0)^{\sT}$. 

On the other hand, $\E \hR(\btheta;\bG) = \E_{Z\sim\normal(0,1)}\ell(Z)>0$ and 
a standard uniform convergence computation (deferred to Section~\ref{section:deferred_proofs} of the Appendix) 
yields
$\sup_{\btheta\in\cC_p}|\hR_n(\btheta;\bG)-\E\hR_n(\btheta;\bG)|\le C\sqrt{p/n}$. This implies
the existence of constants $c>0$, $\sgamma_0>0$ such that 
\begin{equation}
\label{eq:ClaimExample}
\lim_{\substack{n,p\to\infty\\ 
n/p\to \sgamma > \sgamma_0
}}  \P\big(\hR_n^{\star}(\bG)\ge c\big) =1\, .
\end{equation}
Hence, universality fails in this setting as well.
\end{example}

In Example \ref{ex:non_univ_0} the coordinates of $\bx_1$ are highly dependent, while in Example \ref{ex:non_univ} they are not.
The crucial property that both examples share is that the minimum of $\hR_n(\btheta;\bX)$ 
over $\btheta\in \cC_p$ is only achieved at points $\btheta$ such that $\btheta^{\sT}\bx_1$
is not approximately Gaussian.
Our main technical result below
states that this is essentially \emph{the only} mechanism for failure.
While this conclusion may seem `obvious' at a superficial level, it is in fact rather surprising because  
the empirical distribution of $\big\{\hbtheta^{\sT}\bx_i:\, i\le n\big\}$  (with $\hbtheta$
the minimizer) is generally far from Gaussian, even when the $\bx_i$'s are Gaussian (see Remark \ref{rmk:NonGaussian} below).

To be more precise about this notion of approximate Gaussianity for $\btheta^\sT\bx_1$ at $\btheta$, let us introduce
the following definition which will be crucial to our exposition.
\begin{definition}
\label{def:DoG}
   For a (sequence of) random variables $\bx \sim\P_\bx$ in $\R^p$,
   we call a (sequence of) symmetric, convex sets $\cS_p \subseteq \R^p$ 
  a \emph{domain of Gaussianity} for $\P_\bx$ if the following holds for any bounded Lipschitz function $\varphi:\R\to\R$:
   \begin{equation}
   \label{eq:condition_bounded_lipschitz_single}
       \lim_{p\to\infty} \sup_{\btheta \in\cS_p} \left|\E[\varphi(\btheta^\sT \bg)] - \E[\varphi(\btheta^\sT\bx)]\right| = 0,
   \end{equation}
   where $\bg \sim\cN(\bmu_\bg,\bSigma_\bg)$ for some (sequence of) means and covariances $\bmu_\bg \in\R^p,\bSigma_\bg \in\R^{p\times p}.$
\end{definition}
Informally, the `domain of Gaussianity' is the set of directions along which the projections of $\bx$ are asymptotically Gaussian. 
For example, if $\bx\sim\P_\bx$ is a vector of independent entries (with bounded second moments), then $\cS_p := \{\btheta \in\R^p: \|\btheta\|_\infty \le \alpha_p, \|\btheta\|_2 \le C \}$
---with any $\alpha_p\to0$ and $C>0$---
is a domain of Gaussianity for $\P_\bx$ (see Section~\ref{section:linear_example} for more details.) 
Clearly, the minimizers $\hat\btheta$ of Example~\ref{ex:non_univ} are not in this (or any other) domain of Gaussianity for $\P_\bx$.

Let us emphasize that,  for other distributions of the random vectors $\bx$, 
a domain of Gaussianity can take on a different form from the vanishing $\ell_{\infty}$
norm condition in this example (see Section~\ref{section:ntk_example} for an example.)

\subsection{Main technical result}

We state our assumptions about the optimization problem \eqref{eq:OptimizationFirst}, and our main technical result.
Throughout,  we consider a sequence of problems indexed by $n$,  $p=p(n)$, 
both diverging, while $\sfk$
is fixed.
Also $r$, $\cC_p$, the law of the $\bx_i$, the law of $\eps_i$
can depend on $n,p$. The constants in the assumptions below are typically denoted by sans-serif
fonts and assumed fixed as $n,p$ diverge.
\begin{assumption}[Regime]
\label{ass:regime}
\label{ass:1}
 There exists a constant $\sfC>0$ such that
 for all $n,$
 \begin{equation}
 \nonumber
     \sfC^{-1} \le \frac{p(n)}{n} \le \sfC.
 \end{equation}
\end{assumption}

\begin{assumption}[Constraint set]
\label{ass:set}
 The sets $\cC_{j,p}$, $j\le \sfk$ appearing in the constraint in~\eqref{eq:OptimizationFirst}
 are compact with $\ell_2$ radius bounded by $\sfR$.
\end{assumption}

\begin{assumption}[Regularization]
\label{ass:regularizer}
The penalty function $r(\bTheta)$ is locally Lipschitz in Frobenius norm, uniformly in $p$. That is,
for all $p\in\Z_{>0}$, $B>0$, and
$\bTheta,\widetilde\bTheta \in\R^{p\times \sfk}$ satisfying 
$\norm{\bTheta}_F,\|{\widetilde\bTheta}\|_F \le B$, we have
for some $\sfK_r(B) > 0,$
\begin{equation}
\nonumber
\big|r(\bTheta) - r(\widetilde\bTheta) \big|  \le \sfK_r(B) \big\|{\bTheta - 
\widetilde\bTheta}\big\|_F.
\end{equation}
\end{assumption}

Recall that the vectors $\{\bx_i\}_{i\le n}$  are i.i.d., and we denote by $\bx\in\reals^p$
an i.i.d. copy of $\bx_1$.
%


\begin{assumption}[Restricted subgaussianity]
\label{ass:-1}
\label{ass:X}
The following holds for a (sequence of) Gaussian vectors $\bg \sim \cN(\bmu_\bg,\bSigma_\bg)$, and a constant $\sfK$ independent of $n,p$:
\begin{equation}
\label{eq:sg_assump}
 \sup_{\{\btheta\in\cC_{j,p}: \norm{\btheta}_2 \le 1\}} \norm{\bx^\sT \btheta}_{\psi_2} \le \sfK,\quad
 \sup_{\{\btheta\in\cC_{j,p}: \norm{\btheta}_2 \le 1\}}\big\|{\bSigma_\bg^{1/2} \btheta}\big\|_2 \le \sfK,\quad
 \norm{\bmu_\bg}_2 \le \sfK
\end{equation}
with $\cC_{j,p}$ the sets in Assumption~\ref{ass:set} and $\|\cdot\|_{\psi_2}$ denoting the subgaussian norm.
\end{assumption}

\begin{assumption}[Loss functions]
\label{ass:loss-0}
The non-negative loss function $\ell:\R^{\sfk}\times \R\to\R$ satisfies either one of these
conditions:
\begin{enumerate}
    \item Lipschitz:
    for some $\sfK>0$, we have for any
    $(\bv,\eps),(\tilde\bv,\tilde\eps)\in\R^{\sfk}\times\R,$
    \begin{equation}
    \nonumber
        |\ell(\bv,\eps) - \ell(\tilde\bv,\tilde\eps)| \le \sfK (\|\bv - \tilde \bv\|_2 + |\eps - \tilde\eps|);
    \end{equation}
    \item Locally Lipschitz:
    for any $B>0$,
    $(\bv,\eps),(\tilde\bv,\tilde\eps)\in\R^{\sfk}\times\R,$
    with $\|(\bv,\eps)\|_2,\|(\tilde\bv,\tilde\eps)\|_2\le B,$ we have for some $\sfK_\ell(B)>0$
    \begin{equation}
    \nonumber
        |\ell(\bv,\eps) - \ell(\tilde\bv,\tilde\eps)| \le \sfK_\ell(B)( \|\bv - \tilde \bv\|_2 + |\eps -\tilde\eps| );
    \end{equation}
    and has polynomial growth: for some $\sfK_0,\sfK_1,\sfK$ and $\sfb\ge \sfa > 0$, we have for all $(\bv,\eps)\in\R^{\sfk}\times\R,$ 
    \begin{equation}
    \nonumber
 -\sfK_0+\sfK_1(\|\bv\|_2^\sfa + |\eps|^\sfa) \le \ell(\bv,\eps) \le \sfK(1 + \|\bv\|_2^\sfb + |\eps|^\sfb).
    \end{equation}
\end{enumerate}
\end{assumption}
In Section~\ref{section:pf_under_general_loss}
of the Appendix, we provide a more general assumption on $\ell$ under which our results hold.

\begin{theorem}\label{thm:MainFirst}
Let Assumptions~\ref{ass:regime}-\ref{ass:loss-0} hold.
Let $\bG\in\reals^{n\times p}$ be a matrix with i.i.d. rows $\bg_i$ distributed as
in Assumption \ref{ass:X}, and
variables $(\eps_i: i\le n)$ be independent and independent of $\bX$, $\bG$, with $\max_{i\le n}\|\eps_i\|_{\psi_2}\le \sfK$\,  for some constant $\sfK.$
Finally, let $\cS_p$ be a domain of Gaussianity for $\bx$ as in Definition~\ref{def:DoG}.

If $\cC_{j,p}\subseteq \cS_p$ for all $j\le \sfk$, then  for any
bounded Lipschitz function $\psi:\R\to\R$,
\begin{equation}
\label{eq:main_thm_eq_0}
 \lim_{n\to\infty}\left|\E\left[\psi\left(\widehat 
R^\star_n\left(\bX\right)\right)\right] - 
\E\left[\psi\left(\widehat R^\star_n\left(\bG\right) \right)\right]\right| = 0.
\end{equation}
Consequently, for all $\rho \in\R$,
\begin{equation}
\widehat  R^\star_n(\bX)\stackrel{\P}{\to} 
\rho \quad\textrm{if and only if} \quad \widehat R^\star_n(\bG) \stackrel{\P}{\to}\rho.
\end{equation}
\end{theorem}

\begin{remark}\label{rmk:NonGaussian}
At first sight, Theorem \ref{thm:MainFirst}  might seem to hold for the following simple reason:
Denote by $\hbtheta^{\bX}_n$,  $\hbtheta^{\bG}_n$  the minimizers of $\hR_n(\btheta;\bX)$,
$\hR_n(\btheta;\bG)$, respectively, over $\cC_p$.
Here, we consider for simplicity $\sfk=1$,
and assume that the minimizers are unique.
  Since $\hbtheta^{\bX}_n$,  $\hbtheta^{\bG}_n\in\cC_p\subseteq\cS_p$ and 
since moreover $\<\bx_i,\btheta\>$,  
$\<\bg_i,\btheta\>$ are approximately Gaussian with variance depending on $\|\btheta\|_{\bSigma^{\bg}}$
for any fixed $\btheta\in\cS_p$, it might seem 
that $\<\hbtheta^{\bX}_n,\bx_i\>$, $\<\hbtheta^{\bG}_n,\bg_i\>$ are also approximately Gaussian 
and therefore $n^{-1}\sum_{i=1}^n
\ell(\<\hbtheta^{\bX}_n,\bx_i\>)$ is close to $n^{-1}\sum_{i=1}^n
\ell(\<\hbtheta^{\bG}_n,\bg_i\>)$ (provided that $\|\hbtheta^{\bX}_n\|_{\bSigma^{\bg}}\approx \|\hbtheta^{\bG}_n\|_{\bSigma^{\bg}}$).
Of course, this reasoning is flawed because approximate Gaussianity of $\<\btheta,\bx_i\>$ at fixed $\btheta$
does not imply approximate Gaussianity of  $\<\hbtheta^\bX_n,\bx_i\>$ at random $\hbtheta^\bX_n$
(dependent on $\bx_i$).
This flaw cannot be fixed and the very intuition at the basis of this argument is wrong in high dimension. 

Indeed, if $p/n\to\sf\sgamma$ (except in special cases), the distributions of  
$\<\hbtheta^\bX_n,\bx_i\>$, $\<\hbtheta^\bG_n,\bg_i\>$ converge to a non-Gaussian limit (see, e.g. \cite{asgari2025local}). For instance, if $\ell$ is convex and $r$ is quadratic, this limit is the distribution of $\Prox_{b\ell}(G)$, where $\Prox_{b\ell}$ is the proximal operator for a scaling $b>0$ of $\ell$
and $G$ is a Gaussian random variable. Hence the proof of Theorem \ref{thm:MainFirst} must rely on very 
different arguments.
\end{remark}

\subsection{Proof technique}

We present the proof of Theorem \ref{thm:MainFirst} in Section~\ref{sec:Proofs}, with several technical steps deferred to the appendices.  
The proof is based on an interpolation method. 
Namely we consider an ERM problem with feature matrix $\bU_t = \sin(t) \bX+\cos(t) \bG$ that
continuously interpolates between the two cases as $t$ goes from $0$ to $\pi/2$. We then bound
the change in the training error (minimum empirical risk) along this path. 

This approach is analogous to the Lindeberg method~\cite{lindeberg1922neue,chatterjee2006generalization}, 
which was used in the context of statistical learning in  \cite{korada2011applications}
and subsequently  in \cite{montanari2017universality,oymak2018universality,hu2020universality}.
A direct application of the Lindeberg procedure would require 
to swap an entire row of $\bX$ with the corresponding row of $\bG$ and bound the effect
on the minimum empirical risk (we cannot replace one entry at a time since these are 
dependent). We find the use of a continuous path more effective.

In earlier work \cite{hu2020universality}, the effect of a swapping step is controlled by 
first bounding the change in the minimizer $\hbTheta$. This is achieved by assuming strong convexity 
of the empirical risk.  The bound on the change of the minimizer immediately implies a bound on 
the change of
the minimum value. 
In the non-convex setting, we face the challenge of bounding the change of the minimum
without bounding the change of the minimizer. We achieve this by using a 
differentiable approximation of the minimum.

After this sequence of approximations, we are left with the task of bounding 
the expectation of the change in cost function along the interpolation path
with respect to a suitably defined random Gibbs measure. 
This bound is the core of our proof. The difficulty is that this
expectation depends in a rather implicit fashion 
on the random matrices $\bX$, $\bG$ (because the Gibbs measure depends on $\bX$, $\bG$).
The key  innovation is a polynomial approximation
method which simplifies this dependence and 
we believe can be of more general applicability.

\subsection{Organization of the paper}
\begin{itemize}
\item Section \ref{sec:Related} overviews recent related work on universality in 
ERM problems. 
\item In Section \ref{sec:StatLearning}, we apply Theorem \ref{thm:MainFirst}
to empirical risk minimization in statistical learning. We also derive further results that are of 
interest in that context. 
\item In Section \ref{sec:Pointwise}, we derive the domain of Gaussianity for specific distributions
of the random vectors $\bx_i$ which arise in the analysis of neural networks, and check Assumption~\ref{ass:X} in these settings.
\item In general, we can be interested in the optimization 
problem \eqref{eq:OptimizationFirst} over constraint sets $\cC_p$ that are not subsets
of the Gaussianity domain $\cS_p$. In these cases, we can still apply Theorem \ref{thm:MainFirst}
after checking that the minimum is achieved in a subset of $\cS_p$.
In Section~\ref{sec:Deloc}, we give techniques to prove this.
\item Finally, the main steps in the proofs are presented in Section \ref{sec:Proofs}.
\end{itemize}

\subsection{Definitions and notations}
We reserve the \textsf{sans-serif font} for parameters that are considered as 
fixed. 
We use $\norm{X}_{\psi_2}$ and $\norm{f}_{\Lip}$ to denote the subgaussian norm of a random variable 
$X$ and the Lipschitz modulus of a function $f$, respectively, and
$B_q^p(r)$ to denote the $\ell_q$ ball of radius $r$ in $\R^p$. 
We use $C,C_0,C_1,\dots$ or $c,c_0,c_1,\dots$ to denote constants that are independent of $n$ and $p$, but could depend on the fixed parameters such as $\sfk,\sfK,\sfK_r$ and so on.
If a constant $C$ depends \textit{additionally} on some variable, say
$\beta$, we write $C(\beta)$ when we want to emphasize this dependence.

\section{Related literature}
\label{sec:Related}

Universality in empirical risk minimization has recently attracted attention
because of its relevance to the analysis of random features models 
or neural tangent kernel models in machine learning.
This can in turn be viewed as approximations of neural networks 
under certain training scenarios \cite{bartlett2021deep}.
Universality for  random feature models
was proven for the special case  of ridge regression in
 \cite{hastie2019surprises} and \cite{mei2019generalization}. This corresponds to 
the ERM problem \eqref{eq:SecondERM} whereby $\sfk=1$, $\ell(u,y) = (u-y)^2$,
 and $r(\btheta) = \lambda\|\btheta\|_2^2$. Concurrently,
 \cite{goldt2020modeling,goldt2020gaussian} provided heuristic arguments and empirical results 
 indicating that universality holds for other ERM problems as well.

 Universality results for ERM were proven  for feature
vectors $\bx_i$ with independent entries in
\cite{korada2011applications,montanari2017universality,panahi2017universal,han2022universality}. Related results
for randomized dimension reduction were obtained in \cite{oymak2018universality}.
The case of general vectors $\bx_i$ is significantly more challenging.
To the best of our knowledge, the first and only proof of universality 
 beyond independent entries was given by Hu and Lu
 \cite{hu2020universality}.

The result of \cite{hu2020universality} is limited to strongly convex 
ERM problems. Their proof  uses a Lindeberg swapping argument, whereby the 
rows of $\bX$ are replaced one-by-one by Gaussian rows with the same mean and covariance.
This requires bounding at each step the resulting change in train error
$\min_{\bTheta}\hR_n(\bTheta;\bZ,\by)$, which  the authors achieve by bounding the change in 
the minimizer. Strong convexity is crucial in this type of proof to control the change of minimizer 
under a perturbation of the cost.  

As mentioned above, we also use an interpolation argument.
The key challenge in this type of arguments is to bound the derivative
(or finite difference) along the interpolation path. We manage to do this
without resorting to a perturbation method.

A significant line of recent work studies the asymptotic properties of
the ERM \eqref{eq:SecondERM} under the proportional asymptotics $n,p\to\infty$
with $n/p\to{\sgamma}\in (0,\infty)$. A number of phenomena have been elucidated by these 
studies \cite{BayatiMontanariLASSO,thrampoulidis2015,thrampoulidis2018precise}, including the
 design of optimal 
loss functions and regularizers \cite{Donoho2016,el2018impact,celentano2019fundamental,aubin2020generalization},
the analysis of inferential procedures \cite{sur2019likelihood,celentano2020lasso}, 
and the double descent behavior of the generalization error 
\cite{hastie2019surprises,deng2019model,montanari2019generalization,gerace2020generalisation}. 
However, these works often assume  Gaussian feature vectors or feature 
vectors with independent coordinates, and the generalization to dependent non-Gaussian 
features is an open challenge. The present work provides a needed tool for these generalizations.

After a first version of this manuscript was posted online 
(and presented at the Conference on Learning Theory, COLT, 2022),
several groups obtained follow-up results.

Among others, \cite{han2023universality} obtained an improvement of the results of  \cite{korada2011applications,montanari2017universality} (using the same proof technique). Dudeja, Sen, Lu \cite{dudeja2024spectral} proved a general universality result for structured matrices $\bX$ (not necessarily with independent rows), in the case of quadratic $\ell$, convex $r$. 
In \cite{montanari2023universality}, Ruan, Sohn and the present authors
leveraged results established here to prove universality for max-margin classifiers (which cannot be directly written
as ERM), and in \cite{verchand2024high} Verchand and one of the authors 
used universality techniques to study regression with missing entries. Finally, 
Bandeira and Maillard  \cite{maillard2023exact} used the technique developed here
to solve an approximate version of an open problem on interpolation by ellipsoids initially posed by Saunderson et al.
\cite{saunderson2012diagonal}.

\section{Universality in statistical learning}
\label{sec:StatLearning}
 In this section we apply Theorem \ref{thm:MainFirst}
to empirical risk minimization in statistical learning.

\subsection{Background}
In a classical supervised learning problem, we are given $n$ i.i.d. 
samples $\{(y_i,\bz_i)\}_{i\le n}$ where $\bz_i\in \reals^{d}$ are covariate
vectors and $y_i\in\reals$ are labels. We  would like to learn a model $f:\reals^d\to\reals$
to predict the label $y_{\snew}$ given a new input vector $\bz_{\snew}$. We consider the following general approach\footnote{A slightly more general
framework would allow $F(\, \cdot\,;\ba):\reals^{\sfk}\to\reals$ to depend
on additional parameters $\ba\in\reals^{\sfk'}$, $\sfk'=O(1)$. This can be treated
using our techniques, but we refrain from such generalizations for the sake of clarity.}:
\begin{enumerate}
\item  Process the covariates through a featurization map $\bphi:\reals^{d}\to\reals^p$
to obtain feature vectors $\bx_1=\bphi(\bz_1)$, \dots, $\bx_n=\bphi(\bz_n)$.
\item Select a class of functions that depends on $\sfk$ linear projections of the features,
with parameters $\bTheta= (\btheta_1,\dots,\btheta_{\sfk})\in\reals^{p\times \sfk}$,
$\btheta_i\in\reals^p$. Namely, for a fixed $F:\reals^{\sfk}\to \reals$, we consider
\begin{align}
f(\bz;\bTheta) = F(\bTheta^{\sT}\bphi(\bz))\, .\label{eq:FzFirst}
\end{align}
\item Fit the parameters via (regularized) empirical risk minimization (ERM):
\begin{align}
\mbox{minimize }\;\;\; \hR_n(\bTheta;\bZ,\by) := \frac{1}{n}\sum_{i=1}^n
L(f(\bz_i;\bTheta),y_i)+r(\bTheta)\, .\label{eq:FirstERM}
\end{align}
with $L:\reals\times\reals\to \reals_{\ge 0}$ a loss function,
and $r:\reals^{p\times \sfk}\to \reals$ a regularizer, where $\bZ = (\bz_1,\dots,\bz_n),\by=(y_1,\dots,y_n)$.
\end{enumerate}
 A concrete example of the featurization map $\bphi$
is given by the `random features' model of Rahimi and Recht~\cite{rahimi2007random}:
\begin{align}
 \bphi(\bz_i) = \sigma(\bW^{\sT}\bz_i).\label{eq:FirstRF}
\end{align}
Here $\bW\in \reals^{d\times p}$ is a matrix (which might be itself random
but independent of $\bz_i$) and the activation function $\sigma:\reals\to\reals$
is applied entrywise. We will formally analyze this setting in Section~\ref{section:rf_example}.

For mathematical analysis, we can let $\bx_i:=\bphi(\bz_i)$
in Eq.~\eqref{eq:FirstERM}, and therefore the role of $\bphi(\,\cdot\, )$
will be to specify the law of $\bx_i$ (as the push-forward of the law of $\bz_i$).

We will assume that the response $y_i$ depends on the feature vector $\bx_i$
through a low-dimensional projection $\bTheta^{\star\sT}\bx_i$, where
$\bTheta^\star = (\btheta^\star_1,\dots,\btheta^\star_{\sfk^\star})\in 
\R^{p\times \sfk^\star}$ is a fixed matrix of parameters. 
Namely, we let  $\beps := (\eps_1,\dots,\eps_n)$
where $\{\eps_i\}_{i\le n}$ are i.i.d. and set:
\begin{equation}
\label{eq:form_yi}
y_i = 
\eta\left(\bTheta^{\star\sT}\bx_i 
,\eps_i
\right)
\end{equation}
for $\eta: \R^{\sfk^\star+ 1}\to\R$. We write $\by(\bX)$ or $y_i(\bx_i)$
when we want to  make the functional 
dependence of $\by$ on $\bX$ explicit. 

Under the definitions of Eq.~\eqref{eq:FzFirst}-\eqref{eq:form_yi},
we can rewrite the empirical risk \eqref{eq:FirstERM} as:
\begin{align}
 \hR_n(\bTheta;\bX,\beps)&:= \frac{1}{n}\sum_{i=1}^n
\ell(\bTheta^{\sT}\bx_i,\bTheta^{\star\sT}\bx_i ;\eps_i)+r(\bTheta)\, ,\label{eq:SecondERM}\\
\ell(\bu,\bu^{\star};\eps)&:= 
L(F(\bu),\eta(\bu^\star,\eps) )\, .\label{eq:ell-def}
\end{align}
Hence, the ERM problem in the present setting
falls into the general framework of Eq.~\eqref{eq:OptimizationFirst}
whereby, $\sfk$ gets replaced by $\sfk+\sfk^{\star}$.
(In the following we will abuse notation and write either
$\hR_n(\bTheta;\bX,\beps)$ or $\hR_n(\bTheta;\bX,\by)$ depending on whether we want to emphasize dependence on the noise variables or responses.) 

Learning takes place by minimizing the
regularized empirical risk of Eq.~\eqref{eq:SecondERM}, subject to
 $\btheta_j\in \cC_p$, $j\le \sf k$ (recall that $\btheta_j$ denotes the $j$-th column of $\bTheta$).
Namely, we consider the problem
\begin{equation}
\label{eq:min_problem}
 \widehat R^\star_n(\bX,\by) := \min_{\bTheta \in \cC_p^\sfk } 
\widehat R_n(\bTheta; \bX, \by).
\end{equation}
Since $\bTheta^\star$ is not optimized over, this corresponds to
taking $\cC_{j,p}=\cC_p$ for $1\le j\le \sfk$, $\cC_{\sfk+l,p} = \{\btheta^{\star}_l\}$ for $1\le l\le \sfk^{\star}$.

\begin{figure}
\includegraphics[scale=0.46]{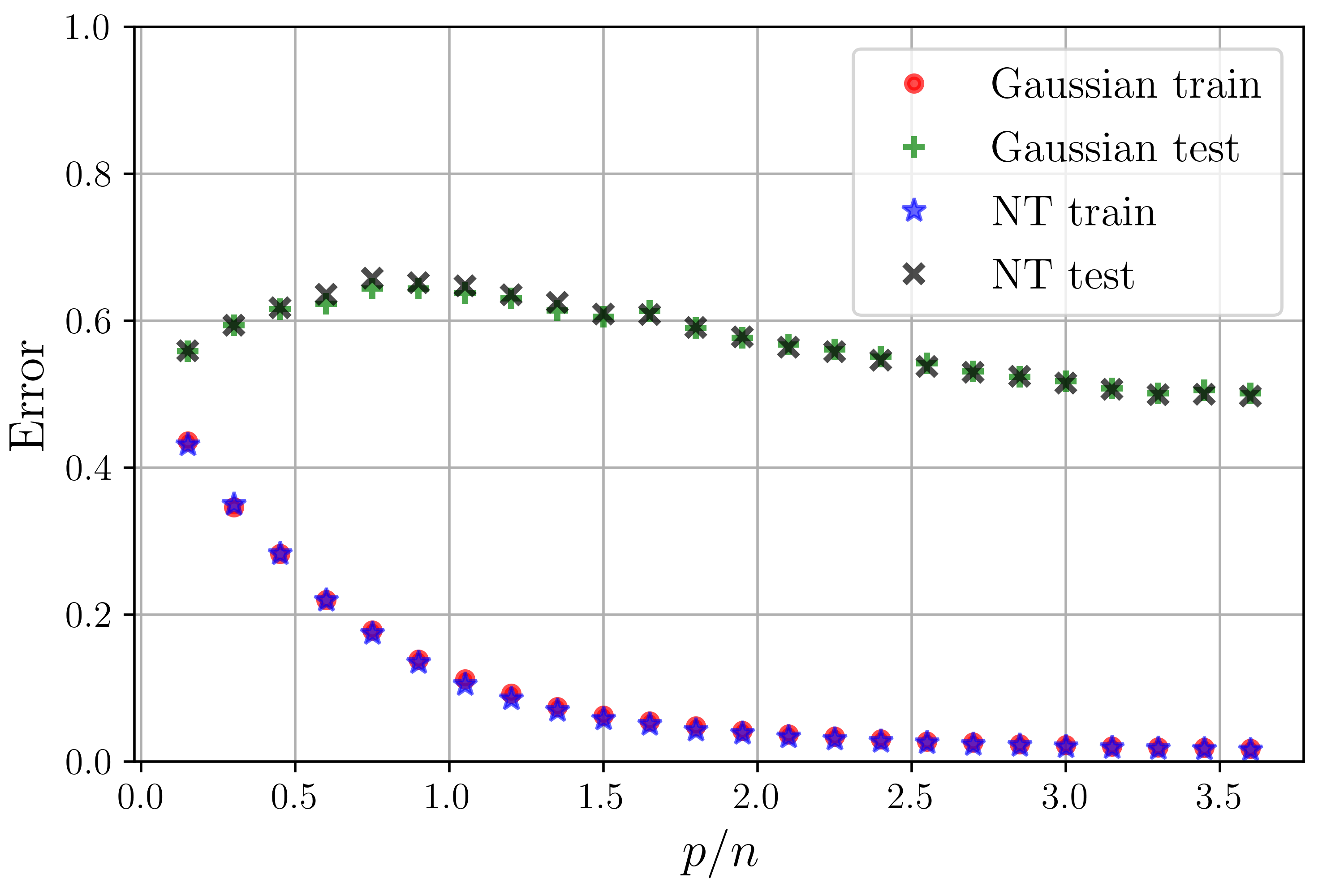}
\includegraphics[scale=0.46]{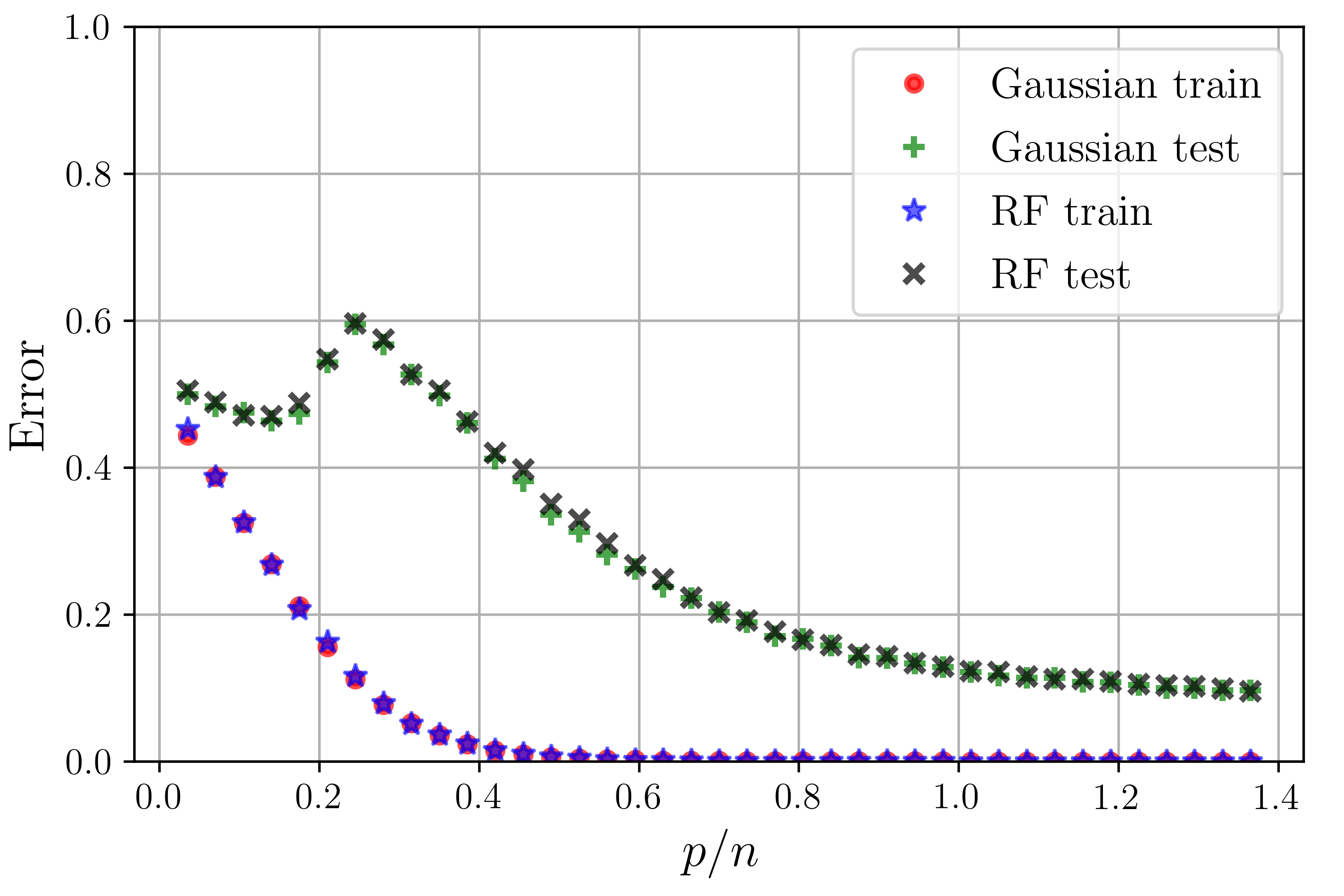}
\caption{Universality of the training and test errors in a simulation experiment:
see main text for description.
In both figures, we take sample size $n=200$ and noise standard deviation $\nu = 0.1$.
In the figure on the left, we take latent dimension $d=30$, regularization $\lambda = 0.02$ and a
neural tangent featurization map.
 In the figure on the right, we take $d=100,\lambda = 0.0002$ and a random features map. 
 We vary the number of features $p$, and at each point we report the average over 
100 realizations.}
\label{fig:univ_sim}
\end{figure}

\begin{figure}
\includegraphics[scale=0.46]{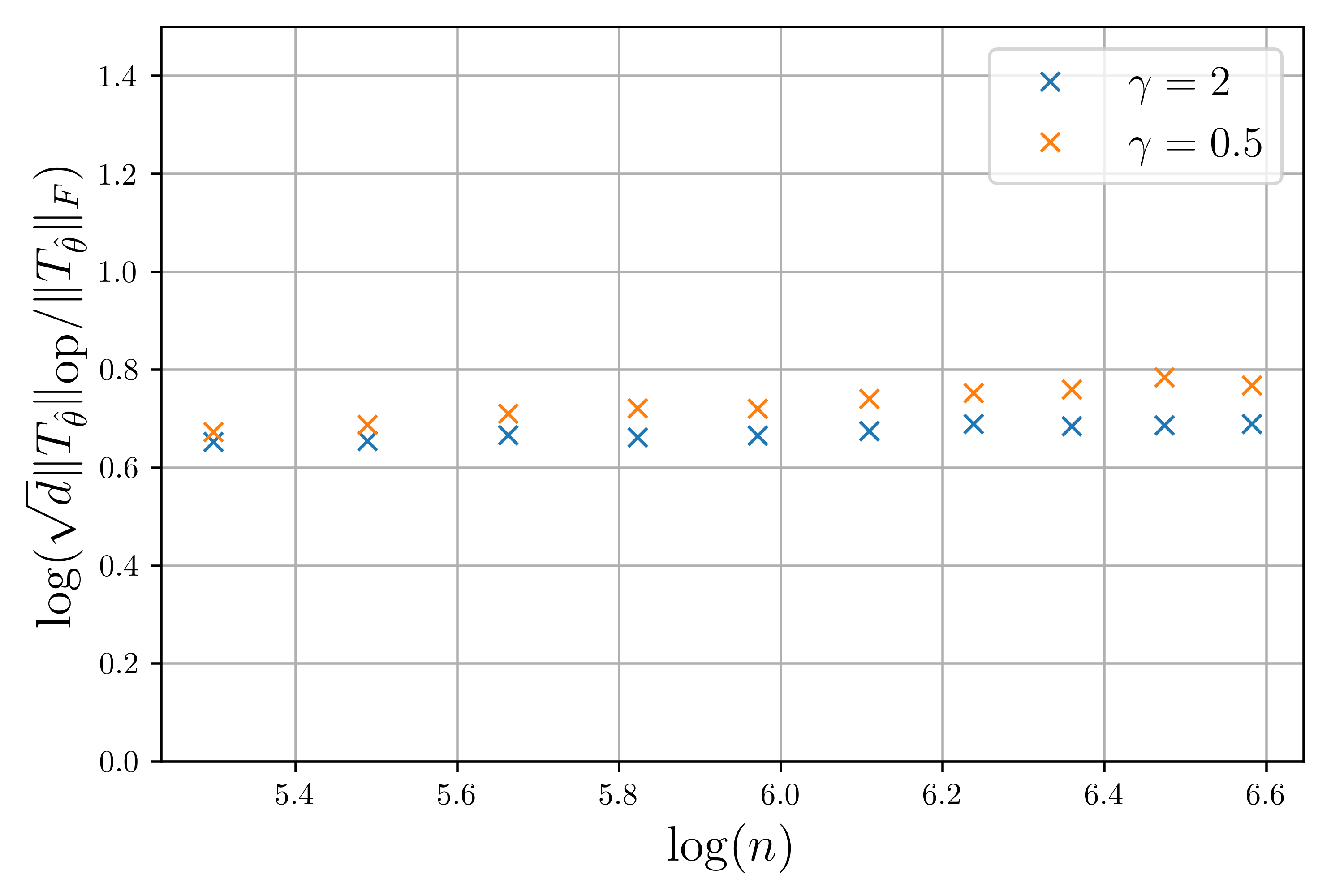}
\includegraphics[scale=0.46]{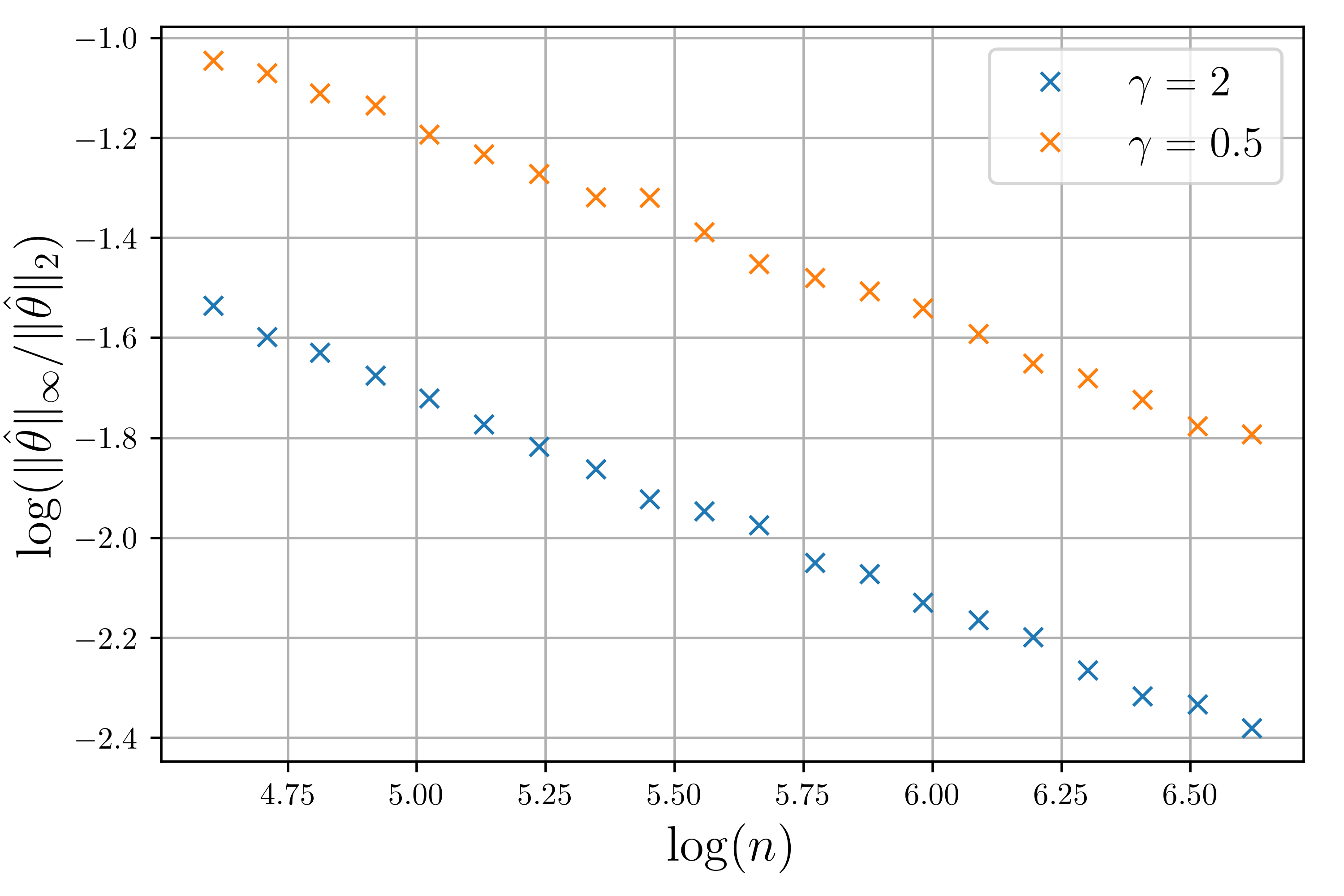}
\caption{
The value of constraints defining the domain of Gaussianity, evaluated at the minimizer:
The figure follows the setting of Example~\ref{ex:1} and Figure~\ref{fig:univ_sim}, where the risk is given in Eq.~\eqref{eq:erm_numerical}.
The noise standard deviation is  $\nu = 0.1$.
Left: 
$\bphi=\bphi_{\NT}$ defined in Section~\ref{section:ntk_example}
 with latent dimension $m = d=\sqrt{\sgamma n}$, regularization $\lambda = 0.02$.
Right: $\bphi=\bphi_{\RF}$ defined in Section~\ref{section:rf_example}, with latent dimension $d= p = \sgamma n,\lambda = 0.0002$.
 }
\label{fig:delocalization}
\end{figure}

\begin{example}[Numerical example]
\label{ex:1}
Figure~\ref{fig:univ_sim} demonstrates universality in a statistical learning
problem via numerical simulations.
We generate synthetic data $(\bz_i,y_i)$ with  $\bz_i \sim\cN(0,\bI_d)$ and 
\begin{align}
y_i = \varphi(\bbeta^{\star\sT}\bz_i + \eps_i)\,,\;\;\;\;\;
\varphi(t)= \begin{cases}
t & \mbox{if }t\in[-1,1]\, ,\\
\sign(t) & \mbox{otherwise,}
\end{cases}\label{eq:ExampleLabels}
\end{align}
for $\eps_i \sim \cN(0, \nu^2)$,
with $\eps_i$ independent of $\bz_i$. Here  $\bbeta^\star\in \reals^d$, 
$\|\bbeta^\star\|_2=1$ is an unknown 
vector of parameters.

Given $n$ data points $(\bz_i,y_i)$, $i\le n$, we generate feature vectors
$\bx_i = \bphi(\bz_i) \in\R^p$ using two different featurization maps:
$(a)$ the neural tangent map
$\bphi=\bphi_{\NT}$  defined in Section~\ref{section:ntk_example}
(with activation function $\sigma(t)=\tanh(t)$);
and $(b)$ the random features map $\bphi=\bphi_{\RF}$  defined in 
Eq.~\eqref{eq:FirstRF} and Section~\ref{section:rf_example}
(with activation function $\sigma(t)=\tanh(t)$).
We also consider Gaussian data $\bg_i$ with matched covariance.

In each case we fit the data by minimizing the empirical risk:
\begin{align}
\label{eq:erm_numerical}
\hR_n(\btheta;\bX,\by) =  \frac1n\sum_{i=1}^n(y_i - \sigma_{\ell}(\btheta^\sT\bx_i))^2  +
 \lambda \norm{\btheta}_2^2\, ,\;\;\;\; \btheta \in\R^p,
 \end{align}
where we take $\sigma_{\ell}(t) = \tanh(t)$.
Notice that the ERM problem is non-convex in the vector $\btheta$.
In each case we compute the train and test errors, and compare them with the train and test errors in a 
similar simulation within the Gaussian equivalent model, 
see  Sections \ref{section:ntk_example}, \ref{section:rf_example}.
The agreement between the Gaussian and non-Gaussian models is excellent. 

Theorem \ref{thm:MainFirst} predicts that the mechanism driving universality is that the empirical risk minimizer lies in the domain of Gaussianity $\cS_p$.
Figure~\ref{fig:delocalization} checks that this holds 
in the present example, even when the minimization is unconstrained. 
Again, we consider the neural tangent features of Section \ref{section:ntk_example} (left frame),
and the random features of Sections \ref{section:rf_example} (right frame).
The relevant
\emph{domain of Gaussianity} criterion is $\|\btheta\|_\infty/\|\btheta\|_2 = p^{-\alpha}$ for $\alpha >0$ for random features, 
and $\|\bT_\btheta\|_{\op}/(\|\bT_{\btheta}\|_{F}\sqrt{d}) =O(1)$
for neural tangent features. (See Section \ref{section:ntk_example} for the definition of $\bT_{\btheta}$.)
\end{example}

\begin{remark}
In examples such as the last one, it is more reasonable to assume
labels $y_i = \varphi(\bU^{\star\sT} \bz_i,\eps_i)$
for some $\bU^\star \in \R^{d\times \sfk^\star}$,  instead of Eq.~\eqref{eq:form_yi}
(Labels depend on the underlying covariates, not on their featurization).
 The analysis of the distribution $y_i = \varphi(\bU^{\star\sT} \bz_i,\eps_i)$
can be reduced to the setting of Eq.~\eqref{eq:form_yi}, by augmenting the features
$\bphi(\bz_i)$ via   $\bx_i := (\bphi(\bz_i),\bz_i) \in\R^{p+d}$. Empirical risk is then
minimized over $\{(\btheta,\bzero) \in \R^{p+d}, \btheta\in\cC_p\}$ in the case $\sfk=1$, and the obvious
generalization for $\sfk>1$.
\end{remark}

Let us emphasize that universality is a very different phenomenon from uniform concentration 
of the empirical risk around its expectation, which is the main 
mathematical tool in statistical learning theory \cite{van2000asymptotic,shalev2014understanding}. 
Indeed, as illustrated very clearly by Figure \ref{fig:univ_sim},
universality holds in a high-dimensional regime in which  
test error and train error 
do not match. Establishing universality requires understanding the dependence of the empirical risk minimizer $\hbTheta_n^\bX$ on the data $\bX,\by$, 
and requires a  more refined analysis than uniform concentration.

\subsection{Universality of training error}

Our main result about the training error $\hR^{\star}_n(\bX,\by)$ of Eq.~\eqref{eq:min_problem}
is a direct consequence of Theorem \ref{thm:MainFirst}. However, we state one additional assumption
(in alternative to Assumption \ref{ass:loss-0}) to cover the important case of binary labels.
\begin{numassumption}[5']
\label{ass:loss-0-bis}
In alternative to Assumption \ref{ass:loss-0}, assume that $L_F(\bv,y) :=L(F(\bv),y)$
is non-negative and satisfies for all $\bv,\widetilde\bv\in\R^{\sfk}$, $y,\widetilde y \in\R$,
\begin{align*}
 &\left|L_F(\bv, y ) - L_F(\widetilde \bv ,y)\right| \le \sfK \left(1 + |y| \right)\norm{\bv - \widetilde \bv}_2 \\
 &\left|L_F(\bv, y ) - L_F(\bv,\widetilde y )\right| \le \sfK \left(1 + \norm{\bv}_2\right) |y- \widetilde y|.
\end{align*}
Further,  labels are binary: $y_i\in\{+1,-1\}$ with
\begin{equation}
 \P\left(y_i= +1 | \bx_i\right) = g\left(\bTheta_{\star}^{\sT}\bx_i\right)\,
\end{equation}
for some $g:\reals^{\sfk^{\star}}\to [0,1]$ satisfying for $\bv,\widetilde\bv \in\R^{\sfk^\star}$
\begin{equation}
 |g(\bv) - g(\widetilde \bv)| \le \sfK( 1 + \norm{\bv}_2 + \norm{\widetilde\bv}_2) \norm{\bv - \widetilde \bv}_2.
\end{equation}
\end{numassumption}

\begin{theorem}
\label{thm:main_empirical_univ}
Let $\{(y_i(\bX),\bx_i): i\le n\}$ be i.i.d. pairs with $\bx_i\in\reals^{p}$ and $y_i = y_i(\bX)$
given by Eq.~\eqref{eq:form_yi}. Similarly, define $\{(y_i(\bG),\bg_i): i\le n\}$
with $\bg_i\sim\cN(\bmu_\bg,\bSigma_\bg)$ as per Assumption \ref{ass:X}.
Finally, let $\cS_p$ be a domain of Gaussianity for the distribution of $\bx$ as in Definition~\ref{def:DoG}.

Suppose that Assumptions~\ref{ass:regime}-\ref{ass:X} hold.
 Suppose that either
$\ell(\bv,\bv^\star;\eps)$ defined in Eq.~\eqref{eq:ell-def} satisfies 
Assumption \ref{ass:loss-0} (with the $\eps_i$ uniformly sub-Gaussian) or $L_F(\bv,y) :=L(F(\bv),y)$
satisfies Assumption~\hyperref[ass:loss-0-bis]{5'}.

If $\cC_{j,p}\subseteq \cS_p$ and $\btheta^\star_j \in \cS_p$ for each $j\le \sfk^{\star}$,
then, for any
bounded Lipschitz function $\psi:\R\to\R$,
\begin{equation}
\label{eq:main_thm_eq}
 \lim_{n\to\infty}\left|\E\left[\psi\left(\widehat 
R^\star_n\left(\bX,\by(\bX)\right)\right)\right] - 
\E\left[\psi\left(\widehat R^\star_n\left(\bG,\by(\bG)\right) \right)\right]\right| = 0.
\end{equation}
Hence, for any constant $\rho \in \R$ and $\delta>0$, we have
\begin{align}
   \limsup_{n\to\infty}\P\left(
   \widehat 
R^\star_n\left(\bX,\by(\bX)\right) \ge  \rho + \delta
   \right)  &\le 
   \limsup_{n\to\infty}\P\left( 
\widehat R^\star_n\left(\bG,\by(\bG)\right)  \ge \rho 
\right), \textrm{ and }\nonumber\\
   \limsup_{n\to\infty}\P\left(
   \widehat 
R^\star_n\left(\bX,\by(\bX)\right) \le  \rho - \delta
   \right)  &\le 
   \limsup_{n\to\infty}\P\left( 
\widehat R^\star_n\left(\bG,\by(\bG)\right)  \le \rho 
\right)
\label{eq:prob_bounds}
\end{align}
and similarly for their limit inferiors,
whence, for all $\rho \in\R$,
\begin{equation}
\label{eq:P_limit_iff}
\hR^\star_n(\bX,\by(\bX))\stackrel{\P}{\to} 
\rho \quad\textrm{if and only if} \quad \widehat R^\star_n(\bG,\by(\bG)) \stackrel{\P}{\to}\rho.
\end{equation}
\end{theorem}

\begin{remark}
As discussed above, an important motivation for our work is provided by models in which the feature
vectors are obtained by applying a featurization map to some latent covariates $\bx_i=\bphi(\bz_i)$.
However, none of Theorems \ref{thm:MainFirst}, \ref{thm:main_empirical_univ},  \ref{thm:universality_bounds}, \ref{thm:test_error} assumes such a setting, and instead assumptions are directly placed on the distribution of $\bx_i$. In Sections \ref{section:ntk_example} and \ref{section:rf_example} we will use these theorems in the context of specific featurization maps.
\end{remark}

\subsection{Universality of the test error}

 The test error is the expectation of the error that a model incurs on a fresh sample:
\begin{align*}
R_n^\bx(\bTheta) :=\E\Big[ \ell\big(\bTheta^\sT \bx, \bTheta^{\star\sT}\bx;\eps\big) \Big]\, ,\;\;\;\;\;
R_n^\bg(\bTheta) :
= \E\Big[ \ell\big(\bTheta^\sT \bg; \bTheta^{\star\sT}\bg,\eps\big) \Big].
\end{align*}
The first expectation is with respect to independent random variables $\bx\sim\P_{\bx}$ and  $\eps\sim\P_{\eps}$,
and the second with respect to independent $\bg\sim\normal(\bmu_\bg,\bSigma_{\bg})$ and  $\eps\sim\P_{\eps}$.
By definition of $\cS_p$, it is easy to see that
$\lim_{n\to\infty}|R_n^\bx(\bTheta)-R_n^\bg(\bTheta)|=0$ at a \emph{fixed} $\bTheta\in \cS_p$.
Here however we are interested in comparing the two at near minimizers
of the respective ERM problems.

We will state two  theorems that provide sufficient conditions for universality of the test
error: Theorem \ref{thm:universality_bounds}  assumes an overparameterized scenario in
which the train error is not locally convex around the minimizer;  
Theorem \ref{thm:test_error} is instead relevant for cases in which the empirical risk is locally convex.

We begin by considering the overparameterized setting because it is very different from 
earlier results, and because of its relevance to deep learning \cite{bartlett2021deep}.

We define the set of approximate empirical risk minimizers:
\begin{align}
\ERM_{t}(\bX):=\big\{\bTheta \in \cC_p^\sfk  \mbox{ s.t. }
 \hR_n(\bTheta;\bX,\by(\bX)) \le t \big\}\, .\label{eq:ERM_Def}
 \end{align}	
In statistical learning jargon, a problem is overparameterized (or `overfits' the data) 
if $\ERM_{t}(\bX)$ is non-empty (and potentially very large) for $t$ small or $t=0$, even if the 
data $(y_i,\bx_i)$ are noisy (i.e. $\Var(y_i|\bx_i)$ is bounded away from $0$).  
A model $\bTheta\in \ERM_{0}(\bX)$ is referred to as an `interpolator' because $L(f(\bz_i;\bTheta),y_i) =0$,
which implies (if $L(\hy,y)=0$ only for $\hy=y$) $f(\bz_i;\bTheta) = y_i$ 
for all $i\le n$. 
\begin{theorem}
\label{thm:universality_bounds}
Assume $\lim_{n\to\infty}\P(\ERM_0(\bG)\neq\emptyset) = 1$.
Then under the assumptions of Theorem \ref{thm:main_empirical_univ},
for all $\delta >0,\alpha>0$ and $\rho\in\R$ we have
\begin{align}
\nonumber
\limsup_{n\to\infty}\P\Big(
 \min_{\bTheta\in {\ERM}_{\alpha}(\bX)} R_n^\bx(\bTheta) \ge \rho + \delta \Big) 
&\le 
\limsup_{n\to\infty}\P\Big( \min_{\bTheta\in\ERM_0(\bG)}R_n^\bg(\bTheta) 
> \rho  \Big), \textrm{ and }\\
\limsup_{t\to 0}\limsup_{n\to\infty}\P\Big(
 \min_{\bTheta\in\ERM_t(\bX)}R_n^\bx(\bTheta) 
\le  \rho - \delta \Big) 
&\le 
\limsup_{n\to\infty}\P\Big(
 \min_{\bTheta\in\ERM_{\alpha}(\bG)}
R_n^\bg(\bTheta) 
\le \rho \Big),
\nonumber
\end{align}
and similarly for the limit inferiors in $n$.
\end{theorem}
In words, Theorem \ref{thm:universality_bounds}  establishes
that the minimum test error over all near-interpolators is universal
under certain regularity conditions on the test error in the associated Gaussian model.
A specific formalization of this remark is provided by the next corollary.
\begin{corollary}\label{coro:UnivInterpolation}
Assume $\lim_{n\to\infty}\P(\ERM_0(\bG)\neq\emptyset) = 1$, and that
the assumptions of Theorem \ref{thm:main_empirical_univ} hold. 
Further assume that the following limit exists for $t\in [0,t_0]$
with $t_0$ a small enough constant (recall that $\plim$ denotes limit in probability):
\begin{align}
\plim_{n\to\infty}
 \min_{\bTheta\in\ERM_t(\bG)}R_n^\bg(\bTheta) & = \rho_{\stst}(t)\, .
 \label{eq:MinimalInterpolation-Assumption}
 \end{align}
 
 Then, for any $0\le t< t_+\le t_0$ 
  (note that $t\mapsto \rho(t)$ is non-increasing by construction),
\begin{align}
\rho_{\stst}(t_+)\le  \plim\inf_{n\to\infty}\min_{\bTheta\in\ERM_t(\bX)}R_n^\bx(\bTheta) \le 
\plim\sup_{n\to\infty}\min_{\bTheta\in\ERM_0(\bX)}R_n^\bx(\bTheta)  \le \rho_{\stst}(0)\, .
\end{align}
In particular, if $\rho$ is continuous at $t=0$, we  have
\begin{align}
\plim_{n\to\infty}\min_{\bTheta\in\ERM_0(\bX)}R_n^\bx(\bTheta)=\rho_{\stst}(0)\, . \label{eq:LimitMinError}
\end{align}
  \end{corollary}
The existence of the limit \eqref{eq:MinimalInterpolation-Assumption} and the continuity of $t\mapsto\rho_{\stst}(t)$
can be checked in specific models. For instance, the techniques of \cite{hastie2019surprises,montanari2019generalization} can be used to verify them ---under technical assumptions---
for overparameterized linear regression and binary classification. 

More generally, we note that applying this corollary only requires to control the asymptotics of the equivalent Gaussian model.

\begin{remark}
Statements analogous to Theorem \ref{thm:universality_bounds} and Corollary \ref{coro:UnivInterpolation} hold
(with identical proof) for the maximum test error over near interpolators,
and if the level $0$ is replaced with any deterministic constant $t_\star$
(i.e., $\ERM_0$ can be replaced by $\ERM_{t_{\star}}$).

\emph{In particular, the analogue of Eq.~\eqref{eq:LimitMinError} 
(i.e. $\plim_{n\to\infty}\min_{\bTheta\in\ERM_{t_{\star}}(\bX)}R_n^\bx(\bTheta)=\rho_{\stst}(t_\star)$)
can be proven if $\rho$ is continuous at $t_{\star}$.}
\end{remark}
  
The next theorem provides alternative sufficient conditions 
that guarantee the universality of the test error. We emphasize that these
are conditions on the Gaussian features model $\bG$ and it is therefore possible to check them
using existing techniques. 
\begin{theorem}
\label{thm:test_error}
Suppose one of the following holds:
\begin{enumerate}[(a)]
 \item
\label{item:a_test_error} 
 The loss $\ell(\;\cdot\;,\bv^\star, v)$ is convex for fixed $\bv^\star,v$, the regularizer $r$ is $\smu$-strongly convex for some fixed 
 constant $\smu>0$ and
 $\cC_p \subseteq\cS_p$ is given by
 $\cC_p = \{\btheta\in\R^p : h(\btheta)\le L\}$
 for some convex $h$ and $L\in \R$.
 Furthermore, we have
 for some $\rho,\widetilde\rho\in\R$
\begin{equation}
\nonumber
\widehat R_n^\star\left(\bG,\by(\bG)\right) \stackrel{\P}{\to} \rho,\quad
R_n^\bg\left(\widehat \bTheta_n^\bG\right) \stackrel{\P}{\to} \widetilde\rho\, ;
\end{equation}
\item
\label{item:b_test_error} 
For some $\rho,\widetilde \rho\in\R$, letting~ 
$$\cU_p(\widetilde\rho,\alpha):=\{\bTheta \in \cC_p^\sfk   : \;| R_n^\bg(\bTheta) - 
\widetilde \rho |\ge \alpha\},$$
we have
$\widehat R_n^\star(\bG,\by(\bG)) \stackrel{\P}{\to} \rho$,
and for all $\alpha >0$, there exists $\delta>0$ so that 
 \begin{equation}  
\nonumber
\lim_{n\to\infty}\P\Big(\min_
{\bTheta\in \cU_p(\widetilde\rho,\alpha)}
|  \widehat R_n(\bTheta;\bG,\by(\bG)) - \widehat 
R_n^\star(\bG,\by(\bG)) | \ge \delta\Big)
= 1;
\end{equation}
\item
\label{item:c_test_error} 
There exists a function $\rho(s)$ differentiable at $s=0$ such that for all $s$ in a 
neighborhood of $0$,
\begin{equation}
\label{eq:DiffCondition}
\min_{\bTheta\in\cC_p^\sfk} \Big\{\widehat R_n(\bTheta; \bG,\by(\bG)) + s R_n^\bg(\bTheta) 
\Big\}
\stackrel{\P}\to \rho(s).
\end{equation}
\end{enumerate}
Then,  under the assumptions of Theorem \ref{thm:main_empirical_univ},
\begin{equation}
\nonumber
 \left| R_n^\bx\left(\widehat\bTheta_n^\bX\right) - 
        R_n^\bg\left(\widehat\bTheta_n^\bG\right)\right| 
\stackrel{\P}{\to} 0
\end{equation}
for any minimizers 
$\widehat\bTheta_n^\bX, \widehat\bTheta_n^\bG$ 
of $\widehat R_n(\bTheta;\bX,\by(\bX))$, $\widehat R_n(\bTheta;\bG,\by(\bG))$, respectively.
\end{theorem}

\paragraph*{Proof technique} 
The proofs of Theorems \ref{thm:universality_bounds} and \ref{thm:test_error} 
are given in Sections~\ref{section:outline_proof_univerality_bounds} and~\ref{section:outline_proof_test_error} respectively.
The basic technique can be gleaned from condition \eqref{eq:DiffCondition}:
We perturb the train error by a term linear in $s$, proportional to the test error 
(this is only a proof device, not an actual algorithm). Then, the test error can be related
to the derivative with respect to $s$ of the resulting minimum value.
Since the minimum value
is universal by our results in the previous section, the technical challenge is to control its derivative.

%
%
\section{Some feature distributions and their domain of Gaussianity}
\label{sec:Pointwise}

In this section we study some concrete examples for the distribution of the feature vectors 
$\bx_i$. In each case, we characterize the set of parameter vectors $\cS_p$
for which the pointwise Gaussianity condition of Eq.~\eqref{eq:condition_bounded_lipschitz_single}
holds. 
For simplicity of exposition, we use $\sfk = \sfk^\star = 1$ throughout this section.

We first consider examples of featurization maps from the deep learning literature.

\noindent\emph{Neural tangent features.}
Section~\ref{section:ntk_example} analyzes the featurization map that
is obtained by linearizing a two-layer neural network around a random initialization.
We establish
asymptotic equivalence (in distributional sense) of ERM under the neural tangent 
model, to ERM under the Gaussian model with matching covariance structure.  
Comparable universality results were not known before our work, even
for convex losses. 

\noindent\emph{Random features.} In 
Section \ref{section:rf_example}, we consider the  featurization map that is 
obtained by applying a one-layer network with
random weights. This is equivalent to the `random features' model of \cite{rahimi2007random}.
Pointwise normality (along the lines of Eq.~\eqref{eq:condition_bounded_lipschitz_single})
and universality of the expected risk at a fixed $\bTheta$ for this model was first shown 
in~\cite{goldt2020gaussian}. 
Universality of test and train error for ridge regression was established in 
\cite{mei2019generalization}, while~\cite{hu2020universality} proved universality of the ERM 
for strongly convex losses. 
Finally, \cite{loureiro2021learning} presented empirical evidence and conjectured 
that universality holds for a wide class of such featurization maps and loss functions.

Our main results (Theorems \ref{thm:MainFirst} and \ref{thm:main_empirical_univ}) 
allow us to generalize the results of
\cite{hu2020universality} to non-convex losses. 

\noindent\emph{Linear functions of vectors with independent entries.} Finally, in  Section~\ref{section:linear_example}, we consider the case in which $\bs_i=\bSigma^{1/2}\obx_i$,
where $\obx_i$
has i.i.d. entries. This is a standard model in random matrix theory
but  was studied in the past mostly for  convex losses
~\cite{montanari2017universality,panahi2017universal,oymak2018universality}.
The only exception\footnote{After the present manuscript was posted, \cite{han2022universality} 
also analyzed non-convex losses with  $\bx_i$ having i.i.d. entries, using the same technique of \cite{korada2011applications}.}  is provided by~\cite{korada2011applications} which studies certain
non-convex losses when $\bSigma=\id$.

\subsection{Neural tangent features}
\label{section:ntk_example}
Consider a two layer neural network with $m$ hidden neurons 
$f(\bz;\bu_1,\dots,\bu_m) :=\sum_{i=1}^m a_i\sigma(\<\bu_i,\bz\>)$, with input $\bz\in\R^d$. 
We assume that the second layer weights $a_i$ are not trained, since they play a secondary role.
Under suitable training conditions \cite{du2019gradient}, such a network is well approximated by a model of the form \eqref{eq:FzFirst} with 
\begin{equation}
\label{eq:ntk_covariates}
    \bphi_\NT(\bz) := \left(\sigma'(\bw_1^\sT \bz)\bz^\sT,\dots,\sigma'(\bw_m^\sT \bz)\bz^\sT\right)^\sT \in\R^p\,,
\end{equation}
where $p=md$, and
$\bw_i$ are the first layer weights at a random initialization 
of gradient descent 
$\bu_i^{0}:=\bw_i$ on ERM with the model $f(\bz;\bu_1,\dots,\bu_m)$.
As in the rest of the paper, we assume to be given training samples 
$\{(y_i,\bz_i)\}_{i\le n}$ and to compute feature vectors $\bx_i=\bphi_\NT(\bz_i)$.

Here we are not concerned with the connection between the original neural network
 and its neural tangent model, for which we refer to the 
 literature 
\cite{jacot2018neural,du2019gradient,lee2019wide,bartlett2021deep,montanari2020interpolation}.
 We will instead focus on the neural tangent model, and show that it can be 
 approximated by an equivalent Gaussian features model.  

We  assume a simple covariates distribution
$\{\bz_i\}_{i\le n}\simiid\cN(0,\bI_d)$ and a standard  network initialization:
$\{\bw_j\}_{j\le m}\simiid\Unif\left(\S^{d-1}(1)\right)$, i.e.,
$\bw_j$ are uniformly distributed on the sphere of radius $1$ in $\R^d$.
\begin{itemize}
\item[$(i)$]~The weights $\bw_j$ are fixed and do not change from sample to sample; 
\item[$(ii)$]~The distribution of  vectors $\bx_i$ 
(the push-forward of $\cN(0,\bI_d)$ via $\bz   \mapsto \bphi_\NT(\bz)$) is highly non-trivial and 
they have dependent entries. In fact, this distribution is supported 
on a $d$-dimensional nonlinear manifold in $\R^p$, with $d\ll p$. In particular, it
is mutually singular with the Gaussian distribution we will compare it to.
\end{itemize}

For $\btheta =(\btheta_{\up{1}}^\sT,\dots,\btheta_\up{m}^\sT)^\sT\in\R^p$, where $\btheta_{\up{j}} 
\in \R^d$ for $j\in[m]$, let $\bT_\btheta\in\R^{d\times m}$ be the matrix 
$\bT_\btheta = \left(\btheta_\up{1} ,\dots, \btheta_\up{m}\right)$,
so that $\btheta^\sT \bx = \bz^\sT \bT_\btheta \sigma'(\bW^\sT \bz)$, where 
    $\bW= (\bw_1,\dots,\bw_m)$
and $\sigma':\R\to\R$ is applied entrywise. We define, for $p\in\Z_{>0}$,
\begin{equation}
    \label{eq:ntk_set}
    \cS_{p} 
    := \left\{ 
    \btheta\in\R^p : \norm{\bT_\btheta}_\op \le
\frac{ \sfR}{\sqrt{d}}
    \right\}.
\end{equation}

We have the following universality result for the neural tangent model~\eqref{eq:ntk_covariates}.
\begin{theorem}
\label{cor:ntk_universality}
Let $\bW \in\R^{d\times m}$ have columns 
$$\{\bw_j\}_{j\le m}\simiid\Unif(\S^{d-1}(1)),$$
$\{\bz_i\}_{i \le n}\simiid\cN(0,\bI_d)$, 
and $\sigma$ four times differentiable satisfying
 $$\E[\sigma'(G)] = 0,\qquad \E[G\sigma'(G)] = 0
\qquad \textrm{for}\quad G\sim\cN(0,1).$$

Let $\bx_i=\bphi_\NT(\bz_i)\in\R^p$ as per Eq.~\eqref{eq:ntk_covariates} so that $p = m d$, with $m= m(p),d = d(p)\in\Z_{>0}$ so that  $\lim_{p\to\infty}m(p)/d(p) = \widetilde\sgamma_\NT$
for some fixed $\widetilde\sgamma_\NT\in(0,\infty)$.
Let 
$$\bg_i|\bW \simiid\cN(0,\bSigma_\bW)
\qquad
\textrm{for}\quad \bSigma_\bW := \E\big[\bx\bx^\sT| \bW\big].$$
Finally, let~$\cS_p$ be as defined in~\eqref{eq:ntk_set}.

Suppose Assumptions~\ref{ass:regime}-\ref{ass:regularizer} hold 
(in particular $p(n)/n\in [\sfC^{-1},\sfC]$)
and that $\cC_p \subseteq \cS_p.$ Further suppose that either
$\ell(\bv,\bv^\star;\eps)$ defined in Eq.~\eqref{eq:ell-def} satisfies 
Assumption \ref{ass:loss-0} (with the $\eps_i$ uniformly sub-Gaussian) or $L_F(\bv,y) :=L(F(\bv),y)$
satisfies 
Assumption~\hyperref[ass:loss-0-bis]{5'}.
Then the following hold:
\begin{enumerate}
\item[$(a)$] For any bounded Lipschitz function $\psi:\R\to\R$,
Eq.~\eqref{eq:main_thm_eq} holds (without conditioning on $\bW$). In particular, as a consequence, 
\begin{equation}
\label{eq:P_limit_iff_}
\widehat  R^\star_n(\bX,\by(\bX))\stackrel{\P}{\to} 
\rho \quad\textrm{if and only if} \quad \widehat R^\star_n(\bG,\by(\bG)) \stackrel{\P}{\to}\rho.
\end{equation}
\item[$(b)$] Under the additional conditions of Theorem \ref{thm:universality_bounds},
Corollary \ref{coro:UnivInterpolation} or Theorem
\ref{thm:test_error}, the universality results for the test error stated there hold. 
\end{enumerate}
\end{theorem}
The specific conditions on the activation $\sigma$ 
are chosen because they allow us to simplify some of the estimates in the proof. We 
defer relaxing them to future work.
\begin{remark}
Theorem \ref{cor:ntk_universality} does not hold if we relax the set $\cS_p$ to
$\cS_p:=B_2^p(\sfR)$.
Indeed, for $\bT_\btheta = \sfR/\sqrt{d}\left(\one_d, 0,\dots,0 \right)$, the random 
variable $\btheta^\sT\bx = \bz^\sT \bT_\btheta \sigma'(\bW^\sT\bz)$ is not asymptotically 
Gaussian. Clearly, this choice of $\btheta$ is not in the set defined in~\eqref{eq:ntk_set}.
\end{remark}

\begin{remark}
    In Theorem \ref{cor:ntk_universality}, we assume the number of parameters $p=md$
    to be of the same order as the sample size $n$, which is the setting of our general results 
    Theorems \ref{thm:MainFirst}, \ref{thm:main_empirical_univ}. Within this general assumption, we focus on the case in which the input dimension is proportional to the number of hidden neurons 
     $\lim_{p\to\infty}m(p)/d(p) = \widetilde\sgamma_{\NT}$, and hence $n\asymp d^2$. This scaling
     is theoretically interesting, and used in practice, but not necessarily the only one in which universality holds.
\end{remark}

\paragraph*{Proof technique} 
We prove Theorem~\ref{cor:ntk_universality} in Section~\ref{proof:ntk_universality}
of the Appendix using 
Theorem~\ref{thm:main_empirical_univ}.
The key technical challenge is to establish that Assumption~\ref{ass:-1} for 
the distribution of the feature vectors $\bx_i = \bphi_\NT(\bz_i)$, cf. Eq.~\eqref{eq:ntk_covariates}. 
We outline the proof of this condition in Section~\ref{section:proof_outline_ntk} where
we make use of Stein's method as done in \cite{hu2020universality} for the random features model. 
However, treating the neural tangent features of Eq.~\eqref{eq:ntk_covariates}
requires extra care due to the more complex covariance structure.
%
%
\subsection{Random features} 
\label{section:rf_example}

Consider a two layer network with $p$ hidden neurons and fixed first layer
weights. Namely $f(\bz;\ba) :=\sum_{i=1}^p a_i\sigma(\<\bw_i,\bz\>)$, where 
the first-layer weights are random
$\{\bw_j\}_{j\le p}\simiid\Unif\left(\S^{d-1}(1)\right)$
and not fit to the data, while second layer weights $a_i$ are.
This is a linear model with parameters $a_i$ with respect to the features 
\begin{equation}
    \label{eq:rf_covariates}
    \bphi_\RF(\bz):= \left(\sigma\big(\bw_1^\sT \bz\big),\dots,\sigma\big(\bw_p^\sT\bz\big)\right)^\sT.
\end{equation}
As before, we consider $\{\bz_i\}_{i\le n}\simiid\cN(0,\bI_d)$ 
and $\bx_i=\bphi_\RF(\bz_i)$.
Finally, fix $\alpha>0$ and define for $p\in\Z_{>0}$
\begin{equation}
    \label{eq:rf_set}
    \cS_p :=  \big\{ \btheta\in\R^p:\; \|\btheta\|_{\infty}\le \sfR p^{-\alpha}\, ,
    \|\btheta\|_{2}\le \sfR\big\}\, .
\end{equation} 
We have the following corollary of Theorem~\ref{thm:main_empirical_univ}.

\begin{corollary}
\label{cor:rf_universality}
Let $\bW \in\R^{p\times d}$ with columns 
$$\{\bw_j\}_{j\le p}\simiid\Unif(\S^{d-1}(1))$$, $\{\bz_{i}\}_{i\le n}\simiid \cN(\bzero,\bI_d)$, 
where $d= d(p)$ satisfies $\lim_{p\to\infty}d(p)/p = \widetilde\sgamma_\RF$
for some fixed $\widetilde\sgamma_\RF\in(0,\infty)$,
and let
$\sigma$ be
three times continuously differentiable with bounded derivatives satisfying $\E\sigma(G) = 0$
for $G\sim \cN(0,1)$.

Let $\bx_i=\bphi_\RF(\bz_i)\in\R^p$ as per Eq.~\eqref{eq:rf_covariates},
and $\bg_i|\bW \simiid\cN(0,\bSigma_\bW)$
for $\bSigma_\bW := \E\left[\bx\bx^\sT| \bW\right]$. Finally, let~$\cS_p$ be as defined in~\eqref{eq:rf_set}.

Suppose Assumptions~\ref{ass:regime}-\ref{ass:regularizer} hold 
(in particular $p = p(n)\in [\sfC^{-1}n,\sfC n]$)
and that $\cC_p \subseteq \cS_p.$ Further suppose that either
$\ell(\bv,\bv^\star;\eps)$ defined in Eq.~\eqref{eq:ell-def} satisfies 
Assumption \ref{ass:loss-0} (with the $\eps_i$ uniformly sub-Gaussian) or $L_F(\bv,y) :=L(F(\bv),y)$
satisfies 
Assumption~\hyperref[ass:loss-0-bis]{5'}.

Then for any bounded Lipschitz function $\psi:\R\to\R$,
Eq.~\eqref{eq:main_thm_eq} holds (without conditioning on $\bW$) along with its consequences: Eq.~\eqref{eq:prob_bounds} and Eq.~\eqref{eq:P_limit_iff}.
\end{corollary}
We derive this corollary in Section~\ref{proof:rf_universality} of the Appendix
as a consequence of Theorem~\ref{thm:main_empirical_univ}.
To do so, we use a result established by~\cite{hu2020universality} implying that the
feature vectors $\bx_i$ satisfy Assumption~\ref{ass:X} and~\eqref{eq:condition_bounded_lipschitz_single} with respect to $\cS_p$ for every $\bW$ in a high probability set.

\subsection{Linear functions of vectors with independent entries}
\label{section:linear_example}
Consider feature vectors $\bx_i = \bSigma^{1/2}\obx_i\in\R^p$, where the vectors $\obx_i$ have $p$ i.i.d 
subgaussian entries of subgaussian norm bounded by $\sfK$ and unit variance. We assume $\|\bSigma\|_{\op}\le\sfK$.
Fix any deterministic sequence $\alpha_p$ such that $\lim_{p\to\infty} \alpha_p = 0$.
An application of the Lindeberg Central Limit Theorem shows that 
Eq.~\eqref{eq:condition_bounded_lipschitz_single} holds for
\begin{equation}
\label{eq:indep_set}
    \cS_p :=  \big\{ \btheta\in\R^p:\; \|\bSigma^{1/2}\btheta\|_{\infty}\le \alpha_p\, ,
    \|\btheta\|_{2}\le \sfR\big\}\, .
\end{equation}

We have therefore the following corollary of Theorem~\ref{thm:main_empirical_univ}.
\begin{corollary}
\label{cor:indep_universality}
Let $\bx_i = \bSigma^{1/2}\obx_i\in\R^p$ where $\obx_i$ has i.i.d. subgaussian entries with unit variance,  $\|\bSigma\|_{\op}\le\sfK$, $\bg_i \sim\cN(0,\bSigma)$,
and~$\cS_p$ be as defined in~\eqref{eq:indep_set}.

Suppose Assumptions~\ref{ass:regime}-\ref{ass:regularizer} hold 
(in particular $p = p(n)\in [\sfC^{-1}n,\sfC n]$)
and that $\cC_p \subseteq \cS_p.$ Further suppose that either
$\ell(\bv,\bv^\star;\eps)$ defined in Eq.~\eqref{eq:ell-def} satisfies 
Assumption \ref{ass:loss-0} (with the $\eps_i$ uniformly sub-Gaussian) or $L_F(\bv,y) :=L(F(\bv),y)$
satisfies 
Assumption~\hyperref[ass:loss-0-bis]{5'}.

Then,
for any bounded Lipschitz function $\psi:\R\to\R$,
Eq.~\eqref{eq:main_thm_eq} holds along with its consequences: Eq.~\eqref{eq:prob_bounds} and Eq.~\eqref{eq:P_limit_iff}.
\end{corollary}

%
%

\section{Verifying that a minimizer is in the domain of Gaussianity}
\label{sec:Deloc}

Note that our main universality results (Theorems \ref{thm:MainFirst},
\ref{thm:main_empirical_univ}, \ref{thm:universality_bounds}, \ref{thm:test_error})
can be applied to cases in which the constraint set $\cC_p$ 
is not a subset of the domain of Gaussianity $\cS_p$, provided we can prove that 
the empirical risk minimizer lies with high probability in such a subset.

In this section we present two classes of statistical learning problems for which 
we prove universality without artificially constraining $\cC_p$ to be a subset of $\cS_p$:
\begin{itemize}
\item In Section~\ref{sec:Controlling} we show this for a broad class of overparameterized models.
\item In Section \ref{sec:M_est} we consider the case $\sfk=1$ and give sufficient conditions for this to hold
when $(i)$~the vectors $\bx_i$ have i.i.d. coordinates and 
$(ii)$~the regularizer $r(\btheta)$ is separable.
We give an example where these conditions can be checked via Gaussian comparison inequalities.
\end{itemize}

\subsection{Overparameterized empirical risk minimization}
\label{sec:Controlling}

For overparameterized models, we will prove that,
if there exists a global empirical risk minimizer with
controlled $\ell_2$ norm (a condition that is relatively easy to check),
then there exists also an empirical risk minimizer with controlled $\ell_{\infty}$ norm.
%

\begin{proposition}
\label{prop:overparameterized_iid}
Consider the ERM problem of Eq.~\eqref{eq:min_problem} with risk function defined by
Eqs.~\eqref{eq:SecondERM}, $\cC_p=\reals^p$ and $r(\bTheta)=0$.
Assume $p(n)/n \in( (1+\delta), \sfC)$ for some constants $\sfC,\delta>0$ and
that either one of the following feature distributions holds:
\begin{enumerate}
    \item $\bSigma^{-1/2}\bx_i$ have i.i.d., mean $0$, unit variance and subgaussian entries.
Furthermore, there exist a constant $K_0>0$ 
such that
\begin{align*}
&\big\|\bSigma^{-1/2}\big\|_{\infty\to\infty} :=\max_{i\le p}\|(\bSigma^{-1/2})_{i,\cdot}\|_1\le K_0,\\
&K_0^{-1} \le \sigma_{\min}(\bSigma^{-1/2}) \le \sigma_{\max}(\bSigma^{-1/2}) \le K_0.
\end{align*}
Recall that a domain of Gaussianity $\cS_p$ is given by Eq.~\eqref{eq:indep_set}.
\item $\bX$ follow the random features distributions of Section~\ref{section:rf_example} for any $\tilde\sgamma_\RF$. 
Recall that $\cS_p$ for this feature distribution was given in Eq.~\eqref{eq:rf_set}.
\end{enumerate}

If there exists $K_1>0$, such that
\begin{equation}
\label{eq:minimizer_whp}
 \lim_{n\to\infty}\P\left(
 \exists
 \widehat\bTheta_n \in\argmin_{\bTheta\in\R^{p\times k}} \hR_n(\bTheta;\bX,\by)
  : \big\|\widehat\bTheta_n\big\|_F \le K_1
 \right) = 1\, ,
\end{equation}
then for any $\alpha < 1/8$, there exists $C>0$ such that
\begin{equation}
\nonumber
\lim_{n\to\infty}\P\left(\exists \widehat \bU_n \in \argmin_{\bTheta\in\R^{p\times k}} 
 \hR_n(\bTheta;\bX,\by): \big\|\widehat\bU_n\big\|_F \le (C+1)K_1, \big\|\widehat\bU_n\big\|_\infty \le K_1 p^{-\alpha} \right) = 1.
\end{equation}
That is, with high probability, there exists a minimizer in $\cS_p^\sfk.$
In particular, under either Assumptions
 \ref{ass:loss-0} (with the $\eps_i$ uniformly sub-Gaussian) or Assumption~\hyperref[ass:loss-0-bis]{5'}, $\bTheta^\star\in\cS_p^{\sfk}$ and the subgaussian condition of Eq.~\eqref{eq:sg_assump},
\begin{equation}
\widehat  R^\star_n(\bX,\by(\bX))\stackrel{\P}{\to}
0 \quad\textrm{if and only if} \quad \widehat R^\star_n(\bG,\by(\bG)) \stackrel{\P}{\to}0,
\end{equation}
where $\widehat R_n^\star$ is the optimum of the unconstrained ERM problem.
\end{proposition}
We give the proof of this result in Section~\ref{section:outline_overparameterized_iid} of the Appendix. 
The main idea behind the proof is 
that in the overparameterized setting,
for a minimizer $\hat\bTheta$, there exists infinitely many
$\bTheta'$ satisfying $\bX\hat\bTheta = \bX\bTheta'$, and hence these must be minimizers as well. If $\hat\bTheta$ is of bounded $\ell_2$ norm, then the minimizer of minimum norm will be in the domain of Gaussianity (which corresponds to the set of delocalized vectors for both feature distributions that the proposition treats).
The following application is a corollary.
\begin{example}
\label{ex:non-convex-nn}
Consider a $2$-layer network
with two hidden layers of width $p$ and $\sfk$:
\begin{align}
f(\bz;\bTheta)= \ba^{\sT} \sigma_2\big(\bTheta^{\sT}\sigma_1(\bW^{\sT}\bz)\big)\,  .
\label{eq:3-layer-example}
\end{align}
Here we denoted by $\bW\in\reals^{d\times p}$ the first-layer weights,
by $\bTheta\in\reals^{p\times \sfk}$ the second layer weights, and by 
$\ba$ the output layer weights.
The activation functions $\sigma_1,\sigma_2$ are understood to act entrywise.

Consider a learning procedure in which the first and last layers $\ba, \bW$ are
not learnt from data, and we learn $\bTheta$ by minimizing the logistic loss for
binary labels $y_i\in\{+1,-1\}$: 
\begin{align}
\hR_n(\bTheta;\bX,\by) = \frac{1}{n}\sum_{i=1}^n
\log\Big\{1+\exp\Big[-y_i 
f(\bz_i;\bTheta)
\Big]\Big\} \, .
\label{eq:LogisticExample}
\end{align}

We note in passing that the model of Eq.~\eqref{eq:3-layer-example} (with the
first and last layer fixed) is not an unreasonable one.
If $\bW$ is random, 
for instance with i.i.d. columns $\bw_i\sim\normal(0,c_d\id_d)$, the
first layer performs a random features map in the sense of Rahimi and Recht~\cite{rahimi2007random}. 
In other words, this layer
embeds the data in the reproducing-kernel Hilbert space (RKHS)
with (finite width) kernel $H_p(\bz_1,\bz_2):= p^{-1}\sum_{i=1}^p
\sigma_1(\<\bw_i,\bz_1\>)\sigma_1(\<\bw_i,\bz_2\>)$, which approximates the kernel $H_{\infty}(\bz_1,\bz_2)=
\E_{\bw}[\sigma_1(\<\bw,\bz_1\>)\sigma_1(\<\bw,\bz_2\>)]$.
Fixing the last layer weights $\ba$ is not a significant reduction
of expressivity, since this layer only comprises $\sfk$ parameters, while we are fitting 
the $p\sfk\gg \sfk$ parameters in $\bTheta$.

Consider the overparameterized regime where $p> n(1+\delta)$ for some $\delta>0$.
Assume $\sigma_1:\R\to\R,\bW\in\R^{d \times p},\bz\in\R^{d}$ are as in Section~\ref{section:rf_example}, and let $y_i\in\{+1,-1\}$ be i.i.d. from the distribution
\begin{equation}
    \P(y_i = +1 | \bz_i) = g(\bz_i^\sT\bbeta^\star)
\end{equation}
for some $g$   as in Assumption~\hyperref[ass:loss-0-bis]{5'}, and some $\bbeta^{\star}\in\reals^d$ such that $\|\bbeta^\star\|_{\infty}\le K p^{-\alpha}$.
Finally, assume that $\sigma_2:\R\to\R$ is Lipschitz, and furthermore, that it is bounded and that the image of $(-\infty,\infty)$ under $\sigma_2$ is closed. In particular, this implies that there exists $a > 0$ such that $\sigma_2([-a,a]) = \sigma_2(\R).$

Under the above setting, we have as a corollary of Proposition~\ref{prop:overparameterized_iid}, that the unconstrained ERM problem
\begin{equation}
   \min_{\bTheta\in\R^{d\times \sfk}}  \widehat R_n(\bTheta;\bX,\by) 
\end{equation}
is universal. 
For details, see Section~\ref{proof:ex-non-convex-nn} of the Appendix.
\end{example}

\subsection{Regularized ERM with vector parameters}
\label{sec:M_est}
We consider the following form of the empirical risk function Eqs.~\eqref{eq:SecondERM}, 
\begin{align}
\hR_n(\btheta;\bX)
= \frac{1}{n}\sum_{i=1}^n\ell(\btheta^{\sT}\bx_i,\btheta^{\star\sT}\bx_i,\eps_i)+
\frac{1}{p}\sum_{i=1}^pr_0(\sqrt{p}\theta_i)\, ,
\end{align}
where $r_0 :\R\to\R$. (We suppress the dependence on $\beps$ in the notation $\hR_n(\btheta;\bX)$ for brevity).
Namely, we restrict to vector parameters, $\sfk =\sfk^{\star}=1$, and further
assume a separable regularizer.
We will further assume feature
vectors $\bx_i$ with independent and identically distributed coordinates.
We will state a general set of sufficient conditions for universality
without restricting the constraint set $\cC^p$
in the ERM problem
involving $\bX$
to explicitly lie in the domain of Gaussianity~$\cS^p$.
Instead, the conditions will be on the Gaussian equivalent problem.
We will then present a technique to check these conditions via Gaussian comparison inequalities.

It will be useful to introduce the following constrained ERM problem:
Given $S\subseteq [p]$, $\bu\in\R^{|S|}$, positive constants $\delta,D >0$ and 
$\bW,\bV \in\R^{n\times p}$, define
\begin{align}
 &\hK_{S,n}^\star(\bu; \bW, \bV, D)
 := \min_{\bbeta\in\R^{p-|S|}}  \Big\{\hK_{S,n}(\bbeta,\bu; \bW, \bV) :\;\;
 \|\bbeta\|_2^2 \le D 
 \Big\}
 +\frac1{p}\sum_{j=1}^{|S|} r_0(\sqrt{p}u_j)
 \, , \nonumber\\
 &\textrm{with}\quad\quad
  \hK_{S,n}(\bbeta,\bu; \bW, \bV)
 :=\frac{1}{n}
 \sum_{i=1}^n \ell(\bu^\sT\bw_{i,S} + \bbeta^\sT\bv_{i,S^c} , 
 \btheta^{\star\sT} \bv_i, \eps_i )
 +\frac1{p}\sum_{j=1}^{p-|S|} r_0(\sqrt{p}\beta_j)\nonumber
\end{align}
where,
for a vector $\ba \in\R^{p}$, we defined
$\ba_S := (a_j)_{j\in S}\in\R^{|S|}$.
In particular, observe that
$\widehat K_{S,n}^\star(\bu;\bX,\bG,D)$ corresponds to the minimum of an ERM problem whereby covariates in $S$ have the distribution of entries
in $\bX$, while covariates in $[p]\setminus S$ are Gaussian,
under a constraint that fixes coefficients with indices in $S$ to be given by $\bu$.

\begin{proposition}
\label{prop:univ_m_estimation}
Fix $D>0$.
Let $\bX$ have i.i.d., centered uniformly sub-gaussian entries with
unit variance and $\bG = (G_{ij}:i\le n, j\le p)\simiid \normal(0,1)$.
Let 
$\ell(\bv,\bv^\star;\eps)$ defined in Eq.~\eqref{eq:ell-def} satisfy 
Assumption \ref{ass:loss-0} (with the $\eps_i$ uniformly sub-Gaussian),
 $\|\btheta^\star\|_\infty \to 0$ as $p\to\infty$,
and
$r$ satisfy Assumption~\ref{ass:regularizer} and furthermore be separable: $r(\btheta) = p^{-1}\sum_{j=1}^p r_0(\sqrt{p}\theta_j)$ for some $r_0$ independent of $n$, $p$.
Assume further that
\begin{enumerate}
    \item There exists a Gaussian limit in probability  $R^\star_\infty$ for the Gaussian ERM problem
\begin{equation}
\label{eq:gaussian_limit}
   \plim_{n\to\infty}\min_{\|\btheta\|_2 \le D} \hR_n(\btheta;\bG)  =  R_\infty^\star\, ,
\end{equation}
\item There exists a delocalized minimizer for this Gaussian problem, i.e., there 
 exists a sequence $\alpha_n \to 0$ such that
\begin{equation}
\label{eq:gaussian_upper_bound}
   \plim_{n\to\infty} \min_{\substack{\|\btheta\|_2 \le D\\ \|\btheta\|_\infty \le \alpha_n}} \hR_n(\btheta;\bG) =
   R_\infty^\star \,;
\end{equation}
\item There exists a sequence $\{m_n\}_{n\ge 1}$ with $m_n\to\infty$, $m_n/n \to 0$ as $n\to \infty$, 
such that
\begin{align}
\label{eq:R_thry_lower_bound}
\pliminf_{n\to\infty}
\min_{\substack{|S| = m_n \\
\|\bu\|_2 \le D}
} \hK_{S,n}^\star(\bu; \bX, \bG,D)
\ge R_{\infty}^\star \, .
\end{align}
\end{enumerate}

Then universality of training error holds, namely
\begin{equation}
\label{eq:erm_problem_M_est}
 \plim_{n\to\infty}\min_{\norm{\btheta}_2 \le D} \widehat R_n(\btheta; \bX) \stackrel{\P}{\to} R^{\star}_\infty.
\end{equation}
\end{proposition}

The proof of this proposition is given in Section~\ref{proof:prop_univ_m_estimation} of the Appendix.
Note in particular, that since adding a delocalization constraint to $\hat R_n(\btheta;\bX)$ only increases the minimum value of the empirical risk, a universal upper bound will always hold (recall that the domain of Gaussianity $\cS_p$ in this i.i.d. coordinate setting is given by this delocalization constraint).
Meanwhile, if the lower bound is non-universal, then there must exist a fraction $m_n/p$ of the coordinates of $\hat\btheta$ that are not delocalized.
Eq.~\eqref{eq:R_thry_lower_bound} requires that these coordinates do not change the limiting value of the minimum risk.
The important feature of this proposition is that Eqs.~\eqref{eq:gaussian_limit},~\eqref{eq:gaussian_upper_bound} are conditions on the model with Gaussian 
covariates, for which we have better  techniques
to check delocalization of the minimizer.
Additionally, since condition \eqref{eq:R_thry_lower_bound} 
involves a sublinear number $m_n$ of non-Gaussian coordinates in the feature vectors, this can 
often be controlled by a uniform convergence argument. 

In what follows, we outline an example for which we can check the conditions
of Proposition \ref{prop:univ_m_estimation} using Gordon's Gaussian comparison inequality.
\begin{example}[Robust M-estimation]
\label{ex:robust_m_estimation}
Consider the case of a ridge regularizer $r_0(x) = \lambda x^2/2$, 
with $\ell(v,v^\star,\eps)$ is strictly convex in $v$ and $\lambda\ge 0$, 
or $\ell$  convex in $v$ and $\lambda >0$.
Assume $\|\btheta^{\star}\|_2=\kappa$ is independent of $n,p$  and
$\|\btheta^{\star}\|_{\infty}\to 0$ as $n,p\to\infty$. We claim
that the assumptions of Proposition~\ref{prop:univ_m_estimation} and hence the conclusion
of Eq.~\eqref{eq:erm_problem_M_est} hold in this case. 

In Section~\ref{sec:ex_robust_m_estimation_details} of the Appendix, we detail the exercise of checking that the assumptions indeed hold in this case.
With some additional work, it is possible to generalize 
this result to certain bowl-shaped non-convex losses. We leave this to future work. 
\end{example}

\section{Proofs}
\label{sec:Proofs}

Throughout the universality proofs, we will assume that
 $\lim_{n\to\infty} p(n)/n = \sgamma \in(0,\infty).$
This is sufficient because, given any sequence with $p/n\in [\sfC^{-1}, \sfC]$, we can always extract a converging subsequence.
\subsection{Proof outline for Theorems~\ref{thm:MainFirst} and~\ref{thm:main_empirical_univ}}
\label{section:proof_outline}

Here, we will present the important details of the proof, leaving the majority of the technical details to the Appendix.
We will present the proof of 
Eq.~\eqref{eq:main_thm_eq_0} of
Theorem~\ref{thm:MainFirst} 
under the following assumption which requires additional smoothness and decay conditions on the loss.
The extension to Assumption~\ref{ass:loss-0}, and the proof of 
Theorem~\ref{thm:main_empirical_univ} under its assumptions will be deduced in the Appendix as a consequence.

\begin{numassumption}[5'']
\label{ass:labeling_prime}
\label{ass:loss_prime}
\label{ass:loss_labeling_prime}
The non-negative loss function $\ell$ is
differentiable with locally Lipschitz gradient that satisfies
\begin{equation*}
\norm{\grad\ell(\bu) }_2 \le \sfK\left(1 + \norm{\bu}_2 \right),\quad
\end{equation*}
for all $\bu \in\R^{\sfk + 1}$.
Furthermore, for any random variables $\bv\in\R^\sfk,V\in\R$ satisfying
\begin{equation}
\label{eq:sg_generic}
\norm{\bv}_{\psi_2}\vee\norm{V}_{\psi_2} \le 2(\sfR +1) \sfK
\end{equation}
and any $\beta > 0$, we have
\begin{equation}
\label{eq:exp_integrability}
\E\left[\exp\left\{\beta\ell\big( \bv,V \big)\right\}\right] \le C(\beta,\sfR,\sfK)
\end{equation}
for some $C(\beta,\sfR,\sfK)$ dependent only on $\beta,\sfR,\sfK$.
\end{numassumption}

We begin by approximating  the ERM value $\hR^\star_n(\bX,\beps)$
by a \textit{free energy} defined by a sum over 
a finite set in $\R^{p\times\sfk}$. Namely, for $\alpha>0$, let
$\cN_\alpha$ be a minimal $\alpha-$net of $\cC_p$ and define 
\begin{equation}
\label{eq:free_energy_def}
    f_\alpha(\beta,\bX):= 
-\frac{1}{n\beta}\log\sum_{\bTheta\in\cN_\alpha^\sfk}
\exp\left\{-n\beta\widehat R_n(\bTheta;\bX,\beps) 
\right\}.
\end{equation}
\begin{lemma}[Universality of the free energy]
\label{lemma:softmin_univ}
Under the assumptions of Theorem~\ref{thm:MainFirst} with the alternative
Assumption~\hyperref[ass:loss_labeling_prime]{5''},
for 
any fixed $\alpha >0$ and any bounded differentiable function $\psi$ with bounded Lipschitz derivative we have
\begin{equation}
\nonumber
 \lim_{n\to\infty} \left|\E\left[\psi\left(f_{\alpha}(\beta,\bX)\right)\right] - 
\E\left[\psi\left(f_{\alpha}(\beta,\bG)\right)\right]\right| = 0.
\end{equation}
\end{lemma}
Here, we give a proof of this lemma.
A standard estimate bounds the difference between the free energy and the minimum 
empirical risk (see Appendix): For $\beta>0$.
\begin{equation}
\nonumber
 \Big| f_\alpha(\beta,\bX) - \min_{\bTheta\in\cN_\alpha^\sfk} \widehat R_n(\bTheta;\bX,\beps)\Big| 
 \le C(\alpha)\,\beta^{-1}.
\end{equation}
Hence, Theorem~\ref{thm:MainFirst} follows from 
Lemma~\ref{lemma:softmin_univ} via approximation arguments detailed in 
Section~\ref{section:proof_of_main_results}
of the Appendix. 
\subsubsection{Universality of the free energy}
We assume, without loss of generality, that $\bX$ and $\bG$ are
defined on the same probability space and are independent, and define
the interpolating paths
\begin{equation}
\label{eq:slepian}
 \bu_{t,i}:= \sin t \left(\bx_i -\bmu_\bg\right) + \cos t \left(\bg_i -\bmu_\bg\right) + \bmu_\bg,\quad \widetilde\bu_{t,i}:= \cos t (\bx_i-\bmu_\bg) - 
\sin t (\bg_i - \bmu_\bg)
\end{equation}
for $t\in[0,\pi/2]$ and $i\in[n]$. We use $\bU_t$ to denote the matrix whose $i$th row is
$\bu_{t,i}$;
note that these rows are i.i.d.~since the rows of $\bX$ and $\bG$ are so. Noting that 
for all 
$\btheta\in\cS_p$, $\bx^\sT\btheta$ and $\bg^\sT\btheta$ are subgaussian with 
subgaussian norms bounded by $\sfR\sfK$ uniformly over $\btheta\in\cS_p$, it is easy to see that 
$\sup_{t\in[0,\pi/2], \btheta\in\cS_p}  \norm{\bu_t^\sT \btheta}_{\psi_2} \le 2\sfR\sfK.$

The goal is to control the difference $\left|\E\left[\psi(f_{\alpha}(\beta,\bX))\right] - 
\E\left[\psi(f_{\alpha}(\beta,\bG))\right] \right|$ by writing 

\begin{align}
&\lim_{n\to\infty}\left|\E\left[\psi(f_{\alpha}(\beta,\bX))\right] - 
\E\left[\psi(f_{\alpha}(\beta,\bG))\right] \right| \\
&= 
\lim_{n\to\infty}\left|\E\left[\psi(f_{\alpha}(\beta,\bU_{\pi/2}))\right] - 
\E\left[\psi(f_{\alpha}(\beta,\bU_0 ))\right] \right| 
\nonumber
\\
&{=}\left|\int_{0}^{\pi/2}\lim_{n\to\infty} \E 
\left[\frac{\partial}{\partial t}
\psi(f_\alpha(\beta,\bU_t ))\right]\right|\de t
\nonumber
\end{align}
where the technical Lemma~\ref{lemma:dct_bound} in the Appendix provides sufficient regularity for this to hold.
So it is sufficient to show that for fixed $t\in[0,\pi/2]$,
\begin{equation}
\label{suff_cond_for_prop_softmax_1}
\lim_{n\to\infty}\left|\E \left[\frac{\partial}{\partial t} 
\psi(f_\alpha(\beta,\bU_t ))\right]\right| = 0.
\end{equation}


Before computing the derivative involved, we introduce some notation to simplify exposition.
For $\bv\in\R^\sfk,v\in\R$,
we use the notation 
\begin{equation}
\nonumber
   \grad \ell(\bv;v) = \left( \frac{\partial}{\partial v_k}\ell\left(\bv;v\right)\right)_{k\in[\sfk]} \in \R^\sfk.
\end{equation}
Furthermore, we will use the shorthand
 $\widehat\ell_{t,i}(\bTheta)$ for  $\ell\left(\bTheta^\sT \bu_{t,i}; \eps_i\right)$
and define the term
\begin{equation}
\label{eq:bd_def}
 \widehat\bd_{t,i}(\bTheta) :=\left(
\bTheta \grad\widehat\ell_{t,i}(\bTheta) 
\right).
\end{equation}
It is convenient to define the probability mass function over  $\bTheta_0\in\cN_\alpha^\sfk$:
\begin{equation}
\label{eq:inner_def}
p^\up{i}(\bTheta_0;t) :=
\frac{
e^{-\beta\left(\sum_{j\neq 
i}\widehat\ell_{t,j}(\bTheta_0) +  n r(\bTheta_0)\right)
}}{\sum_{\bTheta\in\cN_\alpha^\sfk}
e^{-\beta\left(\sum_{j\neq 
i}\widehat\ell_{t,j}(\bTheta) +  n r(\bTheta)\right)
}}
\quad\textrm{and}\quad
\inner{\;\cdot\;}^\up{i}_\bTheta:=
\sum_{\bTheta\in\cN_\alpha^\sfk} (\,\cdot\,)p^\up{i}(\bTheta;t)
\end{equation}
for $i\in[n]$. 
With this notation, we can compute and bound the term involving the derivative of interest as
\begin{align}
\label{eq:derivative_explicit}
 \E\left[\frac{\partial}{\partial t} 
\psi(f_\alpha(\beta,\bU_t))\right]
&= \E\left[\frac{\psi'(f_\alpha(\beta,\bU_t))}{n}  \sum_{i=1}^n 
\frac{\<
\widetilde\bu_{t,i}^\sT
 \widehat\bd_{t,i}(\bTheta) 
e^{-\beta\widehat\ell_{t,i}(\bTheta)}\>_\bTheta^\up{i}}
{\<e^{-\beta\widehat\ell_{t,i}(\bTheta)}\>_\bTheta^\up{i}}\right]\\
&\hspace{-10mm}\le
\frac1n\sum_{i=1}^n\E\left[\left|{\psi'(f_\alpha(\beta,\bU)) - 
\psi'\left(f_\alpha\left(\beta,\bU^\up{i}\right)\right)}
\frac{
\<
\widetilde\bu_{i}^\sT
\widehat\bd_i(\bTheta)
 e^{-\beta\widehat\ell_i(\bTheta)}
 \>_\bTheta^\up{i}
 }
{\<e^{-\beta\widehat\ell_i(\bTheta)}\>_\bTheta^\up{i}} 
\right|\right]
\label{eq:psi'_softmin_bound}
\\
&+
\frac{1}n\sum_{i=1}^n 
\left|\E\left[\psi'\left(f_\alpha\left(\beta,\bU^\up{i}\right)\right)\inner{
\E_\up{i}\left[
\frac{
\widetilde\bu_{i}^\sT
\widehat\bd_i(\bTheta)
 e^{-\beta\widehat\ell_i(\bTheta)}}
{\<e^{-\beta\widehat\ell_i(\bTheta)}\>_\bTheta^\up{i}}\right]}_\bTheta^\up{i}
\right]\right|.\label{eq:psi'_leave_one_out_bound}
\end{align}
where $\bU^\up{i}$ is obtained by setting the $i$th row in $\bU$ to $0$
and $\E_\up{i}$ denotes the expectation conditional on $(\bG^\up{i},\bX^\up{i},\bepsilon^\up{i})$; 
the feature and noise vectors with the $i$th sample set to $0$. 
Note that to reach~\eqref{eq:psi'_leave_one_out_bound}, we used the independence of $p^\up{i}(\bTheta;t)$ and $(\bx_i,\bg_i,\epsilon_i)$ to swap the order of $\E_\up{i}\left[\;\cdot\;\right]$
and $\inner{\;\cdot\;}_{\bTheta}^\up{i}$.
We control~\eqref{eq:psi'_softmin_bound} and~\eqref{eq:psi'_leave_one_out_bound} 
separately.

The term in~\eqref{eq:psi'_softmin_bound} can be controlled via a simple leave-one-out argument. Indeed, since the samples are i.i.d., it is sufficient to control the term $i=1$ in the sum:
\begin{align}
\left|\psi'(f_\alpha(\beta,\bU)) - \psi'\left(f_\alpha\left(\beta,\bU^\up{1}\right)\right)\right|
&\le  \frac{\norm{\psi'}_\Lip}{n\beta} 
\left|\log\frac{\sum_{\bTheta}e^{
-\beta \left(\sum_{j\neq 1} \widehat\ell_j(\bTheta) + n 
r(\bTheta)\right)
} e^{-\beta\widehat\ell_1(\bTheta)}}{\sum_{\bTheta}e^{
-\beta \left(\sum_{j\neq 1} \widehat\ell_j(\bTheta) + n 
r(\bTheta)\right)
} } \right|\nonumber\\
&\hspace{0mm}\stackrel{(a)}{=} -\frac{\norm{\psi'}_\Lip}{n\beta} 
\log\inner{e^{-\beta\widehat\ell_1(\bTheta)}}_\bTheta^\up{1}\nonumber\\
&\hspace{0mm}\stackrel{(b)}{\le}\frac{\norm{\psi'}_\Lip}{n} 
\inner{\widehat\ell_1(\bTheta)}_\bTheta^\up{1},\nonumber
\end{align}
where $(a)$ follows from the non-negativitiy of $\ell$ and $\beta$ and $(b)$ follows by Jensen's 
inequality. The condition in Eq.~\eqref{eq:exp_integrability} of Assumption~\hyperref[ass:loss_labeling_prime]{5''} guarantees 
that $\sup_{\bTheta\in\cS_p^\sfk}\E_\up{i}\left[\widehat \ell_1(\bTheta)^2\right]$
$\le C_0$, while Lemma~\ref{lemma:dct_bound} of the Appendix bounds
\begin{equation}
\label{eq:regularity_bound_softmax} 
\sup_{\bTheta_0\in\cS_p^\sfk}
\E_\up{1}\left[\left(
\frac{
\widetilde\bu_{t,1}^\sT
\widehat\bd_{t,1}(\bTheta_0)
 e^{-\beta\widehat\ell_{t,1}(\bTheta_0)}}
{\<e^{-\beta\widehat\ell_{t,1}(\bTheta)}\>_\bTheta^\up{1}}\right)^2\right]\le C_0'(\beta)
\end{equation}
so that an application of Cauchy--Schwarz to~\eqref{eq:psi'_softmin_bound} yields
\begin{align*}
\limsup_{n\to\infty}\E\left[\left|\left({\psi'(f_\alpha(\beta,\bU)) - 
\psi'(f_\alpha(\beta,\bU^\up{1}))}\right) 
\bigg\<
\frac{
\widetilde\bu_{1}^\sT
\widehat\bd_1(\bTheta)
 e^{-\beta\widehat\ell_1(\bTheta)}}
{\<e^{-\beta\widehat\ell_1(\bTheta)}\>^\up{1}}\bigg\>_\bTheta^\up{1}
\right|\right]= 0.
\end{align*}

Meanwhile, to control the 
term~\eqref{eq:psi'_leave_one_out_bound}, it is sufficient to establish that
\begin{equation}
\label{eq:suff_cond_for_prop_softmax_2}
\lim_{n\to\infty}
\sup_{\bTheta_0}
\left|\E_\up{1}\left[  
\frac{
\widetilde\bu_{1}^\sT
\widehat\bd_{1}(\bTheta_0)
 e^{-\beta\widehat\ell_{1}(\bTheta_0)}}
{\inner{e^{-\beta\widehat\ell_1(\bTheta)}}_\bTheta^\up{1}}\right]\right| = 0  \quad \textrm{a.s.}
\end{equation}
To see that this is sufficient, note that with~\eqref{eq:suff_cond_for_prop_softmax_2}, we can control~\eqref{eq:psi'_leave_one_out_bound} as
\begin{align}
 &\limsup_{n\to\infty}  
\frac{1}n\sum_{i=1}^n 
\bigg|\E\bigg[\psi'\Big(f_\alpha\big(\beta,\bU^\up{i}\big)\Big)\bigg\<
\E_\up{i}\bigg[
\frac{
\widetilde\bu_{i}^\sT
\widehat\bd_i(\bTheta_0)
 e^{-\beta\widehat\ell_i(\bTheta_0)}}
{\<e^{-\beta\widehat\ell_i(\bTheta)}\>_{\bTheta}^\up{i}}\bigg]\bigg\>_\bTheta^\up{i}
\bigg] \bigg|
\nonumber
\\
 &\hspace{40mm}\stackrel{(a)}{\le}
\norm{\psi'}_\infty
\limsup_{n\to\infty}\E\bigg[\bigg\<\bigg|\E_\up{1}\bigg[\frac{
\widetilde\bu_{1}^\sT
\widehat \bd_1(\bTheta_0)
e^{
-\beta\widehat\ell_1(\bTheta_0)}}{\<e^{-\beta\widehat\ell_1(\bTheta)}\>_\bTheta^\up{1}}   \bigg]\bigg|\bigg\>_{\bTheta_0}^\up{1}\bigg] 
\nonumber
\\
&\hspace{40mm}
\stackrel{(b)}{\le}
\norm{\psi'}_\infty\E 
\bigg[\limsup_{n\to\infty}\sup_{\bTheta_0}\bigg|\E_\up{1}\bigg[\frac{ 
\widetilde\bu_{1}^\sT
\widehat\bd_1(\bTheta_0)
e^{
-\beta\widehat\ell_1(\bTheta_0)}}{\<e^{-\beta\widehat\ell_1(\bTheta)}\>^\up{1}_\bTheta}   
\bigg]\bigg| \bigg]
= 0
\nonumber
\end{align}
where $(a)$ follows by the i.i.d assumption on the samples and $(b)$ follows by reverse Fatou's with the bound in~\eqref{eq:regularity_bound_softmax}.

Meanwhile, the following crucial polynomial approximation lemma
allows us to establish~\eqref{eq:suff_cond_for_prop_softmax_2}.
The proof is deferred to the next subsection.

%
%
\begin{lemma}
\label{lemma:poly_approx}
Under the assumptions of Theorem~\ref{thm:MainFirst} with the alternative
Assumption~\hyperref[ass:loss_labeling_prime]{5''},
for any $\delta >0, \beta>0$, there exists a polynomial $P$ of degree and coefficients 
dependent only on $\delta,\beta$ and $\sOmega$ such that for all $\bTheta_0 \in 
\cS_p^\sfk, t\in[0,\pi/2]$ and $n \in \Z_{>0}$
\begin{align}
\nonumber
&\bigg|\E_\up{1}\bigg[  
\frac{\widetilde\bu_{t,1}^\sT\widehat\bd_{t,1}(\bTheta_0) 
e^{-\beta\widehat\ell_{t,1}(\bTheta_0)}}
{\<e^{-\beta\widehat\ell_{t,1}(\bTheta)} \>^\up{1}_\bTheta}\bigg]\bigg|\\
&\hspace{5mm}\le \left|{\E_\up{1}\left[
{\widetilde\bu_{t,1}^\sT\widehat\bd_{t,1}(\bTheta_0)}
e^{-\beta\widehat\ell_{t,1}(\bTheta_0)}
P\Big(\<e^{-\beta\widehat\ell_{t,1}(\bTheta)}\>^\up{1}_\bTheta
\Big)\right]} 
\right| + \delta.
\nonumber
\end{align}
\end{lemma}

This polynomial is achieved by the power series approximation of $x\mapsto 1/x$, however, the  essential aspect of this approximation is that the degree and coefficients of the polynomial are independent of the dimension $n$ so that this approximation holds uniformly, and the 
term in~\eqref{eq:suff_cond_for_prop_softmax_2} can be approximated in terms of a low-dimensional projection of the interpolating feature vectors.
In turn, the term involving these projections is easier to control. 
Indeed, fixing $\delta>0$ and letting $P(s) = \sum_{j=0}^M  b_js^j$ for 
degree $M\in\Z_{>0}$ and
coefficients $\{b_j\}_{j\le [M]}$ as in the lemma, we obtain the bound on~\eqref{eq:suff_cond_for_prop_softmax_2}
\begin{align}
\nonumber
\bigg|\E_\up{1}\bigg[  
\frac{
\widetilde\bu_1^\sT
\widehat\bd_1(\bTheta_0)
e^{-\beta\widehat\ell_1(\bTheta_0)}}
{\<e^{-\beta\widehat\ell_1(\bTheta)}\>_\bTheta}\bigg]\hspace{-0.5mm}\bigg|&
\le\bigg|\E_\up{1}\bigg[
\widetilde\bu_1^\sT
\widehat\bd_1(\bTheta_0)
e^{-\beta\widehat\ell_1(\bTheta_0)}
\sum_{j=0}^M b_j 
\left(\<e^{-\beta\widehat\ell_1(\bTheta)}\>
_\bTheta^\up{1}\right)^j
\bigg]\hspace{-0.7mm}
\bigg| + \delta 
\nonumber
\\
&\hspace{-5mm}\stackrel{(a)}{=}\left|
\sum_{j=0}^M
b_j 
\inner{
\E_\up{1}\left[
\widetilde\bu_1^\sT
\widehat\bd_1(\bTheta_0)
e^{-\beta\widehat\ell_1(\bTheta_0)}
e^{-\beta\sum_{l=1}^j\widehat\ell_1(\bTheta_l)}
\right]
}^\up{1}_{\bTheta_1^j}
\right| + \delta
\nonumber
\\
&\hspace{-5mm}\le
\sum_{j=0}^M
\left| b_j \right|
\sup_{\bTheta_1,\ldots,\bTheta_j\in\cS_p^\sfk}
\left|
\E_\up{1}\left[
{
\widetilde\bu_{1}^\sT
\widehat\bd_1(\bTheta_0)
}
e^{-\beta\sum_{l=0}^j\widehat\ell_1(\bTheta_l)}
\right]
\right|+\delta,
\label{eq:poly_approx_bound}
\end{align}
where in $(a)$ we defined the 
expectation $\inner{\;\cdot\;}_{\bTheta_1^j}$
with respect to $\{\bTheta_l\}_{l\le [j]}$ seen as independent 
samples from $p^\up{1}(\bTheta;t)$. Since the term in~\eqref{eq:poly_approx_bound} now contains a number of projections bounded in $n$,
the next lemma then states that it can be controlled via its Gaussian equivalent.
\begin{lemma}
\label{lemma:gaussian_approx}
Suppose the assumptions of Theorem~\ref{thm:MainFirst} hold with the alternative 
Assumption~\hyperref[ass:loss_labeling_prime]{5''}.
Let $\widetilde\bg_1\sim\cN(\bmu_\bg,\bSigma_\bg)$ and $\widetilde \eps_1$ be an 
independent copy of $\eps_1$, both independent of $\bg_1$ and define
$\bw_{t,1} = \sin(t) (\widetilde\bg_1 - \bmu_\bg) + \cos(t)(\bg_1 - \bmu_\bg) + \bmu_\bg$ and 
$\widetilde\bw_{t,1} = \cos(t) (\widetilde\bg_{1} - \bmu_\bg) - 
\sin(t)(\bg_1 - \bmu_\bg)$. For any fixed $\beta>0$, $t\in[0,\pi/2]$ and $J\in\Z_{>0}$ we have, as $p\to\infty$,
\begin{align}
\label{eq:gaussian_approx_lemma_eq}
\hspace{2mm}
\sup_{
\mathclap{\substack{
\phantom{\bTheta^\star\in\cS_p^{\sfk^\star}}\\
\hspace{7mm}\bTheta_0,\dots,\bTheta_J\in\cS_p^\sfk
}}
}
\hspace{1.5mm}
\Bigg|
\hspace{-0.2mm}\E\hspace{-0.5mm}\left[
{\widetilde\bu_{t,1}^\sT\widehat\bd_{t,1}(\bTheta_0)}
e^{-\beta\sum_{l=0}^J\widehat\ell_{t,1}\left(\bTheta_l\right)}
\right]\hspace{-0.7mm}-\hspace{-0.5mm}
\E\hspace{-0.5mm}\left[\hspace{-0.2mm}
{\widetilde\bw_{t,1}^\sT \widehat\bq_{t,1}(\bTheta_0)}
e^{-\beta\sum_{l=0}^J\ell\left(\bTheta_l^\sT\bw_{t,1};\widetilde\eps_1\right)}
\hspace{-0.5mm}
\right]\hspace{-1.2mm}
\Bigg| \to 0
\end{align}
 where\; $\widehat\bq_{t,1}(\bTheta_0) :=
\bTheta_0 \grad\ell\left(\bTheta_0^\sT \bw_{t,1},\widetilde\eps_1\right).$
\end{lemma}
This lemma can be seen as an extension of the \emph{marginal} pointwise normality condition of Eq.~\eqref{eq:condition_bounded_lipschitz_single} to hold \emph{jointly},
and for a finite number of projections instead of only one.
In fact, in Section~\ref{section:proof_gaussian_approx} of the Appendix, we prove the following more general form, from which Lemma~\ref{lemma:gaussian_approx} will follow after a suitable truncation argument.
\begin{lemma}
\label{lemma:proj}
Suppose Assumption~\ref{ass:X} holds for $\bx$ and $\cS_p$ is as in Definition~\ref{def:DoG}.
Let $K >0$ be a fixed integer, and $\widetilde \bg \sim\cN(\bmu_\bg,\bSigma_\bg)$ an independent copy of $\bg$.
Then for any
bounded Lipschitz function $\varphi:\R^{3K} \to \R$, we have
\begin{equation}
\nonumber
\lim_{p\to\infty}\sup_{\bH = (\btheta_1,\dots,\btheta_K)\in\cS_p^K
}\left|\E\left[\varphi\big(\bH^\sT \bx,\bH^\sT \widetilde \bg, \bH^\sT\bmu_\bg\big)\right] 
- \E\left[\varphi \big(\bH^\sT \bg, \bH^\sT \widetilde \bg, \bH^\sT\bmu_\bg\big)\right]\right| = 0.
\end{equation}
\end{lemma}

%

Note that $\bw_{t,1}$ and $\widetilde \bw_{t,1}$ are jointly Gaussian, and
 $\E_\up{1}\left[\widetilde\bw_{t,1}(\bw_{t,1}-\bmu_\bg)^\sT\right] = 0 $
for all $t\in[0,\pi/2]$, and hence they are independent. Then using that $\E[\widetilde \bw_{t,1}] =0$, the expectation 
involving $\widetilde \bw_{t,1},\bw_{t,1}$ in~\eqref{eq:gaussian_approx_lemma_eq} decouples as
\begin{equation}
\nonumber
\hspace{-1mm}
\E\Big[
\widetilde\bw_{t,1}\Big]
^\sT
\hspace{-1mm} \left[\widehat\bq_{t,1}(\bTheta_0)
e^{-\beta\sum_{l=0}^J\ell\left(\bTheta_l^\sT\bw_{t,1}, \widetilde\eps_1 \right)}
\right]  = 0.
\end{equation}
Hence, for $j\in[M]$, we have
\begin{align}
\nonumber
&\limsup_{n\to\infty}
\quad\sup_{
\mathclap{\substack{
\phantom{\bTheta^\star\in\cS_p^{\sfk^\star}}
\\
\bTheta_0,\dots,\bTheta_j\in\cS_p^\sfk}}}\quad
\left|
\E_\up{1}\left[
{
\widetilde\bu_1^\sT
\widehat\bd_1(\bTheta_0)}
e^{-\beta\sum_{l=0}^j\widehat\ell_1\left(\bTheta_l
\right)}
\right]\right|\\
&\hspace{26mm}=
\limsup_{n\to\infty}
\quad\sup_{
\mathclap{\substack{
\phantom{\bTheta^\star\in\cS_p^{\sfk^\star}}
\\
\bTheta_0,\dots,\bTheta_j\in\cS_p^\sfk}}}\quad
\left|\E_\up{1}
\Big[\widetilde \bw_1\Big]^\sT
\E_\up{1}\left[{\widehat\bq_1(\bTheta_0)}
e^{-\beta\sum_{l=0}^j\ell\left(\bTheta_l^\sT\bw_1;
\widetilde\epsilon_1\right)}
\right]\right|=0,
\nonumber
\end{align}
Combining this with the bound in~\eqref{eq:poly_approx_bound}
yields, for all $\delta>0$,
\begin{align}
\nonumber
\limsup_{n\to\infty}\sup_{\bTheta_0\in\cS_p^\sfk}\left|\E_\up{1}\left[
\frac{
\widetilde\bu_{1}^\sT
\widehat\bd_{1}(\bTheta_0) 
e^{-\beta\widehat\ell_{1}(\bTheta_0)}}
{\<e^{-\beta\widehat\ell_{1}(\bTheta)}\>_\bTheta^\up{1}}\right]\right| 
\le\delta.
\end{align}
Taking $\delta \to0$ then establishes~\eqref{eq:suff_cond_for_prop_softmax_2} for any $t \in[0,\pi/2]$ and concludes the proof.

\subsubsection{Proof of Lemma~\ref{lemma:poly_approx}}
\label{section:proof_poly_approx}
Recall the definitions of $\bu_{t,i}, \widetilde \bu_{t,i}, \widehat\bd_{t,i}(\bTheta)$ and $\inner{\,\cdot\,}_\bTheta^\up{i}$ in equations~\eqref{eq:slepian},~\eqref{eq:bd_def} and~\eqref{eq:inner_def} respectively. 
Further, recall the shorthand notation 
$\widehat \ell_{t,i}(\bTheta)$ for the loss, and $\E_\up{i}$ for the conditional expectation defined in Section~\ref{section:proof_outline}. Throughout, we fix $i=1$ as in the statement of the lemma.

Define the event
\begin{align}
\label{def:G_Theta,B}
 \cG_{\bTheta,B} &:= 
 \Big\{\left|{\btheta_k^\sT\bu_1 }\right|
 \le B \textrm{ for 
all } 
k\in[\sfk]\Big\}
\cap
 \Big\{
 \left|{\btheta_k^\sT\widetilde\bu_1}\right|
 \le B \textrm{ for 
all } 
k\in[\sfk]\Big\}
\cap
\Big\{|\epsilon_1|\le B \Big\},
\end{align}
defined for $\bTheta\in\cS_p^\sfk$ and $B> 0$. 
To avoid centering $\bu$ and $\epsilon$,
we will consider $B > K$ for some $ K \ge 2\left(\sup_{\btheta\in\cS_p}\E\left[|\bx^\sT\btheta|\right]\vee \E\left[|\epsilon|\right]\right)$ depending only on $\sOmega$; the existence of such $K$ is guaranteed 
by the subgaussianity assumption. 

From standard subgaussian tail bounds, we conclude that for any $B>K$, we have constants $C,C' >0$ depending only on~$\sOmega$ such that
\begin{equation}
\label{eq:P_G_c}
\sup_{\bTheta \in\cS_p^\sfk}
\P\left(
\cG_{\bTheta,B}^c
\right)
\le 
C e^{-C' B^2}.
\end{equation}
The proof of this bound is detailed in Section~\ref{section:proof_poly_approx_details}
of the Appendix.

Now consider the power series of $x\mapsto 1/x$ centered at 1, and its associated remainder
\begin{equation}
\nonumber
    P_M(x):= \sum_{l=0}^M (1-x)^l,\quad
    R_M(x):= \frac1x -  P_M(x).
\end{equation}
Elementary calculations show that for $M > 0$,  we have
\begin{enumerate}[(i)]
    \item \label{item:P_M_1} $R_M(x) = (1-x)^{M+1}/{x}$ for $x\neq 0$;
    \item $R_M(x)^2$ is convex on $(0,1]$; \label{item:P_M_4} 
    \item For any $s\in(0,1)$ and $\delta>0$, there exists $M>0$ such that
        $\sup_{t\in[s,1]}\left|R_M(t)\right| < \delta$.
    \label{item:P_M_5}
\end{enumerate}
The proof of these properties is included in Section~\ref{section:proof_poly_approx_details} of the Appendix for the sake of completeness.
The following lemma now bounds the error in the approximation, and is the key for proving Lemma~\ref{lemma:poly_approx}.
\begin{lemma}
\label{lemma:R_M_bound}
For any $\delta>0$ and $\beta>0$, there exists some finite integer $M_{\beta,\delta}>0$, depending only on 
$\beta,\delta$ and $\sOmega$ such that
\begin{equation}
\nonumber
    \E_\up{1}\left[ R_{M_{\beta,\delta}}\left(\< e^{ 
-\beta\widehat\ell_{t,1}(\bTheta)}\>^\up{1}_\bTheta\right)^2\right] < \delta
\end{equation}
uniformly in $n$.
\end{lemma}

\begin{proof}
Recall the definition of $\cG_{\bTheta,B}$ in~\eqref{def:G_Theta,B} for $B>K$ and 
$\bTheta\in\cS_p^\sfk$ and write for arbitrary integer $M>0$,
\begin{align}
\E_\up{1} \left[R_M\left(\< e^{  -\beta\widehat\ell_{t,1}(\bTheta) } \>_\bTheta^\up{1}\right)^2\right]
&\stackrel{(a)}\le
\inner{\E_\up{1}\left[ R_M\left( e^{ -\beta\widehat\ell_{t,1}(\bTheta) } \right)^2 
\right]}^\up{1}_\bTheta\nonumber\\
&= 
\inner{\E_\up{1}\left[ R_M\left( e^{ -\beta\widehat\ell_{t,1}(\bTheta)} \right)^2 
\one_{\cG_{\bTheta,B}} 
\right]}^\up{1}_\bTheta
\label{eq:P_M_approx_bound}
\\
&\hspace{9mm}+
\inner{\E_\up{1}\left[ R_M\left( e^{ -\beta\widehat\ell_{t,1}(\bTheta)} \right)^2 
\one_{\cG^c_{\bTheta, B}}
\right]}^\up{1}_\bTheta
\label{eq:PM_approx_proof_2nd_term}
\end{align}
where $(a)$ follows from Jensen and the convexity of $R_M^2$ on $(0,1]$ from property~\ref{item:P_M_4}.
The expectation 
in the second term can be bounded uniformly over 
$\bTheta\in\cS_p^\sfk$, namely
\begin{align*}
&\E_\up{1}\left[ 
R_M\left(e^{-\beta\widehat\ell_{t,1}(\bTheta)}\right)^2 \one_{\cG_{\bTheta,B}^c}
\right]
\le
\E_\up{1}\left[ 
R_M\left(e^{-\beta\widehat\ell_{t,1}(\bTheta)}\right)^4\right]^{1/2}\P\left(\cG_{\bTheta,B}^c\right)^{1/2}
\\
&\quad\quad\stackrel{(a)}{\le} 
\E_\up{1}\left[\left(\frac{1}{e^{-\beta\widehat\ell_{t,1}(\bTheta)}}\right)^4 \right]^{1/2} C_0 e^{- C_1 B^2}
\quad
\stackrel{(b)}{\le}
\quad
{C_2(\beta) C_0e^{-C_1 B^2}}
\end{align*}
for some constant $C_2(\beta)$ depending only on $\beta$ and $\sOmega$, and $C_0,C_1$ depending only on $\sOmega$.
Here, $(a)$ follows from property~\ref{item:P_M_1} of $R_M$ along with the 
tail bound from Eq.~\eqref{eq:P_G_c}, and that $\ell$ is assumed to be non-negative, and $(b)$ follows from the integrability condition~\eqref{eq:exp_integrability} of
Assumption~\hyperref[ass:loss_labeling_prime]{5''}.

For a given 
$\delta\in(0,1)$, we can find some $B_{\beta,\delta}>0$ sufficiently large, depending only on 
$\beta,\delta$ and $\sOmega$ such  that
$C_2(\beta) C_0e^{-{C_1B_{\beta,\delta}^2}} < \delta/2$, thus bounding the term 
in~\eqref{eq:PM_approx_proof_2nd_term} by $\delta/2$. Then for this fixed $B_{\beta,\delta}$, 
by continuity of the composition of
$\ell$ in $\left(\bTheta^\sT\bu_{t,1}, \epsilon_1\right)$, there exists 
some 
$\widetilde B_{\beta,\delta}>0$, such that, for any $\bTheta\in\cS_p^\sfk$,
\begin{equation*}
\widehat \ell_{t,1}(\bTheta) = \ell\left(\bTheta^\sT\bu_{t,1}, 
\epsilon_1\right)\one_{\cG_{\bTheta,B_{\beta,\delta}}} \in \left[0,\widetilde 
B_{\beta,\delta}\right].
\end{equation*}
Therefore, for any $\bTheta\in\cS_p^\sfk$,
\begin{equation*}
e^{-\beta\widehat\ell_{t,1}\left(\bTheta\right)}\one_{\cG_{\bTheta,B_{\beta,\delta}}} \in 
\left[e^{-\beta\widetilde B_{\beta,\delta}}, 1\right].
\end{equation*}
Then, by property~\ref{item:P_M_5} of $R_M$, we can choose $M = M_{\beta,\delta}$ a sufficiently large integer so that
\begin{equation}
\nonumber
 |R_{M_{\beta,\delta}}(t)| < \sqrt{\delta/2}  \quad \textrm{ for all } t\in \left[e^{-\beta 
\widetilde B_{\beta,\delta}} , 1\right].
\end{equation}
This gives the bound on~\eqref{eq:P_M_approx_bound}:
\begin{align*}
 \inner{\E_\up{1}\left[ R_M\left(e^{-\beta\widehat\ell_{t,1}\left(\bTheta\right)}\right)^2\one_{\cG_{\bTheta,{B_{\beta,\delta}}}}\right]}^\up{1}_\bTheta 
\le \frac\delta2,
\end{align*}
which when combined with the bound on $\eqref{eq:PM_approx_proof_2nd_term}$ yields the 
claim of the lemma.
\end{proof}

Finally, let us complete the proof of Lemma~\ref{lemma:poly_approx}.
\begin{proof}[Proof of Lemma~\ref{lemma:poly_approx}]
Let $C= C(\beta)$ be the constant in Eq.~\eqref{eq:regularity_bound_softmax} guaranteeing that
\begin{equation}
\nonumber
\left|\E_\up{1}\left[\left(
\widetilde\bu_{t,1}^\sT
\widehat\bd_{t,1}(\bTheta_0)
 e^{-\beta\widehat\ell_{t,1}(\bTheta_0)}
\right)^2\right]\right| \le 
\left|\E_\up{1}\left[\left(
\frac{
\widetilde\bu_{t,1}^\sT
\widehat\bd_{t,1}(\bTheta_0)
 e^{-\beta\widehat\ell_{t,1}(\bTheta_0)}}
{\<e^{-\beta\widehat\ell_{t,1}(\bTheta)}\>_\bTheta^\up{1}}\right)^2\right]\right| \le C.
\end{equation}
Fix $\delta>0$, and let $N_{\beta,\delta} := M_{\beta,\delta^2/C}$ so that Lemma~\ref{lemma:R_M_bound} holds with $\delta$ replaced by 
$\delta^2/C.$
Then, we directly have via an application of Cauchy--Schwarz
\begin{align*}
&\bigg|\E_\up{1}\bigg[  
\frac{\widehat\bd_{t,1}(\bTheta_0)^\sT\widetilde\bu_{t,1} e^{-\beta\widehat\ell_{t,1}(\bTheta_0)}}
{\<e^{-\beta\widehat\ell_{t,1}(\bTheta)}\>^\up{1}_\bTheta}\bigg]\bigg|\\
&\le
\left|{\E_\up{1}\left[
\widehat\bd_{t,1}(\bTheta_0)^\sT\widetilde\bu_{t,1}
e^{-\beta\widehat\ell_{t,1}(\bTheta_0)
}
P_{N_{\beta,\delta}}
\left(
\<e^{-\beta\widehat\ell_{t,1} (\bTheta)
}\>^\up{1}_\bTheta\right)
\right]} 
\right| \\
&\hspace{5mm}+ 
\E \left[  
\left(\widehat\bd_{t,1}(\bTheta_0)^\sT \widetilde 
\bu_{t,1} e^{-\beta\widehat\ell_{t,1}
(\bTheta_0)
}
\right)^2
\right]^{1/2}
\E\left[ 
R_{N_{\beta,\delta}}\left(\<e^{-\beta\widehat\ell_{t,1}
(\bTheta)
}\>^\up{1}_\bTheta\right)^2
\right]^{1/2}
\\
&\le
\left|{\E\left[
\widehat\bd_{t,1}(\bTheta_0)^\sT\widetilde\bu 
e^{-\beta\widehat\ell_{t,1}
(\bTheta_0)
}
P_{N_{\beta,\delta}}\left(\<e^{-\beta\widehat\ell_{t,1}
(\bTheta)
}\>^\up{1}_\bTheta\right)
\right]} 
\right| + \delta
\end{align*}
as desired.
\end{proof}

\subsection{Proof of Theorem~\ref{thm:universality_bounds}}
\label{section:outline_proof_univerality_bounds}
Since the proof is independent of the dimension $\sfk$ as long as it is fixed and constant in $n$, we take $\sfk =1$ and further assume, without loss of generality, 
that $\widehat R_n^\star(\bG,\by(\bG)),\widehat R_n^\star(\bX,\by(\bX))$ are non-negative:
Otherwise, we can replace the regularizer $r(\btheta)$ with $\widetilde r(\btheta) := r(\btheta) - \min_{\btheta'\in\cC_p} r(\btheta')$ to obtain a new non-negative regularizer
satisfying Assumption~\ref{ass:regularizer}, and since $\ell$ is assumed to be non-negative, the minimum empirical risk will be nonegative.

Define, for $t>0$ and $n\in\Z_{>0}$ the sequence of events
\begin{equation}
\nonumber
 \cG_{n,t} := \left\{\widehat R_n^\star(\bX,\by(\bX))\le t\right\} \cap \left\{ \widehat R_n^\star(\bG,\by(\bG)) = 0\right\}.
\end{equation}
Recall the assumption that $\lim_{n\to\infty}\P\left(\widehat R_n^\star(\bG,\by(\bG)) =0 \right) =1$ and note that it implies, alongside Theorem~\ref{thm:main_empirical_univ}, that for all $t>0$ 
we have
$\lim_{n\to\infty}\P(\cG_{n,t}) = 1$. 

Working on the extended real numbers $\bar \R$, let us define
\begin{equation}
\nonumber
 F_n^\bg(t,\bX) := 
 \quad
\min_{
\mathclap{\substack{ \btheta \in \cC_p   \\ \widehat 
R_n(\btheta;\bX,\by(\bX)) \le t }}
}
R_n^\bg(\btheta),\hspace{10mm}
 F_n^\bx(t,\bX) := 
 \quad
\min_{
\mathclap{\substack{ \btheta \in \cC_p   \\ \widehat 
R_n(\btheta;\bX,\by(\bX)) \le t }}
}
R_n^\bx(\btheta),
\end{equation}
and similarly  
\begin{equation}
\nonumber
 F_n^\bg(t,\bG) := 
 \quad
\min_{
\mathclap{\substack{ \btheta \in \cC_p   \\ \widehat 
R_n(\btheta;\bG,\by(\bG)) \le t }}
}
R_n^\bg(\btheta)
\end{equation}
for all $t\ge0$, where we set the value of the minimum to $\infty$ whenever the constraints are not feasible. 
Lemma~\ref{lemma:F_diff} of the Appendix shows that 
for all $t\ge s > 0$ and any $\delta>0$, we have
\begin{equation}
\label{eq:deferred_lemma_eq_univ_bound}
\lim_{n\to\infty}\P\left(\left\{\left|  F_n^\bx(t,\bX) - F_n^\bg(t,\bX)  \right| > \delta \right\} \bigcap \cG_{n,s}\right) = 0.
\end{equation}

Now, let us fix $\alpha, \alpha_0$ such that $\alpha\ge \alpha_0 >0$.
We can bound for all $n >0$,
\begin{equation}
\label{eq:bound_on_FG_1}
\sup_{t \ge \alpha_0} F_n^\bg(t,\bX)\one_{\cG_{n,\alpha_0}} \le \sup_{n>0} \sup_{\btheta\in\cC_p } R_n^\bx(\btheta)  \stackrel{(a)}\le C' < \infty,
\end{equation}
where $(a)$ follows from the subgaussianity in Assumption~\ref{ass:X} and the growth condition on the loss,
and hence a similar bound holds for $ F_n^\bx(t,\bX)\one_{\cG_{n,\alpha_0}}$ and $F_n^\bg(t,\bG)\one_{\cG_{n,\alpha_0}}$.
Now take
$
  s := C'/\alpha
$
for the constant $C'$ in~\eqref{eq:bound_on_FG_1}. 

We first lower bound the quantity
\begin{equation}
\nonumber
\widehat R_{n,s}^\star(\bX,\by(\bX)) :=\min_{\btheta\in\cC_p}\left\{ s\widehat R_n(\btheta; \bX, \by(\bX))+
R_n^\bg(\btheta)\right\}
\end{equation}
on $\cG_{n,\alpha_0}$.
Letting
$\widehat\btheta^\bX_s$
denote a minimizer of this problem we write
\begin{align*}
&\left(s\widehat R_n\left(\widehat \btheta^\bX_s; \bX, \by(\bX)\right) +
R_n^\bg\left(\widehat \btheta^\bX_s\right)  \right)\one_{\cG_{n,\alpha_0}}\\
&\hspace{49mm}\ge\left(s\widehat R_n\left(\widehat \btheta^\bX_s; \bX ,\by(\bX)\right) +F_n^\bg\left(\widehat R_n(\widehat \btheta^\bX_s, \bX), \bX\right)\right)\one_{\cG_{n,\alpha_0}}\\
&\hspace{49mm}\ge \min_{t \ge 0} \left\{  t s+ F_n^\bg(t,\bX)\right\} \one_{\cG_{n,\alpha_0}}\\
&\hspace{49mm}\stackrel{(a)}{\ge} \min_{t\ge 0}\left\{ \frac{t C'}{\alpha} + F_n^\bg(\alpha,\bX) \one_{t \le \alpha}\right\} \one_{\cG_{n,\alpha_0}}\\
&\hspace{49mm}\stackrel{(b)}\ge F_n^\bg(\alpha,\bX)\one_{\cG_{n,\alpha_0}}\min_{t\ge 0}\left\{\frac{t}{\alpha} + \one_{t \le \alpha}\right\}\\
&\hspace{49mm}\ge F_n^\bg(\alpha,\bX)\one_{\cG_{n,\alpha_0}},
\end{align*}
where in $(a)$ we used that $F_n^\bg(t,\bX)$ is nonincreasing in $t$ and the definition of 
$s$, and in $(b)$ that $C' \ge F_n^\bg(\alpha,\bX)$ by~\eqref{eq:bound_on_FG_1}.

Meanwhile
we can obtain an upper bound for
\begin{equation*}
\widehat R_{n,s}^\star(\bG,\by(\bG)) :=\min_{\btheta\in\cC_p}\left\{ s\widehat R_n(\btheta; \bG ,\by(\bG))+ 
R_n^\bg(\btheta)\right\}
\end{equation*}
on $\cG_{n,\alpha_0}$
by
\begin{align*}
\min_{\btheta\in\cC_p} \left\{s\widehat R_n(\btheta; \bG, \by(\bG)) + 
R_n^\bg(\btheta) \right\}\one_{\cG_{n,\alpha_0}}
&\le\left( s\alpha_0 +  
\hspace{6mm}
\min_{
\mathclap{\substack{ \btheta \in \cC_p   \\ \widehat 
R_n(\btheta;\bG,\by(\bG)) \le \alpha_0 }}
}
 R_n^\bg(\btheta)\right)\one_{\cG_{n,\alpha_0}} \\
 &=\left( \frac{\alpha_0 C'}{\alpha} + F_n^\bg(\alpha_0,\bG) \right)\one_{\cG_{n,\alpha_0}}.
\end{align*}

%

Now applying Theorem~\ref{thm:main_empirical_univ} 
to deduce the universality of $\widehat R_{n,s}^\star$, we obtain the first assertion of the theorem as follows:
fix $\delta>0$ and $\rho\in\R$, then
\begin{align*}
&\limsup_{n\to\infty}\P\left(  F_n^\bx(\alpha,\bX) \ge 
 \rho + 3\delta \right) \\
&\le 
\limsup_{n\to\infty}\P\left(\left\{
F_n^\bg(\alpha,\bX) \ge
\rho + 2\delta 
\right\}\bigcap \cG_{n,\alpha_0} \right)
+ \limsup_{n\to\infty}\P\left(\cG_{n,\alpha_0}^c\right)
\\
&\hspace{5mm}+ 
\limsup_{n\to\infty}\P\left(\left\{
\left|F_n^\bg(\alpha,\bX) -
 F_n^\bx(\alpha,\bX)\right| \ge \delta 
\right\}\bigcap \cG_{n,\alpha_0} \right)\\
&\stackrel{(a)}\le 
\limsup_{n\to\infty}\P\left(\left\{
\widehat R_{n,s}^\star(\bX,\by(\bX))  \ge \rho +2\delta 
\right\}\bigcap \cG_{n,\alpha_0} \right) \\
&\stackrel{(b)}\le 
\limsup_{n\to\infty}\P\left(\left\{
\widehat R_{n,s}^\star(\bG,\by(\bG))  \ge \rho + \delta \right\}
\bigcap \cG_{n,\alpha_0}
\right) \\
&\stackrel{(c)}\le 
\limsup_{n\to\infty}\P\left(\left\{
F_n^\bg(\alpha_0,\bG) + C' \frac{\alpha_0}{\alpha} \ge \rho +\delta \right\} \bigcap \cG_{n,\alpha_0}\right)\\
&\stackrel{(e)}\le 
\limsup_{n\to\infty}\P\left(
F_n^\bg(0,\bG)  \ge \rho  \right)
+\P\left(
 C' \frac{\alpha_0}{\alpha} \ge \delta \right)
\end{align*}
where $(a)$ follows from~\eqref{eq:deferred_lemma_eq_univ_bound}, that $\lim_n\P(\cG_{n,\alpha_0}) =1$ and
 the upper bound derived above, $(b)$ follows from Theorem~\ref{thm:main_empirical_univ} by absorbing $R^\bg(\btheta)$ into the 
regularization term and the positive constant $s$ into the loss, along with $\lim_n \P(\cG_{n,\alpha_0}) = 1$, 
$(c)$ follows from the bound derived above,
and $(e)$ follows
from the monotonicity of $F_n^\bg(\cdot,\bG)$.
Since $\alpha >0$ and $\delta>0$ were arbitrary, sending $\alpha_0\to 0$ completes the proof of the first statement in the theorem.

Using a similar argument with the roles of $\bX$ and
$\bG$ exchanged gives the second statement.

\subsection{Proof of Theorem~\ref{thm:test_error}}
\label{section:outline_proof_test_error}
Again, we will use $\sfk = 1$ for simplicity, and without losing generality, since
the arguments that follow can be directly extended to the setting where $\sfk >0$ as long as it is a fixed constant.
First, we give the proof under conditions 
$\ref{item:a_test_error}$ and $\ref{item:c_test_error}$ deferring some technical lemmas to Section~\ref{section:proof_test_error} of the Appendix.

For $s\in\R$, let us define the modified empirical risks
\begin{align}
\nonumber
\widehat R_{n,s}(\btheta; \bX ,\by(\bX))&:= \widehat R_n(\btheta;\bX, \by(\bX)) + s 
R_n^\bg(\btheta) \\                     
\widehat R_{n,s}(\btheta; \bG, \by(\bG))&:= \widehat R_n(\btheta;\bG ,\by(\bG)) + s 
R_n^\bg(\btheta)
\label{eq:modified_risk}
\end{align}
(note the asymmetry),
and use $\widehat\btheta_s^\bX,\widehat\btheta_s^\bG$ to denote their unique minimizers respectively. Furthermore, we write  
$\widehat R_{n,s}^\star(\bX ,\by(\bX))$ and $\widehat R_{n,s}^\star(\bG ,\by(\bG))$ for the minima.

Further define, for $r \neq 0$, the differences
\begin{align}
\label{eq:D_def}
 D_n^\bX(r) &:= \frac{\widehat R_{n,r}^\star\left(\bX,\by(\bX)\right)  - \widehat R_{n}^\star\left(\bX,\by(\bG)\right)}{r},\hspace{1mm}\\
 D_n^\bG(r) &:= \frac{\widehat R_{n,r}^\star\left(\bG,\by(\bG)\right)  - \widehat R_{n}^\star\left(\bG,\by(\bG)\right)}{r}.
 \nonumber
\end{align}

The following lemma, which holds as a consequence of Theorem~\ref{thm:main_empirical_univ} and whose proof is deferred to Section~\ref{section:proof_test_error}
of the Appendix, allows us to deduce
the universality of these differences for fixed $r\neq 0$.

\begin{lemma}
\label{lemma:D_universality}
With
the definitions of Eq.~\eqref{eq:D_def},
under the setting of 
$\ref{item:a_test_error}$ and $\ref{item:c_test_error}$
of Theorem~\ref{thm:test_error}, we have
for any $t\in\R, s > 0$ and $\delta>0$,
\begin{align}
\label{eq:lim_P_DX}
\limsup_{n\to\infty}\P\left( D^\bX(-s)  \ge t + \delta  \right) &\le \limsup_{n\to\infty} \P\left(D^\bG(-s) \ge t  \right)\\
\limsup_{n\to\infty}\P\left( D^\bX(s)  \le t - \delta  \right) &\le \limsup_{n\to\infty} \P\left(D^\bG(s) \le t  \right).
\label{eq:lim_P_DG}
\end{align}
\end{lemma}

Meanwhile, note that for $s>0$, $R_n^\bg(\widehat \btheta_0^\bX)$ is sandwiched between
$D_n^\bX(s)$ and $D_n^\bX(-s)$. Indeed, we have
\begin{align}
\nonumber
D_n^\bX(s) &\le 
\frac{\widehat R_{n,s}\left(\widehat\btheta^\bX_0;\bX,\by(\bX)\right)  - \widehat R_n\left(\widehat\btheta^\bX_0;\bX,\by(\bX)\right)}{s}
= R_n^\bg(\widehat\btheta^\bX_0)\nonumber\\
&= \frac{\widehat R_{n,-s}\left(\widehat\btheta^\bX_0;\bX,\by(\bX)\right) - \widehat R_n(\widehat \btheta^\bX_0;\bX,\by(\bX))}{-s}
\le  D_n^\bX(-s). \label{eq:DX_upper_bound}
\end{align}
Analogously, we can derive 
\begin{equation}
\label{eq:DG_lower_bound}
D_n^\bG(s) \le R_n^\bg(\widehat \btheta_0^\bG)\le D_n^\bG(-s).
\end{equation}

So to establish the claim of Theorem~\ref{thm:test_error} under 
conditions~\ref{item:a_test_error} and~\ref{item:c_test_error}, we show that these conditions 
imply that we can make $D_n^\bX(-s)- D_n^\bX(s)$ and $D_n^\bG(-s)- D_n^\bG(s)$ arbitrarily small by taking $s>0$ arbitrarily small.

\paragraph*{Proof under condition~\ref{item:a_test_error}}

Under the strong convexity assumption, we have the following lemma whose proof is deferred to Section~\ref{section:proof_test_error}
of the Appendix.

\begin{lemma}
\label{lemma:theta_lipschitz_in_s}
In the setting of~\ref{item:a_test_error} of Theorem~\ref{thm:test_error},
for all $s\in\R$, $n\in\Z_{>0}$, we have
\begin{equation}
\label{eq:lipschitz_in_s}
 \big\|\widehat \btheta_s^\bX - \widehat \btheta_{-s}^\bX\big\|_2 \le C' |s|
\end{equation} 
for some $ C'>0$ depending only on $\sOmega$.
A similar inequality also holds for $\widehat\btheta_{s}^\bG$.
Consequently, for all $n\in \Z_{>0}$ and $s >0$, we have
\begin{equation}
\label{eq:lip_D}
   D^\bX(-s) - D^\bX(s) \le C\,s \quad\textrm{and}
   \quad
   D^\bG(-s) - D^\bG(s) \le C\,s 
\end{equation}
for $C>0$ depending only on $\sOmega$.
\end{lemma}

Now to establish the universality of the test error,
for any $\delta>0$, take $s_\delta\in(0, \delta/C)$ where $C$ is the constant appearing in Eq.~\eqref{eq:lip_D} of Lemma~\ref{lemma:theta_lipschitz_in_s} and write
\begin{align*}
\limsup_{n\to\infty}\P\left(R_n^\bg\big(\widehat \btheta_0^\bX \big) \ge \widetilde \rho + 3\delta  \right)   
&\stackrel{(a)}\le 
\limsup_{n\to\infty}\P\left( 
D^\bX(-s_\delta) \ge \widetilde \rho + 3\delta 
\right)\\
&\stackrel{(b)}\le 
\limsup_{n\to\infty}\P\left( 
D^\bG(-s_\delta) \ge \widetilde \rho + 2\delta 
\right)\\
&\stackrel{(c)}\le 
\limsup_{n\to\infty}\P\left( 
D^\bG(s_\delta) + C s_\delta \ge \widetilde \rho + 2\delta 
\right)\\
&\stackrel{(d)}\le 
\limsup_{n\to\infty}\P\left( 
R_n^\bg\big(\widehat \btheta_0^\bG\big)  \ge \widetilde \rho + \delta
\right)\\
&\stackrel{(e)}= 0,
\end{align*}
 where $(a)$ follows by~\eqref{eq:DX_upper_bound}, $(b)$ follows by Lemma~\ref{lemma:D_universality}, 
$(c)$ follows by Lemma~\ref{lemma:theta_lipschitz_in_s},
 $(d)$ follows by the lower bound in Eq.~\eqref{eq:DG_lower_bound}
 and the definition of $s_\delta$ and $(e)$ is by 
 the assumption that $R_n^\bg(\widehat \btheta_0^\bG) \stackrel{\P}{\to} \tilde\rho.$
 An analogous argument then shows
 \begin{equation*}
\limsup_{n\to\infty}\P\left(R_n^\bg\big(\widehat \btheta_0^\bX \big) \le \widetilde \rho - 3\delta  \right)   = 0.
 \end{equation*}
 Therefore,  $R^\bg_n(\widehat\btheta_0^\bX) \stackrel{\P}{\to} \widetilde\rho$. 
 To conclude the proof, note that the condition $\cC_p\subseteq\cS_p$
implies that 
\begin{equation}
 \label{eq:test_error_pointwise}
\left|R^g_n\big(\widehat \btheta_0^\bX\big) - R^\bx_n\big(\widehat \btheta_0^\bX\big) \right| \to 0
\end{equation}
 yielding the statement of the theorem under
 condition~$\ref{item:a_test_error}$. 
 Indeed, Eq.~\eqref{eq:test_error_pointwise} is justified by the technical Lemma~\ref{lemma:proj_locally_lip}
of the Appendix.

\paragraph*{Proof under condition~$\ref{item:c_test_error}$}
Under the differentiability condition on $\rho$, we define
\begin{equation}
\nonumber
    \Delta \rho (t) := \frac{\rho(t) - \rho(0)}{t},
\end{equation}
and write, using the bounds in~\eqref{eq:DX_upper_bound} and~\eqref{eq:DG_lower_bound}
\begin{align*}
    \P\left( 
    R_n^\bg\big(\widehat \btheta_0^\bX \big)   \ge R_n^\bg\big(\widehat \btheta_0^\bG\big) + 3\delta
    \right)\hspace{-0.8mm}
    &\le
    \P\left( 
    \left| \Delta \rho(-s)  - D_n^\bX(-s) \right| \ge \delta
    \right)\\
    & 
    \hspace{0mm}
    +
    \P\left( 
    \left|\Delta \rho(s) - D_n^\bG(s) \ \right|\ge \delta
    \right)
    +\P\left( 
    \Delta \rho(-s)  \ge \Delta \rho(s) + \delta
    \right)
\end{align*}
for any $\delta>0$ and $s>0$.
Now recall that by condition~\ref{item:c_test_error}, $\widehat R_{n,s}^\star(\bG,\by(\bG)) \stackrel{\P}\to \rho(s)$ for all $s$ in some neighborhood of $0$. Theorem~\ref{thm:main_empirical_univ} then implies the same for the model with $\bX$, i.e.,
$\widehat R_{n,s}^\star(\bX,\by(\bX)) \stackrel{\P}\to \rho(s)$. 
The universality of $D_n$ from Lemma~\ref{lemma:D_universality} combined with the previous display then gives
\begin{align*}
    \lim_{n\to\infty}\P\left( 
    R^\bg\big(\widehat \btheta_0^\bX \big)   \ge R^\bg\big(\widehat \btheta_0^\bG\big) + 3\delta
    \right) &= \lim_{s\to 0} \P\left(\Delta \rho(-s) \ge \Delta\rho(s) + \delta \right)\\
    &\stackrel{(a)}= 0
\end{align*}
where $(a)$ follows by differentiability of $\rho(s)$ at $s=0$. By exchanging the roles of $\bX$ and $\bG$ in the above and applying the argument once more we conclude
that $\left|R^g(\widehat\btheta_0^\bX) - R^g(\widehat\btheta_0^\bG)  \right| \stackrel{\P}\to 0$. 
The condition $\cC_p\subseteq\cS_p$ again implies the desired result as in the previous subsection.

\paragraph*{Proof under condition~$\ref{item:b_test_error}$}

Let $\cA_{n,\delta,\alpha}$ be the event in condition~$\ref{item:b_test_error}$, namely,
\begin{equation}
\nonumber
 \cA_{n,\delta,\alpha} := \left\{ \min_{\left\{\btheta\in\cC_p : \left| 
R_n^\bg(\btheta) - \widetilde\rho  
\right|\ge \alpha\right\}} \left|  \widehat R_n\left(\btheta;\bG,
\by(\bG)\right) - \widehat 
R_n^\star\left( \bG,\by(\bG)\right) \right| \ge \delta\right\},
\end{equation}
and take $\widehat\btheta_n^\bG$ and $\widehat\btheta_n^\bX$ to be any minimizers of $\widehat R_n(\btheta;\bG,\by(\bG))$ and $\widehat R_n(\btheta;\bX,\by(\bX))$ respectively.
First, note that this directly implies
\begin{equation}
\label{eq:g_test_error}
 \left| R_n^\bg\big(\widehat \btheta_n^\bG\big) - \widetilde \rho\;\right| \stackrel{\P}{\to} 0.
\end{equation}
Indeed, we have for all $\alpha >0$,
\begin{align*}
\P\left(
\big|
R_n^\bg\big(\widehat \btheta_n^\bG\big) - \widetilde\rho \;
\big| \ge \alpha
\right) 
&\le
\P\left(
\left\{\big|
R_n^\bg\big(\widehat \btheta_n^\bG\big) - \widetilde\rho \;
\big| \ge \alpha\right\} \bigcap \cA_{n,\delta,\alpha} 
\right)  + \P\left( \cA_{n,\delta,\alpha}^c\right)\\
&=  \P\left(  
\cA_{n,\delta,\alpha}^c\right)
\end{align*}
for any $\delta>0$. Now choosing $\delta >0$ so that $\lim_{n\to\infty} 
\P(\cA_{n,\delta,\alpha}^c) =0$ proves~\eqref{eq:g_test_error}. Next, we show that
\begin{equation}
\nonumber
 \big| R_n^\bg\big(\widehat \btheta_n^\bX\big) - \widetilde\rho\; \big| \stackrel{\P}{\to} 0
\end{equation}
as a consequence of Theorem~\ref{thm:main_empirical_univ} along with the assumption that 
$\widehat R_n\big(\widehat\btheta_n^\bG; \bG,\by(\bG)\big)\stackrel{\P}\to\rho.$
Indeed, assume the contrary
and choose for any $\alpha>0$, $\delta := \delta_\alpha$ so that 
$\P(\cA_{n,\delta_\alpha, \alpha}^c) \to 0$. We have
\begin{align*}
 &\P\big( \big| \widehat R_n\big(\widehat \btheta_n^\bX;\bX,\by(\bX)\big) - \widehat R_n\big(\widehat 
\btheta_n^\bG; \bG,\by(\bG)\big)\big| < \delta_\alpha\bi)\\
&\le \P\big(\big\{\big| \widehat R_n(\widehat \btheta_n^\bX;\bX,\by(\bX)) 
- \widehat R_n(\widehat\btheta_n^\bG; \bG,\by(\bG)) \big| \hspace{-1mm}<\hspace{-1mm} \delta_\alpha\big\} 
\cap \big\{\big| 
R_n^\bg(\widehat \btheta_n^\bX) -\widetilde\rho \;
\big| \ge \alpha\big\}
\cap \cA_{n,\delta_\alpha, \alpha}
\big)\\
&\hspace{5mm}+ \P\left(\cA_{n,\delta_\alpha, \alpha}^c \right)
+
\P\left(\left| 
R_n^\bg\left(\widehat \btheta_n^\bX\right) - \widetilde \rho 
\right| < \alpha
\right) 
\\
&= \P\left(\cA_{n,\delta_\alpha, \alpha}^c \right)
+
\P\left(\left| 
R_n^\bg\left(\widehat \btheta_n^\bX\right) - \widetilde \rho \;
\right| < \alpha
\right) 
\end{align*}
Sending $n\to\infty$, we have
\begin{align*}
 &\limsup_{n\to\infty} 
 \P\left( \left| \widehat R_n\left(\widehat \btheta_n^\bX;\bX,\by(\bX)\right) - \widehat R_n\left(\widehat 
\btheta_n^\bG; \bG ,\by(\bG)\right)\right| < \delta_\alpha\right) \\
&\hspace{70mm}\le 
\limsup_{n\to\infty} \P\left(\left| 
R_n^\bg\left(\widehat \btheta_n^\bX\right) -\widetilde\rho \;
\right| < \alpha
\right)  \stackrel{(a)}{<} 1,
\end{align*}
where $(a)$ follows since we assumed that 
$\left| R_n^\bg(\widehat\btheta_n^\bX) - \widetilde\rho\; \right|$  does not converge to $0$ in 
probability. This directly contradicts $\left|\widehat R_n^\star(\bX, \by(\bX)) -
\widehat R_n^\star( \bG,\by(\bG))\right| \stackrel{\P}{\to} 0$; a consequence of Theorem~\ref{thm:main_empirical_univ} and the assumption that
$\widehat R_n^\star(\bG,\by(\bG)) \stackrel{\P}{\to} \rho$ in condition~$\ref{item:b_test_error}$.
Meanwhile, note that as in the previous subsection, we have
$\left| R_n^\bg\left(\widehat\btheta_n^\bX\right) - R_n^\bx\left(\widehat\btheta_n^\bG\right) \right| \stackrel{\textrm{a.s.}}{\to} 0$,
hence, we have for all $\alpha >0$,
\begin{align}
\nonumber
\lim_{n\to\infty}\P\left(\left| R_n^\bx\left(\widehat\btheta_n^\bX\right) - \widetilde\rho\; \right|  > \alpha \right) 
\le \lim_{n\to\infty}\P\left(  \left| R_n^\bg\left( \widehat \btheta_n^\bX\right) - \widetilde\rho \;\right| > 
\frac{\alpha}{2}\right) = 0.
\end{align}

\subsection{Proof outline for Theorem~\ref{cor:ntk_universality}}
\label{section:proof_outline_ntk}
Recall the definitions and assumptions in Section~\ref{section:ntk_example}.
The main step in proving Theorem~\ref{cor:ntk_universality} is showing that the distribution of the feature vectors $\{\bx_i\}_{i\le n}$ 
satisfy, on a high probability set, Assumption~\ref{ass:X} and Eq.~\eqref{eq:condition_bounded_lipschitz_single} for the set $\cS_p$ of Eq.~\eqref{eq:ntk_set}. The statement then follows from Theorem~\ref{thm:main_empirical_univ}.
Our proof here is analogous to that of~\cite{hu2020universality} for the random features model.
Let us begin our treatment by defining the event
\begin{equation}
\nonumber
    \cB := \left\{\sup_{\{i,j \in [m] : i\neq j\}} \left|\bw_i^\sT\bw_j\right| \le 
C\left(\frac{\log m}{d}\right)^{1/2}\right\} \cap 
    \left\{
    \norm{\bW }_\op
    \le C'
    \right\}
\end{equation}
for some $C,C'$ depending only on $\widetilde\sgamma_\NT$ so that $\P(\cB^c) \to0$
as $n\to\infty$. The existence of such constants is a standard result (see, for example,~\cite{vershynin2018high}.) However, we include it as Lemma~\ref{lemma:B_tail_bound}
of Section~\ref{section:aux_lemmas_ntK} of the Appendix
for completeness.

Here, we outline the proof of the pointwise normality condition for the distribution of the features, conditional on $\bW\in\cB$. 
Most technical details, along with the deduction of Theorem~\ref{cor:ntk_universality} from this condition, are deferred 
to Section~\ref{section:aux_lemmas_ntK}
of the Appendix. 

Throughout, we will be working conditionally on $\bW\in\cB$, so let us simplify notation by using $\E[\cdot] := \E[\cdot\, \one_\cB \big| \bW]$. 
For a given $\delta>0$, let us define the set
\begin{equation}
\nonumber
\cS_{p,\delta} := \left\{\btheta\in\cS_p : \btheta^\sT \E\left[\bx\bx^\sT\right]\btheta > \delta\right\},
\end{equation}
We will control the difference above along projections of $\btheta$ in $\cS_{p,\delta}$ and $\cS_{p,\delta}^c$ separately.
The following lemma does so along $\cS_{p,\delta}$.

\begin{lemma}
\label{lemma:ntk_bl}
For all $\delta>0$ and any differentiable bounded function $\varphi:\R\to\R$ with bounded derivative, we have
\begin{equation}
    \lim_{n\to\infty} \sup_{\btheta\in \cS_{p,\delta}} 
\left|\E\left[\varphi\left(\btheta^\sT\bx\right)\one_{\cB} \Big| \bW 
\right]  
- 
\E\left[\varphi\left(\btheta^\sT\bg\right)\one_{\cB}\Big
| \bW 
\right]\right| 
= 0.
\end{equation}
\end{lemma}

Meanwhile, taking $\varphi$ to be bounded differentiable with bounded derivative, we have for $\delta >0$,
\begin{align}
&\lim_{n\to\infty} \sup_{\btheta\in \cS_{p}} 
\left|\E\left[\varphi\left(\btheta^\sT\bx\right) 
\right]  
- 
\E\left[\varphi\left(\btheta^\sT\bg\right) 
\right]\right|\nonumber\\
&\hspace{40mm}\stackrel{(a)}\le 
\lim_{n\to\infty} \sup_{\btheta\in \cS_{p,\delta}^c} 
\left|\E\left[\varphi\left(\btheta^\sT\bx\right) 
\right]  
- 
\E\left[\varphi\left(\btheta^\sT\bg\right) 
\right]\right| \nonumber\\
&\hspace{40mm} \le
\lim_{n\to\infty} \sup_{\btheta\in \cS_{p,\delta}^c} 
\norm{\varphi'}_\infty\left(\E\left[\left(\btheta^\sT\bx\right)^2 \right]^{1/2}
+
\E\left[\left(\btheta^\sT\bg\right)^2  \right]^{1/2}
\right)
\nonumber\\
&\hspace{40mm}\stackrel{(b)}\le 2\norm{\varphi'}_{\infty} \delta\label{eq:ntk_bl_delta_bound}
\end{align}
where $(a)$ follows from Lemma~\ref{lemma:ntk_bl} and $(b)$ follows from the definition of $\cS_{p,\delta}^c$.
Now sending $\delta \to 0$ proves the lemma for differentiable Lipschitz functions, which can then be extended
to Lipschitz functions via a standard uniform approximation argument.
Let us now outline the proof of Lemma~\ref{lemma:ntk_bl}.

Define, for $\btheta\in\cS_{p,\delta}$ the notation
\begin{equation}
\nonumber
    \nu^2 = \nu^2_\btheta :=  \btheta^\sT \E\left[\bx\bx^\sT\right] \btheta > \delta.
\end{equation}

For a fixed bounded Lipschitz function $\varphi$, let $\chi = \chi_\varphi$ be the 
solution to Stein's equation for $\varphi$, namely, the function $\chi$ satisfying
\begin{equation}
\nonumber
  \E\left[ \varphi\left(\frac{\btheta^\sT \bx}{\nu}\right) - \varphi\left(\frac{\btheta^\sT \bg}{\nu}\right)\right]
  =
 \E\left[\chi'\left(\frac{\btheta^\sT\bx}{\nu}\right)  - 
\frac{\btheta^\sT\bx}{\nu}\chi\left(\frac{\btheta^\sT\bx}{\nu}\right)   \right]
\end{equation}
(see~\cite{chen2011normal} for more on Stein's method and properties of the solution $\chi$.).
In order to prove Lemma~\ref{lemma:ntk_bl}, it is sufficient to show that
\begin{equation}
\label{eq:stein_equiv_1}
   \lim_{n\to\infty} \sup_{\btheta\in \cS_{p,\delta}}\left|\E\left[\chi'\left(\frac{\btheta^\sT\bx}{\nu}\right)  - 
\frac{\btheta^\sT\bx}{\nu}\chi\left(\frac{\btheta^\sT\bx}{\nu}\right)   \right]\right|  = 0.
\end{equation}

To simplify notation, define
\begin{equation}
\nonumber
    \Delta_i := \frac{\btheta^\sT\bx}{\nu} - \frac{1}{\nu}\sum_{j: j\neq i 
}\btheta_\up{j}^\sT\bP_{i}^\perp\bz \sigma'(\bw_j^\sT\bz - \rho_{i,j}\bw_i^\sT \bz),
\end{equation}
where 
\begin{equation}
\nonumber
    \bP_{i}^\perp := \bI - \bw_i\bw_i^\sT,\quad \rho_{ij}:= 
{\bw_j^\sT\bw_i}.
\end{equation}
In Section~\ref{section:ntk_actual_proof} of the Appendix, we upper bound the quantity~\eqref{eq:stein_equiv_1} as 
\begin{align}
\label{eq:decomposition_1}
    &\left|\E\left[\chi'\left(\frac{\btheta^\sT\bx}{\nu}\right)  - 
\frac{\btheta^\sT\bx}{\nu}\chi\left(\frac{\btheta^\sT\bx}{\nu}\right)   \right]\right|  \\
    &\hspace{7mm}\le \left|
    \E\left[\left(\frac{1}{\nu}\sum_{i=1}^m \btheta_\up{i}^\sT\bz 
\sigma'(\bw_i^\sT\bz)\Delta_i 
- 1 
    \right)\chi'\left(\frac{\btheta^\sT\bx}{\nu}\right)\right]
    \right|\nonumber\\
    &\hspace{15mm}+
    \left|
    \E\left[
    \frac{1}{\nu}\sum_{i=1}^m
    \btheta_\up{i}^\sT\bz\sigma'(\bw_i^\sT\bz)
    \left(\chi\left(\frac{\btheta^\sT\bx}{\nu}\right) 
    -
    \chi\left(\frac{\btheta^\sT\bx}{\nu} - \Delta_i\right) 
    - \Delta_i\chi'\left(\frac{\btheta^\sT\bx}{\nu}\right)
    \right)
    \right]
    \right|.
\nonumber
\end{align}
So we control each of the terms separately. 
\paragraph*{Bounding the first term in Eq.~\eqref{eq:decomposition_1}}
Fixing $\delta>0$ throughout and denoting
\begin{equation}
\nonumber
    U := \frac{1}{\nu}\sum_{i=1}^m \btheta_\up{i}^\sT \bz \sigma'(\bw_i^\sT \bz)\Delta_i,
\end{equation}
we compute the expectation of $U$ and control its variance. The expectation can be computed as
\begin{align*}
&\E[U] = \E\left[
\frac{1}{\nu}\sum_{i=1}^m \btheta_\up{i}^\sT \bz \sigma'(\bw_i^\sT \bz)
\left(\frac{\btheta^\sT\bx}{\nu} +  \Delta_i - \frac{\btheta^\sT\bx}{\nu} 
\right)
\right]\\
&= 
\E\left[
\frac{1}{\nu^2}
\left({\btheta^\sT\bx}\right)^2\right] +
\frac{1}{\nu}\sum_{i=1}^m \E\left[\btheta_\up{i}^\sT \bz \sigma'(\bw_i^\sT 
\bz)\left(\Delta_i - \frac{\btheta^\sT\bx}{\nu}\right)\right]\\
&\stackrel{(a)}{=}1
+\E\left[\btheta_\up{i}^\sT\bP_i^\perp\bz 
\sigma'(\bw_i^\sT\bz)
\left(\Delta_i - \frac{\btheta^\sT \bx}{\nu}\right)
\right] + 
    \btheta_\up{i}^\sT\bw_i\E\left[\bw_i^\sT\bz\sigma'(\bw_i^\sT\bz)\left(\Delta_i - \frac{\btheta^\sT \bx}{\nu}\right)\right] \\
    &\stackrel{(b)}{=}1 + 
    \E\left[\btheta_\up{i}^\sT\bP_i^\perp\bz 
\left(\Delta_i - \frac{\btheta^\sT \bx}{\nu}\right)
\right] 
\E\left[
\sigma'(\bw_i^\sT\bz)
\right]\\
&\hspace{10mm}+ 
\btheta_\up{i}^\sT\bw_i\E\left[\bw_i^\sT\bz\sigma'(\bw_i^\sT\bz)\right]\E\left[\left(\Delta_i - \frac{\btheta^\sT \bx}{\nu}\right)\right] \\
    &\stackrel{(c)}{=}1
\end{align*}
where $(a)$ follows by the definition of $\nu$, $(b)$ follows by independence of $\Delta_i - \btheta^\sT \bx/\nu$ and $\bw_i^\sT \bz$, which 
can be seen from the definition of $\Delta_i$, and $(c)$ follows by the assumption on $\sigma'$, namely, that 
$\E[\sigma'(G)] = \E[ G\sigma'(G)] = 0$ for $G$ standard normal.\\

We then control $\Var(U)$ by expressing it via a Taylor expansion as
\begin{align}
    U =& \frac{1}{\nu^2}\sum_{i=1}^m \left(\btheta_\up{i}^\sT \bz \sigma'(\bw_i^\sT 
\bz)\right)^2\label{eq:u_1_z_1}\\
    &+ \frac1{\nu^2}\sum_{i,j: j\neq i} \btheta_\up{i}^\sT\bz 
\sigma'(\bw_i^\sT\bz)\btheta_\up{j}^\sT \bw_i  \bw_i^\sT \bz 
\sigma'(\bw_j^\sT \bz)\label{eq:u_2_z_1}\\
    &+ \frac1{\nu^2} \sum_{i,j: j\neq i}\btheta_\up{i}^\sT\bz 
\sigma'(\bw_i^\sT\bz)\widetilde\btheta_{j,i}^\sT\bz 
    \left\{ \rho_{ij} \bw_i^\sT \bz \sigma''(\bw_j^\sT \bz)
    -\frac12\rho_{ij}^2(\bw_i^\sT \bz)^2 \sigma'''(\bw_j^\sT\bz)\right\}\label{eq:u_3_z_1}\\
    &+ \frac1{6\nu^2}\sum_{i,j: j\neq i}\btheta_\up{i}^\sT\bz
    \sigma'(\bw_i^\sT\bz)\widetilde\btheta_{j,i}^\sT\bz \rho_{ij}^3 (\bw_i^\sT\bz)^3 
\sigma^\up{4}(v_{ij}(\bz)).\label{eq:u_4_z_1}
\end{align}
for some $v_{ij}$
where
$\widetilde \btheta_{j,i} := \bP_{i}^\perp\btheta_\up{j}.$
Writing $u_1(\bz)$-$u_4(\bz)$ for the terms on the right-hand side of  
lines \eqref{eq:u_1_z_1}-\eqref{eq:u_4_z_1} respectively,
we observe that
 \begin{equation}
 \nonumber
\Var(U)^{1/2} \le \sum_{l=1}^4 \Var(u_l(\bz))^{1/2} \stackrel{(a)}{\le}C_0 \sum_{l=1}^3 
\left(\E\left[\norm{\grad u_l(\bz)}_2^2\right]\right)^{1/2} + C_0\Var(u_4(\bz))^{1/2}
 \end{equation}
where $(a)$ follows from the Gaussian Poincar\'e inequality. 
We control each summand directly in Section~\ref{section:ntk_actual_proof} of the Appendix to conclude that
\begin{equation}
 \nonumber
\lim_{n\to\infty} \sup_{\btheta\in\cS_{p,\delta}} \Var(U)^{1/2} = 0.
\end{equation}
Therefore, we can control the first term in~\eqref{eq:decomposition_1} as
\begin{align*}
   \lim_{n\to\infty} \sup_{\btheta\in\cS_{p,\sfk}} \left|\E\left[(U-1)\chi'\left(\frac{\btheta^\sT \bx}{\nu}\right)\right]\right|
    \le \norm{\chi'}_\infty
    \lim_{n\to\infty} \sup_{\btheta\in\cS_{p,\sfk}} \left( \Var(U)^{1/2} + |\E[U - 1]|  \right)
    = 0.
\end{align*}

\paragraph*{Bounding the second term in Eq.~\eqref{eq:decomposition_1}}
Let us define the event 
\begin{align*}
\cA&:= \left\{
\sup_{i\in[m]} \left| \frac{1}{\norm{\btheta_\up{i}}} \btheta_\up{i}^\sT\bz\right| \le \left(\log m\right)^{50} \right\}
\bigcap\left\{
\sup_{i\in[m]} \left| \bw_i^\sT\bz\right| \le \left(\log m\right)^{50} \right\}\\
&\hspace{55mm}\bigcap
\left\{\sup_{\{(i,j)\in[m]^2 : i\neq j\}}
 \left|
  \frac{1}
 {\norm{\widetilde\btheta_{j,i}}_2 }\widetilde\btheta_{j,i}^\sT\bz
 \right|
 \le \left(\log m\right)^{50} 
\right\}
\end{align*}

Using that for $v_i$, not necessarily independent, subgaussian with subgaussian norm $1$ 
\begin{equation*}
\P\left(\sup_{i\in[m]}|v_i| > \sqrt{2\log m} + t  \right)  \le\exp\left\{-\frac{t^2}{2\sfK_v^2}\right\},
\end{equation*}
we obtain
\begin{align*}
\P\left(
\cA^c
\right) \le  
3\exp\left\{-\frac{c_0(\log m)^{99}}{2}\right\}
\end{align*}
for some universal constant $c_0\in(0,\infty)$.
Hence, it is sufficient to establish the desired bound on the set $\cA$. Indeed, suppose 
\begin{equation}
    \label{eq:second_term_ntk_target_1}
    \lim_{n\to\infty}\sup_{\btheta\in\cS_{p,\delta}}
    \bigg|
    \E\bigg[
    \frac{1}{\nu}\sum_{i=1}^m
    \btheta_\up{i}^\sT\bz\sigma'(\bw_i^\sT\bz)
    \Big(\chi\big(\frac{\btheta^\sT\bx}{\nu}\Big)
    -
    \chi\Big(\frac{\btheta^\sT\bx}{\nu} - \Delta_i\Big)
    - \Delta_i\chi'\Big(\frac{\btheta^\sT\bx}{\nu}\Big)
    \bigg)
    \one_{\cA}
    \bigg]
    \bigg|= 0,
\end{equation}
then 
\begin{align*}
    &\lim_{n\to\infty}\sup_{\btheta\in\cS_{p,\delta}} 
    \Big|
    \E\bigg[
    \frac{1}{\nu}\sum_{i=1}^m
    \btheta_\up{i}^\sT\bz\sigma'(\bw_i^\sT\bz)
    \bigg(\chi\Big(\frac{\btheta^\sT\bx}{\nu}\Big)
    -
    \chi\Big(\frac{\btheta^\sT\bx}{\nu} - \Delta_i\Big)
    - \Delta_i\chi'\Big(\frac{\btheta^\sT\bx}{\nu}\Big)
    \bigg)
    \bigg]
    \bigg|  \\
    &\hspace{5mm}\stackrel{(a)}{\le }
    \lim_{n\to\infty}\sup_{\btheta\in\cS_{p,\delta}} 
    \frac{C_1 m}{\nu}
    \left(\norm{\chi}_\infty
    \vee\|\chi'\|_\infty\right)
    \sup_{i\in[m]}\|\btheta_\up{i}\|_2
    \E\Big[
    \norm{\bz}_2
    \Big(2+ \sup_{i\in[m]}|\Delta_i|\Big)
    \one_{\cA^c}
    \Big]\\
    &\hspace{5mm}\stackrel{(b)}{\le }
    \lim_{n\to\infty}\sup_{\btheta\in\cS_{p,\delta}} 
    C_2(\nu) m^2 \exp\left\{-\frac{c_0(\log m)^{99}}{2}\right\}\\
    &\hspace{5mm}=0.
\end{align*}
where $(a)$ follows by a naive bound on $\Delta_i$ and $(b)$ follows by an application of H\"older's.
Hence, throughout we work on the event $\cA$.\\

By Lemma 2.4 of~\cite{chen2011normal}, $\chi' = \chi'_\varphi$ is differentiable and
 $\norm{\chi''}_\infty \le  C_3$ since $\varphi$ is assumed to be differentiable with bounded derivative.
Hence,
\begin{align}
\label{eq:2nd_term_taylor_1}
\left|\chi\left(\frac{\btheta^\sT\bx}{\nu}\right) -\chi\left( \frac{\btheta^\sT\bx}{\nu} 
- \Delta_i\right)  - \Delta_i \chi' \left( \frac{\btheta^\sT\bx}{\nu}\right)\right| 
&\le C_3 \left|\Delta_i \right|^2.
\end{align}
Using this in~\eqref{eq:second_term_ntk_target_1} we obtain
\begin{align}
    &\bigg|
    \E\bigg[
    \frac{1}{\nu}\sum_{i=1}^m
    \btheta_\up{i}^\sT\bz\sigma'(\bw_i^\sT\bz)
    \bigg(\chi\bigg(\frac{\btheta^\sT\bx}{\nu}\bigg) 
    -
    \chi\bigg(\frac{\btheta^\sT\bx}{\nu} - \Delta_i\bigg) 
    - \Delta_i\chi'\bigg(\frac{\btheta^\sT\bx}{\nu}\bigg)
    \bigg)
    \one_{\cA}
    \bigg]
    \bigg|    
\nonumber
    \\
    &\hspace{5mm}\stackrel{(a)}{\le}
    C_3
    \E\bigg[
    \frac{1}{\nu}\sum_{i=1}^m
    \Big|\btheta_\up{i}^\sT\bz\sigma'(\bw_i^\sT\bz)\Big|
    \Delta_i^2 \one_\cA
    \bigg]
    \nonumber
    \\
    &\hspace{5mm}\stackrel{(b)}{\le}
    \frac{C_4}{\nu}
    \E\bigg[\sup_{i\in[m]}\frac{
    |\btheta_\up{i}^\sT\bz|
    }{\|\btheta_\up{i}\|_2} 
    \sum_{i=1}^m \norm{\btheta_\up{i}}_2
     \Delta_i^2\one_{\cA}
    \bigg]\nonumber\\
    &\hspace{5mm}\stackrel{(c)}{\le} \frac{C_5}{\nu} \frac{(\log m)^{50}}{m^{1/2}} 
    \sum_{i=1}^m \E\left[
     \Delta_i^2\one_{\cA}
    \right],
\end{align}
where $(a)$ follows from~\eqref{eq:2nd_term_taylor_1}, $(b)$ follows from 
boundedness of $\norm{\sigma'}_\infty$, and $(c)$ follows from $\norm{\btheta_\up{i}}_2 \le \sfR/\sqrt{d}$ and the definition of $\cA$. 
Via another Taylor expansion of $\sigma'$, we write 
\begin{align}
    \Delta_i &= \frac{1}{\nu} \btheta_\up{i}^\sT \bz \sigma'(\bw_i^\sT \bz) 
    + \frac1{\nu}\sum_{j: j\neq i} \btheta_\up{j}^\sT  \bw_i  
\bw_i^\sT \bz \sigma'(\bw_j^\sT \bz)\nonumber\\
    &\hspace{10mm}+ \frac1{\nu} \sum_{j: j\neq i}\widetilde\btheta_{j,i}^\sT\bz 
     \rho_{ij} \bw_i^\sT \bz \sigma''(\bw_j^\sT \bz)-
     \frac1{\nu}\sum_{j:j\neq i}
     \widetilde \btheta_{j,i}^\sT \bz
     \rho_{ij}^2(\bw_i^\sT \bz)^2 \sigma'''(v_{j,i}(\bz))\nonumber\\
     &=: d_{1,i} + d_{2,i} + d_{3,i} + d_{4,i}
\end{align}
for some $v_{j,i}(\bz)$ between $\bw_j^\sT \bz$ and $\bw_j^\sT\bz - \rho_{ij}\bw_i^\sT\bz$.
In Section~\ref{section:second_term_ntk}
of the Appendix, we show directly that for each $k\in [4]$,
\begin{equation}
    \lim_{n\to\infty} \sup_{\btheta\in\cS_{p,\delta}} \frac{1}{\nu}
    \frac{(\log m)^{50}}{m^{1/2}} \sum_{i=1}^m\E\left[d_{k,i}^2\one_\cA\right]=0.
\end{equation} 
This then controls the second term in~\eqref{eq:decomposition_1}.

%
%

\subsection*{Acknowledgements}
This work was supported by the NSF through award DMS-2031883, the Simons Foundation through
Award 814639 for the Collaboration on the Theoretical Foundations of Deep Learning,
the NSF grant CCF-2006489, the ONR grant N00014-18-1-2729, and 
the National Science Foundation Graduate Research Fellowship Program under Grant No.~DGE-1656518.

\bibliographystyle{amsalpha} 
\bibliography{all-bibliography}

\newpage

\appendix

\section{Proof of Theorems~\ref{thm:MainFirst} and~\ref{thm:main_empirical_univ}}
\label{section:proof_of_main_results}
In this section, we complete the proof of Theorems~\ref{thm:MainFirst} and~\ref{thm:main_empirical_univ} by deducing them from Lemma~\ref{lemma:softmin_univ} and extending the result to hold under the remaining assumptions. 

\subsection{Proof of Theorems~\ref{thm:MainFirst} and~\ref{thm:main_empirical_univ}
}
\label{section:main_thm_proof}
Recall that for $\alpha>0$, 
in Section~\ref{section:proof_outline} 
we let $\cN_\alpha$ be a minimal $\alpha-$net of
$ \cC_p \subseteq B_2^p(\sfR)$, so that 
$|\cN_\alpha| \le C(\alpha)^p$ for some $C(\alpha)$.
Let us define the discretized minimization over $\bTheta\in\cN_\alpha^\sfk$
\begin{equation}
\label{eq:opt}
 \Opt_n^\alpha(\bX,\beps) := \min_{\bTheta \in \cN_\alpha^\sfk} 
\widehat R_n(\bTheta;\bX,\beps).
\end{equation}

We have the following consequence of Lemma~\ref{lemma:softmin_univ}.
\begin{lemma}[Universality of $\Opt_n^\alpha$]
\label{prop:eps_net_univ} 
Under the assumptions of Theorem~\ref{thm:MainFirst}, along with
the alternative Assumption~\hyperref[ass:loss_labeling_prime]{5''},
we have for any bounded differentiable function $\psi$ with bounded Lipschitz derivative
\begin{equation}
\nonumber
\lim_{n\to\infty}\left|\E\left[\psi\left(\Opt_n^\alpha\left(\bX,\beps\right)\right) \right] - 
\E\left[\psi\left(\Opt_n^\alpha\left(\bG,\beps\right)\right)\right]\right| = 
0.
\end{equation}
\end{lemma}

The proof of this result is deferred to Section~\ref{section:universality_eps_net_proof}.
Here, we show that Theorem~\ref{thm:MainFirst}, under the alternative Assumption~\hyperref[ass:loss_labeling_prime]{5''}
is a direct consequence of this lemma. First, we 
need a few technical lemmas.
Let us define the restricted operator norm 
\begin{equation}
\nonumber
 \norm{\bX}_{\cS_p} := \sup_{\left\{\btheta\in\cS_p : \norm{\btheta}_2 \le 1\right\}} \norm{\bX 
\btheta}_2.
\end{equation}
\begin{lemma}
\label{lemma:op_norm_x}
 Let $\bX,\bG, \cS_p, p(n)$ be as in Theorem~\ref{thm:MainFirst}, we have for some $C \in(0,\infty)$,
 \begin{equation}
\E\left[\norm{\bX}_{\cS_p}^2\right] \le  C p,\quad 
\E\left[\norm{\bG}_{\cS_p}^2\right] \le  C p.
\nonumber
 \end{equation}
\end{lemma}

\begin{lemma}
\label{lemma:Rn_lipschitz_bound}
Under the assumptions of Theorem~\ref{thm:MainFirst} with the alternative 
Assumption~\hyperref[ass:loss_labeling_prime]{5''}, we have for all 
$\bTheta,\widetilde\bTheta\in\cS_p^\sfk$
\begin{equation}
\left|\widehat 
R_n(\bTheta;\bX,\beps) - 
 \widehat R_n(\widetilde\bTheta;\bX,\beps)\right|
\le C \left( \frac{\norm{\bX}^2_{\cS_p}}{n} + 
\frac{\norm{\bX}_{\cS_p}\norm{\beps}_2}{n}+ 1 \right) \norm{\bTheta -\widetilde \bTheta}_F,
\nonumber
\end{equation}
for some constant $C > 0$. A similar bound also holds for the Gaussian model.
\end{lemma}
The proofs are deferred to  
Sections~\ref{proof:op_norm_x} and~\ref{proof:Rn_lipschitz_bound} respectively. Here, we derive Theorem~\ref{thm:MainFirst}.

\subsubsection[Proof of Theorem 1 under Assumption 5'']{Proof of Theorem~\ref{thm:MainFirst} under Assumption~\hyperref[ass:loss_labeling_prime]{5''}}
\label{section:opt_to_erm}

We state the claim as a lemma for future reference.
\begin{lemma}
\label{lemma:eq_MainFirst}
Consider the setting of Theorem~\ref{thm:MainFirst}.
Let Assumptions~\ref{ass:regime}-\ref{ass:X} and \hyperref[ass:loss_labeling_prime]{5''} hold.
Let $\bG\in\reals^{n\times p}$ be a matrix with i.i.d. rows $\bg_i$ distributed as
in Assumption \ref{ass:X}, and
variables $(\eps_i: i\le n)$ be independent and independent of $\bX$, $\bG$, with $\max_{i\le n}\|\eps_i\|_{\psi_2}\le \sfK$\,  for some constant $\sfK.$

If $\cC_{j,p}\subseteq \cS_p$ for all $j\le \sfk$, then  for any
bounded Lipschitz function $\psi:\R\to\R$,
\begin{equation}
 \lim_{n\to\infty}\left|\E\left[\psi\left(\widehat 
R^\star_n\left(\bX\right)\right)\right] - 
\E\left[\psi\left(\widehat R^\star_n\left(\bG\right) \right)\right]\right| = 0.
\end{equation}    
\end{lemma}
\begin{proof}

Let $$\widehat\bTheta_\bX := \left(\widehat 
\btheta_{\bX,1},\dots,\widehat\btheta_{\bX,\sfk}\right) $$
be a minimizer of 
$\widehat R_n(\bTheta;\bX,\by(\bX))$, and then let
 $$\widetilde\bTheta_\bX := \big(\widetilde \btheta_{\bX,1},\dots,\widetilde 
\btheta_{\bX,\sfk}\big)$$
where $\widetilde\btheta_{\bX,k}$ is the closest point in $\cN_\alpha$ to $\widehat\btheta_{\bX,k}$ in 
$\ell_2$ norm.
We have
\begin{align}
   \left| \widehat R_n^\star(\bX,\beps) - \Opt_n^\alpha(\bX,\beps)\right| 
    &\stackrel{(a)}{=} \left(
    \Opt_n^\alpha(\bX,\beps) - \widehat R^\star_n(\bX,\beps)
    \right)\nonumber\\
    &\stackrel{(b)}{\le} \left( \widehat R_n(\widetilde \bTheta_{\bX};\bX, \beps ) 
    - \widehat R_n^\star(\bX,\beps)\right)\nonumber\\
&\stackrel{(c)}{=}\left|\widehat R_n (\widehat\bTheta_\bX; \bX,\beps) - \widehat R_n(\widetilde 
\bTheta_\bX;\bX,\beps)\right|\nonumber\\
&\stackrel{(d)}{\le}
  C_0 \left( \frac{\norm{\bX}^2_{\cS_p}}{n} + 
\frac{\norm{\bX}_{\cS_p}\norm{\beps}_2}{n}+ 1 \right) \sfk^{1/2} \alpha,\nonumber
\end{align}
where in $(a)$ we used that 
$\Opt_n^\alpha(\bX,\beps) 
\ge \widehat R_n^\star(\bX,\beps)$, 
in $(b)$ we used the inequality $\Opt_n^\alpha(\bX,\beps) \le \widehat R_n(\widetilde \bTheta_\bX; \bX,\beps)$, in $(c)$ we used $\widehat R_n(\widetilde \bTheta_{\bX};\bX, \beps )\ge \widehat R_n^\star(\bX,\beps)$,
and in $(d)$ we used Lemma~\ref{lemma:Rn_lipschitz_bound}.  Now letting $\psi:\R\to\R$ be a bounded differentiable function with bounded Lipschitz derivative, we have
\begin{align*}
   &\Big|\E\left[\psi\left( \widehat R_n^\star(\bX,\beps)\right)\right] - \E \Big[
\psi(\Opt^\alpha_n 
(\bX,\beps))\Big]\Big|\\
   &\hspace{60mm}\le  \E\left[\left|    \psi\left( \widehat R_n^\star(\bX,\beps)\right) -  
\psi\left(\Opt^\alpha_n (\bX,\beps)\right)\right|\right]\nonumber\\
   &\hspace{60mm}\le\norm{\psi'}_\infty \E \left| \widehat R_n^\star(\bX,\beps) - 
\Opt_n^\alpha(\bX,\beps)\right|\\
   &\hspace{60mm}\le  C_0 \norm{\psi'}_\infty \E\left[\frac{\norm{\bX}^2_{\cS_p}}{n} + 
\frac{\norm{\bX}_{\cS_p}\norm{\beps}_2}{n}+ 1 \right] \sfk^{1/2} \alpha\\
   &\hspace{60mm}\stackrel{(a)}{\le}  C_0  \norm{\psi'}_\infty  \left(C_1 + C_2^{1/2} 
\E\left[\eps_1^2\right]^{1/2}  + 1 \right)  \sfk^{1/2}\alpha\\
&\hspace{60mm}\stackrel{(b)}\le C_1 \norm{\psi'}_\infty  \alpha
\end{align*}
where in $(a)$ we used Lemma~\ref{lemma:op_norm_x} and in $(b)$ the subgaussianity condition on $\eps_1$ and $\bx$. An analogous argument then 
shows that
\begin{align}
   \left|  \E \left[ \psi\left( \widehat R_n^\star(\bG,\beps)\right)\right] 
   - \E \left[\psi\left(\Opt^\alpha_n 
(\bG,\beps)\right)\right]\right|
  \le C_1 \norm{\psi'}_\infty\alpha,
\end{align}
allowing us to write
\begin{align*}
    &\lim_{n\to\infty}\left| \E \left[\psi\left(\widehat R^\star_n(\bX,\beps)\right)\right] - 
    \E 
    \left[\psi\left(\widehat R^\star_n(\bG,\beps)\right)\right] \right|\\
    &\hspace{25mm}\le 
    \lim_{n\to\infty}\left|
    \E \psi\left(\Opt^\alpha_n (\bX,\beps) \right)  - 
    \E \psi\left(\Opt^\alpha_n (\bG,\beps) \right)
    \right|
    + 2 C_2\norm{\psi'}_\infty \alpha\\
    &\hspace{25mm}= 2C_2\norm{\psi'}_\infty\alpha,
\end{align*}
where the last equality is by Lemma~\ref{prop:eps_net_univ}. Now using that 
$\norm{\psi'}_\infty < \infty$ and sending $\alpha \to 0$ concludes the proof of 
Eq.~\eqref{eq:main_thm_eq} for $\psi$ bounded differentiable with bounded Lipschitz derivative. To extend it to $\psi$ bounded Lipschitz, it is sufficient to find a sequence of bounded differentiable functions with bounded Lipschitz derivative approximating $\psi$ uniformly (see for example the following section for a similar argument).
\end{proof}
%

%

\subsubsection{Proof of Eq.~\eqref{eq:main_thm_eq} of Theorem~\ref{thm:main_empirical_univ} under Assumption~\hyperref[ass:loss_labeling_prime]{5''}}
\begin{lemma}
\label{lemma:smooth_assumption_thm2}
Consider the setting of Theorem~\ref{thm:main_empirical_univ}. Namely,
let $\{(y_i(\bX),\bx_i): i\le n\}$ be i.i.d. pairs with $\bx_i\in\reals^d$ and $y_i = y_i(\bX)$
given by Eq.~\eqref{eq:form_yi}. Similarly, define $\{(y_i(\bG),\bg_i): i\le n\}$
with $\bg_i\sim\cN(\bmu_\bg,\bSigma_\bg)$ as per Assumption \ref{ass:X}.

Suppose that Assumptions \ref{ass:set} to \ref{ass:X} hold for
$\cC_p$, $r$, and the law of $\bx$ and the domain $\cS_p$. 

Assume that
$\ell(\bv,\bv^\star;\eps)$
defined in Eq.~\eqref{eq:ell-def} satisfies 
$\ell(\bv,\bv^\star;\eps)= \tilde \ell(\bu ;\eps)$ 
for $\bu=(\bv,\bv^\star) \in\R^{\sfk + \sfk^\star}$ and $\tilde \ell$ satisfying 
Assumption~\hyperref[ass:loss_labeling_prime]{5''} 
with the $\eps_i$ uniformly sub-Gaussian.

If $\cC_{p}\subseteq \cS_p$ and $\btheta^\star_j \in \cS_p$ for each $j\le \sfk^{\star}$,
then, for any
bounded Lipschitz function $\psi:\R\to\R$,
\begin{equation}
 \lim_{n\to\infty}\left|\E\left[\psi\left(\widehat 
R^\star_n\left(\bX,\by(\bX)\right)\right)\right] - 
\E\left[\psi\left(\widehat R^\star_n\left(\bG,\by(\bG)\right) \right)\right]\right| = 0.
\end{equation}
\end{lemma}
\begin{proof}
Apply Lemma~\ref{lemma:eq_MainFirst} with $\cC_{j,p} = \cC_p$ for $j\in[\sfk]$, and $\cC_{j,p} = \{\btheta_j^\star\}$ for $j\in\{\sfk+1,\dots,\sfk+\sfk^\star$.
\end{proof}

\subsubsection{Proof of Eq.~\eqref{eq:main_thm_eq} of Theorem~\ref{thm:main_empirical_univ} under
Assumption~\hyperref[ass:loss-0-bis]{5'}}
\label{section:proof_binary_response}
We now move on to proving Eq.~\eqref{eq:main_thm_eq} of Theorem~\ref{thm:main_empirical_univ} under Assumption~\hyperref[ass:loss-0-bis]{5'}. This is done via an approximation argument.

%
For $m\in\Z_{>0}$, $\delta>0$, define the following mollifier on $\R^m$:
\begin{equation}
\label{eq:mollifier}
 \zeta_{\delta,m}(\bv) := \begin{cases}
                    C \delta^{-m} \exp\left\{\delta^2/(\norm{\bv}_2^2 - \delta^2)\right\}&, \norm{\bx}_2 <\delta\\
                    0 &, \norm{\bx}_2\ge \delta 
                   \end{cases}
\end{equation}
where $C$ is chosen so that $\zeta_{\delta,m}$ integrates to $1$. 
For $f:\R^m\to \R$, the convolution
\begin{equation}
\nonumber
 f_\delta(\bv) := (f \ast \zeta_{\delta,m})(\bv) =  \int_{B_2^m(\delta)} \zeta_{\delta,m}(\bw) f(\bv - \bw) \de \bw
\end{equation}
is infinitely differentiable (see~\cite{evans2010partial}, Appendix C.4.).
%
%
%
Additionally, we have the following properties of $f_\delta$.
\begin{lemma}
\label{lemma:pseudo_approx}
Assume $f:\R^m \to\R$ satisfies
\begin{equation*}
|f(\bv)  - f(\widetilde\bv) | \le C(1 + \norm{\bv}_2 + \norm{\widetilde\bv}_2)\norm{\bv - \widetilde\bv}_2
\end{equation*}
for some $C>0$.
Then for $\delta\in(0,1)$, we have
\begin{equation}
\label{eq:pseudo_approx_1}
   \left\|\grad f_\delta(\bv) \right\|_2\le \widetilde C(1 + \norm{\bv}_2), 
\end{equation} 
and
 \begin{equation}
\label{eq:pseudo_approx_2}
   |f_\delta(\bv) - f(\bv)| \le \bar C(1 + \norm{\bv}_2)\delta,
 \end{equation}
for some $\bar C,\widetilde C >0$. Furthermore, if for some positive integer $l< m$, $f$ satisfies
\begin{equation*}
|f(\bv,\bu)  - f(\widetilde \bv,\bu) | \le C'(1 + \norm{\bu}_2) \norm{\bv - \widetilde \bv}_2
\end{equation*}
for $\bv\in\R^{l},\bu\in\R^{m-l}$, then $f_\delta$ satisfies a similar property for a different constant $C'>0$.
\end{lemma}
\begin{proof}
For the bound in~\eqref{eq:pseudo_approx_1}, we have
\begin{align*}
\left|f_\delta(\bv) - f_\delta(\widetilde\bv)\right|
&\le C_0\int_{B_2^m(\delta)}\zeta_{\delta,m}(\bw)\left(1 + \norm{\bv}_2 + \norm{\widetilde\bv}_2 + \norm{\bw}_2\right) \norm{\bv -\widetilde\bv}_2 \de \bw\\
&\stackrel{(a)}{\le} C_1 (1 + \norm{\bv}_2 + \norm{\widetilde\bv}_2)\norm{\bv - \widetilde\bv}_2,
\end{align*}
where in $(a)$ we used $\norm{\bw}_2 \le \delta <1$.
Hence, for any $\bs\in\R^m$ with $\norm{\bs}_2 = 1$, we have
\begin{equation*}
|\bs^\sT \grad f_\delta(\bv)|  = \lim_{t\to 0} \frac{|f_\delta(\bv + t\bs) - f_\delta(\bv)|}{|t|} \le C_2\left(1 + \norm{\bv}_2\right).
\end{equation*}
Optimizing over $\bs$ gives the claim.
Meanwhile, the bound in~\eqref{eq:pseudo_approx_2} is obtained as
\begin{align*}
 \left|f(\bv) - f_\delta(\bv)\right|
&\le C_3\int_{B_2^m(\delta)}\zeta_{\delta,m}(\bw) \left(1+ \norm{\bv}_2 + \norm{\bw}_2\right) \norm{\bw}_2  \de \bw\\
&\le C_4(1 + \norm{\bv}_2)\delta.
\end{align*}
Finally, the last property is obtained via a similar argument, namely,
\begin{align*}
\left|f_\delta(\bv,\bu) - f_\delta(\widetilde \bv,\bu)\right|
&\le C_5\int_{B_2^m(\delta)}\zeta_{\delta,m}(\bw,\bz)\left(1 + \norm{\bu}_2 + \norm{\bz}_2 \right) \norm{\bv -\widetilde \bv}_2 \de (\bw,\bz)\\
&\le C_6 (1 + \norm{\bu}_2 )\norm{\bv - \widetilde\bv}_2.
\end{align*}

\end{proof}

Recall now the conditions on the loss and labels in
Assumption~\hyperref[ass:loss-0-bis]{5'}. Define for $L_F$ satisfying this assumption
$L_\delta := L_F \ast \zeta_{\delta,\sfk+1}$. First note, that $L_\delta$ is nonnegative if $L_F$ is, and locally Lipschitz since it is infinitely differentiable.
Furthermore, we have for $\bv,\widetilde\bv\in\R^{k}, v,\widetilde v\in\R$,
\begin{align}
 \left|L_F(\bv,v) - L_F(\widetilde \bv,\widetilde v) \right| &\le  \sfK(1 + \norm{\bv}_2)|v - \widetilde v| + \sfK (1 + |\widetilde v|)\norm{\bv - \widetilde \bv}_2\\
&\le C(1 + \norm{\bv}_2 + \norm{\widetilde \bv}_2 + |v|+ |\widetilde v|) \left(\norm{\bv-\widetilde\bv}_2^2 + |v - \widetilde v|^2\right)^{1/2}.
\end{align}
Hence, by the previous lemma,
we have

\begin{equation}
\label{eq:L_delta_smooth_enough}
 \norm{\grad_{\bv,v}L_\delta(\bv,v)}_2 \le C(1+\norm{\bv}_2 + |v|).
\end{equation}

Now for the labels $y_i(\bx_i)$, note that we can write
$y_i(\bx_i) \stackrel{d}{=} \chi(g(\bTheta^{\star\sT}\bx_i) - \epsilon_i)$
where $\epsilon_i \stackrel{i.i.d.}\sim\Unif([0,1])$ for $i\in[n]$ and 
\begin{equation}
\chi(t) :=  2\cdot\one_{\left\{t \ge 0\right\}} - 1.
\end{equation}
Define the smoothed functions  $g_\delta := (g \ast \zeta_{\delta,\sfk^\star})$ and $\chi_\delta := \chi \ast \zeta_{\delta,1}$  and finally,
for $\bv \in\R^{\sfk^\star}$ and $v\in\R$,  define the labeling function
\begin{equation}
\eta_\delta(\bv, v) :=  \chi_\delta(g_\delta(\bv )- v).
\end{equation}

Once again, $\eta_\delta$ is locally Lipschitz, differentiable and has
\begin{equation}
 \norm{\grad_{\bv,v} \eta_\delta(\bv,v)}_2 \le |\chi_\delta'(g_\delta(\bv) - v)| \big( \norm{\grad_\bv g_\delta(\bv)}_2^2 + 1 \big)^{1/2} \stackrel{(a)}{\le} C(\delta) (1 + \norm{\bv}_2)
\end{equation}
where in $(a)$ we used that $\chi'_\delta$ is continuous and supported on a bounded interval, along with
Lemma~\ref{lemma:pseudo_approx} applied to $(g,g_\delta)$.
Defining
\begin{equation}
    \ell_\delta(\bv,\bv^\star,v) := L_\delta(\bv, \eta_\delta(\bv^\star, v)),
\end{equation}
finally note that
for all $\beta>0$, and random variables $\bv,\bv^\star,V$ as in Eq.~\eqref{eq:sg_generic},
\begin{equation}
\E\left[\exp\{\beta|\ell_\delta(\bv,\bv^\star, V)|\}\right] \stackrel{(a)}\le \E\left[\exp\{C(1 + \norm{\bv}_2)(1  + |\eta_\delta(\bv^\star,V)|) \}\right] \stackrel{(b)}\le C(\beta,\sfR,\sfK)
\end{equation}
where $(a)$ is by Lemma~\ref{lemma:pseudo_approx} and $(b)$ is by boundedness of $\eta_\delta$. 

Hence, we conclude that $\ell_\delta$  satisfies Assumption~\hyperref[ass:loss_labeling_prime]{5''} for fixed $\delta\in(0,1)$. 
Therefore to conclude the proof of 
Eq.~\eqref{eq:main_thm_eq} of Theorem~\ref{thm:main_empirical_univ} under Assumption~\hyperref[ass:loss-0-bis]{5'}
, we only need the following lemma.

In what follows, we use the notation $\widehat R_n(\bTheta;\bX,\by(\bX)),\widehat R_n^\delta(\bTheta;\bX,\beps)$ for the empirical risk with losses $L_F,\ell_\delta$ respectively,
while the penalty function $r$ is the same in both quantities.

\begin{lemma}
In the setting of 
Theorem~\ref{thm:main_empirical_univ} with the alternative Assumption~\hyperref[ass:loss-0-bis]{5'},
for any $\delta\in(0,1)$ and bounded Lipschitz test functions $\varphi$, there exists a constant $C>0$ 
such that
\begin{align}
\nonumber
\lim_{n\to\infty} &
\left|
\E\left[\varphi\left(\min_{\bTheta \in\cC_p}
\widehat R_n(\bTheta;\bX,\by(\bX))
\right)\right]
-
\E\left[\varphi\left(\min_{\bTheta \in \cC_p}
\widehat R_n^\delta(\bTheta;\bX,\eps)
\right)\right]
\right|\le C \delta^{1/2}.
\end{align}
\end{lemma}

\begin{proof}
Define $\bfeta_\delta(\bX;\bepsilon) := \left(\eta_\delta(\bTheta^{\star\sT}\bx_i,\epsilon_i)\right)_{i\in[n]}.$
Let $\widehat \bTheta$ and $\widehat\bTheta_\delta$ denote the minimizers of the empirical risk $\widehat R_n(\bTheta;\bX,\by(\bX))$ and
$\widehat R_n^\delta(\bTheta; \bX, \eps)$ respectively.  
Since $\varphi$ is Lipschitz, it is sufficient to bound
\begin{equation*}
 \E\left[\left|\widehat R_n(\widehat\bTheta; \bX, \by(\bX))  - \widehat R^\delta_n(\widehat\bTheta_\delta; \bX, \beps)\right|\right] \le  C \delta^{1/2}
\end{equation*}
for $C>0$.
First, let us obtain an upper bound on 
\begin{align}
 \widehat R_n(\widehat\bTheta; \bX, \by(\bX))  - \widehat R^\delta_n(\widehat\bTheta_\delta; \bX, \beps)&\le
\left|\widehat R_n(\widehat\bTheta_\delta; \bX, \by(\bX))  -
 \widehat R_n(\widehat\bTheta_\delta; \bX, \bfeta_\delta(\bX;\bepsilon))\right|
\label{eq:term_1_approx} 
 \\
 &\hspace{-5mm}+
\left|\widehat R_n(\widehat\bTheta_\delta; \bX, \bfeta_\delta(\bX;\bepsilon))
 -\widehat R^\delta_n(\widehat\bTheta_\delta; \bX, \eps)\right|.
\label{eq:term_2_approx} 
\end{align}

For the term in~\eqref{eq:term_1_approx}, letting $\{\widehat\btheta_{\delta,j}\}_{j\in[\sfk]}$ be the columns of $\widehat\bTheta_\delta$,
\begin{align*}
 &\left|\widehat R_n(\widehat\bTheta_\delta; \bX, \by(\bX))  - \widehat R_n(\widehat\bTheta_\delta; \bX, \bfeta_\delta(\bX;\bepsilon))\right|\\
 &\hspace{30mm}\le \left|\frac1n \sum_{i=1}^n L_F( \widehat \bTheta_\delta^\sT\bx_i, y_i(\bx_i)) 
 - L_F(\widehat \bTheta_\delta^\sT\bx_i, \eta_\delta(\bTheta^{\star\sT}\bx_i,\epsilon_i) )\right|\\
 &\hspace{30mm}\stackrel{(a)}\le\frac{\sfK}n \sum_{i=1}^n  \left(1 + \norm{\widehat\bTheta_\delta^\sT\bx_i}_1\right) 
 \left|y_i(\bx_i) - \eta_\delta(\bTheta^{\star\sT}\bx_i,\epsilon_i)\right|\\
 &\hspace{30mm}\le\frac{\sfK}{n}\sum_{j=1}^\sfk \norm{\bX\widehat\btheta_{\delta,j}}_2 \norm{\by(\bX) - \bfeta_\delta(\bX;\bepsilon)}_2\\
 &\hspace{30mm}\hspace{15mm}+ \frac{\sfK}{\sqrt{n}}\norm{\by(\bX) - \bfeta_\delta(\bX;\bepsilon)}_2
 \\
 &\hspace{30mm}\stackrel{(b)}
\le\frac{\sfK}{\sqrt{n}}\left(
\sfk\sfR\frac{\norm{\bX}_{\cS_p} }{\sqrt{n}}
+ 1 \right) \norm{\by(\bX) - \bfeta_\delta(\bX;\bepsilon)}_2
\end{align*}
where in $(a)$ we used the condition on $L_F$ in Assumption~\hyperref[ass:loss-0-bis]{5'}, and in $(b)$ we used the notation
 $\norm{\bX}_{\cS_p} := \sup_{\left\{\btheta\in\cS_{p} : \norm{\btheta}_2\le 1\right\}}\norm{\bX \btheta}_2$
and that $\cC_p \subseteq B_2^p(\sfR)$.
Meanwhile, for the term in~\eqref{eq:term_2_approx}, we have
\begin{align*}
\Big|\widehat R_n(\widehat\bTheta_\delta; \bX, &\bfeta_\delta(\bX;\bepsilon))
 -\widehat  R^\delta_n(\widehat\bTheta_\delta; \bX, \eps)\Big|\\
 &\le 
\frac1n \sum_{i=1}^n \left|L_F\left(\widehat\bTheta_\delta^{\sT} \bx_i; \eta_\delta\left(\bTheta^{\star\sT}\bx_i,\epsilon_i\right)\right)
 -
L_\delta\left(\widehat\bTheta_\delta^{\sT}\bx_i; \eta_\delta\left(\bTheta^{\star\sT}\bx_i, \epsilon_i\right)\right)\right|\\
&\stackrel{(a)}{\le}
\frac{C_0 \delta}n \sum_{i=1}^n  \left(1 + \norm{\widehat\bTheta_\delta^{\sT}\bx_i}_2 + \left|\eta_\delta\left(\bTheta^{\star\sT}\bx_i, \epsilon_i\right)\right| \right)\\
&\le C_1\delta \left(1 + \left(\frac{1}{n}\sum_{j=1}^\sfk \norm{\bX\widehat\btheta_{\delta,j}}_2^2\right)^{1/2} \right)\\
&\le C_1\delta \left(1 + \sfk^{1/2}\sfR\frac{\norm{\bX}_{\cS_p}}{\sqrt{n}} \right),
\end{align*}
where in $(a)$ we applied Lemma~\ref{lemma:pseudo_approx}.

By symmetry, we can obtain a similar lower bound on the left-hand side of~\eqref{eq:term_2_approx} (by replacing $\widehat\bTheta_\delta$ throughout with $\widehat\bTheta$), which allows
us to write
\begin{align}
 &\E\left[\left|\widehat R_n(\widehat\bTheta; \bX, \by(\bX))  - \widehat R^\delta_n(\widehat\bTheta_\delta; \bX, \beps)\right|\right]\nonumber\\
 &\hspace{50mm}\le 
C_2 \E\left[\left(1 + \frac{\norm{\bX}_{\cS_p} }{\sqrt{n}}\right)\left( \frac{\norm{\by(\bX) - \bfeta_\delta(\bX;\bepsilon)}_2}{\sqrt{n}} + \delta\right)\right] \\
&\hspace{50mm}\stackrel{(a)}{\le} C_3 \left(\E\left[\frac{\norm{\by(\bX) - \bfeta_\delta(\bX;\bepsilon)}^2_2}{n}\right]^{1/2}  + \delta \right)
\label{eq:final_bound_approx}
\end{align}
for large enough $n$ and $C_2,C_3>0$. Here, in $(a)$ we used Lemma~\ref{lemma:op_norm_x}.

To conclude the proof, we show that the expectation on line~\eqref{eq:final_bound_approx} is bounded by a positive constant times $\delta$. This follows via the following computation:
\begin{align*}
 \E\left[\norm{\by(\bX) -\bfeta_\delta(\bX;\bepsilon)}_2^2\right]&\le 
 2\sum_{i=1}^n
 \E\left[|\chi(g(\bTheta^{\star\sT}\bx_i) - \epsilon_i) - \chi(g_\delta(\bTheta^{\star\sT}\bx_i) - \epsilon_i)|^2\right]\\
&\hspace{10mm}+  
 2\sum_{i=1}^n\E\left[|\chi(g_\delta(\bTheta^{\star\sT}\bx_i) - \epsilon_i) - \chi_\delta(g_\delta(\bTheta^{\star\sT}\bx_i) - \epsilon_i)|^2\right]\\
 &\stackrel{(a)}\le 8 \sum_{i=1}^n \E\left[\P{\left(\epsilon_i\textrm{ between } g(\bTheta^{\star\sT}\bx_i) \textrm{ and }g_\delta(\bTheta^{\star\sT}\bx_i)\big| \bx_i \right)} \right]\\
 &\hspace{10mm}+ 8 \sum_{i=1}^n \E\left[\P{\left(\epsilon_i \in [g_\delta(\bTheta^{\star\sT}\bx_i) -\delta,g_\delta(\bTheta^{\star\sT}\bx_i) +\delta] \big| \bx_i\right)}\right]\\
 &\le 8\sum_{i=1}^n 
\E\left[\left|g(\bTheta^{\star\sT}\bx_i) - g_\delta(\bTheta^{\star\sT}\bx_i)  \right| \right]
+ 16 \sum_{i=1}^n  \delta\\
&\stackrel{(b)}{\le} C_4 \sum_{i=1}^n \delta \left( 1 +  \E\norm{\bTheta^{\star\sT}\bx_i}_2 \right)\\
&\stackrel{(c)}\le C_5 n  \delta,
\end{align*}
for some $C_4,C_5>0$.
Here, in $(a)$ we used $\chi_\delta(t) = 1$ for all $t \ge \delta$ and $-1$ for all $t\le \delta$, in $(b)$ we used
Lemma~\ref{lemma:pseudo_approx} with $(g,g_\delta)$, and in $(c)$ we used subgaussianity of $\bx_i$.

\end{proof}

\subsubsection{
Proof of Theorems~\ref{thm:MainFirst} and~\ref{thm:main_empirical_univ} under Assumption~\ref{ass:loss-0}}
\label{section:pf_under_general_loss}

We will present the proof of Eq.~\eqref{eq:main_thm_eq} 
under Assumption~\ref{ass:loss-0} on the loss
of Theorem~\ref{thm:main_empirical_univ} by deducing it from Lemma~\ref{lemma:smooth_assumption_thm2}. The proof of 
Eq.~\eqref{eq:main_thm_eq_0} of Theorem~\ref{thm:MainFirst} similarly follows from Lemma~\ref{lemma:eq_MainFirst}. In fact, we will first give the proof under the following more general assumption.

\begin{numassumption}[5''']
\label{ass:loss-general}
The following hold:
\begin{enumerate}[$(i)$]
    \item
    $0 \le \ell(\bv,\eps) \le \sfK(1 + \|\bv\|_2^\sfb + |\eps|^\sfb)$ for some constants $\sfK$, $\sfb>0$.
    \item For all $M>0$ sufficiently large, with
    $\cA_M := \{(\bv,\eps) : \ell(\bv,\eps) \le M \},$ 
    there exists $\ell_M :\R^{k  + 1}$ such that 
    \begin{itemize}
           \item[$(a)$] $\|\ell_M\|_{\Lip}\le L_M$ for some $L_M >0$ depending on $M$.
    \item[$(b)$] On $\cA_M$,
    we have $\ell_M(\bv,\eps) = \ell(\bv,\eps)$.
    \item[$(c)$] 
    On $\cA_M^c$, we have
    for some $q \ge 1$,
    \begin{align}
    M <  \ell_M(\bv,\eps)  \le \ell(\bv,\eps) \le 
    C(1  + \ell_M(\bv, \eps)^q).\label{eq:loss-0-c}
    \end{align}
    \end{itemize}
\end{enumerate}
\end{numassumption}

Let us first show that this assumption is indeed more general than Assumption~\ref{ass:loss-0}.
\begin{lemma}
If $\ell$ satisfies Assumption~\ref{ass:loss-0}, then it also satisfies Assumption~\hyperref[ass:loss-general]{5{'''}}.
 \end{lemma}
 
\begin{proof}
If $\ell$ is Lipschitz, then it clearly satisfies Assumption~\hyperref[ass:loss-general]{5'''}.
Let us assume it satisfies the local Lipschitz condition and the polynomial growth condition.

Note that $(i)$ holds trivially, and we need therefore to prove parts $(a)$, $(b)$ $(c)$ of condition $(ii)$.
Without loss of generality, we can assume that $\sfK\ge 1$ and $\sfa \in(0,1)$.

For $M>0$, let $R(M):=[(M+\sfK_0)/\sfK_1]^{1/\sfa}$.
Then 
   \begin{equation}
       \cA_M  \subseteq \{(\eps,\bv): \|\bv\|_2\vee|\eps| \le 
       R(M)\}
   \end{equation}
   so that $\cA_M$ bounded and compact by the local Lipschitz condition on $\ell$.
   Let $L_M := \|\ell|_{\cA_M}\|_{\Lip} \vee 1$, where
   $\|\ell|_{\cA_M}\|_{\Lip}$ denotes the
   Lipschitz constant of the restriction of $\ell$ to $\cA_M$
   (with $\|\ell_{\cA_M}\|_{\Lip}=0$ if $\cA_M$ is empty). 
   Since $\cA_M$ is compact and $f$ is locally Lipschitz, we have $L_M <\infty$ for all $M > 0$.
   Define
   \begin{equation}
       \ell_{M}(\bw) :=  \inf_{\bz \in \R^{k+1}} \big(\ell(\bz) + L_M \|\bz - \bw\|_2\big).
   \end{equation}
   
 \noindent\emph{Condition $(a)$.}  
    To see that $\ell_M$ is Lipschitz with constant $L_M$, note that for any 
    $\bz,\bx \in\R^{k+1}$ we have
    \begin{align*}
    \ell_M(\bw) - \left(\ell(\bz)  + L_M \|\bz - \bx\|_2 \right)&\le
    \ell(\bz)  + L_M \|\bz - \bw\|_2
    - \left(\ell(\bz)  + L_M \|\bz - \bx\|_2 \right)\\
    &\le
    L_M \|\bx - \bw\|_2.
    \end{align*}
    Evaluating the above at a sequence $\bz_j$ so that $\ell(\bz_j) + L_{M}\|\bz_j - \bw| \to  \ell_M(\bw)$, 
    we conclude that
    \begin{equation}
       \ell_M(\bw) - \ell_M(\bx)  \le L_M \|\bx - \bw\|_2\, ,
    \end{equation}
    which proves that $\ell_M$ is $L_M$-Lipschitz

   \noindent\emph{Condition $(b)$.}  To see that that $\ell_M = \ell$ on $\cA_M$, note that
   \begin{equation}
       \ell_{M}(\bw) =  \left(\min_{\bz \in \cA_M} \big(\ell(\bz) + L_M \|\bz - \bw\|_2\big) \right)\wedge \left(
       \inf_{\bz \in \cA_M^c} \big(\ell(\bz) + L_M \|\bz - \bw\|_2\big)\right).
   \end{equation}
   If $w\in\cA_M$, the first quantity is upper bounded by $M$ corresponding to the choice $\bz = \bw$, while the second quantity strictly larger than  $M$, since $L_M =\|\ell_{\cA_M}\|_\Lip>0$.
Therefore, for $\bw\in \cA_M$,  $\ell_{M}(\bw) = \min_{\bz \in \cA_M} (\ell(\bz) + L_M \|\bz - \bw\|_2)$.
Of course $\ell_M(\bw)\le \ell(\bw)$. On the other hand, for any $\bz \in \cA_M$,
$\ell(\bz) + L_M \|\bz - \bw\|_2\ge \ell(\bw)$ by the Lipschitz property, and therefore $\ell_M(\bw)=\ell(\bw)$. 

 \noindent\emph{Condition $(c)$.}    To prove $\ell_M(\bw) > M$ for $\bw \in\cA_M^c$, assume that infimum in the definition of $\ell_M$ is achieved inside $\cA_M$ at point $\bz_0$ (otherwise, the property clearly holds),
 sot that $\ell_M(\bw) = \ell(\bz_0)+L_M\|\bw-\bz_0\|$. 
 Choose a point $\bu\in S(\bw):=\argmin\{\|\tilde\bu - \bw\|_2: \tilde\bu \in\cA_M\}$, and assume that $\bz_0
 \not \in S_M(\bw)$ (otherwise, we clearly have the desired bound). By continuity and compactness of $\cA_M$, we have $\ell(\bu)= M$. Now write
   \begin{align*}
      \ell_M(\bw) - M &=
      \ell(\bz_0) - \ell(\bu)+  L_M \|\bz_0 - \bw\|_2  \\
      &\ge L_M\left(-\|\bz_0 - \bu\|+  \|\bz_0 - \bw\|_2    \right)> 0.
   \end{align*}

   Finally, to prove that the second inequality in Eq.~\eqref{eq:loss-0-c},
   we have
   \begin{align*}
       \ell_M(\bw) & \ge \min_{\bz}\Big(-\sfK_0+\sfK_1\|\bz\|^\sfa+L_M\|\bz-\bw\|\Big) = \ell_M(\|\bw\|)\, ,\\
       \ell_M(t) & = \min_{s\ge 0} \Big(-\sfK_0+\sfK_1 s^\sfa+L_M|t-s|\Big)\, .
   \end{align*}
   Simple calculus reveals that $\ell_M(t) = -\sfK_0+\sfK_1 t^\sfa$ for $t\ge (\sfK_1/L_M)^{1/(1-\sfa)}:=R_0(M)$.
   Hence, for $\bw\in\cA_M^c$, $\|\bw\|\ge R_0(M)$, we have
   $\ell_M(\bw) \ge M \vee( -\sfK_0+\sfK_1\|\bw\|^\sfa)$, whence, for a sufficiently small $c_M$,
   \begin{align*}
       \ell_M(\bw) & \ge c_M(1+\|\bw\|^\sfa)\, ,
    \end{align*}
which implies the desired upper bound.
\end{proof}


We now extend the universality results to Assumption~\hyperref[ass:loss-general]{5'''}.
\begin{proof}[
Proof of Eq.~\eqref{eq:main_thm_eq_0} and Eq.~\eqref{eq:main_thm_eq}  under Assumption~{5'''}.
]

 Take $M>0$ sufficiently large and recall $\ell_M$ stated in the assumption.
Let 
\begin{equation}
    \widehat R_M(\bTheta) := \frac1n \sum_{i=1}^n \ell_M(\bTheta^\sT\bx_i, \bTheta^{\star\sT}\bx_i, \eps_i)+ 
    r(\bTheta),
\end{equation}
and let $\widehat R(\bTheta)$ denote the empirical risk with $\ell$ instead of $\ell_M$.

We first show that 
\begin{equation*}
    \min_{\bTheta \in\cC_p} \widehat R(\bTheta) \le 
    \min_{\bTheta \in\cC_p} \widehat R_M(\bTheta) + \Delta_M
\end{equation*}
for some $\Delta_M$ such that 
\begin{equation}
\label{eq:Delta_error_limit}
    \lim_{M\to\infty } \lim_{n\to\infty} \E[\Delta_M].
\end{equation}

Let $\widehat \bTheta_M$ be the minimizer of $\widehat R_M$, and  $\widehat\bTheta_M^{\up{i}}$ the minimzer of 

$$
\widehat R_M^\up{i} := 
\frac1n \sum_{j\neq i} \ell_M(\bTheta^\sT\bx_j, \bTheta^{\star\sT}\bx_j, \eps_j)+ 
    r(\bTheta).
$$

Clearly, we have
\begin{equation*}
    \widehat R_M^\up{i}(\widehat \bTheta_M) \ge \widehat R_M^{\up{i}} (\widehat\bTheta_M^{\up{i}}),\quad\textrm{and}\quad
    \widehat R_M(\widehat\bTheta_M)  \le \widehat R_M(\widehat\bTheta_M^\up{i}).
\end{equation*}
The second relation gives 
\begin{equation*}
    \widehat R_M^\up{i}(\widehat \bTheta_M) + \ell_M(\widehat\bTheta_M^\sT\bx_i,\bTheta^{\star\sT}\bx_i, \eps_i)
    \le \widehat R_M^{\up{i}} (\widehat\bTheta_M^{\up{i}})+
    \ell_M(\widehat\bTheta_M^{\up{i}\sT}\bx_i,\bTheta^{\star\sT}\bx_i, \eps_i).
\end{equation*}
Applying the first relation then shows that
\begin{equation}
\label{eq:loo_comparison}
    \ell_M(\widehat\bTheta_M^\sT\bx_i,\bTheta^{\star\sT}\bx_i, \eps_i)
    \le
    \ell_M(\widehat\bTheta_M^{\up{i}\sT}\bx_i,\bTheta^{\star\sT}\bx_i, \eps_i).
\end{equation}
\noindent
Let us denote $\ell_M(\bTheta;i)\equiv 
\ell_M(\bTheta^\sT\bx_i,\widehat\bTheta^{\star\sT}\bx_i, \eps_i)$ and $\ell(\bTheta;i)
\equiv \ell(\bTheta^sT\bx_i,\widehat\bTheta^{\star\sT}\bx_i, \eps_i)$.
Now since $\ell\neq \ell_M$ only when $\cA_M^c= \{\ell_M > M\}$, we have
\begin{align*}
    \ell_M(\widehat\bTheta_M;i )
    &=
    \ell(\widehat\bTheta_M;i)
     -\left(
     \ell(\widehat\bTheta_M;i)
     -
     \ell_M(\widehat\bTheta_M;i)
     \right) \one_{\ell_M(\widehat\bTheta_M;i) > M}.
\end{align*}
We can upper bound
\begin{align*}
&\left(
     \ell(\widehat\bTheta_M;i)
     -
     \ell_M(\widehat\bTheta_M;i)
     \right) \one_{\left\{
\ell_M(\widehat\bTheta_M;i) > M
     \right\}}\\
     &\stackrel{(a)}{\le} \left( C( 1 + \ell_M(\widehat\bTheta_M;i)^q) - \ell_M(\widehat\bTheta_M; i)\right)
     \one_{\left\{
\ell_M(\widehat\bTheta_M;i) > M
     \right\}}\\
     &\stackrel{(b)}{\le}  C_1 \left(1 +  \ell_M(\widehat\bTheta_M^\up{i}; i)^q\right) \one_{\left\{ \ell_M(\widehat\bTheta_M^\up{i};i)> M \right\}}
\end{align*}
where $(a)$ follows by the assumption on $\ell_M$ and $(b)$ follows by
Eq.~\eqref{eq:loo_comparison}.
Defining
\begin{equation*}
    \Delta_M := \frac{C_1}n \sum_{i=1}^n 
    \left(1 +  \ell_M(\widehat\bTheta_M^\up{i}; i)^q\right) \one_{\left\{ \ell_M(\widehat\bTheta_M^\up{i};i) > M \right\}},
\end{equation*}
we have by taking $M$ sufficiently large,
\begin{align*}
    \E\left[|\Delta_M|\right] 
    &\le C_2 \E\left[\left(1 + \|\widehat\bTheta_M^\up{i}\bx_i\|_2^{2\sfb} +
    \|\bTheta^{\star\sT}\bx_i\|_2^{2\sfb}  + |\eps_i|^{2\sfb}
    \right)\right]^{1/2} \\
    &\qquad\cdot\P\left(
     \|\widehat\bTheta_M^\up{i}\bx_i\|_2^{\sfb} +
    \|\bTheta^{\star\sT}\bx_i\|_2^{\sfb}  + |\eps_i|^{\sfb} \ge c_1 M
    \right)^{1/2}\\
    &\le C_3 e^{-c_2 M}
\end{align*}
for some constants $C_3,c_2 > 0$ independent of $n$, by the subgaussian assumption and the assumption on $\cC_p.$
This implies the desired claim of Eq.~\eqref{eq:Delta_error_limit}.

Since $\ell_M \le \ell$ by assumption, we conclude then that 
\begin{equation*}
\min_{\bTheta \in\cC_p} \widehat R_M(\bTheta) \le \min_{\bTheta \in\cC_p} \widehat R(\bTheta) \le 
    \min_{\bTheta \in\cC_p} \widehat R_M(\bTheta) + \Delta_M.
\end{equation*}
so that for any Lipschitz test function $\psi$, we have
\begin{equation*}
\E\left[\psi\left(\min_{\bTheta \in\cC_p} \widehat R_M(\bTheta)\right)\right] -
\E\left[\psi\left(
\min_{\bTheta \in\cC_p} \widehat R(\bTheta) \right)\right]
\le  \|\psi\|_{\Lip}\E[|\Delta_M|].
\end{equation*}

Since for any fixed $M>0$ sufficiently large, $\ell_M$ is Lipschitz,
it satisfies the integrability condition of
Assumption~\hyperref[ass:loss_labeling_prime]{5''}.
Then 
by a smoothing argument analogous to the previous section, 
one deduces from Lemma~\ref{lemma:smooth_assumption_thm2}
that Eq.~\eqref{eq:main_thm_eq} holds,
so that taking $n\to\infty$ followed by $M\to\infty$ completes the proof.
\end{proof}

\subsubsection{Proof of the bounds in Eq.~\eqref{eq:prob_bounds} of Theorem~\ref{thm:main_empirical_univ}}
Having proved that Eq.~\eqref{eq:main_thm_eq} of Theorem~\ref{thm:main_empirical_univ} holds under both assumptions on the loss and labels, we show that the bounds in~\eqref{eq:prob_bounds} 
are a direct consequence.

Fix $\delta>0$ and $\rho\in\R$ and define
$\chi_\delta : \chi \ast \zeta_{\delta,1}$ as in the previous section, where we again have $\chi(t) = \one_{t\ge0}$.
Recall that $\chi_{\delta,\rho}$ satisfies
\begin{equation}
\nonumber
  \one_{\{t \ge \rho + \delta\}}  \le \chi_{\delta}(t - \rho) \le \one_{\{t\ge\rho - \delta \}}.
\end{equation}
 and that$\norm{\chi_{\delta,\rho}}_\Lip = C(\delta)$ for some constant depending only on $\delta$. 
 Hence, we can apply~\eqref{eq:main_thm_eq} with $\psi(t)= \chi_{\delta}(t - \rho)$ to conclude
\begin{align*}
   \limsup_{n\to\infty}\P\left( \widehat R_n^\star \left(\bX,\by\left(\bX\right)\right) \ge \rho + \delta  \right) 
   &\le \limsup_{n\to\infty}\E\left[\chi_{\delta}\left(\widehat R_n^\star\left(\bX,\by\left(\bX\right) \right) - \rho\right)\right]\\
   &= \limsup_{n\to\infty}\E\left[\chi_{\delta}\left(\widehat R_n^\star\left(\bG,\by\left(\bG\right) \right) - \rho\right)\right]\\
   &\le
   \limsup_{n\to\infty}\P\left( \widehat R_n^\star \left(\bG,\by\left(\bG\right)\right) \ge \rho - \delta  \right),
\end{align*}
which establishes the first bound in~\eqref{eq:prob_bounds}. The second bound follows via a similar argument.
\qed


\subsection{Universality of the minimum over the discretized space: Proof of Lemma~\ref{prop:eps_net_univ}}
\label{section:universality_eps_net_proof}
Recall the minimization problem over the set $\cN_\alpha^\sfk$ defined in~\eqref{eq:opt}. 
We show in this section that Lemma~\ref{prop:eps_net_univ} is a direct consequence of
Lemma~\ref{lemma:softmin_univ}.
\begin{proof}[Proof of Lemma~\ref{prop:eps_net_univ}]

Fix $\alpha>0$.
Let us first bound the derivative of the free energy.
Define the probability mass function for $\bTheta\in\cN_\alpha^\sfk$,
\begin{equation}
\nonumber
p(\bTheta;\bX, t) :=\frac{e^{-t n \widehat 
R_n(\bTheta; \bX,\beps)}}{\sum_{\bTheta\in\cN_\alpha^\sfk}
e^{- t n \widehat R_n(\bTheta;\bX,\beps)}}
\end{equation}
and define similalry $p(\bTheta;\bG,t)$ for the Gaussian model.
Recall that the Shannon entropy of a distribution $H(p(\,\cdot\,;\bX,t)) := -\sum_{\bTheta\in\cN_\alpha^\sfk} 
p(\bTheta; \bX, t) \log p(\bTheta ; \bX, t)$ satisfies
\begin{equation}
 0 \le H(p(\bTheta;\bX,t)) \le \log \left|\cN_\alpha^\sfk \right| = \log  
C_0(\alpha,\sfR)^{p\sfk}\label{eq:entropy_bound}
\end{equation}
for some $C_0>0$.
Therefore, the derivative of the free energy with respect to $t$ can be bounded as
\begin{align*}
 \frac{\partial}{\partial t} f_\alpha(t,\bX) 
 &=\frac{1}{t} \frac{\sum_{\bTheta}
 \widehat R_n(\bTheta; \bX,\beps)
 e^{ 
-tn \widehat R_n(\bTheta; \bX,\beps)} }{
\sum_{\bTheta} e^{-t n \widehat R_n(\bTheta; \bX,\beps)}
}  + \frac{1}{t^2 n} \log\sum_\bTheta e^{- t n \widehat R_n(\bTheta; \bX,\beps)}\\
 &=-\frac{1}{t^2n} \frac{\sum_{\bTheta}
 \log  p(\bTheta; \bX,t)
 e^{ 
-tn \widehat R_n(\bTheta; \bX,\beps)} }{
\sum_{\bTheta} e^{-t n \widehat R_n(\bTheta; \bX,\beps)}
}\\
&=\frac{1}{t^2 n} H\left(p(\bTheta,t)\right)\\
&\stackrel{(a)}{\le}
C_1(\alpha)\frac{p(n)}{n} \frac{1}{t^2}
\end{align*}
where $(a)$ follows by~\eqref{eq:entropy_bound}.
This bound on the derivative implies that $f_\alpha(\beta,\bX)$ approximates $\Opt_n^\alpha(\bX,\beps)$ uniformly:
\begin{align*}
    \left| 
    f_\alpha(\beta,\bX) - \Opt_n^\alpha(\bX,\beps)
    \right|
    &=
    \lim_{s\to\infty}\left| 
    f_\alpha(\beta,\bX) -
    f_\alpha(s,\bX)
    \right|\\
    &\le C_1(\alpha) \frac{p(n)}{n} \lim_{s\to\infty} \int_{\beta}^\infty \frac{1}{t^2}\de t \\
    &= C_1(\alpha) \frac{p(n)}{n} \frac1\beta.
\end{align*}
Clearly, a similar bound holds with $\bG$ replacing $\bX$. 
Hence, we have
\begin{align*}
&\lim_{n\to\infty} \left|\E\left[\psi\left(\Opt_n^\alpha(\bX,\beps)\right) - \psi\left(\Opt_n^{\alpha}(\bG,\beps)\right)\right]\right| \\
&\hspace{25mm}\le  
\lim_{n\to\infty } \left| \E\left[ \psi( f_{\alpha}(\beta,\bX)) - \psi (f_{\alpha}(\beta,\bG))\right]\right|+ 
\frac{2\norm{\psi'}_\infty C_1(\alpha)}{\beta} 
\lim_{n\to\infty}\frac{p(n)}{n}\\
&\hspace{25mm}\stackrel{(a)}{=}  \frac{\norm{\psi'}_\infty C_2(\alpha)}{\beta}
\end{align*}
where $(a)$ follows from Lemma~\ref{lemma:softmin_univ} along with the assumption that $p(n)/n \to \sgamma$.
Sending $\beta\to\infty$ completes the proof.
\end{proof}

\subsection{Deferred proofs of technical lemmas for the universality of the free energy}
\label{section:proof_softmin_univ}
Let us recall the interpolating paths 
 $\bu_{t,i}:= \sin(t) \left(\bx_i  - \bmu_\bg\right)+ \cos(t) \left(\bg_i - \bmu_\bg\right) + \bmu_\bg$ and 
 $\widetilde\bu_{t,i}:= \cos(t) \left(\bx_i  - \bmu_\bg\right)- \sin(t) \left(\bg_i - \bmu_\bg\right)$
 defined in~\eqref{eq:slepian} for $t\in[0,\pi/2]$ and $i\in[n]$, and the associated matrix $\bU_t$ whose $i$th row is
$\bu_{t,i}$. Further, recall the gradient notation introduced in Section~\ref{section:proof_outline}:
\begin{equation}
\nonumber
   \grad \ell(\bv;v) = \left( \frac{\partial}{\partial v_k}\ell\left(\bv;v\right)\right)_{k\in[\sfk]}\hspace{-3mm}
\end{equation}
for $\bv\in\R^\sfk,v\in\R$ and the shorthand
 $\widehat\ell_{t,i}(\bTheta)$ for  $\ell\left(\bTheta^\sT \bu_{t,i};\epsilon_i\right)$. Now, recall the definition in~\eqref{eq:bd_def}:
\begin{equation}
\nonumber
 \widehat\bd_{t,i}(\bTheta) :=
\bTheta \grad\widehat\ell_{t,i}(\bTheta).
\end{equation}
Finally, recall the probability mass function and its associated expectation defined in~\eqref{eq:inner_def} 
\begin{equation}
\nonumber
p^\up{i}(\bTheta_0;t) :=
\frac{
e^{-\beta\left(\sum_{j\neq 
i}\widehat\ell_{t,j}(\bTheta_0) +  n r(\bTheta_0)\right)
}}{\sum_{\bTheta}
e^{-\beta\left(\sum_{j\neq 
i}\widehat\ell_{t,j}(\bTheta) +  n r(\bTheta)\right)
}}
\quad\textrm{and}\quad
\inner{\;\cdot\;}^\up{i}_\bTheta:=
\sum_{\bTheta} (\,\cdot\,)p^\up{i}(\bTheta;t),
\end{equation}
where all sums are implicitly over $\cN_\alpha$; the minimal $\alpha-$net of $\cC_p$ introduced in Section~\ref{section:proof_outline}.

Let us use $\E_\up{i}$ to denote the expectation conditional on $(\bG^\up{i},\bX^\up{i},\bepsilon^\up{i})$; 
the feature vectors and the noise vector with the $i$th sample set to $0$ (or equivalently, since 
the samples are i.i.d, the expectation with respect to $(\bx_i,\bg_i,\epsilon_i)$.)

\begin{lemma}
\label{lemma:dct_bound}
Suppose the assumptions of Theorem~\ref{thm:MainFirst} hold, along with the alternative Assumption~\hyperref[ass:loss_labeling_prime]{5''}.
For all $n\in\Z_{>0}, t\in[0,\pi/2]$ and $\beta>0$, we have
\begin{equation}
\label{eq:square_integrability_main_lemma}
\sup_{\bTheta_0\in\cS_p^\sfk}
\E_\up{1}\left[\left(
\frac{
\widetilde\bu_{t,1}^\sT
\widehat\bd_{t,1}(\bTheta_0)
 e^{-\beta\widehat\ell_{t,1}(\bTheta_0)}}
{\inner{e^{-\beta\widehat\ell_{t,1}(\bTheta)}}_\bTheta^\up{1}}\right)^2\right]\le C(\beta),
\end{equation}
for some $C(\beta)$.
In particular, we have for any fixed $\beta>0$ and bounded differentiable function $\psi:\R\to\R$ with bounded Lipschitz derivative,
\begin{equation}
\label{eq:dct_bound_softmax}
 \int_{0}^{\pi/2} \sup_{n \in \Z_{>0} }\left|\E\left[\frac{\partial}{\partial t}\psi(f_\alpha(\beta, 
\bU_t)) \right]\right| \de t < \infty,
\end{equation}
where $f_\alpha(\beta,\cdot)$ is the free energy defined in~\eqref{eq:free_energy_def}.
\end{lemma}

\begin{proof}
Let us expand this via the definition of $\widehat\bd_{t,1}(\bTheta_0)$ in~\eqref{eq:bd_def}:
\begin{align}
\nonumber
\widehat\bd_{t,1}(\bTheta_0)^\sT \widetilde\bu_{t,1} &=
\sum_{k=1}^\sfk\partial_k \ell\left(\bTheta_0^\sT \bu_{t,1}; 
\epsilon_1 \right) \btheta_k^\sT \widetilde\bu_{t,1}
\label{eq:bd_expansion}
\end{align}
where we defined 
\begin{equation}
\nonumber
\partial_k \ell(\bv;v) := \frac{\partial}{\partial v_k} \ell(\bv;v),
\end{equation}
for $\bv\in\R^\sfk$ and $v\in\R$. 
Now recall the condition on $\ell$ and $\eta$ in Assumption~\hyperref[ass:loss_labeling_prime]{5''} and note that 
this allows us to bound
\begin{align*}
\left|\partial_k \ell(\bv;v)\right|
&\le C_1 \left( 
\norm{\bv}_2 + |v|  + 1
\right)
\end{align*}
for some $C_1>0$.
However, for any fixed $m>0$ and $\bTheta\in\cS_p^\sfk$ we have
\begin{align}
\nonumber
 \E\left[\norm{\bTheta^\sT\bu_{t,1}}_2^m\right]\le 
 C_2(\sfk)\sum_{k=1}^\sfk\E\left[\left(\btheta_k^\sT\bu_{t,1}\right)^m \right] \le C_3
\end{align}
 since $\sup_{t\in[0,\pi/2]}\sup_{\btheta\in\cS_p} \norm{\btheta^\sT \bu_{t,1}} \le 2\sfR\sfK$ by Assumption~\ref{ass:X}. 
A similar bound clearly holds for $\widetilde\bu_{t,1}$. 
Hence, 
using that $\sfk$ are assumed to be fixed, an application of H\"older's gives
\begin{equation}
\nonumber
\E\left[
\left(\widehat\bd_{t,1}(\bTheta_0)^\sT \widetilde\bu_{t,1}\right)^4
\right]  \le C_4
\end{equation}
for some $C_4>0$, where we also used that $\epsilon_1$ is assumed to be subgaussian.
Therefore, we have
\begin{align*}
\E_\up{1}\left[\left(
\frac{
\widehat\bd_{t,i}(\bTheta_0)
^\sT\widetilde\bu_{t,i} e^{-\beta\widehat\ell_{t,i}(\bTheta_0)}}
{\inner{e^{-\beta\widehat\ell_{t,i}(\bTheta)}}_\bTheta^\up{i}}\right)^2\right]
&\stackrel{(a)}{\le}
C_5^{1/2}\E_\up{i}\left[  
\frac{1}
{\left(\inner{e^{-\beta\widehat\ell_{t,i}(\bTheta)}}_\bTheta^\up{i}\right)^4}\right]^{1/2} \\
&\stackrel{(b)}{\le}
C_5^{1/2}
\left(\inner{\E_\up{i}\left[e^{4\beta\widehat\ell_{t,i}(\bTheta)}\right]}_\bTheta^\up{i} \right)^{1/2}\\
&\stackrel{(c)}{\le} C_5^{1/2} C(\beta)^{1/2}
\end{align*}
for $C_5>0$ and $C(\beta)>0$.
Here, in $(a)$ we used that $\ell$ and $\beta$ are nonnegative, in $(b)$ we used Jensen's and that $p^\up{i}(\bTheta;t)$ as defined in~\eqref{eq:inner_def} is independent of $(\bx_i,\bg_i,\epsilon_i)$, and in $(c)$ we
used the integrability condition of Assumption~\hyperref[ass:loss_prime]{5''}. 
Taking the supremum over $\bTheta_0 \in \cS_p^\sfk$ establishes the first inequality in the statement of the lemma.

To establish the second inequality, recall the explicit form of the derivative from~\eqref{eq:derivative_explicit} 
and note that $\bu_{t,i}$ are i.i.d. for different $i$ so that
\begin{align*}
\left| \E\left[\frac{\partial}{\partial t}\psi(f_\alpha(\beta,\bU_t)) \right]\right|
&\le 
\norm{\psi'}_\infty \E\left[\sup_{\bTheta_0 \in\cC_p^\sfk}  \E_{\up{1}}\left[
\left(
\frac{
\widehat\bd_{t,1}(\bTheta_0)
^\sT\widetilde\bu_{t,1} e^{-\beta\widehat\ell_{t,1}(\bTheta_0)}}
{\inner{e^{-\beta\widehat\ell_{t,1}(\bTheta)}}_\bTheta^\up{1}}\right)^2
\right]^{1/2} \right]\\
&\le \norm{\psi'}_\infty C_6(\beta)
\end{align*}
for $C_6(\beta)>0$, where we used that $\cC_p \subseteq \cS_p$ by Assumption~\ref{ass:set}. Hence, the bound holds uniformly in $t\in[0,\pi/2]$ and $n\in\Z_{>0}$ as desired.

\end{proof}

\section{Auxiliary lemma for Theorem~\ref{thm:universality_bounds}}
\label{section:proof_univerality_bounds}
\begin{lemma}
\label{lemma:F_diff}
In the setting of Section~\ref{section:outline_proof_univerality_bounds},
for all $t\ge s > 0$ and any $\delta>0$, we have
\begin{equation}
\lim_{n\to\infty}\P\left(\left\{\left|  F_n^\bx(t,\bX) - F_n^\bg(t,\bX)  \right| > \delta \right\} \bigcap \cG_{n,s}\right) = 0.
\end{equation}
\end{lemma}

\begin{proof}
Fix $t\ge s> 0$. On $\cG_{n,s}$, let
\begin{equation}
\nonumber
\widehat \btheta_\bx \in
 \quad
\argmin_{
\mathclap{\substack{ \btheta \in \cC_p   \\ \widehat 
R_n(\btheta;\bX,\by(\bX)) \le t }}
}
R_n^\bx(\btheta),\hspace{12mm}
\widehat \btheta_\bg \in 
\argmin_{
\mathclap{\substack{ \btheta \in \cC_p   \\ \widehat 
R_n(\btheta;\bX,\by(\bX)) \le t }}
}
R_n^\bg(\btheta)
\end{equation}
be any minimizers of the respective functions so that $ F_n^\bx(t,\bX) =  R_n^\bx(\widehat\btheta_\bx)$ and $F_n^\bg(t,\bX) =  R_n^\bg(\widehat \btheta_\bg)$. 
Then note that we can upper bound
\begin{align*}
\left(R_n^\bg(\widehat \btheta_\bg) - R_n^\bx(\widehat \btheta_\bx) \right)\one_{\cG_{n,s}}
  &\stackrel{(a)}\le 
 \left|R_n^\bg(\widehat \btheta_\bx) - R_n^\bx(\widehat 
\btheta_\bx)\right| \one_{\cG_{n,s}}\\
  &\le 
 \sup_{\btheta\in\cS_p}
 \left|R_n^\bg(\btheta) - R_n^\bx(\btheta)\right|
\end{align*}
where in $(a)$ we used that $R_n^\bg(\widehat \btheta_\bg) \le R_n^\bg(\widehat 
\btheta_\bx)$ on $\cG_{n,s}$.
An analogous argument with the roles of $\bx$ and $\bg$ exchanged shows that we also have
\begin{align}
\nonumber
\left( R_n^\bx(\widehat \btheta_\bx) - R_n^\bg(\widehat \btheta_\bg) \right) \one_{\cG_{n,s}}\le
 \sup_{\btheta\in\cS_p}
 \left|R_n^\bg(\btheta) - R_n^\bx(\btheta)\right|.
\end{align}
Hence, for all $\delta>0$,
\begin{align*}
&\lim_{n\to\infty}\P\left(\bigg\{\left| F_n^\bx(t,\bX) 
- F_n^\bg(t,\bX)  \right| > \delta \bigg\} \bigcap \cG_{n,s}\right)\\
&\hspace{61mm}\le 
\lim_{n\to\infty}\P\left(\bigg\{
\sup_{\btheta\in\cS_p}
 \left|R_n^\bg(\btheta) - R_n^\bx(\btheta)\right|  > \delta \bigg\} \bigcap \cG_{n,s}\right) \\
 &\hspace{61mm}\stackrel{(a)}=0,
\end{align*}
where in $(a)$ we used $\P(\cG_{n,s}) \to 1$ for all fixed $s$ and
Lemma~\ref{lemma:proj_locally_lip} along with the assumptions
on $\ell$:
Indeed, via an approximation argument like the one outlined in
Section~\ref{section:proof_binary_response}, one can apply the statement of this lemma
to $R_n^\bg(\btheta), R_n^\bx(\btheta)$ when the response variables $\by$ are 
discrete as in Assumption~\hyperref[ass:loss-0-bis]{5'}.
\end{proof}

\section{Auxiliary lemmas for Theorem~\ref{thm:test_error}}
\label{section:proof_test_error}

\begin{proof}[Proof of Lemma~\ref{lemma:D_universality}]
The proof is a direct application of Theorem~\ref{thm:main_empirical_univ}. Indeed, we have
\begin{align*}
    \limsup_{n\to\infty}\P\left( D^\bX(-s)  \ge t + 3\delta  \right) 
    &=
    \limsup_{n\to\infty}
    \P\Big(
   \frac{\widehat R_{n,-s}^\star(\bX,\by(\bX)) - \widehat R_n^\star(\bX,\by(\bX))  }{-s}  \ge t + 3\delta
    \Big)\\
    &\le  
   \limsup_{n\to\infty} \P\Big(
   \frac{\widehat R_{n,-s}^\star(\bX,\by(\bX)) - \rho  }{-s}  \ge t + 2\delta
    \Big)\\
    &\hspace{10mm}+ \limsup_{n\to\infty}\P\Big( \Big|\widehat R_n^\star(\bX,\by(\bX)) - \rho \Big| \ge s\delta  \Big)\\
    &\stackrel{(a)}= \limsup_{n\to\infty} 
    \P\Big(
   \frac{\widehat R_{n,-s}^\star(\bX,\by(\bX)) - \rho  }{-s}  \ge t + 2\delta
    \Big)\\
    &\stackrel{(b)}\le \limsup_{n\to\infty} 
    \P\Big(
   \frac{\widehat R_{n,-s}^\star(\bG,\by(\bG)) - \rho  }{-s}  \ge t + \delta
    \Big)\\
    &\le \limsup_{n\to\infty} 
    \P\Big(
   \frac{\widehat R_{n,-s}^\star(\bG,\by(\bG)) - \widehat R^\star_n(\bG,\by(\bG))  }{-s}  \ge t 
    \Big) \\
    &\hspace{10mm}+ \limsup_{n\to\infty}\P\left( \left|\widehat R_n^\star(\bG,\by(\bG)) - \rho \right| \ge s\delta  \right)\\
    &\stackrel{(c)}= 
    \limsup_{n\to\infty} \P\left( D^\bG(-s)  \ge t   \right)
\end{align*}
where $(a)$ follows from Theorem~\ref{thm:main_empirical_univ} applied to $\widehat R_n^\star$ along with the assumption that $\widehat R_n^\star(\bG) \stackrel{\P}{\to} \rho$, $(b)$ follows from Theorem~\ref{thm:main_empirical_univ} applied to $\widehat R_{n,-s}^\star$ by absorbing the term $-s R^\bg$ into 
the regularizer,
and $(c)$ follows directly from the assumption $\widehat R_n^\star(\bG) \stackrel{\P}{\to} \rho$.
This proves~\eqref{eq:lim_P_DX}. 
A similar argument establishes the inequality~\eqref{eq:lim_P_DG}.
\end{proof}

\begin{proof}[Proof of Lemma~\ref{lemma:theta_lipschitz_in_s}]
Write $L_F(v,u) := L(F(v), u)$. 
We assume without loss of generality that $L_F(v,u)$ is differentiable in $v$ and that $r$ and $h$ are differentiable in $\btheta$. Otherwise, we can replace all derivatives with subgradients in what follows.
We prove the statement by upper and lower bounding the quantity
\begin{equation}
\nonumber
\widehat R_{n}\left(\widehat\btheta_s^\bX;\bX,\by(\bX)\right) - \widehat R_{n} \left(\widehat\btheta_0^\bX;\bX,\by(\bX)\right).
\end{equation}
For the lower bound, we have
\begin{align*}
    \widehat R_n\left(\widehat\btheta_s^\bX;\bX,\by(\bX)\right) - \widehat R_n\left(\widehat \btheta_0^\bX; \bX,\by(\bX)\right)
    &\ge \frac1n \sum_{i=1}^n \partial_1 L_F\big(\bx_i^\sT \widehat \btheta_0^{\bX};y_i \big)\bx_i^\sT 
    \left( \widehat\btheta_s^\bX -\widehat \btheta_0^\bX  \right)\\
    &\hspace{3mm}+ \grad r\big(\widehat \btheta_0^\bX\big)^\sT \left( \widehat\btheta_s^\bX - \widehat \btheta_0^\bX \right)
    + \frac{\smu}{2} \norm{\widehat\btheta_s^\bX - \widehat \btheta_0^\bX }_2^2\\
    &\stackrel{(a)}= \frac{\smu}{2} \norm{\widehat\btheta_s^\bX - \widehat\btheta_0^\bX}_2^2
\end{align*}
where $(a)$ follows from the KKT conditions for $\widehat R_n$; namely, 
for some $\lambda \ge 0$, we have
\begin{align*}
 \frac1n \sum_{i=1}^n \partial_1L_F\big(\bx_i^\sT \widehat \btheta_0^{\bX};y_i \big)\bx_i
    + 
    \grad r\big(\widehat \btheta_0^\bX\big)
    +
    \lambda \grad h\big(\widehat\btheta_0^\bX\big)  &=0\\
    \lambda \left(h \big(\widehat\btheta_0^\bX \big)  - L\right) &= 0.
\end{align*}
And hence,
\begin{align*}
 &\frac1n \sum_{i=1}^n \partial_1 L_F\big(\bx_i^\sT \widehat \btheta_0^{\bX};y_i \big)\bx_i^\sT 
    \left( \widehat\btheta_s^\bX -\widehat \btheta_0^\bX  \right)+ \grad r\big(\widehat \btheta_0^\bX\big)^\sT \left( \widehat\btheta_s^\bX - \widehat \btheta_0^\bX \right)\\
    &\hspace{-0mm}= \lambda \grad h\big(\widehat\btheta_0^\bX\big)^\sT \left( \widehat\btheta_0^\bX -\widehat \btheta_s^\bX  \right)\\
    &\hspace{-0mm}\ge \lambda h\big(\widehat\btheta_0^\bX\big) - \lambda h\big(\widehat\btheta_s^\bX\big) \\
    &\hspace{-0mm}= \lambda L -   \lambda h\big(\widehat\btheta_s^\bX\big) \\
    &\hspace{-0mm}\ge 0.
\end{align*}

Meanwhile, for the upper bound we write

\begin{align*}
 &\widehat R_{n}\left(\widehat\btheta_s^\bX;\bX,\by(\bX)\right) - \widehat R_{n} \left(\widehat\btheta_0^\bX;\bX, \by(\bX)\right)\\
&= \widehat R_{n,s} \left(\widehat\btheta_s^\bX;\bX,\by(\bX)\right) - \widehat R_{n,s}\left(\widehat\btheta_0^\bX;\bX,\by(\bX)\right)\\
&\hspace{10mm}+ s\left(R^\bg_n\left(\widehat\btheta^\bX_0\right) - R^\bg_n\left(\widehat\btheta^\bX_s\right)\right)\\
&\stackrel{(a)}{\le} \left|s\right| \left| R^\bg_n\left(\widehat\btheta^\bX_0\right) - R^\bg_n\left(\widehat \btheta^\bX_s\right)\right|\\
&\stackrel{(b)}{\le} C_1 |s| \norm{\widehat \btheta^\bX_0 - \widehat \btheta^\bX_s }_2,
\end{align*}
where $(a)$ follows by noting that $\widehat \btheta_s^\bX$ minimizes $\widehat R_{n,s}(\btheta;\bX,\by(\bX))$, and $(b)$ follows since
 $R^\bg_n(\btheta)$ is locally Lipschitz and the constraint set has a bounded radius. Indeed we have
\begin{equation}
\nonumber
 \left| R_n^\bg(\btheta) -  R_n^\bg(\btheta') \right| \le 
 C_2\E\left[\left|\left(\btheta -\btheta'\right)^\sT\bg  \right|\right] = C_2  \E\left[|G|\right] \norm{\btheta - \btheta'}_2.
\end{equation}

Combining the upper and lower bounds and rearranging gives
\begin{equation}
\nonumber
\norm{\widehat\btheta^\bX_s - \widehat \btheta^\bX_0}_2 \le C_3 |s|,
\end{equation}
and hence
\begin{equation}
\nonumber
\norm{\widehat \btheta^\bX_s - \widehat\btheta^\bX_{-s}}_2 \le  \norm{\widehat\btheta^\bX_s - 
\widehat\btheta^\bX_0}_2 + \norm{\widehat\btheta^\bX_{-s} - \widehat\btheta^\bX_0}_2 \le C_4|s|.
\end{equation}
This proves the first statement for $\bX$. A similar argument clearly holds for the Gaussian model.

Now to prove~\eqref{eq:lip_D}, we can write
\begin{align*}
D^\bX(-s) - D^\bX(s)
&= -\frac{1}{s}\left(\widehat R_n(\widehat \btheta_{-s}^\bX;\bX,\by(\bX)) - \widehat R_n(\widehat 
\btheta_0^\bX;\bX,\by(\bX))\right)
\\
&\hspace{10mm}-\frac1s \left(\widehat R_n(\widehat \btheta_s^\bX;\bX, \by(\bX)) - \widehat R_n(\widehat \btheta_0^\bX;\bX, \by(\bX)) 
\right)\\
&\hspace{10mm}+
R_n^\bg\left(\widehat 
\btheta_{-s}\left(\bX\right)\right) - R_n^\bg\left(\widehat\btheta_s\left(\bX\right)\right)
\\
&\stackrel{(a)}{\le}  \left|R_n^\bg(\widehat \btheta^\bX_{-s}) - R_n^\bg(\widehat \btheta^\bX_s )\right|\\
&\stackrel{(b)}{\le} C_5  \norm{\widehat\btheta^\bX_{-s} - \widehat\btheta^\bX_s}_2\\
&\le C_6\,s
\end{align*}
where in $(a)$ we used that 
$$\widehat R_n(\widehat \btheta^\bX_{-s};\bX,\by(\bX))\ge \widehat R_n(\widehat \btheta^\bX_0;\bX,\by(\bX))
\qquad
\textrm{and that}
\qquad
\widehat R_n(\widehat 
\btheta^\bX_{s},\bX) \ge \widehat R_n(\widehat \btheta^\bX_0,\bX)$$
and in $(b)$ we used that 
that $R_n^\bg(\btheta)$ is
is locally Lipschitz and the constraint set has a bounded radius. A similar argument then shows the same property for $D^\bG(s)$.
\end{proof}

\section{The neural tangent model: Proof of Theorem~\ref{cor:ntk_universality}}
\label{proof:ntk_universality}

Let us begin by recalling the definitions and assumptions on the model defined in Section~\ref{section:ntk_example}.
Recall the activation function $\sigma$ that is assumed to be four times differentiable with bounded derivatives and satisfying
 $\E[\sigma'(G)] = 0$, and $E[G\sigma'(G)] = 0$ for $G\sim\cN(0,1)$.
Further recall the weight matrix $\bW$ whose $m$ columns are $\bw_j\stackrel{i.i.d}{\sim}\textsf{Unif}\left(\S^{d-1}(1)\right)$, $j\in[m]$.
The feature vectors for the neural tangent model were then defined in~\eqref{eq:ntk_covariates} as
\begin{equation}
\nonumber
    \bx_i = \left(\bz_i \sigma'\left(\bw_1^\sT \bz_i\right),\dots,\bz_i \sigma'\left(\bw_m^\sT \bz_i\right)\right) \in\R^p,
\end{equation}

where $\bz_i\stackrel{\textrm{i.i.d.}}{\sim}\cN(0,\bI_d)$ for $i\in[n]$. Additionally, for the Gaussian model we defined
 $\bg | \bW \sim \cN\left(0, \E\left[\bx\bx^\sT | \bW \right] \right)$.
We assume $m(n)/d(n) \to \widetilde\sgamma_\NT$ and $p(n)/n \to \sgamma$ as $n\to\infty$. As we have done so far, we suppress the dependence of these integers on $n$.

For a given $\btheta =\left(\btheta_{\up{1}}^\sT,\dots,\btheta_\up{m}^\sT\right)^\sT\in\R^p$, where $\btheta_{\up{j}} 
\in \R^d$ for $j\in[m]$, we introduced the notation $\bT_\btheta\in\R^{d\times m}$ to denote 
the matrix $\bT_\btheta = \left(\btheta_\up{1} ,\dots, \btheta_\up{m}\right) \in\R^{d\times m}$
so that we can write $\btheta^\sT \bx = \bz^\sT \bT_\btheta \sigma'\left(\bW^\sT \bz\right)$, 
where $\sigma':\R\to\R$ is applied element-wise. Finally, recall the set 
\begin{equation}
\nonumber
    \cS_{p}  = \left\{ 
    \btheta\in\R^p : \norm{\bT_\btheta}_\op \le 
\frac{ \sfR}{\sqrt{d}}
    \right\}.
\end{equation}
Note that $\cS_p$ is symmetric, convex, and $\cS_p\subseteq B_2^p(\sfR)$. Furthermore, for all $\btheta\in\cS_p$ we have $\norm{\btheta_\up{j}}_2 \le \sfR/\sqrt{d}$ for all $j\in[m]$.

The key to proving Theorem~\ref{cor:ntk_universality} is showing that the distribution of the feature vectors $\{\bx_i\}_{i\le[n]}$ 
satisfy, on a high probability set, Assumption~\ref{ass:X} along with Eq.~\eqref{eq:condition_bounded_lipschitz_single} for the set $\cS_p$ above.
Our proof here is analogous to that of~\cite{hu2020universality} for the random features model.
Let us begin our treatment by defining the event
\begin{equation}
\nonumber
    \cB := \left\{\sup_{\{i,j \in [m] : i\neq j\}} \left|\bw_i^\sT\bw_j\right| \le 
C\left(\frac{\log m}{d}\right)^{1/2}\right\} \bigcap 
    \left\{
    \norm{\bW }_\op
    \le C'
    \right\}
\end{equation}
for some $C,C'$ depending only on $\widetilde\sgamma_\NT$ so that $\P(\cB^c) \to0$
as $n\to\infty$. The existence of such constants is a standard result (see for example~\cite{vershynin2018high}.) However, we include it as Lemma~\ref{lemma:B_tail_bound} of Section~\ref{section:aux_lemmas_ntK} for completeness.

\subsection{Asymptotic Gaussianity on a subset of $\cS_p$}
\label{subsection:bounded_diff_delta}

For a given $\delta>0$, let us define the set
\begin{equation}
\nonumber
\cS_{p,\delta} := \left\{\btheta\in\cS_p : \btheta^\sT \E\left[\bx\bx^\sT\right]\btheta > \delta\right\},
\end{equation}
We establish in this section the following lemma.

\begin{lemma}
\label{lemma:ntk_bl-1}
For all $\delta>0$ and any differentiable bounded function $\varphi:\R\to\R$ with bounded derivative, we have
\begin{equation}
    \lim_{n\to\infty} \sup_{\btheta\in \cS_{p,\delta}} 
\left|\E\left[\varphi\left(\btheta^\sT\bx\right)\one_{\cB} \Big| \bW 
\right]  
- 
\E\left[\varphi\left(\btheta^\sT\bg\right)\one_{\cB}\Big
| \bW 
\right]\right| 
= 0.
\end{equation}
\end{lemma}

Define, for $\btheta\in\cS_{p,\delta}$ the notation
\begin{equation}
\nonumber
    \nu^2 = \nu^2_\btheta :=  \btheta^\sT \E\left[\bx\bx^\sT\right] \btheta > \delta.
\end{equation}

For a fixed bounded Lipschitz function $\varphi$, let $\chi = \chi_\varphi$ be the 
solution to Stein's equation for $\varphi$, namely, the function $\chi$ satisfying
\begin{equation}
\label{eq:stein_solution}
  \E\left[ \varphi\left(\frac{\btheta^\sT \bx}{\nu}\right) - \varphi\left(\frac{\btheta^\sT \bg}{\nu}\right)\right]
  =
 \E\left[\chi'\left(\frac{\btheta^\sT\bx}{\nu}\right)  - 
\frac{\btheta^\sT\bx}{\nu}\chi\left(\frac{\btheta^\sT\bx}{\nu}\right)   \right]
\end{equation}
(see~\cite{chen2011normal} for more on Stein's method and properties of the solution $\chi$.).
In order to prove Lemma~\ref{lemma:ntk_bl-1}, it is sufficient to show that
\begin{equation}
\label{eq:stein_equiv}
   \lim_{n\to\infty} \sup_{\btheta\in \cS_{p,\delta}}\left|\E\left[\chi'\left(\frac{\btheta^\sT\bx}{\nu}\right)  - 
\frac{\btheta^\sT\bx}{\nu}\chi\left(\frac{\btheta^\sT\bx}{\nu}\right)   \right]\right|  = 0.
\end{equation}

To simplify notation, define
\begin{equation}
\label{eq:Delta_i_def}
    \Delta_i := \frac{\btheta^\sT\bx}{\nu} - \frac{1}{\nu}\sum_{j: j\neq i 
}\btheta_\up{j}^\sT\bP_{i}^\perp\bz \sigma'(\bw_j^\sT\bz - \rho_{i,j}\bw_i^\sT \bz),
\end{equation}
where 
\begin{equation}
\nonumber
    \bP_{i}^\perp := \bI - \bw_i\bw_i^\sT,\quad \rho_{ij}:= 
{\bw_j^\sT\bw_i}.
\end{equation}
In Section~\ref{section:ntk_actual_proof}, we upper bound the quantity~\eqref{eq:stein_equiv} as 
\begin{align}
\label{eq:decomposition}
    &\left|\E\left[\chi'\left(\frac{\btheta^\sT\bx}{\nu}\right)  - 
\frac{\btheta^\sT\bx}{\nu}\chi\left(\frac{\btheta^\sT\bx}{\nu}\right)   \right]\right|  \\
    &\hspace{7mm}\le \left|
    \E\left[\left(\frac{1}{\nu}\sum_{i=1}^m \btheta_\up{i}^\sT\bz 
\sigma'(\bw_i^\sT\bz)\Delta_i 
- 1 
    \right)\chi'\left(\frac{\btheta^\sT\bx}{\nu}\right)\right]
    \right|\nonumber\\
    &\hspace{15mm}+
    \left|
    \E\left[
    \frac{1}{\nu}\sum_{i=1}^m
    \btheta_\up{i}^\sT\bz\sigma'(\bw_i^\sT\bz)
    \left(\chi\left(\frac{\btheta^\sT\bx}{\nu}\right) 
    -
    \chi\left(\frac{\btheta^\sT\bx}{\nu} - \Delta_i\right) 
    - \Delta_i\chi'\left(\frac{\btheta^\sT\bx}{\nu}\right)
    \right)
    \right]
    \right|.
\nonumber
\end{align}
So first, let us control the terms on the right hand side: We do this in Sections~\ref{section:first_term_ntk} and~\ref{section:second_term_ntk}, respectively.
Before doing this, we make the following definitions which will be used throughout. 
Define $\widetilde \btheta_{j,i} := \bP_{i}^\perp\btheta_\up{j},$
along with the matrix notation
\begin{align*}
    \bD_l &:= \diag\left\{\bsigma^\up{l}(\bW^\sT\bz)\right\},\quad\quad
\bM:=\diag\left\{\bW^\sT\bz\right\},\quad\quad \widetilde \bM := \diag\left\{\bT_\btheta^\sT 
\bz\right\},\\
    \bA &:= \bW^\sT\bT_\btheta - 
    \left(\bW^\sT \bT_{\btheta}\right) \odot \bI_{m},\quad\quad
\bR := \bW^\sT\bW - \bI_m,\quad\quad
\bN:= \left(\widetilde \btheta_{j,i}^\sT \bz\right)_{i,j \in[m]},
\end{align*}
where we write $\bsigma^\up{l}(\bv)$ to denote the element-wise application of $\sigma^\up{l}:\R\to\R$, the $l$th derivative of $\sigma$ to a vector $\bv$. Additionally, here $\odot$ denotes the Hadamard product,
and $\diag\{\bv\}$ for a vector $\bv$ denotes the matrix whose elements on the main diagonal are the elements of $\bv$, and whose elements off the main diagonal are $0$.

We prove the following bounds.
\begin{lemma}
\label{lemma:op_norm_bounds}
For $\bW\in\cB$, we have for any fixed integers $k>0$ and $l\le 4 $
\begin{align*}
&\norm{\bD_l}_\op \le C_0,\hspace{4mm}
\E\left[\norm{\bM}_\op^k \right] \le  C_1(\log m)^{k/2} , \hspace{4mm}
\E\left[\norm{\widetilde\bM}^k_\op\right]\le C_2 \frac{(\log{m})^{k/2}}{m^{k/2}},\\
&\norm{\bA}_\op \le \frac{C_3}{m^{1/2}},\hspace{4mm}
\norm{\bR}_\op \le C_4,\hspace{4mm}
\E\left[\norm{\bN\odot \bR }_\op^k  \right] \le C_5\frac{(\log m)^{k/2}}{m^{k/2}},\hspace{4mm}
\norm{\bR\odot \bR}_\op \le C_6,\\
&\E\left[\norm{\bN\odot \bR \odot \bR }_\op^k \right] \le C_7\frac{(\log 
m)^{k/2}}{m^{k/2}},\hspace{2mm}
\norm{\bA \odot \bR}_\op \le C_8 \frac1{\sqrt{m}},\\
&\norm{\bA \odot \bR \odot \bR}_\op \le C_{9} \frac1{\sqrt{m}},
\end{align*}
for some constants $C_i>0.$
\end{lemma}

\begin{proof}
Using Lemma~\ref{lemma:subgaussian_max},
the first five inequalities are direct. Indeed, recalling that $m/d \to\widetilde\sgamma_\NT$, we have
\begin{align*}
&\norm{\bD_l}_\op =  \sup_{i\in[m]}\left|\sigma^\up{l}(\bw_i^\sT \bz)\right| \le 
\norm{\sigma^\up{l}}_\infty \le C_0,\\
&\E\left[\norm{\bM}_\op^k  \right]= \E\left[\sup_{i\in[m]} \left|\bw_i^\sT \bz\right|^k\right] 
\le C_1\left(\log m\right)^{k/2},\\
&\E\left[\norm{\widetilde\bM}^k_\op\right] = \E \left[ \sup_{i\in[m]} \left| \btheta_\up{i}^\sT 
\bz\right|^k\right] \le C_2 \frac{\left(\log{m}\right)^{k/2}}{m^{k/2}},\\
&\norm{\bA}_\op \le \norm{\bW}_\op\norm{\bT_\btheta}_\op  + 
\sup_{i}\left|\bw_i^\sT \btheta_\up{i}\right|
\le C_3 \frac{1}{m^{1/2}},\\
&\norm{\bR}_\op \le \norm{\bW}_\op^2 + 
\norm{\bI}_\op \le C_4.
\end{align*}
For the remaining inequalities, let $\bB\in\R^{m\times m}$ be an arbitrary fixed matrix and note 
that we have
\begin{align}
    \bN \odot \bB &= \left(\btheta_\up{j}^\sT \bz\right)_{i,j\in[m]} \odot \bB \quad-\quad \left({\btheta_\up{j}^\sT 
\bw_i \bw_i^\sT \bz} \right)_{i,j \in[m]}\odot \bB\nonumber\\
    &= \bB \widetilde \bM - (\bM \bW^\sT \bT_\btheta ) \odot \bB \nonumber\\
    &= \bB \widetilde \bM - 
\left(\bW^\sT\bT_\btheta\right) \odot (\bM \bB),
\label{eq:hadamard_eq}
\end{align}
where the last equality holds because $\bM$ is a diagonal matrix. Now recall 
that for the two square matrices $\bW^\sT\bT_\btheta$ and $\bM\bB$, we have (see for example~\cite{johnson1990matrix}, (3.7.9))
\begin{equation}
\label{eq:hadamard_ineq}
   \norm{\left(\bW^\sT\bT_\btheta\right)\odot \left(\bM\bB\right) }_\op \le \left(\norm{\bI 
\odot \bT_\btheta^\sT\bT_\btheta}_\op \norm{\bI \odot \bW^\sT\bW}_\op\right)^{1/2} \norm{\bM\bB}_\op.
\end{equation}

\noindent Combining~\eqref{eq:hadamard_eq} with~\eqref{eq:hadamard_ineq} we can write
\begin{align}
\E\left[\norm{\bN \odot \bB}_\op^k\right] &\le
\E\left[\left(\norm{\bB\widetilde \bM}_\op + 
\norm{\left(\bW^\sT\bT_\btheta\right)\odot \left(\bM\bB\right) }_\op
\right)^k\right]
\nonumber\\
&\le
\E\left[\left(\norm{\bB\widetilde\bM}_\op + 
\sup_{i\in[m]} \norm{\btheta_\up{i}}_2
\sup_{i\in[m]} \norm{\bw_i}_2
\norm{\bM\bB}_\op
\right)^k\right]\nonumber\\
&\le 
C_{10} \norm{\bB}_\op^k
\E\left[ \norm{\widetilde\bM}_\op^k\right] 
+
C_{10}\sup_{i\in[m]} \norm{\btheta_\up{i}}_2^k
\norm{\bB}_\op^k
\E\left[
\norm{\bM}_\op^k
\right]\nonumber\\
&\le 
C_{11}\norm{\bB}^k_\op
\left(
\frac{ (\log m)^{k/2} }{m^{k/2}} + 
\frac{ (\log m)^{k/2}}{d^{k/2}}
\right)\label{eq:hadamard_ineq_2}.
\end{align}
Hence, using $\norm{\bR}_\op\le C_4$ and $m/d\to\widetilde \sgamma_\NT$, we have
$$\E\left[
\norm{\bN \odot \bR}_\op^k
\right] = C_5 (\log m)^{k/2}/{m^{k/2}}$$
which establishes the sixth bound.

Now, note that 
\begin{align*}
\norm{\bR \odot \bR}_\op &= \norm{\bW^\sT \bW \odot \bR}_\op\\
&\stackrel{(a)}{\le} \norm{\bI \odot\ \bW^\sT \bW }_\op  \norm{\bR}_\op\\
&=  \sup_{i\in [m]} \left|\bw_i^\sT \bw_i\right| \norm{\bR}_\op\\
&\le C_6,
\end{align*}
where $(a)$ follows using the same bound we
applied to~\eqref{eq:hadamard_ineq}. This establishes the seventh bound in the lemma.

Now, using~\eqref{eq:hadamard_ineq_2} and the bound applied to~\eqref{eq:hadamard_ineq} again gives
$$\E\left[
\norm{\bN \odot \bR \odot \bR}_\op^k\right] \le C_7 {(\log m)^{k/2}}/{m^{k/2}}
,$$
establishing the eighth bound.

For the ninth bound, we first note that by definition of $\bA$ and $\bR$, 
$\bA \odot \bR  =\bW^\sT  
\bT_\btheta \odot \bR$ so that
\begin{align*}
 \norm{\bA \odot \bR}_\op &\le  \left(\norm{\bI \odot \bW^\sT 
\bW }_\op \norm{\bI \odot 
\bT_\btheta^\sT\bT_\btheta}_\op\right)^{1/2} \norm{\bR}_\op\\
&\le 
\sup_{i\in[m]} \norm{\btheta_\up{i}}_2  \norm{\bR}_\op\\
&\stackrel{(a)}{\le} C_8 \frac{1}{\sqrt{m}}
\end{align*}
where in $(a)$ we used that $\norm{\bR}_\op \le C_4$ and $\norm{\btheta_\up{i}}_2\le\sfR/\sqrt{d}$ along with $m/d\to\widetilde\sgamma_\NT$. 

Finally, using $\norm{\bR \odot \bR}_\op \le C_6$, a similar argument shows 
that $\norm{\bA \odot \bR \odot \bR}_\op \le C_{9}/\sqrt{m}$, yielding the final bound of the lemma.
\end{proof}

\subsubsection{Bounding the first term in Eq.~\eqref{eq:decomposition}}
\label{section:first_term_ntk}
\begin{lemma}
\label{lemma:first_term_ntk}
For any $\delta>0$, we have
\begin{equation}
\label{eq:first_term_ntk}
 \lim_{n\to\infty}\sup_{\btheta\in\cS_{p,\delta}}\left|
    \E\left[\left(\frac{1}{\nu}\sum_{i=1}^m \btheta_\up{i}^\sT\bz 
\sigma'(\bw_i^\sT\bz)\Delta_i 
- 1 
    \right)\chi'\left(\frac{\btheta^\sT\bx}{\nu}\right)\right]
    \right|  = 0 
\end{equation}
\end{lemma}
\begin{proof}
Fix $\delta>0$ throughout. Define for convenience
\begin{equation}
\nonumber
    U := \frac{1}{\nu}\sum_{i=1}^m \btheta_\up{i}^\sT \bz \sigma'(\bw_i^\sT \bz)\Delta_i.
\end{equation}
Let us compute the expectation of $U$ and control its variance.\\

The expectation can be computed as
\begin{align*}
&\E[U] = \E\left[
\frac{1}{\nu}\sum_{i=1}^m \btheta_\up{i}^\sT \bz \sigma'(\bw_i^\sT \bz)
\left(\frac{\btheta^\sT\bx}{\nu} +  \Delta_i - \frac{\btheta^\sT\bx}{\nu} 
\right)
\right]\\
&= 
\E\left[
\frac{1}{\nu^2}
\left({\btheta^\sT\bx}\right)^2\right] +
\frac{1}{\nu}\sum_{i=1}^m \E\left[\btheta_\up{i}^\sT \bz \sigma'(\bw_i^\sT 
\bz)\left(\Delta_i - \frac{\btheta^\sT\bx}{\nu}\right)\right]\\
&\stackrel{(a)}{=}1
+\E\left[\btheta_\up{i}^\sT\bP_i^\perp\bz 
\sigma'(\bw_i^\sT\bz)
\left(\Delta_i - \frac{\btheta^\sT \bx}{\nu}\right)
\right] + 
    \btheta_\up{i}^\sT\bw_i\E\left[\bw_i^\sT\bz\sigma'(\bw_i^\sT\bz)\left(\Delta_i - \frac{\btheta^\sT \bx}{\nu}\right)\right] \\
    &\stackrel{(b)}{=}1 + 
    \E\left[\btheta_\up{i}^\sT\bP_i^\perp\bz 
\left(\Delta_i - \frac{\btheta^\sT \bx}{\nu}\right)
\right] 
\E\left[
\sigma'(\bw_i^\sT\bz)
\right]\\
&\hspace{10mm}+ 
\btheta_\up{i}^\sT\bw_i\E\left[\bw_i^\sT\bz\sigma'(\bw_i^\sT\bz)\right]\E\left[\left(\Delta_i - \frac{\btheta^\sT \bx}{\nu}\right)\right] \\
    &\stackrel{(c)}{=}1
\end{align*}
where $(a)$ follows by the definition of $\nu$, $(b)$ follows by independence of the difference $\Delta_i - \btheta^\sT \bx/\nu$ and $\bw_i^\sT \bz$, which 
can be seen from the definition of $\Delta_i$, and $(c)$ follows by the assumption on $\sigma'$, namely, that 
$\E[\sigma'(G)] = \E[ G\sigma'(G)] = 0$ for $G$ standard normal.\\

Now, we control $\Var(U)$. First, note that we can write $\Delta_i$ as 
\begin{equation}
\label{eq:Delta_i_new_form}
    \Delta_i = \frac{\btheta_\up{i}^\sT\bz \sigma'(\bw_i^\sT\bz)}{\nu} 
    +\frac{1}{\nu}\sum_{j:j\neq i}
    \left\{\btheta_\up{j}^\sT\bz \sigma'(\bw_j^\sT\bz) - \btheta_\up{j}^\sT \bP_i^\perp \bz 
\sigma'(\bw_j^\sT \bz - \rho_{ij}\bw_i^\sT\bz)\right\}.
\end{equation}
Taylor expanding $\sigma'$ to the third order gives
\begin{align}
\nonumber
    \sigma'(\bw_j^\sT\bz - \rho_{ij}\bw_i^\sT\bz) &= \sigma'(\bw_j^\sT\bz) - 
\rho_{ij}\bw_i^\sT\bz \sigma''(\bw_j^\sT\bz) + \frac12 \rho_{ij}^2(\bw_i^\sT\bz)^2 
\sigma'''(\bw_i^\sT\bz)\\
    &\hspace{10mm}-\frac16\rho_{ij}^3(\bw_i^\sT\bz)^3 \sigma^\up{4}(v_{ij}(\bz))\nonumber
\end{align}
for some $v_{ij}(\bz)$ between $\bw_j^\sT \bz - \rho_{ij}\bw_i^\sT\bz$ and $\bw_j^\sT\bz$. Using this expansion and the notation $\widetilde \btheta_{j,i}$ defined earlier,
$\Delta_i$ can be re-written as
\begin{align}
\label{eq:delta_i_expansion}
    \Delta_i = \frac{1}{\nu} \btheta_\up{i}^\sT \bz \sigma'(\bw_i^\sT \bz) 
    &+ \frac1{\nu}\sum_{j: j\neq i} \btheta_\up{j}^\sT \bw_i  \bw_i^\sT 
\bz \sigma'(\bw_j^\sT \bz)\\
    &+ \frac1{\nu} \sum_{j: j\neq i}\widetilde\btheta_{j,i}^\sT\bz 
    \left( \rho_{ij} \bw_i^\sT \bz \sigma''(\bw_j^\sT \bz)
    -\frac12\rho_{ij}^2(\bw_i^\sT \bz)^2 \sigma'''(\bw_j^\sT\bz)\right)\nonumber\\
    &+ \frac1{6\nu}\sum_{j: j\neq i}\widetilde\btheta_{j,i}^\sT\bz \rho_{ij}^3 
(\bw_i^\sT\bz)^3 \sigma^\up{4}(v_{ij}(\bz)).\nonumber
\end{align}
\noindent Using the expansion~\eqref{eq:delta_i_expansion} in the expression for $U$ 
gives
\begin{align}
    U =& \frac{1}{\nu^2}\sum_{i=1}^m \left(\btheta_\up{i}^\sT \bz \sigma'(\bw_i^\sT 
\bz)\right)^2\label{eq:u_1_z}\\
    &+ \frac1{\nu^2}\sum_{i,j: j\neq i} \btheta_\up{i}^\sT\bz 
\sigma'(\bw_i^\sT\bz)\btheta_\up{j}^\sT \bw_i  \bw_i^\sT \bz 
\sigma'(\bw_j^\sT \bz)\label{eq:u_2_z}\\
    &+ \frac1{\nu^2} \sum_{i,j: j\neq i}\btheta_\up{i}^\sT\bz 
\sigma'(\bw_i^\sT\bz)\widetilde\btheta_{j,i}^\sT\bz 
    \left\{ \rho_{ij} \bw_i^\sT \bz \sigma''(\bw_j^\sT \bz)
    -\frac12\rho_{ij}^2(\bw_i^\sT \bz)^2 \sigma'''(\bw_j^\sT\bz)\right\}\label{eq:u_3_z}\\
    &+ \frac1{6\nu^2}\sum_{i,j: j\neq i}\btheta_\up{i}^\sT\bz
    \sigma'(\bw_i^\sT\bz)\widetilde\btheta_{j,i}^\sT\bz \rho_{ij}^3 (\bw_i^\sT\bz)^3 
\sigma^\up{4}(v_{ij}(\bz)).\label{eq:u_4_z}
\end{align}
Let us write $u_1(\bz)-u_4(\bz)$ for the terms on the right-hand on 
lines \eqref{eq:u_1_z}-\eqref{eq:u_4_z} respectively.
Observe that

 \begin{equation}
 \label{eq:var_U_bound_in_terms_of_u}
\Var(U)^{1/2} \le \sum_{l=1}^4 \Var(u_l(\bz))^{1/2} \stackrel{(a)}{\le}C_0 \sum_{l=1}^3 
\left(\E\left[\norm{\grad u_l(\bz)}_2^2\right]\right)^{1/2} + C_0\Var(u_4(\bz))^{1/2}
 \end{equation}
where $(a)$ follows from the Gaussian Poincar\'e inequality. 
We control each summand directly. In doing so, we will make heavy use of the bounds in 
Lemma~\ref{lemma:op_norm_bounds} and hence we will often do so without reference.
First let us bound the expected norm of the gradients in the above display.

For $\E\left[\norm{\grad u_1(\bz)}_2^2\right]$ we have the bound
\begin{align*}
    &\E\left[\big\|{\grad u_1(\bz)\|}_2^2\right]\\
    &=
    \E\Big[\Big\|\frac2{\nu^2} \sum_{i=1}^m 
\btheta_\up{i}^\sT\bz \sigma'(\bw_i^\sT\bz)^2 
\btheta_\up{i} 
    +
    \frac2{\nu^2} \sum_{i=1}^m (\btheta_\up{i}^\sT\bz)^2 \sigma''(\bw_i^\sT 
\bz)\sigma'(\bw_i^\sT\bz)\bw_i\Big\|_2^2\Big]\nonumber\\
&\le 
    \frac8{\nu^4}\E\Big[\Big\|{  \sum_{i=1}^m 
\btheta_\up{i}^\sT\bz \sigma'(\bw_i^\sT\bz)^2 
\btheta_\up{i}}\Big\|_2^2\Big]\\
    &\hspace{10mm}+
    \frac8{\nu^4}\E\Big[\Big\| \sum_{i=1}^m (\btheta_\up{i}^\sT\bz)^2 \sigma''(\bw_i^\sT 
\bz)\sigma'(\bw_i^\sT\bz)\bw_i\Big\|_2^2\Big]\nonumber\\
&\le 
\frac{8}{\nu^4}\E\left[\norm{\bT_\btheta\bD_1^2\bT_\btheta^\sT\bz}_2^2\right] + 
\frac{8}{\nu^4}
    \E\left[\norm{\bW\bD_1\bD_2\left((\btheta_\up{i}^\sT \bz)^2\right)_{i\in[m]}}_2^2\right]\nonumber\\
&\le \frac{8}{\nu^4}  \norm{\bT_\btheta}_\op^4 \E\left[ \norm{\bD_1}^4_\op 
\norm{\bz}_2^2\right] 
+ \frac{8}{\nu^4}\norm{\bW}^2_\op \E\Big[ \norm{\bD_1\bD_2}^2_\op 
\Big\|\left( (\btheta_\up{i}^\sT \bz)^2 \right)_{i\in[m]} \Big\|_2^2\Big]\nonumber\\
&\stackrel{(a)}{\le} \frac{C_1}{\nu^4} \frac{1}{d} + \frac{C_2}{\nu^4} 
\sum_{i=1}^m \E\left[(\btheta_\up{i}^\sT \bz)^4 \right]\nonumber\\
&= \frac{C_1}{\nu^4} \frac{1}{d}  + \frac{C_2}{\nu^4}\E[G^4] 
\sum_{i=1}^m\norm{\btheta_\up{i}}_2^4\\
&\stackrel{(b)}{\le} C_3(\delta) \left(\frac{1}{d}+ 
\frac{m}{d^2}\right)
\end{align*}
for all $\btheta\in\cS_{p,\delta}$, where $C_1,C_2,C_3(\delta) >0$, and $G$ is a standard normal variable.
Here, $(a)$ follows from the bound $\norm{\bT_\btheta}_\op\le \sfR /\sqrt{d}$ for $\btheta\in\cS_p$ along with 
the bounds in Lemma~\ref{lemma:op_norm_bounds}, and $(b)$ 
follows from $\nu = \nu_{\btheta} > \delta$ for $\btheta\in\cS_{p,\delta}$, and the bound $\norm{\btheta_\up{i}}_2 \le \sfR/\sqrt{d}$
Taking the supremum over $\cS_{p,\delta}$ then sending $n\to\infty$ shows that
\begin{equation}
\label{eq:u_1_var_asymp}
\lim_{n\to\infty} \sup_{\btheta\in\cS_{p,\delta}} \E\left[\norm{\grad u_1(\bz)}_2^2 \right] = 0.
\end{equation}

Now, the gradient of $u_2(\bz)$ can be computed as
\begin{align}
    \grad u_2(\bz) &=\frac{1}{\nu^2} \sum_{i,j:i\neq j}
    \sigma'(\bw_i^\sT\bz) \btheta_\up{j}^\sT\bw_i \bw_i^\sT\bz 
\sigma'(\bw_j^\sT \bz) \btheta_\up{i}\nonumber\\
    &\hspace{10mm}+
    \frac{1}{\nu^2} \sum_{i,j:i\neq j}
    \btheta_\up{i}^\sT\bz
    \sigma''(\bw_i^\sT\bz) \btheta_\up{j}^\sT\bw_i
     \bw_i^\sT\bz \sigma'(\bw_j^\sT \bz) 
     \bw_i\nonumber\\
    &\hspace{10mm}+\frac{1}{\nu^2} \sum_{i,j:i\neq j}
    \btheta_\up{i}^\sT\bz
    \sigma'(\bw_i^\sT\bz) \btheta_\up{j}^\sT\bw_i
    \sigma'(\bw_j^\sT \bz)\bw_i\nonumber\\
    &\hspace{10mm}+  
    \frac{1}{\nu^2} \sum_{i,j:i\neq j}
    \btheta_\up{i}^\sT\bz
    \sigma'(\bw_i^\sT\bz) 
    \btheta_\up{j}^\sT\bw_i
     \bw_i^\sT \bz
    \sigma''(\bw_j^\sT \bz)\bw_j\nonumber\\ 
    &=\frac1{\nu^2} \bT_\btheta \bD_1 \bM \bA \bsigma'(\bW^\sT \bz)
    + \frac1{\nu^2} \bW \bD_2 \widetilde \bM \bM \bA \bsigma'(\bW^\sT \bz) \nonumber\\
    &\hspace{10mm}+ \frac1{\nu^2} \bW \bD_1 \widetilde\bM \bA \bsigma'(\bW^\sT \bz) 
    + \frac1{\nu^2} \bW \bD_2 \bA \bD_1 \bM \bT_\btheta^\sT \bz \label{eq:grad_u_2},
\end{align}
where we recall that $\bsigma(\bv)$ denotes the vector whose $i$th entry is $\sigma(v_i)$. We 
have the following bounds on the expected norm squared of each term 
in~\eqref{eq:grad_u_2}: for the first of these terms,
\begin{align}
\frac{1}{\nu^4} 
\E\left[ 
\norm{
\bT_\btheta
\bD_1 \bM \bA \bsigma'(\bW^\sT \bz)
}_2^2
\right] &\le \frac{1}{\nu^4}\norm{\bT_\btheta}_\op^2 
\norm{\bA}_\op^2
\E\left[
\norm{\bD_1}_\op^2
\norm{\bM}_\op^2 
\norm{\bsigma'(\bW^\sT\bz)}^2_2
\right]\nonumber\\
&\le\frac{C_4}{\nu^4 d m}\E\left[
\norm{\bM}_\op^4\right]^{1/2}\E\left[\norm{\bsigma'(\bW^\sT\bz)}^4_2
\right]^{1/2}\nonumber\\
&\le\frac{C_4 \log m}{\nu^4 d m} \E\left[ \left(\sum_{i=1}^m 
\sigma'(\bw_i^\sT \bz)^2 \right)^2\right]^{1/2} \nonumber\\
&\stackrel{(a)}{\le}C_5(\delta)\left(\frac{ \log m  }{d}\right),\nonumber
\end{align}
for all $\btheta\in\cS_{p,\delta}$, where $C_4,C_5(\delta)>0$ are constants.
Note that in $(a)$ we used $\norm{\sigma'}_\infty$ is finite.

Moving on to bound the norm squared of the second term in~\eqref{eq:grad_u_2}, we have
\begin{align}
&\frac{1}{\nu^4 }
\E\left[
\norm{
\bW \bD_2\widetilde\bM\bM \bA
\bsigma'(\bW^\sT\bz)
}_2^2
\right]\nonumber\\
&\hspace{25mm}\le \frac{1}{\nu^4}
\norm{\bW}^2_\op
\norm{\bA}_\op^2
\E\left[ 
\norm{\bD_2}_\op^2 \norm{\widetilde \bM}_\op^2\norm{\bM}_\op^2  
\norm{\bsigma'(\bW^\sT\bz)}_2^2
\right]\nonumber\\
&\hspace{25mm}\le C_6  \frac{ 1}{m} \E\left[\norm{\widetilde \bM}_\op^6 
\right]^{1/3} \E\left[ \norm{\bM}_\op^6 \right]^{1/3} \E\left[ 
\norm{\bsigma'(\bW^\sT\bz)}_2^6\right]^{1/3}\nonumber\\
&\hspace{25mm}= C_7(\delta)\frac{(\log m)^2}{m}.\nonumber
\end{align}
Similarly, the expected norm squared of the third term in~\eqref{eq:grad_u_2} is bounded 
as
\begin{align}
\frac{1}{\nu^4}\E\big[
\big\|\bW \bD_1 \widetilde\bM \bA \bsigma'(\bW^\sT\bz)\big\|_2^2
\big]
&\le
\frac{C_8}{\nu^4} 
\norm{\bW}_\op^2 \norm{\bA}_\op^2 
\E\big[ 
\big\|\widetilde \bM\big\|_\op^4\big]^{\frac12}
\E\big[\big\|{\bsigma'(\bW^\sT \bz)}\big\|_2^4\big]^{\frac12}\nonumber\\
&=C_9(\delta)\frac{\log m}{m},\nonumber
\end{align}
and finally, for the fourth term in~\eqref{eq:grad_u_2} we have
\begin{align}
&\frac{1}{\nu^4}\E\big[
\big\|\bW \bD_2 \bA \bD_1 \bM \bT_\btheta^\sT \bz \big\|_2^2
\big]\\
& \le \frac{1}{\nu^4}
\norm{\bW}_\op^2 
\norm{\bA}_\op^2
\norm{\bT_\btheta}_\op^2
\E\left[
\norm{\bD_2}_\op^2  \norm{\bD_1}_\op^2 \norm{\bM}_\op^2  \norm{\bz}_2^2
\right]\nonumber\\
&=C_{10}(\delta)\frac{\log m}{m}\nonumber
\end{align}
Hence, we similarly conclude that 
\begin{equation}
\label{eq:u_2_var_asymp}
\lim_{n\to\infty} \sup_{\btheta\in\cS_{p,\delta}} \E\left[\norm{\grad u_2(\bz)}_2^2 \right] = 0.
\end{equation}

Now moving on to $u_3(\bz)$, we can write
\begin{align}
   \grad u_3(\bz)=&\frac1{\nu^2 }\sum_{i,j: i\neq j} \Bigg(
   \sigma'(\bw_i^\sT \bz) \widetilde \btheta_{j,i}^\sT \bz \left(\rho_{ij}\bw_i^\sT\bz 
\sigma''(\bw_j^\sT \bz) - \frac12\rho_{ij}^2 (\bw_i^\sT \bz)^2 \sigma'''(\bw_j^\sT \bz) 
\right)\btheta_\up{i}\nonumber\\
   &\hspace{0mm}+
   \btheta_\up{i}^\sT\bz\sigma''(\bw_i^\sT \bz) \widetilde \btheta_{j,i}^\sT \bz 
\left(\rho_{ij}\bw_i^\sT\bz \sigma''(\bw_j^\sT \bz) - \frac12\rho_{ij}^2 (\bw_i^\sT \bz)^2 
\sigma'''(\bw_j^\sT \bz) \right)\bw_i\nonumber\\
   &\hspace{0mm}+
   \btheta_\up{i}^\sT\bz\sigma'(\bw_i^\sT \bz) \left(\rho_{ij}\bw_i^\sT\bz \sigma''(\bw_j^\sT 
\bz) - \frac12\rho_{ij}^2 (\bw_i^\sT \bz)^2 \sigma'''(\bw_j^\sT \bz) \right)\widetilde 
\btheta_{j,i}\nonumber\\
   &\hspace{0mm}+\btheta_\up{i}^\sT\bz\sigma'(\bw_i^\sT \bz) \widetilde \btheta_{j,i}^\sT \bz 
\left(\rho_{ij} \sigma''(\bw_j^\sT \bz) - \rho_{ij}^2 \bw_i^\sT \bz \sigma'''(\bw_j^\sT 
\bz) \right)\bw_i\nonumber\\
   &\hspace{0mm}+
   \btheta_\up{i}^\sT\bz\sigma'(\bw_i^\sT \bz) \widetilde \btheta_{j,i}^\sT \bz 
\left(\rho_{ij}\bw_i^\sT\bz \sigma'''(\bw_j^\sT \bz) - \frac12\rho_{ij}^2 (\bw_i^\sT \bz)^2 
\sigma^\up{4}(\bw_j^\sT \bz) \right)\bw_j
   \Bigg)\nonumber.
   \end{align}
This can be rewritten as
   \begin{align}
  \grad u_3(\bz) =&\frac1{\nu^2}
   \Bigg(
    \bT_\btheta \bD_1 \bM \left( (\bN \odot \bR) \bsigma''(\bW^\sT\bz) - \frac12\bM (\bN \odot \bR 
\odot \bR)\sigma'''(\bW^\sT \bz)  \right)
   \label{eq:big_eq_1}\\ 
   &\hspace{0mm}+
   \bW\widetilde\bM \bD_2 \bM \left( (\bN \odot \bR) \sigma''(\bW^\sT\bz)  - \frac12\bM(\bN \odot \bR 
\odot \bR )\sigma'''(\bW^\sT \bz) \right)
   \label{eq:big_eq_2}\\ 
   &\hspace{0mm}+
    \bT_\btheta \left(  \bD_2 \bR  - \frac12\bD_3 (\bR \odot \bR)\bM \right)\bD_1 \bM \bT_\btheta^\sT 
\bz \label{eq:big_eq_3}\\
    &\hspace{0mm}+  \bW\widetilde \bM \bD_1 \bM \bF \left( (\bA \odot 
\bR)\sigma''(\bW^\sT \bz) - \frac12\bM (\bA \odot \bR \odot \bR) \sigma'''(\bW^\sT \bz)\right) \label{eq:big_eq_4}\\
    &\hspace{0mm}+ 
    \bW \bD_1 \widetilde \bM \left( 
    (\bN \odot \bR) \sigma''(\bW^\sT \bz) -  \bM (\bN \odot \bR \odot \bR) 
\sigma'''(\bW^\sT \bz)
    \right) \label{eq:big_eq_5}\\
    &\hspace{0mm}+ \bW \left( \bD_3 (\bN \odot \bR) - \frac12\bD_4 (\bN \odot\bR\odot \bR) \bM 
\right) \bM \bD_1 \bT_\btheta^\sT \bz
    \Bigg). \label{eq:big_eq_6}
\end{align}
Let us again bound the expected norm squared of each of the terms in the previous display.

For the terms on lines~\eqref{eq:big_eq_1} and~\eqref{eq:big_eq_2}
we have
\begin{align*}
&\E\left[ \norm{ \left(\bT_\btheta\bD_1 \bM  \right)\left( (\bN 
\odot \bR) \bsigma''(\bW^\sT \bz) - \frac12\bM (\bN \odot \bR \odot \bR) \bsigma'''(\bW^\sT \bz) 
\right)}_2^2\right]\\
&\hspace{10mm}\le
\norm{\bT_\btheta}_\op^2
\Bigg(
\E\left[
\norm{\bD_1 \bM \left(\bN \odot \bR \right) \bsigma''\left(\bW^\sT \bz\right)}_\op^2
\right]\\
&\hspace{20mm}+
\frac12\E\left[
\norm{\bD_1 \bM^2 \left(\bN \odot \bR \odot\bR \right) \bsigma'''\left(\bW^\sT 
\bz\right)}_2^2
\right]
\Bigg)\\
&\hspace{10mm}\le C \norm{\bT_\btheta}_2^2 \bigg( 
\E \left[  \norm{\bM}^6_\op\right]^{1/3} \E\left[ \norm{\bN \odot \bR}_\op^6\right]^{1/3}
\E\left[\norm{\bsigma''(\bW^\sT \bz)}_2^6 \right]^{1/3}\\
&\hspace{20mm}+ 
\E \left[  \norm{\bM}^{12}_\op\right]^{1/3} \E\left[ \norm{\bN \odot \bR 
\odot \bR }_\op^6\right]^{1/3}
\E\left[\norm{\bsigma'''(\bW^\sT \bz)}_2^6 \right]^{1/3}
\bigg)\\
&\hspace{10mm}\stackrel{(a)}{\le} C_{11} \frac{(\log m)^3}{d},
\end{align*}
where in $(a)$ we used that $\norm{\sigma^\up{l}}_\infty < \infty$.
A similar calculation shows that
\begin{align*}
&\E\left[ \norm{ \left(\bW\widetilde \bM \bD_2\bM  \right)\left( (\bN 
\odot \bR) \bsigma''(\bW^\sT \bz) - \frac12\bM (\bN \odot \bR \odot \bR) \bsigma'''(\bW^\sT \bz) 
\right)}_2^2\right]\\
&\hspace{0mm}\le
C_{12}\norm{\bW}_\op^2
\bigg(
\E\left[
\norm{\widetilde \bM \bD_2\bM \left(\bN \odot \bR \right) \bsigma''\left(\bW^\sT 
\bz\right)}_2^2
\right]\\
&\hspace{5mm}+
\E\left[
\norm{\widetilde\bM \bD_2 \bM^2 \left(\bN \odot \bR \odot\bR \right) 
\bsigma'''\left(\bW^\sT 
\bz\right)}_2^2
\right]
\bigg)\\
&\hspace{0mm}\le C_{12} \norm{\bW}_\op^2 \Big( 
\E \big[ \|\widetilde\bM\|^8_\op\big]^{1/4}
\E \big[  \norm{\bM}^8_\op\big]^{1/4}
\E\big[ \norm{\bN \odot 
\bR}_\op^8\big]^{1/4}
\E\big[\|\bsigma''(\bW^\sT \bz)\|_2^8 \big]^{1/4}\\
&\hspace{5mm}+ 
\E \left[  \|{\bM}\|^{16}_\op\right]^{1/4}
\E \left[  \|\widetilde\bM\|^8_\op\right]^{1/4} 
\E\left[ \norm{\bN \odot \bR 
\odot \bR }_\op^8\right]^{1/4}
\E\left[\|\bsigma'''(\bW^\sT \bz)\|_2^8 \right]^{1/4}
\Big)\\
&\hspace{0mm}\le C_{13} \frac{(\log m)^4}{m}.
\end{align*}\\

For the term on line~\eqref{eq:big_eq_3},
\begin{align*}
 &\E\left[\norm{\bT_\btheta\left( 
\bD_2 \bR -\frac12 \bD_3 \left(\bR \odot \bR\right)\bM \right) \bD_1 \bM \bT_\btheta^\sT \bz
 }_2^2 \right]
 \\
&\hspace{20mm}\le C \norm{\bT_\btheta}_\op^4
\Big( \norm{\bR}_\op^2\E\left[ 
\norm{\bD_2}_\op^2 \norm{\bD_1}_\op^2 \norm{\bM}_\op^2\norm{\bz}_2^2
\right]\\
&\hspace{30mm}+
\norm{\bR\odot \bR}_\op^2\E\left[
\norm{\bD_3}_\op^2 \norm{\bM}_\op^2 \norm{\bD_1}_\op^2 \norm{\bM}_\op^2 
\norm{\bz}_2^2
\right]
\Big)\\
&\hspace{20mm}\stackrel{(a)}{\le}
C_{14} \frac{(\log m)^2}{d}.
\end{align*}

For the term on line~\eqref{eq:big_eq_4}, an analogous calculation shows that
\begin{align*}
 &\E\Big[
 \Big\|\bW\widetilde \bM \bD_1 \bM \bF\hspace{-0.6mm} \Big(\hspace{-1mm}(\bA \odot 
\bR)\sigma''(\bW^\sT \bz) - \frac12\bM (\bA \odot \bR \odot \bR) \sigma'''(\bW^\sT 
\bz)\Big) \Big\|_2^2
 \Big]\\
 &\le C_{15}\frac{(\log m)^3}{m},
\end{align*}
and then similarly for~\eqref{eq:big_eq_5}, and~\eqref{eq:big_eq_6} we have

\begin{align*}
\E\left[ 
\norm{
    \bW \bD_1 \widetilde \bM \left( 
    (\bN \odot \bR) \sigma''(\bW^\sT \bz) -  \bM (\bN \odot \bR \odot \bR) 
\sigma'''(\bW^\sT \bz)
    \right) 
 }_2^2\right]
\le C_{16}\frac{(\log m)^3}{m},
\end{align*}
and
\begin{align*}
\E\left[ 
\norm{\bW \left( \bD_3 (\bN \odot \bR) - \frac12\bD_4 (\bN \odot\bR\odot \bR) \bM 
\right) \bM \bD_1 \bT_\btheta^\sT \bz}_2^2 
\right] 
\le C_{17}\frac{(\log m)^3}{m},
\end{align*}
respectively.

These bounds then give
\begin{equation}
\label{eq:u_3_var_asymp}
\lim_{n\to\infty} \sup_{\btheta\in\cS_{p,\delta}} \E\left[\norm{\grad u_3(\bz)}_2^2 \right]= 0.
\end{equation}

What remains is the term $\Var(u_4(\bz))^{1/2}$. However, this can be bounded naively as
\begin{align*}
\Var(u_4(\bz)) &\le \frac1{36\nu^4}\E\left[u_4(\bz)^2 \right] \\
&\le \frac{m^4}{36\nu^4} \E\left[\sup_{i\neq j} \left| \btheta_\up{i}^\sT \bz \sigma'(\bw_i^\sT \bz) 
\widetilde \btheta_{j,i}^\sT \bz \rho_{i,j}^3 (\bw_i^\sT \bz)^3 
\sigma^\up{4}(v_{ij}(\bz))\right|^2 \right] \\
&\le C_{18} \frac{m^4}{\nu^4} \left(
\E\left[\sup_{i\in[m]}\left| \btheta_\up{i}^\sT \bz\right|^2
\sup_{i\neq j}\left| \widetilde\btheta_{j,i}^\sT \bz\right|^2
\sup_{i\in[m]}\left| \bw_i^\sT \bz\right|^6
\right]
\sup_{i\neq j} \left|\rho_{i,j}\right|^6 \right)\\
&\stackrel{(a)}{\le}
C_{19} \frac{m^4}{\nu^4}
\left(
\frac{\log m}{m}\frac{\log m}{m} (\log m)^3\left(\frac{\log m}{d}\right)^3
\right)\\
&\le
C_{20} \frac{m^2}{\nu^4}
\frac{(\log m)^8}{d^3},
\end{align*}
where $(a)$ follows from an application of H\"older's and Lemma~\ref{lemma:op_norm_bounds}.
Hence we have
\begin{equation}
 \lim_{n\to\infty}\sup_{\btheta\in\cS_{p,\delta}}  \Var(u_4(\bz)) = 0.
\end{equation}
Combining this with 
 \eqref{eq:u_1_var_asymp}, 
 \eqref{eq:u_2_var_asymp}, and
 \eqref{eq:u_3_var_asymp} gives
\begin{equation}
 \lim_{n\to\infty} \sup_{\btheta\in\cS_{p,\delta}} \Var(U)= 0.
\end{equation}

Therefore, we can control~\eqref{eq:first_term_ntk} as
\begin{align*}
   \lim_{n\to\infty} \sup_{\btheta\in\cS_{p,\sfk}} \left|\E\left[(U-1)\chi'\left(\frac{\btheta^\sT \bx}{\nu}\right)\right]\right|
    &\le \norm{\chi'}_\infty
    \lim_{n\to\infty} \sup_{\btheta\in\cS_{p,\sfk}} \left( \Var(U)^{1/2} + |\E[U - 1]|  \right)\\
    &= 0
\end{align*}
by the previous display and the computation showing $\E[U] = 1$.

\end{proof}

\subsubsection{Bounding the second term in Eq.~\eqref{eq:decomposition}}
\label{section:second_term_ntk}
\begin{lemma}
\label{lemma:second_term_ntk}
For any $\delta>0$, we have
\begin{equation}
\nonumber
    \lim_{n\to\infty}\sup_{\btheta\in\cS_{p,\delta}} 
    \bigg|
    \E\bigg[
    \frac{1}{\nu}\sum_{i=1}^m
    \btheta_\up{i}^\sT\bz\sigma'(\bw_i^\sT\bz)
    \bigg(\chi\Big(\frac{\btheta^\sT\bx}{\nu}\Big)
    -
    \chi\Big(\frac{\btheta^\sT\bx}{\nu} - \Delta_i\Big)
    - \Delta_i\chi'\Big(\frac{\btheta^\sT\bx}{\nu}\Big)
    \bigg)
    \bigg]
    \bigg|= 0.
\end{equation}
\end{lemma}

\begin{proof}
Let us define the event 
\begin{align*}
\cA&:= \Big\{
\sup_{i\in[m]} \Big| {\|{\btheta_\up{i}}\|^{-1}} \btheta_\up{i}^\sT\bz\Big| \le \left(\log m\right)^{50} \Big\}
\bigcap\Big\{
\sup_{i\in[m]} \Big| \bw_i^\sT\bz\Big| \le \left(\log m\right)^{50} \Big\}\\
&\hspace{40mm}\bigcap
\Big\{\sup_{\{(i,j)\in[m]^2 : i\neq j\}}
 \big|{\|\widetilde\btheta_{j,i}}_2 \|^{-1}
 \widetilde\btheta_{j,i}^\sT\bz\big|
 \le \left(\log m\right)^{50} 
\Big\}
\end{align*}

Using that for $v_i$, not necessarily independent, subgaussian with subgaussian norm $1$ 
\begin{equation*}
\P\left(\sup_{i\in[m]}|v_i| > \sqrt{2\log m} + t  \right)  \le\exp\left\{-\frac{t^2}{2\sfK_v^2}\right\},
\end{equation*}
we obtain
\begin{align*}
\P\left(
\cA^c
\right) \le  
3\exp\left\{-\frac{c_0(\log m)^{99}}{2}\right\}
\end{align*}
for some universal constant $c_0\in(0,\infty)$.
Hence, it is sufficient to establish the desired bound on the set $\cA$. Indeed, suppose 

\begin{equation}
    \label{eq:second_term_ntk_target}
    \lim_{n\to\infty}\sup_{\btheta\in\cS_{p,\delta}}
    \bigg|
    \E\bigg[
    \frac{1}{\nu}\sum_{i=1}^m
    \btheta_\up{i}^\sT\bz\sigma'(\bw_i^\sT\bz)
    \bigg(\chi\Big(\frac{\btheta^\sT\bx}{\nu}\Big)
    -
    \chi\Big(\frac{\btheta^\sT\bx}{\nu} - \Delta_i\Big)
    - \Delta_i\chi'\Big(\frac{\btheta^\sT\bx}{\nu}\Big)
    \bigg)
    \one_{\cA}
    \bigg]
    \bigg|= 0,
\end{equation}
then 
\begin{align*}
    &\lim_{n\to\infty}\sup_{\btheta\in\cS_{p,\delta}} 
    \bigg|
    \E\bigg[
    \frac{1}{\nu}\sum_{i=1}^m
    \btheta_\up{i}^\sT\bz\sigma'(\bw_i^\sT\bz)
    \Big(\chi\Big(\frac{\btheta^\sT\bx}{\nu}\Big)
    -
    \chi\Big(\frac{\btheta^\sT\bx}{\nu} - \Delta_i\Big)
    - \Delta_i\chi'\Big(\frac{\btheta^\sT\bx}{\nu}\Big)
    \Big)
    \bigg]
    \bigg|  \\
    &\hspace{10mm}\stackrel{(a)}{\le }
    \lim_{n\to\infty}\sup_{\btheta\in\cS_{p,\delta}} 
    \frac{C_1 m}{\nu}
    \left(\norm{\chi}_\infty
    \vee\norm{\chi'}_\infty\right)
    \sup_{i\in[m]}\norm{\btheta_\up{i}}_2
    \E\left[
    \norm{\bz}_2
    \Big( 2+ \sup_{i\in[m]}|\Delta_i|\Big)
    \one_{\cA^c}
    \right]\\
    &\hspace{10mm}\le
    \lim_{n\to\infty}\sup_{\btheta\in\cS_{p,\delta}} 
    \frac{C_1 m}{\nu^2}
    \left(\norm{\chi}_\infty
    \vee\norm{\chi'}_\infty\right)
    \sup_{i\in[m]}\norm{\btheta_\up{i}}_2\dots\\
    &\hspace{40mm}\dots \E\left[
    \norm{\bz}_2
    \Big( 2\nu + \norm{\btheta}_2\norm{\bx}_2 + C_2 m \sup_{j\in[m]}\norm{\btheta_\up{j}}_2\norm{\bz}_2\Big)
    \one_{\cA^c}
    \right]\\
    &\hspace{10mm}\stackrel{(b)}{\le }
    \lim_{n\to\infty}\sup_{\btheta\in\cS_{p,\delta}} 
    C_4(\nu) m^2 \exp\left\{-\frac{c_0(\log m)^{99}}{2}\right\}\\
    &\hspace{10mm}{\le }
    \lim_{n\to\infty} C_4(\delta) m^2 \exp\left\{-\frac{c_0(\log m)^{99}}{2}\right\}\\
    &\hspace{10mm}=0.
\end{align*}
where $(a)$ follows by a naive bound on $\Delta_i$, while $(b)$ follows by an application of H\"older's inequality.
Hence, throughout we work on the event $\cA$.\\

By Lemma 2.4 of~\cite{chen2011normal}, $\chi' = \chi'_\varphi$ is differentiable and
 $\norm{\chi''}_\infty \le  C_0$ since $\varphi$ is assumed to be differentiable with bounded derivative.
Hence,
\begin{align}
\label{eq:2nd_term_taylor}
\left|\chi\left(\frac{\btheta^\sT\bx}{\nu}\right) -\chi\left( \frac{\btheta^\sT\bx}{\nu} 
- \Delta_i\right)  - \Delta_i \chi' \left( \frac{\btheta^\sT\bx}{\nu}\right)\right| 
&\le C_0 \left|\Delta_i \right|^2.
\end{align}
Using this in~\eqref{eq:second_term_ntk_target} we obtain
\begin{align}
    &\Big|
    \E\Big[
    \frac{1}{\nu}\sum_{i=1}^m
    \btheta_\up{i}^\sT\bz\sigma'(\bw_i^\sT\bz)
    \Big(\chi\Big(\frac{\btheta^\sT\bx}{\nu}\Big) 
    -
    \chi\Big(\frac{\btheta^\sT\bx}{\nu} - \Delta_i\Big) 
    - \Delta_i\chi'\Big(\frac{\btheta^\sT\bx}{\nu}\Big)
    \Big)
    \one_{\cA}
    \Big]
    \Big|    
\nonumber
    \\
    &\hspace{50mm}\stackrel{(a)}{\le}
    C_0
    \E\Big[
    \frac{1}{\nu}\sum_{i=1}^m
    \Big|\btheta_\up{i}^\sT\bz\sigma'(\bw_i^\sT\bz)\Big|
    \Delta_i^2 \one_\cA
    \Big]
    \nonumber
    \\
    &\hspace{50mm}\stackrel{(b)}{\le}
    \frac{C_1}{\nu}
    \E\Big[\sup_{i\in[m]}\Big|\frac{1}{\norm{\btheta_\up{i}}_2} \btheta_\up{i}^\sT\bz\Big|
    \sum_{i=1}^m \norm{\btheta_\up{i}}_2
     \Delta_i^2\one_{\cA}
    \Big]\nonumber\\
    &\hspace{50mm}\stackrel{(c)}{\le} \frac{C_2}{\nu} \frac{(\log m)^{50}}{m^{1/2}} 
    \sum_{i=1}^m \E\left[
     \Delta_i^2\one_{\cA}
    \right],
    \label{eq:delta_i_square_bound}
\end{align}
where $(a)$ follows from~\eqref{eq:2nd_term_taylor}, $(b)$ follows from 
boundedness of $\norm{\sigma'}_\infty$, and $(c)$ follows from $\norm{\btheta_\up{i}}_2 \le \sfR/\sqrt{d}$ and the definition of $\cA$. 
Now recall the form of $\Delta_i$ introduced in Eq.~\eqref{eq:Delta_i_new_form} and let us again Taylor expand $\sigma'$ to write
\begin{align}
    \Delta_i &= \frac{1}{\nu} \btheta_\up{i}^\sT \bz \sigma'(\bw_i^\sT \bz) 
    + \frac1{\nu}\sum_{j: j\neq i} \btheta_\up{j}^\sT  \bw_i  
\bw_i^\sT \bz \sigma'(\bw_j^\sT \bz)\nonumber\\
    &\hspace{10mm}+ \frac1{\nu} \sum_{j: j\neq i}\widetilde\btheta_{j,i}^\sT\bz 
     \rho_{ij} \bw_i^\sT \bz \sigma''(\bw_j^\sT \bz)-
     \frac1{\nu}\sum_{j:j\neq i}
     \widetilde \btheta_{j,i}^\sT \bz
     \rho_{ij}^2(\bw_i^\sT \bz)^2 \sigma'''(v_{j,i}(\bz))\nonumber\\
     &=: d_{1,i} + d_{2,i} + d_{3,i} + d_{4,i}
     \label{eq:delta_i_to_d_k_i}
\end{align}
for some $v_{j,i}(\bz)$ between $\bw_j^\sT \bz$ and $\bw_j^\sT\bz - \rho_{ij}\bw_i^\sT\bz$. We show that for each $k\in [4]$,
\begin{equation}
    \label{eq:d_ki_bounds}
    \lim_{n\to\infty} \sup_{\btheta\in\cS_{p,\delta}} \frac{1}{\nu}
    \frac{(\log m)^{50}}{m^{1/2}} \sum_{i=1}^m\E\left[d_{k,i}^2\one_\cA\right]=0.
\end{equation} 
For the contributions of $d_{1,i}$, we have
\begin{align*}
\frac{1}{\nu}\frac{(\log m)^{50}}{m^{1/2}}\sum_{i=1}^m \E\left[ d_{1,i}^2\one_\cA\right]
&\le C_3\frac1{\nu^3} \frac{(\log m)^{50}}{m^{1/2}}
\sum_{i=1}^m\E\left[
(\btheta_\up{i}^\sT \bz)^2 \sigma'(\bw_i^\sT \bz)^2\one_\cA
\right]\\
&\le C_4\frac1{\nu^3} \frac{(\log m)^{50}}{m^{1/2}} \E\left[\norm{\bT_\btheta ^\sT \bz}_2^2 \right]^{1/2}\\
&\le C_4\frac1{\nu^3}\frac{(\log m)^{50}}{m^{1/2}}\norm{\bT_\btheta}_\op^2 \E\left[\norm{\bz}_2^2\right]\\
&\le C_5(\delta)\left( \frac{(\log m)^{50} }{m^{1/2}}\right)
\end{align*}
uniformly over $\btheta\in\cS_{p,\delta}$. Taking supremum over $\btheta\in\cS_{p,\delta}$ and sending $n\to\infty$ proves~\eqref{eq:d_ki_bounds} 
for $k=1$. 

For $d_{2,i}$,
\begin{align*}
    \frac{1}{\nu}\frac{(\log m)^{50}}{m^{1/2}}\sum_{i=1}^m \E\left[d_{2,i}^2\one_{\cA} \right]
    &\le \frac{C_{6}}{\nu^3} \frac{(\log m)^{50}}{m^{1/2}}\E\Big[ \sum_{i=1}^m( \bw_i^\sT \bz)^2  
\Big(\sum_{j\neq i} \btheta_\up{j}^\sT  \bw_i \sigma'(\bw_j^\sT \bz)  
\Big)^2 \one_{\cA}\Big]
\\
    &\stackrel{(a)}\le \frac{C_{7}}{\nu^3} \frac{ (\log m)^{150}}{m^{1/2}}\E\left[\norm{ 
    \bA^\sT \bsigma'(\bW^\sT \bz)
    }_2^2 
    \right]
    \\
    &\le C_{8}(\delta)\frac{ (\log m)^{150}}{m^{1/2}}
\end{align*}
uniformly over $\cS_{p,\delta}$,
where $(a)$ holds by the definition of $\cA$.
Sending $n\to\infty$ shows~\eqref{eq:d_ki_bounds} for $k=2$. 

Similarly, we have
\begin{align*}
   \frac{1}{\nu}\frac{(\log m)^{50}}{m^{1/2}} \sum_{i=1}^m \E\Big[d_{3,i}^2\one_\cA\Big]
   &\le \frac{C_{9}}{\nu^3}\frac{(\log m)^{50}}{m^{1/2}} \E\Big[\sum_{i=1}^m (\bw_i^\sT \bz)^2 
\Big(\sum_{j\neq i} \widetilde\btheta_{j,i}^\sT \bz \rho_{ij} \sigma''(\bw_j^\sT \bz) 
\Big)^2 \one_\cA\Big]\\
   &\le \frac{C_{9}}{\nu^3}\frac{(\log m)^{150}}{m^{1/2}}
   \E\Big[\sum_{i=1}^m 
\Big(\sum_{j\neq i} \widetilde\btheta_{j,i}^\sT \bz \rho_{ij} \sigma''(\bw_j^\sT \bz) 
\Big)^2 \one_\cA\Big]\\
   &= \frac{C_{9}}{\nu^3}\frac{ (\log m)^{150}}{m^{1/2}} \E\Big[
    \norm{(\bN \odot \bR) \bsigma''(\bW^\sT \bz)}^2
   \Big]\\
   &\le C_{10}(\delta)\frac{(\log m)^{151}}{m^{1/2}}
\end{align*}
uniformly over $\btheta\in\cS_{p,\delta}$, establishing~\eqref{eq:d_ki_bounds} for $k=3$.

Finally, $d_{4,i}$ can be bounded almost surely on $\cA$:
\begin{align*} 
|d_{4,i}|\one_{\cA} &\le \frac{C_{11} m}{\nu}  
\sup_{ i\neq j}
 \left|
  \frac{1}
 {\norm{\widetilde\btheta_{j,i}}_2 }\widetilde\btheta_{j,i}^\sT\bz
 \right| \sup_{i\neq j}\norm{\widetilde \btheta_{j,i}}_2
\sup_{i\neq j}\rho_{ij}^2 \sup_{i\in [m]}(\bw_i^\sT \bz)^2 \one_\cA\\
&\stackrel{(a)}{\le} \frac{C_{12} (\log m)^{50} }{\nu} \sup_{i\in[m]}\norm{\bP_i^\perp}_\op \sup_{j\in[m]}\norm{\btheta_\up{j}}_2\\
&\stackrel{(b)}{\le} C_{13}(\delta) \frac{(\log m)^{50}}{d^{1/2}}
\end{align*}
uniformly over $\cS_{p,\delta}$,
where $(a)$ follows from the definition of the event $\cA$ and $(b)$ follows because $\bP_i^\perp$ is a projection matrix for all $i$
and that $\norm{\btheta_\up{j}}_2 \le \sfR/\sqrt{d}$.
Therefore, we have
\begin{equation}
\nonumber
    \lim_{n\to\infty}\sup_{\btheta\in\cS_{p,\delta}}
    \frac1\nu
    \frac{(\log m)^{50}}{m^{1/2}}\sum_{i=1}^m \E[d_{4,i}^2\one_{\cA}] = 0,
\end{equation}
establishing~\eqref{eq:d_ki_bounds} for $k=4$. 

Hence we showed
\begin{align*}
    &\lim_{n\to\infty}\sup_{\btheta\in\cS_{p,\delta}}\Big|
    \E\Big[
    \frac{1}{\nu}\sum_{i=1}^m
    \btheta_\up{i}^\sT\bz\sigma'(\bw_i^\sT\bz)
    \Big(\chi\Big(\frac{\btheta^\sT\bx}{\nu}\Big) 
    -
    \chi\Big(\frac{\btheta^\sT\bx}{\nu} - \Delta_i\Big) 
    - \Delta_i\chi'\Big(\frac{\btheta^\sT\bx}{\nu}\Big)
    \Big)
    \one_{\cA}
    \Big]
    \Big|\\
    &\hspace{10mm}\stackrel{(a)}{\le}
    \lim_{n\to\infty}\sup_{\btheta\in\cS_{p,\delta}}
    \frac{C_2}{\nu} \frac{(\log m)^{50}}{m^{1/2}} 
    \sum_{i=1}^m \E\left[
     \Delta_i^2\one_{\cA}
    \right]\\
    &\hspace{10mm}\stackrel{(b)}\le
    \lim_{n\to\infty}\sup_{\btheta\in\cS_{p,\delta}}\frac{C_{14}}{\nu} \frac{(\log{m})^{50}}{m^{1/2}} 
    \sum_{k=1}^4\sum_{i=1}^m \E\left[
     d_{k,i}^2\one_{\cA}
    \right]\\
    &\hspace{10mm}\stackrel{(c)}=0,
\end{align*}
where $(a)$ follows from~\eqref{eq:delta_i_square_bound}, $(b)$ follows from~\eqref{eq:delta_i_to_d_k_i} and $(c)$ follows 
from~\eqref{eq:d_ki_bounds} holding for $k\in[4]$. Hence, we have shown~\eqref{eq:second_term_ntk_target} and completed
the proof.
\end{proof}

\subsubsection{Proof of Lemma~\ref{lemma:ntk_bl-1}}
\label{section:ntk_actual_proof}

\begin{proof}
Recall the definition of $\Delta_i$ in~\eqref{eq:Delta_i_def}
and note that for all $i\in[m]$, 
\begin{equation}\nonumber
\frac1\nu\btheta^\sT\bx - \Delta_i = \frac1\nu \sum_{j:j\neq i} \btheta_\up{j}^\sT \bP_i^\perp 
\bz \sigma'(\bw_j^\sT \bz - \rho_{i,j} \bw_i^\sT \bz).
\end{equation}
Since $\bz$ is Gaussian, $\bw_i^\sT \bz$ is independent of any function of $\bw_j^\sT\bz - \rho_{i,j}\bw_i^\sT\bz$ and hence is independent of
$\btheta^\sT \bx/\nu - \Delta_i$. Therefore, we have
\begin{align}
 \E\left[ \btheta_\up{i}^\sT \bz \sigma'(\bw_i^\sT \bz)
 \chi\left( 
 \frac{\btheta^\sT \bx}{\nu} - \Delta_i
\right)\right]
&= \E\left[ 
\btheta_\up{i}^\sT\bw_i \bw_i^\sT \bz \sigma'(\bw_i^\sT \bz)
\chi\left( 
 \frac{\btheta^\sT \bx}{\nu} - \Delta_i
\right)
\right]\\
&\hspace{10mm}+ 
\E\left[  \btheta_\up{i}^\sT \bP^\perp_i \bz \sigma'(\bw_i^\sT \bz)
\chi\left( 
 \frac{\btheta^\sT \bx}{\nu} - \Delta_i
\right)
\right]
\nonumber
\\
&=  
\btheta_\up{i}^\sT\bw_i 
\E\left[
\bw_i^\sT \bz \sigma'(\bw_i^\sT \bz)
\right]
\E\left[
\chi\left( 
 \frac{\btheta^\sT \bx}{\nu} - \Delta_i
\right)
\right]\nonumber\\
&\hspace{10mm}+ 
\E\left[  \btheta_\up{i}^\sT \bP^\perp_i \bz 
\chi\left( 
 \frac{\btheta^\sT \bx}{\nu} - \Delta_i
\right)
\right]
\E\left[\sigma'(\bw_i^\sT \bz) \right]
\nonumber
\\
&\stackrel{(a)}{=} 0 \label{eq:cross_term_0}
\end{align}
where $(a)$ follows from the assumption that $\E[\sigma'(G)] = \E[G\sigma'(G)] = 0$ for a standard normal $G$.
Hence, we can write
\begin{align}
  &\left|\E\left[ \varphi\left(\frac{\btheta^\sT \bx}{\nu}\right) - \varphi\left(\frac{\btheta^\sT \bg}{\nu}\right)\right]\right|\nonumber\\
  &\hspace{15mm}\stackrel{(a)}{=}
    \left|\E\left[\frac{\btheta^\sT\bx}{\nu}\chi\left(\frac{\btheta^\sT\bx}{\nu}\right)   
    -\chi'\left(\frac{\btheta^\sT\bx}{\nu}\right)
    \right]\right|  \nonumber\\
&\hspace{15mm}=
    \Bigg|
    \E\left[\left(\frac{1}{\nu}\sum_{i=1}^m \btheta_\up{i}^\sT\bz 
\sigma'(\bw_i^\sT\bz)\Delta_i 
- 1 
    \right)\chi'\left(\frac{\btheta^\sT\bx}{\nu}\right)\right]\nonumber\\
    &\hspace{15mm}+
    \E\left[
    \frac{1}{\nu}\sum_{i=1}^m
    \btheta_\up{i}^\sT\bz\sigma'(\bw_i^\sT\bz)
    \left(\chi\left(\frac{\btheta^\sT\bx}{\nu}\right) 
    -
    \chi\left(\frac{\btheta^\sT\bx}{\nu} - \Delta_i\right) 
    - \Delta_i\chi'\left(\frac{\btheta^\sT\bx}{\nu}\right)
    \right)
    \right]\nonumber\\
   &\hspace{15mm}+ 
   \E\left[\frac{1}{\nu}\sum_{i=1}^m \btheta_\up{i}^\sT \bz \sigma'(\bw_i^\sT \bz)
 \chi\left( 
 \frac{\btheta^\sT \bx}{\nu} - \Delta_i
\right)\right]
    \Bigg|\nonumber\\
    &\hspace{15mm}
     \stackrel{(b)}\le\left|
    \E\left[\left(\frac{1}{\nu}\sum_{i=1}^m \btheta_\up{i}^\sT\bz 
\sigma'(\bw_i^\sT\bz)\Delta_i 
- 1 
    \right)\chi'\left(\frac{\btheta^\sT\bx}{\nu}\right)\right]
    \right|\label{eq:term1_ntk}\\
    &+
    \left|
    \E\left[
    \frac{1}{\nu}\sum_{i=1}^m
    \btheta_\up{i}^\sT\bz\sigma'(\bw_i^\sT\bz)
    \left(\chi\left(\frac{\btheta^\sT\bx}{\nu}\right) 
    -
    \chi\left(\frac{\btheta^\sT\bx}{\nu} - \Delta_i\right) 
    - \Delta_i\chi'\left(\frac{\btheta^\sT\bx}{\nu}\right)
    \right)
    \right]
    \right|\label{eq:term2_ntk}
\end{align}
where $(a)$ follows by Eq.~\eqref{eq:stein_solution} and $(b)$ follows by Eq.~\eqref{eq:cross_term_0}. 
Taking the supremum over $\btheta\in\cS_{p,\delta}$ then $n\to\infty$ and applying
Lemmas~\ref{lemma:first_term_ntk} and~\ref{lemma:second_term_ntk} completes the proof.
\end{proof}

\subsection{Truncation}
Let us define $\cG := \left\{ \norm{\bz}_2 \le 2 \sqrt{d}\right\}$ and the 
random variable $\bar \bx := \bx \one_\cG.$  The following Lemma establishes the subgaussianity condition of Assumption~\ref{ass:X} for $\bar \bx$.

\begin{lemma}
\label{lemma:bar_x_subgaussian}
Conditional on $\bW\in\cB$ we have
\begin{equation}
\nonumber
\sup_{\btheta\in\cS_{p}}  \norm{\bar\bx^\sT \btheta}_{\psi_2} \le C 
\end{equation}
for some constant $C>0$.
\end{lemma}
\begin{proof}
Take arbitrary $\btheta\in\cS_p$. Let 
\begin{equation}
\nonumber
     u(t):= \begin{cases}
                        1 & t \le 2\\
                        3 - t & t\in(2,3]\\
                        0 & t > 3
                \end{cases},
\end{equation}
then consider the function $f(\bz) := \bz^\sT \bT_\btheta  
\bsigma'(\bW^\sT \bz) u\left(\norm{\bz}_2/\sqrt{d}\right)$. 
Note that $f$ is continuous and differentiable almost everywhere with gradient 
\begin{equation}
\nonumber
\grad f(\bz)  = \Big(\bT_\btheta \bsigma'(\bW^\sT \bz) 
+ \bW\diag\Big\{ \bsigma''(\bw_i^\sT \bz) \Big\}  \bT_\btheta^\sT  \bz\Big)u\bigg(\frac{\norm{\bz}_2}{\sqrt{d}}\bigg)
+ u'\left(\frac{\norm{\bz}_2}{\sqrt{d}}\right) \frac{\bz}{\sqrt{d}\norm{\bz}} f(\bz)
\end{equation}
almost everywhere.
Noting that $u'(t) = u'(t) \one_{t \le 3}$ and  $u(t) \le \one_{t\le 3}$ we can bound
\begin{align} 
 \norm{\grad f(\bz)}_2  &\le \left(\norm{\bT_\btheta}_\op \norm{\bsigma'\left(\bW^\sT \bz\right)}_2 + 
\norm{\bW}_\op \sup_{i\in[m]}\sigma''\left(\bw_i^\sT \bz\right) \norm{\bT_\btheta}_\op \norm{\bz}_2\right)\one_{\norm{\bz}_2 \le 3\sqrt{d}}\nonumber\\
&\hspace{15mm}+ u'\left(\frac{\norm{\bz}_2}{\sqrt{d}}\right) \frac{\norm{\bz}_2}{\sqrt{d}}\norm{\bT_\btheta}\norm{\bsigma'\left(\bW^\sT \bz\right)}_2\one_{\norm{\bz}_2 \le 3\sqrt{d}}\nonumber\\
&\stackrel{(a)}{\le}  C_0 \nonumber
\end{align}
almost everywhere, where $C_0>0$ is a constant. In $(a)$ we used that $\norm{\bT_\btheta}_\op \le \sfR/\sqrt{d}$ for $\btheta\in\cS_p$. 
Hence, $\norm{f}_\Lip \le C_0$ so that $f(\bz)$ is subgaussian with constant subgaussian norm. This implies that
\begin{align}
\nonumber
 \P\left( \left| \bar\bx^\sT \btheta  \right| \ge t \right) &\stackrel{(a)}\le
 \P\left(  \left| f(\bz)\right| \ge t \right) 
 \le C_2\exp\left\{ -{c_0t^2}\right\}.
\end{align}
where $(a)$ follows by nothing that $\one_{t \le 2} \le u(t)$. This shows that $\bar\bx^\sT \btheta$ is subgaussian with subgaussian norm constant in $n$ and $\btheta$. Since $\btheta\in\cS_p$ was arbitrary, this proves the claim.

\end{proof}

Now, let us show that the condition of Eq.~\eqref{eq:condition_bounded_lipschitz_single} holds for the truncated variables $\bar\bx$.
\begin{lemma}
\label{lemma:bar_x_bl}
For any bounded Lipschitz function $\varphi :\R\to\R$, we have
\begin{equation}
\nonumber
\lim_{n\to\infty}  \sup_{\btheta\in \cS_p} \left|  
\E \left[\left(\varphi(\bar \bx^\sT \btheta) - \varphi(\bg^\sT \btheta)\right) \one_\cB 
\big| \bW \right] 
\right|=0.
\end{equation}
\end{lemma}
\begin{proof}[Proof]
Let us use the notation $\E[(\cdot) ] := \E[(\cdot)\one_{\cB}| 
\bW].$ We have
\begin{align*}
 \left|  
\E \left[\left(\varphi(\bar \bx^\sT \btheta) - \varphi(\bx^\sT \btheta)\right) 
\right] 
\right| 
&\le \norm{\varphi}_\Lip \E\left[\left|\bx^\sT \btheta \right|\one_{\cG^c}
\right]\\
&\le\norm{\varphi}_\Lip   \E\left[\left(\bx^\sT \btheta\right)^2\right]^{1/2} 
\P\left(\cG^c\right).
\end{align*}
Recalling that $\P\left(\cG^c\right)\le \exp\{-c_0 d\}$ since $\bz$ is Gaussian, we can write
\begin{align*}
 \lim_{n\to\infty} \sup_{\btheta\in\cS_p}\big|  
\E \big[\big(\varphi(\bar \bx^\sT \btheta) &- \varphi(\bg^\sT \btheta)\big) 
 \big] 
\big| \le 
\lim_{n\to\infty}\sup_{\btheta\in\cS_p} \left|  
\E \left[\left(\varphi(\bar \bx^\sT \btheta) - \varphi(\bx^\sT \btheta)\right) 
 \right] 
\right|\\
&\hspace{25mm}+ 
 \lim_{n\to\infty}\sup_{\btheta\in\cS_p}\left|  
\E \left[\left(\varphi(\bx^\sT \btheta) - \varphi(\bg^\sT \btheta)\right)  
\right] 
\right|\\
&\stackrel{(a)}{\le}  \lim_{n\to\infty}\sup_{\btheta\in\cS_p} \norm{\varphi}_{\Lip}
\E\left[\left(\bx^\sT \btheta\right)^2\right]^{1/2} e^{-c_0 d}\\
&\le
\lim_{n\to\infty}\sup_{\btheta\in\cS_p}
\norm{\varphi}_{\Lip} \E\left[\norm{\bz}^2_2\norm{\bT_\btheta}_\op^2 \norm{\bsigma'(\bW^\sT \bz)}_2^2\right]e^{-c_0 d}\\
& = 0.
\end{align*}
\end{proof}

\subsection{Proof of Theorem~\ref{cor:ntk_universality}}

\begin{proof}
Let $\cG_i := \{\norm{\bz_i}_2 \le 2 \sqrt{d}\}$ where $\bz_i$ is the Gaussian vector defining the $i$th sample $\bx_i$ of the neural tangent model. Now let 
Let $\bar\bX := (\bar \bx_1,\dots,\bar \bx_n)^\sT$  where $\bar\bx_i := \bx_i \one_{\cG_i}$.
Take any compact $\cC_p\subseteq \cS_p$ and let
$\widehat R_n^\star\left(\cdot\right)$ be the optimal empirical risk 
for a choice of $L_F,\eta,\btheta^\star,\epsilon,r$
satisfying the assumptions of the theorem.
Since $\bar\bx$ verifies Eq.~\eqref{eq:condition_bounded_lipschitz_single} and Assumption~\ref{ass:X} for $\bW\in\cB$
by Lemmas~\ref{lemma:bar_x_subgaussian} and~\ref{lemma:bar_x_bl},
then Theorem~\ref{thm:main_empirical_univ} can be applied to $\bar\bx$ to conclude that for 
for any bounded Lipschitz $\psi$
\begin{equation}
\label{eq:bar_x_universality}
\lim_{n\to\infty}\left|\E\left[
\psi\left(
\widehat R_n^\star\left(\bar \bX\right) \right)\one_\cB - 
\psi\left(
\widehat R_n^\star\left(\bG\right) \right)\one_\cB
\Big| \bW \right]  \right|= 0
\end{equation}

Now, note that we have for some $C_0,c_0 >0$,
\begin{align}
\label{eq:G_i_union_bound}
 \P\left(\bigcup_{i\in[n]}\cG^c_i\right) \le n\P\left(\norm{\bz}_2 > 2\sqrt{n}\right) \le C_0n\exp\{-c_0 d\} \to 0
\end{align}
as $n\to\infty$, so that 
\begin{align}
\label{eq:risk_bar_x_diff}
\lim_{n\to\infty}\left|\E\left[
\psi\left(\widehat R_n^\star(\bX)\right) - \psi\left(\widehat R^\star_n(\bar\bX)\right) \right]\right|
&\le 2\norm{\psi}_\infty\lim_{n\to\infty} \P\left( \bigcup_{i\in[n]}\cG_i^c \right) = 0.
\end{align}
Meanwhile, 
\begin{align}
\label{eq:risk_bar_g_diff}
\left|\E\left[
\psi\left(\widehat R^\star_n(\bar\bX)\right)- \psi\left(\widehat R^\star_n(\bG)\right)
\right]\right| 
&\le
\left|\E\left[\E\left[
\left(\psi\left(\widehat R^\star_n(\bar\bX)\right)- \psi\left(\widehat R^\star_n(\bG)\right) \right)\one_{\cB}\Big| \bW\right]\right]\right| \nonumber\\
&\hspace{10mm}+ 
2\norm{\psi}_\infty\P\left(\cB^c \right).
\end{align}
Combining the displays~\eqref{eq:bar_x_universality},~\eqref{eq:risk_bar_x_diff} and~\eqref{eq:risk_bar_g_diff} gives

\begin{align}
\lim_{n\to\infty}\big|\E\big[
\psi\big(\widehat R^\star_n(\bX)\big) - \psi\big(\widehat R^\star_n(\bG)\big) \big]\big|
&\le
\lim_{n\to\infty} \big|\E\big[\E\big[
\big(\psi\big(\widehat R^\star_n(\bar\bX)\big)- \psi\big(\widehat R^\star_n(\bG)\big) \big)\one_{\cB}\big| \bW\big]\big]\big| \nonumber\\
&\hspace{5mm}+ C_1\norm{\psi}_\infty\big(\lim_{n\to\infty}  n\exp\{-c_0 d\}
+ \lim_{n\to\infty} \P(\cB^c)\big)\nonumber\\
&\stackrel{(a)}\le
\E\big[\lim_{n\to\infty}\big| \E\big[
\big(\psi\big(\widehat R^\star_n(\bar\bX)\big)- \psi\big(\widehat R^\star_n(\bG)\big) \big)\one_{\cB}\big| \bW\big]\big|\big] \nonumber\\
&= 0\nonumber
\end{align}
where $(a)$ follows by dominated convergence.
\end{proof}

\subsection{Auxiliary lemmas}
 \label{section:aux_lemmas_ntK}
We include the following auxiliary lemmas for the sake of completeness.
\begin{lemma}
\label{lemma:subgaussian_max}
Let $V_i$ be mean zero subgaussian random variables with $\sup_{i\in[m]}\norm{V_i}_{\psi_2}\le \sfK$. 
We have for all integer $k \ge 1$,
\begin{equation}
\nonumber
\E\left[ \sup_{i\in[m]}  |V_i|^k \right] \le \left( C k \sfK^2 \log{m}\right)^{k/2}
\end{equation}
for some universal constant $C >0$. 
\end{lemma}
\begin{proof}
This follows by integrating the bound
\begin{equation}
\nonumber
\P\left(\sup_{i\in [m]} |V_i| \ge \sqrt{2 \sfK^2 \log m} + t \right)  \le C_1\exp\left\{- \frac{t^2}{2\sfK^2}\right\} 
\end{equation}
holding for $V_i$ subgaussian. 

\end{proof}
\begin{lemma}
\label{lemma:B_tail_bound}
There exist constants $C,C' \in(0,\infty)$ depending only on $\widetilde \sgamma_\NT$ such that
\begin{equation}
\nonumber
    \lim_{n\to\infty}\P\left(\left\{\sup_{\{i,j \in [m] : i\neq j\}} \left|\bw_i^\sT\bw_j\right| > 
\frac{C(\log m)^{1/2}}{d^{1/2}}\right\} \bigcup 
    \left\{
    \norm{\bW }_\op
    > C'
    \right\}\right) =0.
\end{equation}
\end{lemma}
\begin{proof}
Let $V_{i,j} = \bw_i^\sT \bw_j$ for $i,j\in[m],i\neq j$. Note that $V_{i,j}$ are subgaussian with subgaussian norm $C_1/\sqrt{d}$ for some universal constant $C_1$. Indeed, we have for $\lambda \in\R$,
\begin{equation}
\nonumber
    \E\left[  
    \exp\{\lambda V_{i,j}\}
    \right] = 
    \E\left[   \E\left[ 
    \exp\{\lambda \bw_i^\sT \bw_j\}
    | \bw_i\right]
    \right]  \le  \exp\left\{C_1\frac{\lambda^2}{d}\right\},
\end{equation}
where we used that $\bw_i$ and $\bw_j$ are independent for $i\neq j$, $\norm{\bw_i}=1$ and that $\bw_{j}$ is 
 subgaussian with subgaussian norm $C_0/\sqrt{d}$. Hence, we have
 \begin{equation}
 \nonumber
    \P\left(
    \sup_{i\neq j}  \left|V_{i,j} \right| > 4C_0\left(\frac{ \log m}{d}\right)^{1/2} 
    \right)  \le C_2 \exp\left\{ - 2\log m\right\}.
 \end{equation} 
 This proves the existence of the constant $C$ in the statement of the lemma. Meanwhile the existence of $C'$ is a consequence of Theorem 4.6.1 in~\cite{vershynin2018high}.
\end{proof}


\section{The random features model: Proof of Corollary~\ref{cor:rf_universality}}
\label{proof:rf_universality}
We recall the definitions and assumptions introduced in Section~\ref{section:rf_example}. 
Recall the activation function $\sigma$ assumed to be a three times differentiable function with bounded derivatives satisfying
$\E[\sigma(G)] =0$ for $G\sim\cN(0,1)$, the covariates 
 $\{\bz_{i}\}_{i\le [n]}~\simiid\cN(0,\bI_{d})$ and the matrix $\bW$ whose columns are the weights $\{\bw_{j}\}_{j\le[p]}\stackrel{i.i.d.}{\sim}\textsf{Unif}(\S^{d-1}(1))$.
We assume $d/p \to \widetilde\sgamma_\RF$. Now recall the definition of the feature vectors in~\eqref{eq:rf_covariates}:
$\bx:= \left(\sigma\left(\bw_1^\sT \bz\right),\dots,\sigma\left(\bw_p^\sT\bz\right)\right)$ and the set in~\eqref{eq:rf_set}:
 $\cS_p = B_\infty^p\left(\sfR/\sqrt{p}\right)$. 
 
 Define the event
\begin{equation}
\nonumber
    \cB := \left\{\sup_{i,j\in[m]: i\neq j} \left|\bw_i^\sT\bw_j \right| \le 
C\left(\frac{\log d}{d}\right)^{1/2}\right\} \bigcap 
    \left\{
    \norm{\bW }_\op
    \le C'
    \right\}
\end{equation}
for some $C,C'>0$ universal constants so that $\P(\cB^c) \to 0$ as $d\to\infty$ (see Lemma~\ref{lemma:B_tail_bound} for the existence of such $C,C'$.)

The following lemma is a direct consequence of Theorem 2 and Lemma 8
from~\cite{hu2020universality}.

\begin{lemma}
\label{lemma:rf_ass_x}
Let 
$\bSigma_\bW := \E\left[\bx\bx^\sT | \bW\right]$ and 
$\bg\big| \bW \sim\cN(0,\bSigma_\bW).$
For any bounded differentiable Lipschitz function $\varphi$ we have 
\begin{equation}
 \lim_{p\to\infty}
 \sup_{\btheta\in\cS_p}
 \left|
 \E\left[
 \left(\varphi\left(\bx^\sT \btheta \right) - \varphi\left( \bg^\sT \btheta \right)
 \right) \one_\cB \big| \bW
 \right]
 \right| = 0.\label{eq:rf_bl}
\end{equation}
Furthermore, conditional on $\bW \in\cB$, $\bx$ is subgaussian with subgaussian norm constant in $n$.
\end{lemma}

\begin{remark}
We remark that the setting of~\cite{hu2020universality} differs slightly from the one considered above. Indeed, they take
\begin{enumerate}
\item the activation function to be odd and the weight vectors to be $\{\bw_j\}_{j\le[p]} \simiid \cN(0,\bI_d/d)$, and
\item the ``asymptotically equivalent'' Gaussian vectors to be $\widetilde\bg := c_1 \bW^\sT \bz + c_2 \bh$ for $\bh\sim\cN(0,\bI_p)$ instead of $\bg$, where
$c_1$ and $c_2$ are defined so that 
 \begin{equation}
 \label{label:asymptotic_cov_rf}
\lim_{p\to\infty}\norm{\E\left[\widetilde \bg \widetilde \bg^\sT \one_\cB|\bW\right] - \E\left[\bx\bx^\sT \one_\cB |\bW\right]}_\op = 0.
 \end{equation}
\end{enumerate}
However, an examination of their proofs reveals that their results hold when $\sigma$ is assumed to satisfy $\E[\sigma(G)]=0$ for $G\sim\cN(0,1)$ instead of being odd,
and $\{\bw_j\}_{j\le [p]}\simiid {\normalfont\textsf{Unif}(\S^{d-1}(1))}$, provided $\widetilde \bg$ is replaced with $\bg$.
Indeed, the only part where the odd assumption on $\sigma$ is used in their proofs, other than to ensure that $\E\left[\sigma\big(\bw_j^\sT \bz\big)\big|\bW\right]=0$, is in showing 
that~\eqref{label:asymptotic_cov_rf} holds for their setting of $c_1$ and $c_2$ (Lemma 5 of~\cite{hu2020universality}). We circumvent this by our choice of $\bg$.
\end{remark}

\begin{remark}
Theorem 2 of~\cite{hu2020universality} prove a more general result than the one stated here for their setting. 
Additionally, they give bounds for the rate of convergence for a fixed $\btheta$
in terms of $\norm{\btheta}_2,\norm{\btheta}_\infty$ and $\norm{\varphi}_\Lip$ (and other parameters irrelevant to our setting.). However, here we are only interested in the consequence given above.
\end{remark}

\begin{proof}[Proof of Corollary~\ref{cor:rf_universality}]

First note that via a standard argument uniformly approximating Lipschitz functions wtih differentiable Lipschitz functions, 
Lemma~\ref{lemma:rf_ass_x} can be extended to hold for $\varphi$ that are bounded Lipschitz.

Now note that $\cS_p$ as defined in~\eqref{eq:rf_set} is symmetric, convex and a subset of $B_2^p(\sfR)$. Let $\cC_p$ be any compact subset of $\cS_p$ and let
$\widehat R_n^\star\left(\bX,\by(\bX)\right)$ be the minimum of the empirical risk over $\cC_p$, where the empirical risk is defined with 
 a choice of $L_F,\eta,\btheta^\star, \epsilon$ and $r$
satisfying the assumptions of the corollary.
By Lemma~\ref{lemma:rf_ass_x}, $\bx$ is subgaussian conditional on $\bW$ and hence satisfies the subgaussianity condition of Assumption~\ref{ass:X}. 
Furthermore, conditional on $\bW \in\cB$, $\bx$ satisfies the condition in~\eqref{eq:condition_bounded_lipschitz_single} for the given $\bg$, therefore, 
Theorem~\ref{thm:main_empirical_univ} implies that
for any bounded Lipschitz 
$\psi$ 
\begin{equation}
\label{eq:rf_conditional_universality}
\lim_{n\to\infty}\left|\E\left[
\psi\left(
\widehat R_n^\star\left(\bX,\by(\bX)\right) \right)\one_\cB - 
\psi\left(
\widehat R_n^\star\left(\bG,\by(\bG)\right) \right)\one_\cB
\Big| \bW \right]  \right|= 0
\end{equation}

\noindent Hence, we can write
\begin{align*}
&\lim_{n\to\infty}\left| 
\E\left[ 
\psi\left(
\widehat R_n^\star(\bX,\by(\bX))\right) -\psi\left( \widehat R_n^\star(\bG,\by(\bG)) 
\right)
\right]
\right|\\
&\hspace{20mm}\le 
\lim_{p\to\infty}\left| 
\E\left[\E\left[ 
\left(\psi\left(
\widehat R_n^\star(\bX,\by(\bX))\right) - \left(\widehat R_n^\star(\bG,\by(\bG)) 
\right)\right)\one_\cB
\Big| \bW\right]\right]
\right|\\
&\hspace{30mm}+
2\norm{\psi}_\infty  \lim_{p\to\infty}\P\left( \cB^c\right)\\
&\hspace{20mm}\stackrel{(a)}{\le} 
\E\left[
\lim_{p\to\infty}\left|
\E\left[ 
\left(\psi\left(
\widehat R_n^\star(\bX,\by(\bX))\right) - \left(\widehat R_n^\star(\bG,\by(\bG)) 
\right)\right)\one_\cB
\Big| \bW\right]\right|\right]
\\
&\hspace{20mm}\stackrel{(b)}{=} 0
\end{align*}
where $(a)$ follows by the dominated convergence theorem and $(b)$ follows from Eq.~\eqref{eq:rf_conditional_universality}.

\end{proof}


\section{Deferred proofs}
\label{section:deferred_proofs}

\subsection{Proof of non-universality in Example~\ref{ex:non_univ}}
\label{section:non-univ_proof}
Let $\cN_\alpha$ be a minimal $\alpha-$net of $B_2^p(1)$ so that $|\cN_\alpha| \le C(\alpha)^p$.
It is easy to show that for $\bg$ centered isotropic Gaussian,
\begin{equation}
\nonumber
\min_{\btheta \in\cN_\alpha} \E\left[\ell(\btheta^\sT\bg)\right] \ge 4 \Delta
\end{equation}
for some $\Delta >0$. 
Let 
\begin{equation}
\nonumber
\Opt_\alpha^n(\bG) := \min_{\btheta\in\cN_\alpha} \frac1n \sum_{i=1}^n \ell(\btheta^\sT\bg_i).
\end{equation}
Define the event $\cB := \{\norm{\bG}_\op \le C_0 \sqrt{n}\}$ and recall that $\P(\cB^c) \le 2e^{-c_0n}$ 
for some $C_0,c_0 >0$ (see for example~\cite[Theorem 4.4.5]{vershynin2018high}).
%
By an argument similar to that in the proof of Lemma~\ref{lemma:Rn_lipschitz_bound},
one can show that
\begin{equation}
\nonumber
 \left|\hR_n^\star(\bG) - \Opt_\alpha^n(\bG) \right| \le \frac{C_1}{\sqrt{n}} \norm{\bG}_\op \alpha 
 \le C_2 \alpha
\end{equation}
for some constants $C_1,C_2 > 0$, where the last inequality holds on $\cB$. 
(A similar argument was carried out in the proof of Lemma~\ref{lemma:Rn_lipschitz_bound}.)

Choose $\alpha \le \Delta/C_2$.
By union bound over $\cN_{\alpha}$,
for sufficiently large $n$, the following holds with probability at least $1-\delta$:
\begin{align}
\nonumber
\left|\Opt_\alpha^n(\bG) - \min_{\btheta\in\cN_\alpha}\E\left[\ell(\btheta^\sT\bg)\right] \right|
&\le \left( \frac{\log{(2|\cN_\alpha|)}}{2n} + \frac{\log{(1/\delta)}}{2} \right)^{1/2}\\
\nonumber
&\le  \left( \frac{C_1(\alpha)}{\sgamma} \right)^{1/2} + \left( \frac{\log{(1/\delta)}}{2} \right)^{1/2}\,.
\end{align}
 Let $\cA_\delta$ be the event that this inequality holds.
Having chosen $\alpha$, choose $\sgamma>0$ to satisfy $(C_1(\alpha)/\sgamma)^{1/2} < \Delta$
and $\delta = e^{-2\Delta^2} < 1$
so that we have
\begin{equation}
\nonumber
 \left| \hR_n^\star(\bG) - \min_{\btheta \in\cN_\alpha} \E\left[\ell(\btheta^\sT \bg)\right]\right| \le 3 \Delta
\end{equation}
on $\cA_{e^{-\Delta^2}} \cap \cB$. 
Since $\P\left(\cB\right) \to 1$ as $n\to\infty$,
we have
\begin{equation}
\label{eq:LimInfGaussian}
\liminf_{n\to\infty}\P\left(\hR_n^\star(\bG) > 2\Delta\right) \ge 1-e^{-\Delta^2} > 0.
\end{equation}
Finally notice that, for any two matrices $\bG$, $\widetilde\bG$,
\begin{align*}
 \big|\hR_n^\star(\bG)- \hR_n^\star(\widetilde\bG)\big|
 &\le \sup_{\|\btheta\|_2\le 1}\Big|\frac1n \sum_{i=1}^n \ell(\btheta^\sT\bg_{i})-
 \frac1n \sum_{i=1}^n \ell(\btheta^\sT\widetilde\bg_{i})\Big| \\
 &\le \sup_{\|\btheta\|_2\le 1}\frac1n \sum_{i=1}^n \big|\<\btheta,\bg_{i}-\widetilde\bg_{i}\>\big|\\
 &\le \frac{1}{\sqrt{n}}\|\bG-\widetilde\bG\|_F\, .
  \end{align*}

  Hence, by Gaussian concentration,  
\begin{align*}
  \P\Big(\big|\hR_n^\star(\bG)-\E\hR_n^\star(\bG)\big|\ge  \Delta\Big)\, \le 2
  \, e^{-n\Delta^2/2}\, .
\end{align*}
In conjunction with Eq.~\eqref{eq:LimInfGaussian}, this proves the claim of 
Eq.~\eqref{eq:ClaimExample}.

\subsection{Proof of Proposition~\ref{prop:overparameterized_iid}}
\label{section:outline_overparameterized_iid}
The claim of the proposition is a direct corollary of the following lemma.
\begin{lemma}\label{lemma:KeyLemmaForOverparametrized}
Assume $p/n \ge (1+\delta)$ and that the feature vectors
$\bx_i$ have either of the following distributions:
\begin{enumerate}[(1.)]
    \item 
    i.i.d. mean $0$, unit variance and subgaussian entries; or
    \item a random features distributions as defined by Section~\ref{section:rf_example}.
\end{enumerate}
Fix $\alpha < 1/8$. Then the following holds with 
high probability:
For any $\btheta$,
 there exists $ \bu =\bu (\btheta)$ such that $\bX\bu = \bX\btheta$ 
 satisfying 
 $$\norm{\bu}_\infty \le 2 \norm{\btheta}_2 p^{-\alpha} \textrm{\; and \;}
   \norm{\bu}_2 \le \norm{\btheta}_2 \left( 1  + C\right)$$
   for some $C>0$ depending only on $\sOmega$.
\end{lemma}
Note that this lemma assumes $\bSigma=\id_p$ while Proposition~\ref{prop:overparameterized_iid} is stated for more general $\bSigma$. The statement of Proposition 
\ref{prop:overparameterized_iid} follows by noting that, under the assumptions
of the proposition, we have  $\bX\btheta = 
\obX (\bSigma^{-1/2}\btheta)$ where the entries of $\obX$ are independent.
Therefore, Lemma \ref{lemma:KeyLemmaForOverparametrized} implies the existence of 
global empirical risk minimizer $\widehat\bu$ satisfying
$\norm{\widehat\bu}_\infty \le 2 \norm{\bSigma}^{-1/2}_{\infty\to\infty} \big\|\bSigma^{1/2}\widehat\btheta\big\|_2  p^{-\alpha}$
and 
$\big\|\widehat\bu\big\|_2 \le (C+1)\norm{\bSigma^{-1/2}}_\op \big\|\bSigma^{1/2}\widehat\btheta\big\|_2$ 
where $\widehat\btheta$ is a minimizer from~\eqref{eq:minimizer_whp}.
The claim of the proposition then follows by the assumptions on $\bSigma$.

\begin{proof}[Proof of Lemma \ref{lemma:KeyLemmaForOverparametrized}]
For $A\subseteq [p]$, and $\bv\in\reals^p$, we denote by $\bv_A := (v_i:\; i\in A)$ the
vector comprising the entries of $\bv$ with indices in $A$, and 
$\bX_A := (\bX_{\cdot,i}:\; i\in A)$ the submatrix of $\bX$ with columns indexed by $A$.

Let $m = \lceil p^{2\alpha}\rceil$ and denote by $L=L(\btheta)\subseteq [p]$ the set of indices
corresponding to the $m$ entries in $\btheta$ with largest absolute value.
Namely if $|\theta_{i(1)}|\ge |\theta_{i(2)}|\ge \cdots\ge |\theta_{i(p)}|$,
then we let $L:=\{i(1),\dots,i(m)\}$ (ties are broken arbitrarily).
We also let $S=S(\btheta):= [p]\setminus L(\btheta)$ denote the set of indices of `small'
entries.

Note that $m\,|\theta_{i(m)}|^2\le \|\btheta\|_2^2$, whence
\begin{align}
\max_{i\in S}|\theta_i|\le \frac{1}{\sqrt{m}}\|\btheta\|_2\le \|\btheta\|_2 p^{-\alpha}\, .
\end{align}
%

%

We claim that the following holds with 
with high probability:
for any $\btheta \in\reals^p$, there exists
 $\bfeta =\bfeta(\btheta)$ such that $\supp(\bfeta)\subseteq S(\btheta)$,
$\norm{\bfeta}_\infty \le \norm{\btheta}_\infty p^{-\alpha}$,
 $\norm{\bfeta}_2 \le C \norm{\btheta}_2$, and
\begin{equation}
 \bX \bfeta = \bX_L\btheta_L.\label{eq:EtaEquality}
\end{equation}
Postponing the proof of this claim, we define  $\bu= \bu(\btheta)$ by
\begin{equation}
\nonumber
 u_j := 
 \begin{cases}
 \theta_j  + \eta_j & j \in S\\
 0 & j \in L\, ,
 \end{cases},
\end{equation}
whence
\begin{equation*}
\bX\bu =
\bX_{S}\btheta_{S}
+
 \bX_{S} \bfeta_S
=
 \bX_{S} \btheta_{S}+
\bX_L \btheta_L 
= \bX \btheta\, .
\end{equation*}
Further
$\norm{\bu}_\infty \le (\norm{\btheta}_2 + \norm{\btheta}_\infty) p^{-\alpha}$,  and
$\norm{\bu}_2 \le (1 + C)\norm{\btheta}_2$, thus proving the lemma, once the claim is shown.

We are left with the task of proving the existence of $\bfeta=\bfeta(\btheta)$ with the properties 
stated above.
We construct $\bfeta$ by setting $\bfeta_L=\bfzero$ and 
\begin{align*}
\bfeta_S&:= \argmin_{\bxi\in\reals^S}\left\{\norm{\bxi}_2^2 : \bX_{S} \bxi = \bX_L\btheta_L\right\}
= \bX_S^{\sT}(\bX_S\bX_S^{\sT})\bX_L\btheta_L\, .
\end{align*}
This vector satisfies the condition \eqref{eq:EtaEquality} by construction,
and we are therefore left with the task of proving that it satisfies the 
norm constraints, with the claimed probability.

Recalling that $m = \lceil p^{2\alpha}\rceil$, we define the  
\begin{align*}
 \cA & := \left\{\norm{\bX_Q}_\op \le C_1\sqrt{p}\; \textrm{ for all } Q\subseteq[p] \textrm{ with }  |Q|=m \right\},\\
  \cB & := \left\{ \sigma_{\min}\left(\bX_{R} \bX_{R}^\sT\right) \ge \frac{p}{C_1}
  \textrm{ for all } R\subseteq[p] \textrm{ with } |R|=p-m  \right\}\, ,\\
 \cB_{*} & := \left\{ \sigma_{\min}\left(\bX_{R\setminus s} \bX_{R\setminus s}^\sT\right) \ge \frac{p}{C_1}
   \textrm{ for all }  R\subseteq[p] \textrm{ with } |R|=p-m
  ,\textrm{ and } s\in R
   \right\}\, ,\\
    \cD & := \bigg\{ \max_{l\in Q}\left| \bX_s^\sT \left(\bX_{R\setminus s} \bX_{R\setminus s}^\sT\right)^{-1}  \bX_l\right|\le \frac{p^{-3\alpha}}{2C_1}
   \textrm{ for all } Q,R\subseteq[p]\\
   &\hspace{75mm}\textrm{ with } |Q|=m ,  R=[p]\setminus Q \textrm{ and } s\in R\bigg\}\, ,
\end{align*}
where $\bX_s = \bX_{\{s\}}$ is the $s$-th column of $\bX$.

Here $C_1$ is a constant that will be specified below. On the intersection of these events,
we have
 \begin{align*}
  \norm{\bfeta}_2^2 
    &\le \norm{\bX_S}_\op\norm{\bX_L}_\op \Big\|{\left(\bX_{S}\bX_{S}^\sT\right)^{-1}}\Big\|_\op \norm{\btheta_L}_2^2 \\
    &\le C_1^3 \norm{\btheta}_2^2 ,
 \end{align*}
 which verifies the $\ell_{2}$ bound on $\bfeta$.
  
In order to bound the $\ell_{\infty}$ norm of $\bfeta$, 
note that for $s\in S$, 
\begin{equation*}
  \eta_{s} := \frac{\bX_s^\sT \left(\bX_{S/s} \bX_{S/s}^\sT\right)^{-1}  \bX_L \btheta_L}{ 1 + \bX_s^\sT \left(\bX_{S/s} \bX_{S/s}^\sT\right)^{-1} \bX_s }\, ,
\end{equation*}
where $\bX_s = \bX_{\{s\}}$ is the $s$-th column of $\bX$.
We therefore have, on the event $\cA\cap\cB\cap \cB_*\cap \cD$, 
\begin{align*}
|\eta_{s}| &\le \sum_{l\in L}\left| \bX_s^\sT \left(\bX_{S/s} \bX_{S/s}^\sT\right)^{-1}  \bX_l \right|
\left|\theta_l\right|\\
&\le \lceil p^{2\alpha}\rceil \norm{\btheta}_\infty \max_{l\in L}\left| \bX_s^\sT \left(\bX_{S/s} \bX_{S/s}^\sT\right)^{-1}  \bx_l \right|\, ,\\
&\le 2C_1\,  p^{2\alpha} \norm{\btheta}_\infty\cdot p^{-3\alpha}/(2C_1) \\  
&\le  p^{-\alpha} \norm{\btheta}_\infty\, .
\end{align*}

In order to conclude the proof of the lemma, we need to prove that each of events 
$\cA$, $\cB$, $\cB_*$, $\cD$ holds with 
high probability
for a suitable choice of $C_1$. 

We split the analysis into two cases depending on the feature distribution.

\noindent\textbf{(1.) Independent features.}
For event $\cA$, note that $\|\bX_Q\|_{\op}\le \|\bX\|_{\op}\le 2(\sqrt{p}+\sqrt{n})$
with probability at least $1- C\exp(-p/C)$, see \cite[Theorem 4.4.5]{vershynin2018high},
and hence the claimed probability bound follows.

For event $\cB$, by Theorem 1.1 of~\cite{rudelson2009smallest},
for any set $R$, $|R|=p-m$, we have
\begin{equation*}
\P\left(\sigma_n(\bX_{R}) \le \epsilon \left(\sqrt{p-m - 1} - \sqrt{n-1}\right)\right)  
\le \left(C_3 \epsilon\right)^{p-m - n } + e^{-c_3\left(p-m\right)},
\end{equation*}
for $C_3,c_3>0$ and any $\epsilon >0$, where $\sigma_n$ is the $n$-th largest singular value.
Hence, for a suitable choice of $C_1$, $\sigma_{\min}(\bX_{R}\bX_R^{\sT})\ge 2p/C_1$
with probability at least $1-c_0\exp(-c_0'p)$. The claim follows by taking a union bound over the
$\binom{p}{m} =  \exp(O(p^{2\alpha}\log p))$ choices of set $R$. 

For event $\cB_*$, the bound follows in the same way (the only difference being that 
the union bound  is over $m\binom{p}{m}$ terms). 

Finally, for event $\cD$, let 
\begin{equation}
\cB_{**}
:= 
\left\{
\Big\|\left(\bX_{R\setminus s} \bX_{R\setminus s}^\sT\right)^{-1}\bx_s\Big\|_2\le C/\sqrt{p}
\textrm{ for all } R \subseteq [p] \textrm{ with } |R| = p-m, \textrm{ and } s\in R
\right\}.
\end{equation} 

It is immediate to see that
$\P(\cB_{**})\ge 1-c\exp(-c'p)$ for some constants $c,c''$, because of
the lower bound on the probability of event $\cB_*$ and $\|\bx_s\|_2\le c''\sqrt{n}$
with similar probability since $\bx_s$ is subgaussian.

Next note that, defining $\bv_{R,s}:=\left(\bX_{R\setminus s} \bX_{R\setminus s}^\sT\right)^{-1}\bx_s$,
we have
\begin{align*}
\P(\cD^c )  &\le \P(\cD^c\cap\cB_{**} ) +\P(\cB_{**}^c)\\
&\le \sum_{R:|R| = p-m} \sum_{s\in R}\sum_{l\in Q=[p]\setminus R}
\P\left(\left\{|\bv_{R,s}^{\sT}\bx_l|\ge
p^{-3\alpha}/(2C_1)\right\}\cap\cB_{**}\right)+ \P(\cB_{**}^c) \\
& \stackrel{(a)}{\le} 2 m(p-m)\binom{p}{m} \exp\Big\{-\frac{C''p}{(p^{-3\alpha})^2}\Big\} + c\, e^{-c'p}\\
& \le C\exp(-p^{1-6\alpha}/C)\, ,
\end{align*}
where the inequality $(a)$ follows because of the previous bound on
$\P(\cB_{**}^c)$ and because $\bx_l$ is a subgaussian vector with subgaussian norm of order one,
independent of $\bv_{R,s}$ and of $\cB_{**}$. 

\noindent\textbf{(2.) Random Features.}
Recall that under the random features model considered, 
 $\bZ \in \R^{n \times d}$ has i.i.d. entries $\normal(0, 1)$. Let $(\bz_i)_{i\le n}$,
$\bz_i\in\reals^d$ denote the rows of $\bZ$.

Lemma 4 of~\cite{montanari2023universality} states that
there are universal constants $c, C > 0$ such that the following holds:
For every $n, p, d$, the event 
	\begin{align*}
		\Omega_{\bZ} = \Big\{\norm{\bZ}_\op \le C(\sqrt{d}+\sqrt{n}),~~\max_{i\le n} \big|&\norm{\bz_i}_2-\sqrt{d}\big|
        \le C\sqrt{\log n}, \\
			&\max_{i, j\le n, i \neq j}|\langle \bz_i, \bz_j\rangle| \le C\sqrt{d}\log n\Big\}
	\end{align*}
satisfies the following three properties:  
\begin{enumerate}[(i)]
\item $\Omega_{\bZ}$ has probability at least $1-n^{-1} - e^{-cd}$. 	
\item Conditional on $\bZ$, the columns of $\bX$ are i.i.d.
	subgaussian with parameter $\nu^2(\bZ)$ (depending on $\bZ$). Furthemore, on the event $\Omega_{\bZ}$,
	\begin{equation*}
		\nu(\bZ) \le C (1+\sqrt{p/d}) \norm{\sigma}_{{\rm Lip}}.
	\end{equation*} 
\item Let $\bSigma_{\bZ}\in \reals^{n\times n}$ denote the conditional covariance if the columns of $\bX$ given $\bZ$.
	Then on the event $\Omega_{\bZ}$,
	\begin{equation*}
		\mu_2 - \delta_n 
			\le \lambda_{\min}(\bSigma_{\bZ}) \le \lambda_{\max}(\bSigma_{\bZ}) \le C\mu_1 \cdot (1+\sqrt{p/d}) + \mu_2 
			+ \delta_n
	\end{equation*}
	where $|\delta_n| \le \frac{C}{\sqrt{n}}((1+\sqrt{p/d})^4 + \log^4 n)$, for some $\mu_1,\mu_2 > 0$, independent of $n$.
\end{enumerate}

With this lemma in hand, showing that the event $\cA$ has high probability is straightforward.

For $R\subseteq [p]$, let $\bSigma_{\bZ,R}$ denote the covariance of the columns of $\bX_{R}$ conditional on $\bZ$ (note that they are i.i.d. conditional on $\bZ$).
The above implies that for sufficiently large $n$,
	\begin{equation*}
		 c_4
			\le \lambda_{\min}(\bSigma_{\bZ, R}) \le \lambda_{\max}(\bSigma_{\bZ, R}) \le C_4
	\end{equation*}
    uniformly over all choices of $R$ with $|R| = p-m$.

This implies the existence of constants $c_5,c_6>0$ such that
for any $R\subseteq[p]$ with $|R| = p - m$, the following bound holds on the event $\Omega_{\bZ}$:
\begin{align}
\label{eqn:key-bound-high-probability-two}
	&\P\left(\inf_{\bu \in \R^{n}, \norm{\bu}_2 = 1} \norm{\bX^{\sT}_R \bu}_2 \ge \frac{c_5}{4}\sqrt{p}
   \; \Big|\bZ\right) \ge 1-e^{-c_6(p-m)}. 
\end{align} 

Now we can apply an argument similar to the one applied for the independent features model, that $\cB,\cB_{*}$ and $\cB_{**}$ are high probability sets conditional on $\bZ$ in $\Omega_\bZ$. Then since the columns of $\bX$ are once again independent conditional on $\bZ$, we can once again deduce that $\cD$ is a high probability set from our conclusion about $\cB_{**}$.
We conclude the proof by noting that $\Omega_\bZ$ is a high probability set.

\end{proof}

\subsection{Proof of universality in Ex.~\ref{ex:non-convex-nn}}
\label{proof:ex-non-convex-nn}

Let us denote $\bX = \bsigma_1(\bW^\sT\bZ)$ (where $\sigma_1$ is applied element-wise).
By assumption on $\sigma_2$, for any $\bTheta$, there exists some $\bS(\bTheta,\bX)\in\R^{\sfk\times n}$ with $\|\bS\|_\infty \le C$ such that
\begin{equation}
\bS = \sigma_2(\bTheta^\sT \bX).
\end{equation}

So there exists some $\widehat\bS = (\widehat\bs_1,\dots,\widehat\bs_n)$, $\widehat\bs_i \in\R^{\sfk}$ such that
\begin{align}
\min_{\bTheta} \widehat R_n(\bTheta) &=  \frac1n \sum_{i=1}^n \log \left\{ 1 + \exp\left[-y_i \ba^\sT \widehat\bs_i\right] \right\}
\end{align}
satisfying $\|\widehat\bs_i\|_\infty \le C.$ Let $\widehat\bV \in\R^{\sfk\times n}$ be any choice of matrix satisfying 
\begin{equation}
    \widehat\bS = \sigma_2(\widehat\bV).
\end{equation}
Then for $j\in [\sfk],$ with $\widehat\bv_j$ denoting the $j$th column of $\widehat\bV$, choose
\begin{equation}
    \widehat\btheta_{j} := \argmin_{\bX \widehat \btheta_j = \widehat\bv_j} \|\hat \btheta_j\|_2.
\end{equation}
We then have for any $j$, 
\begin{equation}
    \|\hat \btheta_j\|_2 = \|\bX^\sT(\bX\bX^\sT)^{-1} \widehat \bv_j\|_2 \le \frac1{\sigma_{\min}(\bX)} \|\hat\bv_j\|_2.
\end{equation}
In the proof of Proposition~\ref{prop:overparameterized_iid}, we have already obtained a high probability bound of the form $\sigma_{\min}(\bX)\ge c \sqrt{p}$  for $\bX$ following the random features distribution.
A union bound, along with the $\ell^\infty$ bound on $\|\widehat\bv_j\|$ then gives that
$\|\widehat \bTheta\|_F \le C$ for some $C>0$ independent of $n$, with high probability, where $\widehat\bTheta := (\widehat\btheta_1,\dots,\widehat\btheta_\sfk).$ Finally, from the definition, it is clear that $\widehat R_n(\widehat \bTheta) = \min_{\bTheta} \widehat R_n(\bTheta).$ Invoking Proposition~\ref{prop:overparameterized_iid} concludes the proof of universality for the setting of the example.

\subsection{Proof of Proposition~\ref{prop:univ_m_estimation}}
\label{proof:prop_univ_m_estimation}
Fix $D>0$ and $\btheta$ with $\|\btheta\|_2\le D$.
First, note that for any $m\in[p]$, there exists some $S\subseteq [p]$ with $|S| = m$
satisfying $\min_{j \in S}|\theta_j| \ge \max_{j\in S^c} |\theta_j|$. Writing
$\btheta^\sT = (\btheta_S^\sT, \btheta_{S^c}^\sT)$, we have
\begin{equation}
\nonumber
\norm{\btheta_{S^c}}_{\infty} \le
\min_{j\in S} |\theta_j| \le \left(\frac1m \|\btheta_S\|^2 \right)^{1/2}
\le  \frac{D}{\sqrt{m}}.
\end{equation}

So for any $m$ and $D>0$, there exists some $S$ with $|S|=m$ so that
\begin{align}
    \{\|\btheta\|_2\le D\}
    &\subseteq
    \{\btheta= (\btheta_S,\btheta_{S^c}):  \|\btheta_S\|_2,\|\btheta_{S^c}\|_2\le D, \|\btheta_{S^c}\|_\infty \le Dm^{-1/2}\}.
\end{align}

With
\begin{equation}
    \widehat R_n(\btheta;\bX) := \frac1n \sum_{i=1}^n \ell(\bx_i^\sT\btheta, \bx_i^\sT\btheta^\star, \eps_i)  + r(\btheta),
\end{equation}
 for any $m$, we have the bound
\begin{align}
\label{eq:bound_for_general_k}
 \min_{\|\btheta\|_2 \le D} \widehat R_{n}(\btheta; \bX)
  &\ge \min_{\substack{
  |S| = m
  \\\|\btheta_S\|_2 \le D } } 
  \min_{\substack{\|\btheta_{S^c}\|_2 \le D \\ \|\btheta_{S^c}\|_\infty \le D m^{-1/2} }}
  \widehat R_{n}((\btheta_S, \btheta_{S^c}); \bX)\\
  &=
  \min_{\substack{
  |S| = m
  \\\|\bu\|_2 \le D } } 
  \min_{\substack{\|\bbeta\|_2 \le D \\ \|\bbeta\|_\infty \le D m^{-1/2} }}
  \widehat K_{S,n}(\bbeta, \bu; \bX, \bX) + r(\bu).
\end{align}

Now, by inspecting the proof of Theorem~\ref{thm:main_empirical_univ}, one can see that for any $t \in \R,  \delta >0$ and sequence $m_n\to \infty$ with $m_n/n \to 0$ as $n\to \infty$,

\begin{align*}
&\lim_{n\to\infty}\P\left(
     \min_{\substack{
  |S| = m
  \\\|\bu\|_2 \le D } } 
  \min_{\substack{\|\bbeta\|_2 \le D \\ \|\bbeta\|_\infty \le D m^{-1/2} }}
  \widehat K_{S,n}(\bbeta, \bu; \bX, \bX) + r(\bu)
< t - \delta
\right)\\
&\le
\lim_{n\to\infty}\P\left(
     \min_{\substack{
  |S| = m
  \\\|\bu\|_2 \le D } } 
  \min_{\substack{\|\bbeta\|_2 \le D \\ \|\bbeta\|_\infty \le D m^{-1/2} }}
  \widehat K_{S,n}(\bbeta, \bu; \bX, \bG) + r(\bu)
< t - \delta/2
\right).
\end{align*}

Indeed, one can modify the proof of Theorem~\ref{thm:main_empirical_univ} to only interpolate between the coordinates of $\bX$ and $\bG$ with indices in $S^c$, and follow through with the argument, since $\bX$ and $\bG$ have independent entries. Since $m_n\to\infty$, the delocalization condition allows us to replace the non-Gaussian samples corresponding to indices in $S^c$ with Gaussian ones.

Combining this with the above bound we have
\begin{align*}
\lim_{n\to\infty}&\P\left(
\min_{\norm{\btheta}_2 \le D } \widehat R_n(\btheta; \bX)
< R_{\infty}^\star - \delta
\right)\\
&\le
\lim_{n\to\infty}\P\left(
     \min_{\substack{
  |S| = m
  \\\|\bu\|_2 \le D } } 
  \min_{\substack{\|\bbeta\|_2 \le D \\ \|\bbeta\|_\infty \le D m^{-1/2} }}
  \widehat K_{S,n}(\bbeta, \bu; \bX, \bG) + r(\bu)
< R_{\infty}^\star - \delta
\right)\\
&{\le}
\lim_{n\to\infty}\P\left(
\min_{\|\bu\|_2 \le D, |S| = m_n} \widehat K_{S,n}^\star(\bu; \bX, \bG, D)
+ r(\btheta_S)
< R_{\infty}^\star - \delta/2
\right)\\
&
=0,
\end{align*}
where in the last step we used Eq.~\eqref{eq:R_thry_lower_bound}, the assumption of the proposition.

\noindent
Now let us give a matching upper bound. We have by Theorem~\ref{thm:main_empirical_univ} and the upper bound on the Gaussian limit in~\eqref{eq:gaussian_upper_bound} that, for any $\eps_n\to\infty,$
\begin{align*}
\lim_{n\to\infty}&\P\left(
\min_{\norm{\btheta}_2 \le D } \widehat R_n(\btheta; \bX)
> R_{\infty}^\star + \delta
\right)\\
&\le
\lim_{n\to\infty}\P\left(
\min_{\norm{\btheta}_2 \le D, \|\btheta\|_\infty < \eps_n } \widehat R_n(\btheta; \bX)
> R_{\infty}^\star + \delta
\right)\\
&\le
\lim_{n\to\infty}\P\left(
\min_{\norm{\btheta}_2 \le D, \|\btheta\|_\infty < \eps_n } \widehat R_n(\btheta; \bG)
> R_{\infty}^\star + \delta/2
\right) \\
&=0.
\end{align*}
This concludes the proof of the statement.

\subsection{Details for Example~\ref{ex:robust_m_estimation}}
%
\label{sec:ex_robust_m_estimation_details}
Let $\bx\in\R^{m}$ be a vector with i.i.d. entries with the same distribution as the entries of $\bx_i$.
For $\bu\in\R^{m}$, $r\ge 0,s>0, \alpha\in[-1,1]$, define
\begin{align*}
   H_m(\bu; r,\alpha, s):=&\;
   \E\left[ 
   M_\ell\Big(\bx^\sT\bu + r\sqrt{1-\alpha^2}G_1 + r\alpha G_0 ,\kappa G_0, \eps;s\Big)
   \right] \\
   &- \frac{\sgamma}{2}\frac{r^2(1-\alpha^2)}{s} +\frac{\lambda}{2}r^2+  \frac{\lambda}{2} \|\bu\|_2^2\, ,
\end{align*}
where the expectation is over $G_0,G_1 \stackrel{i.i.d.}\sim\cN(0,1),\eps,\bx$, mutually independent, and
we defined the Moreau envelope of $\ell(v,v^\star,\eps)$ with respect to its first argument:
\begin{equation}
\nonumber
 M_{\ell}(t, v^{\star}, \eps;s) := \min_{v\in\reals} \left\{
    \ell( v, v^{\star},\eps)
    + \frac{1}{2 s} \left( v - t\right)^2
   \right\}    \, .
\end{equation}

Define 
\begin{equation}
\nonumber
    R^\star_\infty := \min_{\substack{r \in[0,D]\\ \alpha\in [-1,1]}} \sup_{s>0} H_1(0; r,\alpha, s).
\end{equation}
Fix a sequence $m_n \to \infty$ with $(m_n\log(n))/n \to 0$.
A by-now standard analysis via Gordon's Gaussian comparison inequality,
and an $\eps$-net argument over $\bu\in B_2^{m_n}(D)$, yields lower bound
\begin{align*}
&\P\Bigg(\;
\bigg\{
\min_{\substack{|S| = m_n \\
\|\bu\|_2 \le D}
} \widehat K_{S,n}^\star(\bu; \bX, \bG,D)
< R_{\infty}^\star - \delta
\bigg\} \cap \cG_n
\Bigg)\\
&\le 2\, \P\Bigg(\;
\bigg\{
\min_{
\|\bu\|_2 \le D,
\substack{r \in[0,D]\\ \alpha\in [-1,1]}} \sup_{s>0} H_{m_n}(\bu; r,\alpha, s)
< R_{\infty}^\star - \delta/2\bigg\} 
\cap \cG_n
\Bigg)
\end{align*}
where $\cG_n$ is high a probability event, i.e. $\lim_{n\to\infty}\P(\cG_n)=1$.
We note that this lower bound does not require convexity of $\ell$.

To conclude the proof of condition \eqref{eq:R_thry_lower_bound}, it is sufficient to show that 
$(i)$~$R_\infty^\star$ satisfies Eq.~\eqref{eq:gaussian_limit}, and $(ii)$~for all $r\ge 0,s>0, \alpha\in[-1,1]$ and $m$,
\begin{equation}
\label{eq:required_ineq}
   \min_{\|\bu\|_2\le D} H_m(\bu; r,\alpha, s ) \ge H_m(\bzero; r,\alpha, s) =  H_1(0; r,\alpha, s)\, .
\end{equation}

%
Under strictly convex $\ell$,
\begin{enumerate}
    \item[$(i)$] Eq.~\eqref{eq:gaussian_limit}  holds by Theorem 4 of~\cite{asgari2025local}.
    \item[$(ii)$] The Moreau envelope $t\mapsto  M_{\ell}(t, v^{\star}, \eps;s)$
 is convex in its first argument. Hence,  for any fixed $m$ and $\bu$, Jensen's inequality implies
\begin{align*}
&\E\Big[ \E\Big[
 M_\ell\Big(\bx^\sT\bu + r\sqrt{1-\alpha^2}G_1 + r\alpha G_0 ,\kappa G_0, \eps;s\Big)
   \big| G_1,G_0,\eps\Big] \Big]\\
&\quad\quad\quad\ge
\E\Big[
 M_\ell\Big( r\sqrt{1-\alpha^2}G_1 + r\alpha G_0 ,\kappa G_0, \eps;s\Big)
 \Big] \, .
\end{align*}
So we indeed have the inequality of Eq.~\eqref{eq:required_ineq}.
\end{enumerate}

\subsection{Proof of Lemma~\ref{lemma:op_norm_x}}
\label{proof:op_norm_x}
The proof is a standard argument following the argument for bounding 
$\E\left[\norm{\bZ}_\op\right]$ for a matrix $\bZ$ with i.i.d. subgaussian rows (see for example~\cite{vershynin2018high}, Lemma 4.6.1). Note, however, that such a bound, or subgaussian matrix deviation bounds
such as Theorem 9.1.1 of~\cite{vershynin2018high} that assume that the rows of $\bX$, are subgaussian are not directly applicable in our case, since
the projections of $\bX$ are subgaussian only along the directions of $\cS_p$.
Indeed, we are interested in cases such as the example in Section~\ref{section:ntk_example} where the feature vectors $\bx_i$ are not subgaussian.
Although the statement of Lemma~\ref{lemma:op_norm_x} is a direct extension of such results, we include its proof here for the sake of completeness.

We only need to prove the bound for $\bX$; indeed, $\bG$ itself satisfies Assumption~\ref{ass:X}. 
Furthermore, let $\overline \bX := \bX - \E[\bX]$, and recall the definition
\begin{equation}
\nonumber
 \norm{\bX}_{\cS_p} := \sup_{\left\{\btheta\in\cS_p : \norm{\btheta}_2 \le 1\right\}} \norm{\bX 
\btheta}_2.
\end{equation}

We begin with the following lemma.
\begin{lemma}
\label{lemma:X_sp_norm_tail_bound}
Assume $\bX$ satisfies Assumption~\ref{ass:X}, and $\cC_p\subseteq\cS_p$.
There exist constants $C,\widetilde C,c >0$
such that for all $t>0$,
\begin{equation*}
\P\left(\norm{\overline\bX}_{\cS_p}^2 \ge n \widetilde C 
\left((\delta_t^2 \vee \delta_t)+ 1\right) \right)
\le
2e^{-c  t^2}\\
\end{equation*}
where 
\begin{equation*}
\delta_t := C \frac{\sqrt {p}}{\sqrt n} + \frac{t}{\sqrt n}.
\end{equation*}
\end{lemma}

\begin{proof}
Letting $\overline\bx_i$ be the rows of $\overline \bX$,
note that by Lemma 2.6.8 of~\cite{vershynin2018high} we have
\begin{equation*}
 \norm{\btheta^\sT\overline\bx_i}_{\psi_2} \le C\norm{\btheta^\sT \bx_i}_{\psi_2}.
\end{equation*}
Recall that $\cS_p$ $\subseteq B_2^p(\sfR)$, 
and hence there exists an
$\alpha$-net $\cN_{\alpha}$ of $\cS_p$ of size $|\cN_{\alpha}| \le C(\sfR,\alpha)^p$ for 
some constant 
depending only on $\sfR,\alpha$. Fix $\btheta\in\cN_{\alpha}$ and 
note that 
\begin{equation*}
 \norm{\obX\btheta}_2^2 = \sum_{i=1}^n (\overline\bx_i^\sT \btheta)^2.
\end{equation*}
By Assumption~\ref{ass:X},
 $(\obx_i^\sT \btheta)^2$ are squares of i.i.d subgaussian random variables
with subgaussian norm $K_\btheta \le \sfK$ uniformly in $\btheta$, and with means
\begin{equation*}
 \E\left[\left(\obx_i^\sT \btheta\right)^2\right] =
\norm{\E\left[\obx\obx^\sT \right]^{1/2}\btheta}_2^2  =: V_\btheta \le \sfK,
\end{equation*}
where the last inequality holds uniformly over $\btheta$ (see Proposition 2.5.2 of~\cite{vershynin2018high} for the properties of subgaussian variables). 
Hence, via Bernstein's inequality (2.8.3 of~\cite{vershynin2018high}), we have for any $s >0$,
\begin{align*}
\P\Bigg(\sum_{i=1}^n (\obx_i^\sT \btheta)^2 \ge & n (s + 1)\sfK \Bigg) =
\P\left(\sum_{i=1}^n (\obx_i^\sT \btheta)^2 - n V_\btheta \ge n (s +1)\sfK
- n V_\btheta
\right)\\
&\le 2\exp\left\{- c n\min\left\{  
\left(\frac{(s+1)\sfK - V_\btheta}{K_\btheta}\right)^2,
\left(\frac{(s+1)\sfK - V_\btheta}{K_\btheta}\right)
\right\} \right\}\\
&\stackrel{(a)}{\le} 2 e^{ -c n(s^2\vee s )}
\end{align*}
where for $(a)$ we used that $\sup_\btheta V_\btheta \le \sfK$ and 
$\sup_\btheta K_\btheta \le \sfK$. 
Taking $C \ge ( \log C(\sfR ,\alpha)/ c )^{1/2}$
and $s = \delta_t^2 \vee \delta_t$,
we have via a union bound over $\cN_\alpha$
\begin{align}
\P\left(\sup_{\btheta\in\cN_\alpha}\sum_{i=1}^n (\obx_i^\sT \btheta)^2 \ge  \sfK n
\left((\delta_t^2 \vee \delta_t) + 1\right) \right)
&{\le} 2C(\sfR,\alpha)^p e^{-c n (s^2 \vee s)}\nonumber\\
&\stackrel{(a)}{=}
2C(\sfR,\alpha)^pe^{-c n \delta_t^2}\nonumber\\
&\stackrel{(b)}{\le}
2C(\sfR,\alpha)^pe^{-c  \left(C^2 p + t^2 \right)}\nonumber\\
&\stackrel{(c)}{\le}
2e^{-c  t^2}.\label{eq:epsilon_not_op_bound}
\end{align}
where for $(a)$ we used that $s^2\vee s = \delta_t^2$, for $(b)$ we used the 
definition of $\delta_t$, and for $(c)$ that $C \ge (\log C(\sfR,\alpha) / c)^{1/2}.$ 
Now via a standard epsilon net argument (see for example the proof of Theorem 4.6.1 in~\cite{vershynin2018high}), one can show that
\begin{equation}\nonumber
 \sup_{\btheta\in\cS_p}\norm{\obX\btheta}_2^2 \le C_0(\sfR,\alpha)
\sup_{\btheta\in\cN_{\alpha}}\norm{\obX\btheta}_2^2
\end{equation}
for some $C_0$ depending only on $\sfR$ and $\alpha$. Combining this with~\eqref{eq:epsilon_not_op_bound}
 gives the desired result.
\end{proof}

\begin{lemma}
\label{lemma:simpler_tail_bound_X_sp_norm}
Assume $\bX$ satisfies Assumption~\ref{ass:X}.
There exist constants $C,c>0$ such that for all 
$t >0$,
\begin{equation}
\nonumber
    \P\left(\norm{\obX}_{\cS_p} > C\left(\sqrt n + \sqrt 
p + t\right)  \right)\le  2e^{-ct^2}.
\end{equation}
\end{lemma}

\begin{proof}
Let $\cA$ be the high probability event of Lemma~\ref{lemma:X_sp_norm_tail_bound}, i.e.
\begin{equation}
\nonumber
\cA := \left\{ \frac{\norm{\obX}^2_{\cS_p}}{C_0^2 n}  - 1 \le
(\delta_t^2 \vee \delta_t)\right\}
\end{equation}
where $C_0 := \sqrt{\widetilde C}$ for the constant $\widetilde C$ appearing in the statement of the lemma.
Next define the event
\begin{equation}
\nonumber
    \cG := \left\{\frac{\norm{\obX}_{\cS_p}}{C_0\sqrt{ n}} \le 1\right\}.
\end{equation}

We have on $\cG^c\cap \cA$,
\begin{align*}
\max\left\{
\left(\frac{\norm{\obX}_{\cS_p}}{C_0\sqrt{n}} -  1
\right)^2, 
\left|\frac{\norm{\obX}_{\cS_p}}{C_0\sqrt{n}}- 
1\right|
\right\}
&\stackrel{(a)}\le 
    \left|\left(\frac{\norm{\obX}_{\cS_p}}{C_0\sqrt{n}}\right)^2 - 
1\right|\\
    &\stackrel{(b)}{=}
    \left(\frac{\norm{\obX}_{\cS_p}}{C_0\sqrt{ n}}\right)^2 - 
1\\
    &\stackrel{(c)}{\le} \delta_t^2 \vee \delta_t,
\end{align*}
where $(a)$ follows from
\begin{equation}
\nonumber
    \max\left\{(a-b)^2,|a-b|\right\} \le |a^2 -b^2|,
\end{equation}
holding for $a,b >0$, $a+b \ge 1$. Meanwhile, $(b)$ holds on 
$\cG^c$ and $(c)$ is from the definition of $\cA$. Hence, by the definition of
$\delta_t$ in Lemma~\ref{lemma:X_sp_norm_tail_bound} we have
\begin{equation}
\nonumber
\cA\cap\cG^c
\subseteq 
    \left\{ \left|\frac{\norm{\obX}_{\cS_p}}{C_0\sqrt{n}} - 1 \right| 
\le C \sqrt{\frac{p}{n}} + \frac{t}{\sqrt n} \right\}
    \subseteq 
\left\{\norm{\obX}_{\cS_p}  \le C_1 \left(\sqrt{n} + \sqrt{p} +  t\right)\right\}.
\end{equation}
Meanwhile, from the definition of $\cG$, we directly have
\begin{equation}
\nonumber
\cA \cap \cG \subseteq \left\{ \norm{\obX}_{\cS_p} \le  \sqrt{n} C_0\right\},
\end{equation}
so for some $C_2$ we have 
\begin{equation}
\nonumber
\cA \subseteq \left\{ \norm{\obX}_{\cS_p} \le C_2\left( \sqrt{n} + \sqrt{p} 
+ t\right)\right\},
\end{equation}
implying that 
\begin{align}
\nonumber
    \P\left(\norm{\obX}_{\cS_p} > C\left(\sqrt n + \sqrt 
p  + t\right) \right) &\le \P(\cA^c) \le 2e^{-ct^2}
\end{align}
by Lemma~\ref{lemma:X_sp_norm_tail_bound}.
\end{proof}

Finally, we prove Lemma~\ref{lemma:op_norm_x}.

\begin{proof}[Proof of Lemma~\ref{lemma:op_norm_x}]

By an application of Lemma~\ref{lemma:simpler_tail_bound_X_sp_norm}
with $t := \sqrt{s}/ C - \sqrt{n} -
\sqrt{p}$, 
we have for all $s  > C^2(\sqrt n + \sqrt p)^2,$ 

\begin{equation}
\nonumber
    \P\left( \norm{\obX}_{\cS_p} \ge \sqrt{s} \right)
    \le 2 \exp\left\{
    -c \left(\frac{\sqrt{s}}{C} - \sqrt{n} - \sqrt{p}\right)^2
    \right\}.
\end{equation}
Hence, we can bound the desired expectation as  
\begin{align*}
    \E\left[\norm{\obX}_{\cS_p}^2 \right]
    &= \int_{0}^\infty \P\left(
    \norm{\obX}_{\cS_p}^2 > s \right)\de s\\
    &= \int_{0}^{ C^2\left(\sqrt{n} + \sqrt{p}\right)^2}
    \P\left( \norm{\obX}_\op \ge \sqrt{s}  \right)\de s + \int_{C^2 \left(\sqrt{n} 
+ \sqrt{p}\right)^2}^\infty \P\left( \norm{\obX}_\op \ge \sqrt{s} \right)\de s\\
    &\le C^2 \left(\sqrt{n} +  \sqrt{p} \right)^2 + 2\int_{C^2 ( \sqrt{n} +   
\sqrt{p})^2}^\infty  \exp\left\{ -c\left(\frac{\sqrt{s}}{C} - \sqrt{n} -  
\sqrt{p}\right)^2\right\}\de s\\
    &\le 
    C_0 (n+ p) + C_1 (\sqrt{n} + \sqrt{p}) \\
    &\le C_2 p
\end{align*}
for some sufficiently large $C_2 >0$ since $\lim_{n\to\infty} p(n)/n = \sgamma.$ 
Using that $\bx_i^\sT\btheta$ are i.i.d. subgaussian for $\btheta\in\cS_p$, we have
\begin{equation*}
 \norm{\E\left[\bX\right]}^2_{\cS_p} = \sup_{\btheta\in\cS_p, \norm{\btheta}_2\le 1} \sum_{i=1}^n \E\left[\bx_i^\sT\btheta\right]^2 \le  C_3 p,
\end{equation*}
and hence
\begin{align*}
\E\left[\norm{\bX}_{\cS_p}^2 \right]
&\le 
2\E\left[\norm{\obX}_{\cS_p}^2 \right]+
2\norm{\E\left[\bX\right]}_{\cS_p}^2 \\
&\le C_4 p
\end{align*}
for some $C_3,C_4 >0$.
\end{proof}

\subsection{Proof of Lemma~\ref{lemma:Rn_lipschitz_bound}}
\label{proof:Rn_lipschitz_bound}
We prove the bound for the model with $\bX$. 
Let us define $L_n(\bTheta;\bX,\beps):= \sum_{i=1}^n \ell(\bTheta^\sT 
\bx_i;
\bTheta^{\star\sT}
\bx_i,
\eps_i)/n$ so that $\widehat R_n(\bTheta;\bX,\by) = L_n(\bTheta; \bX,\beps) + r(\bTheta).$  We have
\begin{align}
\left|\widehat R_n(\bTheta;\bX,\beps) - 
\widehat R_n(\widetilde\bTheta;\bX,\beps) \right| 
&\le
\left|L_n\left(\bTheta;\bX,\beps\right) 
- L_n\left(\widetilde\bTheta;\bX,\beps\right)  \right|
+ \left|r(\bTheta) - r\left(\widetilde\bTheta_k\right)\right|
\nonumber
\\
&\stackrel{(a)}\le  \sup_{\bTheta'\subseteq 
\cS_p^{\sfk}} \left| 
\inner{\grad_\bTheta L_n(\bTheta'; \bX, \beps), \left(\bTheta - \widetilde \bTheta 
\right)}_F\right|\nonumber\\
&\hspace{10mm}+ \sfK_r\left(\sqrt{\sfk}\sfR\right) \norm{\bTheta - \widetilde \bTheta}_F,\label{eq:lipschitz_bound_Rn}
\end{align}
where in $(a)$ we used that the regulizer $r$ is assumed to be locally Lipschitz in Frobenius norm and that $\norm{\bTheta}_F \le \sqrt{\sfk}\sfR$ for $\bTheta\in\cS_p^\sfk$.
Now using $\partial_k$ to denote the partial derivative with respect to the $k$th entry, we compute the gradient
\begin{align}
\grad_{\bTheta}
L_n\left(\bTheta; \bX, \beps\right)
&=\frac1n \sum_{i=1}^n \sum_{k=1}^\sfk \partial_k \ell(\bTheta^\sT \bx_i;\eps_i)
\grad_\bTheta \left(\btheta_k^\sT\bx_i\right)
\nonumber\\
&=\frac1n \bX \bD(\bTheta,\bX,\beps)\label{eq:grad_Ln_computation}
\end{align}
where we defined
\begin{equation}
\nonumber
    \bD(\bTheta,\bX,\beps) := (\bd_1(\bTheta,\bX,\beps), \dots,\bd_\sfk (\bTheta,\bX,\beps))\in\R^{n\times \sfk}
\end{equation}
for
 $\bd_k(\bTheta,\bX,\beps) :=  \left(\partial_k \ell\left(\bTheta^\sT \bx_i,\eps_i \right) 
\right)_{i\in[n]} \in \R^n$. Before applying Cauchy-Schwarz, let us bound the norm
$\norm{\bD(\bTheta,\bX,\beps)}_F$. 
Recall the condition on the gradient of
the loss in Assumption~\hyperref[ass:loss_labeling_prime]{5''}, which implies
\begin{equation}
\nonumber
 \norm{\grad\ell(\bTheta^\sT\bx_i, \bTheta^{\star\sT}\bx_i, \eps_i)}_2^2 \le  C_1 \left(\|\bTheta^\sT  \bx_i\|_2^2 + \|\bTheta^{\star\sT} \bx_i\|_2^2  + |\eps_i|^2  + 1\right)
\end{equation}
for all $i\in[n]$ and some $C_1 > 0$.
Hence, we 
have
\begin{align}
 \norm{\bD(\bTheta,\bX,\beps)}_F^2
&=
\sum_{i=1}^n
 \norm{\grad \ell(\bTheta^\sT \bx_i,\bTheta^{\star\sT} \bx_i, \eps_i) }_2^2\nonumber\\
&\le  
C_1 \sum_{i=1}^n\left(\|\bTheta^\sT  \bx_i\|_2^2 + \|\bTheta^{\star\sT} \bx_i\|_2^2  + |\eps_i|^2  + 1\right)\nonumber\\
&= 
C_1 \left(\|\bX\bTheta\|_F^2 + \|\bX\bTheta^\star\|_F^2 + \|\beps\|_2^2 + n\right)\nonumber\\
&\le C_2 
\left(\|\bX\bTheta\|_F + \|\bX\bTheta^\star\|_F + \|\beps\|_2 + \sqrt{n}\right)^2\nonumber\\
&\le C_2 
\left(\|\bX\|_{\cS_p}\|\bTheta\|_F + \|\bX\|_{\cS_p}\|\bTheta^\star\|_F + \|\beps\|_2 + \sqrt{n}\right)^2
\label{eq:bound_on_norm_d}.
\end{align}
 Combining 
equations~\eqref{eq:grad_Ln_computation} and~\eqref{eq:bound_on_norm_d} allows us to 
bound the first term in~\eqref{eq:lipschitz_bound_Rn} as
\begin{align*}
&\sup_{\bTheta'\in 
\cS_p^{\sfk}} \left| 
\inner{\grad_{\bTheta}L_n(\bTheta', \bX, \beps),\left(\bTheta - \widetilde \bTheta 
\right)}_F\right|\\
&\hspace{10mm}\stackrel{(a)}{=}
\frac1n  \sup_{\bTheta'\in\cS_p^\sfk} \left|
\inner{\bD(\bTheta',\bX,\beps),\bX^\sT\left(\bTheta - \widetilde \bTheta 
\right)}_F\right|\\
&\hspace{10mm}\stackrel{(b)}\le 
\frac1n  \sup_{\bTheta'\in\cS_p^\sfk}\norm{\bD(\bTheta',\bX,\beps)}_F \norm{\bX}_{\cS_p} \norm{\bTheta - \widetilde\bTheta}_F\\
&\hspace{10mm}\stackrel{(c)}{\le}
\frac{C_2}{n} \sup_{\bTheta'\in\cS_p^\sfk}
\left(\|\bX\|_{\cS_p}\|\bTheta'\|_F + \|\bX\|_{\cS_p}\|\bTheta^\star\|_F + \|\beps\|_2 + \sqrt{n}\right)
 \norm{\bX}_{\cS_p}  \norm{\bTheta - \widetilde \bTheta}_F,
\end{align*}
where $(a)$ follows from Eq.~\eqref{eq:grad_Ln_computation}, $(b)$ follows from the 
assumption that $\cS_p$ is symmetric and convex, and $(c)$ follows from~\eqref{eq:bound_on_norm_d}.
Finally, combining with~\eqref{eq:lipschitz_bound_Rn} we obtain 

\begin{align}
\nonumber
\left|\widehat R_n(\bTheta,\bX,\beps) - 
\widehat R_n(\widetilde\bTheta,\bX,\beps) \right| 
&\le C_3 \left( \frac{\norm{\bX}^2_{\cS_p}}{n} + 
\frac{\norm{\bX}_{\cS_p}\norm{\beps}_2}{n}+ 1 \right) \norm{\bTheta -\widetilde \bTheta}_F
\end{align}
for some constant $C_6>0$. This 
concludes the proof.

\subsection{Auxiliary Lemmas for the proof of Lemma~\ref{lemma:poly_approx}}
\label{section:proof_poly_approx_details}

\begin{lemma}
\label{lemma:P_G_c}
    For any $B>K$, we have constants $C,C' >0$ such that
\begin{equation}
\nonumber
\sup_{\bTheta \in\cS_p^\sfk}
\P\left(
\cG_{\bTheta,B}^c
\right)
\le 
C e^{-C' B^2}.
\end{equation}
\end{lemma}

\begin{proof}
From the definition of $\bu_{t,1}$ along with Assumption~\ref{ass:X}, we have
$\norm{\btheta^\sT\bu_{t,i}}_{\psi_2} \le 2  \sfR\sfK$, and similarly for $\widetilde \bu_{t,1}.$ 
Furthermore, the assumption on the noise asserts that $\norm{\epsilon_1}_{\psi_2} \le \sfK$. 
So a union bound directly gives
\begin{align}
\P(\cG_{\bTheta,B}^c)
&\le \sum_{k\le \sfk} \left(\P\left(|\btheta_k^\sT \bu_1| > B\right)
+
\P\left(|\btheta_k^\sT \widetilde\bu_1| > B\right)\right)\nonumber\\
&\hspace{10mm}+
 \sum_{k\le \sfk^*}\left( \P\left(|\btheta_k^{*\sT} \bu_1| > B\right) 
 +\P\left(|\btheta_k^{*\sT} \widetilde\bu_1| > B\right)\right)
 +
\P\Big(|\epsilon_1| 
> B\Big)
\nonumber\\
&\stackrel{(a)}{\le} C_0(\sfk + \sfk^* +1)
\exp\left\{-\frac{C_1 B^2}{(2\sfR + 1)^2\sfK^2}\right\}
\label{eq:P_G_c_2}
\end{align}
for some universal constants $C_0,C_1\in(0,\infty).$
\end{proof}

\begin{lemma}
\label{lemma:P_M}
For $M > 0$,  we have
\begin{enumerate}[(i)]
    \item \label{item:P_M_12} $R_M(x) = (1-x)^{M+1}/{x}$ for $x\neq 0$;
    \item $R_M(x)^2$ is convex on $(0,1]$; \label{item:P_M_4-1} 
    \item For any $s\in(0,1)$ and $\delta>0$, there exists $M>0$ such that
        $\sup_{t\in[s,1]}\left|R_M(t)\right| < \delta$.
    \label{item:P_M_5-1}
\end{enumerate}
\end{lemma}

\begin{proof}
For~$\ref{item:P_M_12}$, we write for $x >0$, 
\begin{align*}
    R_M(x) &= \frac1{x} \left(1 - \Big(1 - (1-x)\Big)\sum_{l=0}^M(1-x)^l \right)\\
    &= \frac1{x} \left(1 - \sum_{l=0}^M(1-x)^l + \sum_{l=1}^{M+1}(1-x)^{l} \right)\\
    &=\frac{(1-x)^{M+1}}{x}
\end{align*}
as desired. 

For~$\ref{item:P_M_4-1}$, the convexity of $R_M(x)^2$ can be shown by noting 
that~$\ref{item:P_M_12}$ gives
\begin{equation}
\nonumber
    \frac{\de^2}{\de x^2}\left(R_M(x)\right)^2= 
    \frac
    {2(1-x)^{2M}\left(M(2M -1)x^2 + (4M -2)x  + 3\right)}
    {x^4}
     \ge 0 
\end{equation}
for all $x\in(0,1]$ and  $M>0$.

Finally,~$\ref{item:P_M_5-1}$ can be shown by verifying that $P_M$ is indeed the power series 
of $1/x$ with a radius of convergence of $1$.

\end{proof}

\subsection{Proof of Lemma~\ref{lemma:gaussian_approx}}
\label{section:proof_gaussian_approx}
This section is dedicated to proving Lemma~\ref{lemma:gaussian_approx}. 
The first step is extending Eq.~\eqref{eq:condition_bounded_lipschitz_single} as follows.
\begin{proof}[Proof of Lemma~\ref{lemma:proj}]

Fix $\bH = (\btheta_1,\dots,\btheta_K) \in\cS_p^K$ be arbitrary. Let $M=3K$ and define
\begin{equation}
\label{eq:concat_def}
\widetilde \bH:=\left(\begin{array}{@{}ccc@{}}
    \bH
  & \bigzero &\bigzero
  \vspace{2mm}
  \\
  \bigzero &
  \bH &\bigzero
  \vspace{2mm}
  \\
  \bigzero & \bigzero & \bH
\end{array}\right) \in\R^{3p\times M}, 
\quad
\bv:= \left(\bx^\sT ,\widetilde \bg^\sT, \bmu_\bg^\sT \right)^\sT \in\R^{3p}, \quad
\bh:= \left(\bg^\sT ,\widetilde \bg^\sT, \bmu_\bg^\sT \right)^\sT\in\R^{3p},
\end{equation}
so that $\left(\bH^\sT \bx, \bH^\sT \widetilde \bg,\bH^\sT\bmu_\bg\right) = \widetilde \bH^\sT \bv$ 
and $\left(\bH^\sT \bg, \bH^\sT \widetilde \bg,\bH^\sT\bmu_\bg\right) = \widetilde \bH^\sT \bh$. Consider any bounded Lipschitz function $\varphi:\R^M\to\R$, and define
 $\balpha \sim \cN(0,\delta^2 \bI_{M})$ for  $\delta >0$.
We can decompose
\begin{align}
&\left|\E\left[\varphi\left(\bH^\sT \bx,\bH^\sT\widetilde \bg,\bH^\sT\bmu_\bg \right)\right] 
- \E\left[\varphi \left(\bH^\sT \bg, \bH^\sT \widetilde \bg, \bH^\sT\bmu_\bg\right)\right]\right|\nonumber\\
&\hspace{8mm}=
\left|\E\left[\varphi\left(\widetilde\bH^\sT \bv\right)\right] 
- \E\left[\varphi\left(\widetilde\bH^\sT \bh\right)\right]\right|\nonumber\\
&\hspace{8mm}\le 
\left|\E\left[\varphi\left( \widetilde\bH^\sT \bv+ \balpha\right)\right]
-
\E\left[\varphi\left( \widetilde\bH^\sT \bv
\right)\right]\right|
+
\left|
\E\left[\varphi\left( \widetilde\bH^\sT \bh + \balpha\right)\right]
-\E\left[\varphi\left(\widetilde\bH^\sT \bh\right)\right]
\right|
\label{eq:first_term_proj}
\\
&\hspace{18mm}+
\left|
\E\left[\varphi\left(\widetilde\bH^\sT \bv + \balpha\right)\right]
-\E\left[\varphi\left( \widetilde\bH^\sT \bh + \alpha\right)\right]
\right|.
\label{eq:third_term_proj}
\end{align}

Both terms on the right hand side on line~\eqref{eq:first_term_proj} are 
similar and can be bounded in an analogous manner. Namely,
we can write for the first of these
\begin{align}
\left|\E\left[\varphi\left(\widetilde\bH^\sT \bh + \balpha \right) - \varphi\left( \widetilde\bH^\sT\bh \right)\right]\right|
&\le \E\left[\left|\varphi\left(\widetilde\bH^\sT\bh + \balpha \right) - \varphi\left(\widetilde\bH^\sT \bh\right)\right|\right] \nonumber\\
&\le \norm{\varphi}_\Lip \E \norm{\balpha}_2\nonumber\\
&\le \sqrt{M} \norm{\varphi}_\Lip\delta
\label{eq:projections_first_bound}
\end{align}
and similarly for the second term. Now for the term on line~\eqref{eq:third_term_proj}, we have for any random variable $\bw \in\R^{M}$,

\begin{align}
\nonumber
   \E\left[\varphi(\bw + \balpha)\right]
   &= \frac{1}{(2\pi)^M}\int\int \varphi(\bs)\exp\left\{i \bt^\sT \bs- 
\delta^2\frac{\norm{\bt}_2^2}{2}\right\} 
   \phi_\bw(\bt)
   \de\bt \de\bs,
\end{align}
where $\phi_\bw(\bt) := \int \exp\left\{-i \bt^\sT \by\right\} \P_\bw(\de \by)$ is the (reflected) characteristic function of $\bw$. Using this representation 
and
denoting the characteristic functions of $\widetilde\bH^\sT\bv$ and $\widetilde\bH^\sT\bh$ by $\phi_{\bv,\bH}$ and $\phi_{\bh,\bH}$ 
respectively,
we have
\begin{align}
&\left|
\E\left[\varphi\left(\widetilde\bH^\sT\bv + \balpha\right)\right]
-\E\left[\varphi\left(\widetilde\bH^\sT \bh+ \balpha\right)\right]
\right|
\nonumber
\\
&= \left|\frac{1}{(2\pi)^M}\int\int \varphi(\bs)e^{i\bt^\sT \bs - 
\delta^2\norm{\bt}^2/2}\big(\phi_{\bv,\bH}(\bt) - 
\phi_{\bh,\bH}(\bt) 
\big)\de\bt \de\bs \right|
\nonumber
\\
&\le \frac{1}{(2\pi)^M}\int \left|\varphi(\bs)\right|\Bigg(\int e^{2i\bt^\sT \bs - 
\delta^2\norm{\bt}^2/2}
 \de\bt
 \cdot
 \int
\big(\phi_{\bv,\bH}(\bt) - 
\phi_{\bh,\bH}(\bt) 
\big)^2e^{-\delta^2{\norm{\bt}^2}/{2}}\de\bt\Bigg)^{1/2} \de\bs 
\nonumber
\\
&\stackrel{(a)}{=} \frac{1}{(\delta^2)^{M/4} (2\pi)^{3M/4}}\int \left|\varphi(\bs)\right| 
e^{- 
{\norm{\bs}^2}/{\delta^2} 
}d\bs
\left(\int\left(\phi_{\bv,\bH}(\bt) - 
\phi_{\bh,\bH}(\bt) 
\right)^2e^{-\delta^2{\norm{\bt}^2}/{2}}\de \bt \right)^{1/2} \nonumber\\
&= \frac{1}{2^{M/2}}\left(\frac{\delta^2}{2\pi}\right)^{M/4}\E\left[\left|\varphi\left(\frac{\balpha}{\sqrt2}\right)\right|\right] \left(
\int\left(\phi_{\bv,\bH}(\bt) - 
\phi_{\bh,\bH}(\bt) 
\right)^2e^{-\delta^2{\norm{\bt}^2}/{2}}\de\bt \right)^{1/2}\nonumber\\
&\le \frac{\norm{\varphi}_\infty}{2^{M/2}}  
\left(
\left(\frac{\delta^2}{2\pi}\right)^{M/2}
\int\left(\phi_{\bv,\bH}(\bt) - 
\phi_{\bh,\bH}(\bt) 
\right)^2e^{-\delta^2{\norm{\bt}^2}/{2}}\de \bt \right)^{1/2}\nonumber\\
&=  \frac{\norm{\varphi}_\infty }{2^{M/2}}  
\E\left[
\Big(\phi_{\bv,\bH}(\btau_\delta) -  \phi_{\bh,\bH}(\btau_\delta)\Big)^2
\right]^{1/2},
\label{eq:projections_second_bound}
\end{align}
where $\btau_\delta \sim\cN(0,\bI_M/\delta^2)$.
Note that in $(a)$ we used
\begin{align}
\nonumber
\int \exp\left\{2i\bt^\sT \bs - 
\delta^2\frac{\norm{\bt}^2}{2}\right\}
 \de \bt
 &= \left(\frac{2\pi}{\delta^2}\right)^{M/2}  \exp\left\{-2 \frac{\norm{\bs}_2^2}{\delta^2} 
\right\}.
\end{align}

Fix $\bs\in \R^K$ such that $\bs\neq0$. We have for any
$\bH = (\btheta_1,\dots,\btheta_K) \in \cS_p^K$, 
\begin{equation}
\nonumber
\frac{\bH\bs}{\norm{\bs}_1} = \sum_{j=1}^K  \frac{|s_j|}{\norm{\bs}_1}\sign\{s_j\} \btheta_j.
\end{equation}
Recalling that $\cS_p$ is symmetric, we see that $\sign\{s_j\}\btheta_j \in\cS_p$ for all $j\in[K]$,
and then the convexity of $\cS_p$ implies that $\bH\bs/\norm{\bs}_1\in\cS_p$ for $\bs \neq 0$. 
Letting $\phi_{\bx,\bH},\phi_{\bg,\bH}$ be the 
characteristic functions of $\bH^\sT\bx, \bH^\sT\bg$ respectively and fixing  $\bt = (\bs,\widetilde \bs, \bs') \in \R^M$, we have if $\bs\neq 0$,
\begin{align}
   \label{eq:s_equal_0}
  &\limsup_{p\to\infty}\sup_{\bH\in\cS_p^K}\left|\phi_{\bv,\bH}(\bt) - \phi_{\bh,\bH}(\bt) \right|^2\\
  &\stackrel{(a)}{=}
   \limsup_{p\to\infty}\sup_{\bH\in\cS_p^K} \left|\phi_{\bg,\bH}(\widetilde\bs) e^{-i\bs'^\sT \bH^\sT \bmu_\bg}\right|^2
   \left|\phi_{\bx,\bH}(\bs) - \phi_{\bg,\bH}(\bs) \right|^2\nonumber\\
   &\stackrel{(b)}{\le}
   2
   \limsup_{p\to\infty}\sup_{\bH\in\cS_p^K}
   \left|\phi_{\bx,\bH}(\bs) - \phi_{\bg,\bH}(\bs) \right|
   \nonumber\\
   &\le2 \limsup_{p\to\infty}\sup_{\bH\in\cS_p^K}\left|\E\left[ \exp\{-i{\bx^\sT\bH\bs} \} \right] -  
   \E\left[ \exp\{-i{\bg^\sT \bH\bs} \} \right] 
   \right|\nonumber\\
&= 
\limsup_{p\to\infty}\sup_{\bH \in \cS_p^K}\left|\E\left[ 
\exp\left\{-i\norm{\bs}_1  \bx^\sT \left(\frac{\bH\bs}{\norm{\bs}_1}\right)\right\} \right] -  
   \E\left[ \exp\left\{-i
   \norm{\bs}_1  \bg^\sT \left(\frac{\bH\bs}{\norm{\bs}_1}\right)\right\}\right] 
   \right|\nonumber\\
   &\stackrel{(c)}{\le}
\limsup_{p\to\infty}\sup_{\btheta \in \cS_p}\left|\E\left[ 
\exp\left\{-i\norm{\bs}_1  \bx^\sT \btheta\right\} \right] -  
   \E\left[ \exp\left\{-i
   \norm{\bs}_1  \bg^\sT\btheta\right\}\right] 
   \right|
\nonumber
   \\
   \nonumber
   &\stackrel{(d)}=0
\end{align}
where $(a)$ holds because of the independence of $\widetilde \bg$ and $(\bx,\bg)$, $(b)$ holds because $|\phi(\bs)| \in[0,1]$ for all $\bs\in\R^M$, $(c)$ holds
since $\bH \bs/\norm{\bs}_1 \in\cS_p$ and $(d)$ holds by Eq.~\eqref{eq:condition_bounded_lipschitz_single} since $x\mapsto \exp({i\norm{\bs}_1x})$ is bounded Lipschitz for fixed $\bs \in\R^M, \in\R$. Further, for $\bs =0$, by the equality on line~\eqref{eq:s_equal_0}
we immediately have $\left|\phi_{\bv,\bH}(\bt) - \phi_{\bh,\bH}(\bt) \right|^2 =0$ and hence for any fixed $\bt\in\R^M,$ 
\begin{equation}
\label{eq:lim_equal_zero_char}
  \lim_{p\to\infty}\sup_{\bH\in\cS_p^K}\left|\phi_{\bv,\bH}(\bt) - \phi_{\bh,\bH}(\bt) \right|^2 = 0.
\end{equation}

In conclusion, we have
\begin{align*}
&\lim_{p\to\infty}\sup_{\bH\in\cS_p^K}\left|\E\left[\varphi\left(\bH^\sT\bv\right)\right] 
- \E\left[\varphi(\bH^\sT \bh)\right]\right| 
\\
&\hspace{30mm}\stackrel{(a)}{\le} 2\sqrt{M}\norm{\varphi}_\Lip \delta  + 
\frac{\norm{\varphi}_\infty}{2^{M/2}} 
\lim_{p\to\infty}\sup_{\bH\in\cS_p^K}
\E\left[
\Big(\phi_{\bv,\bH}(\btau_\delta) -  \phi_{\bh,\bH}(\btau_\delta)\Big)^2
\right]^{1/2}\\
&\hspace{30mm}\le 2\sqrt{M}\norm{\varphi}_\Lip \delta  +
\frac{\norm{\varphi}_\infty}{2^{M/2}}
\lim_{p\to\infty}
\E\left[
\sup_{\bH\in\cS_p^K}
\Big(
\phi_{\bv,\bH}(\btau_\delta) -  \phi_{\bh,\bH}(\btau_\delta)\Big)^2
\right]^{1/2}\\
&\hspace{30mm}\stackrel{(b)}= 2\sqrt{M} \norm{\varphi}_\Lip \delta,
\end{align*}
where $(a)$ follows from the decomposition in~\eqref{eq:first_term_proj}
and the
bounds in~\eqref{eq:projections_first_bound} 
and~\eqref{eq:projections_second_bound},
and $(b)$ follows from the dominated convergence theorem along with the
limit in~\eqref{eq:lim_equal_zero_char} and domination of the integrand
$ \sup_{\bH\in\cS_p^M} \left(\phi_{\bv,\bH}
(\bt) - \phi_{\bh,\bH}(\bt) 
\right)^2 \le 2$.
Sending $\delta \to 0$ completes the proof.
\end{proof}

Now, via a truncation argument, we show that this can be extended to square integrable locally Lipschitz functions.
\begin{lemma}
\label{lemma:proj_locally_lip}
Let $\cS_p$ be as in Definition~\ref{def:DoG}. Let $K>0$ be a fixed integer and $\widetilde \bg \sim\cN(\bmu_\bg,\bSigma_\bg)$ an independent copy of $\bg$,
and let $\varphi:\R^{3K} \to\R$ be a locally Lipschitz function satisfying
    \begin{align}
        \sup_{p\in \Z_{>0}}\sup_{\bH=(\btheta_1,\dots,\btheta_K)\in\cS_p^K}\E\left[\left| \varphi\left(\bH^\sT \bx,\bH^\sT \widetilde \bg, \bH^\sT \bmu_\bg\right) \right|^2\right] &< \infty,\textrm{ and}\nonumber\\
        \sup_{p\in \Z_{>0}}\sup_{\bH=(\btheta_1,\dots,\btheta_K)\in\cS_p^K}\E\left[\left| \varphi\left(\bH^\sT \bg, \bH^\sT \widetilde \bg, \bH^\sT\bmu_\bg\right)\right|^2\right] &< \infty.
    \label{eq:square_integrability_varphi}
    \end{align}
    Then
\begin{equation}
\nonumber
    \lim_{p\to\infty} \sup_{\bH\in \cS_p^K} 
    \left|\E\left[ \varphi\left(\bH^\sT \bx,\bH^\sT \widetilde \bg, \bH^\sT\bmu_\bg\right) \right] - 
          \E\left[ \varphi\left(\bH^\sT \bg, \bH^\sT \widetilde \bg, \bH^\sT\bmu_\bg\right) \right] \right| = 0.
\end{equation}
\end{lemma}

\begin{proof}[Proof of Lemma~\ref{lemma:proj_locally_lip}]

Fix $\bH = (\btheta_1,\dots,\btheta_K) \in\cS_p^K$ be arbitrary. Let $K=3M$ and again define
$\widetilde \bH \in\R^{3p\times M}, \bv \in \R^{3p}, \bh\in\R^{3p}$ as 
in~\eqref{eq:concat_def}.
First, we bound the probability of the tail event $\left\{ \norm{\bH^\sT \bu}_2 >B\right\}$ for 
$B \ge 2\sqrt{M}\left(\sfR\norm{\bmu_\bg}_2 \vee \sup_{\btheta\in\cS_p}|\bx^\sT\btheta| \right)$. We 
have
\begin{align}
   \P\left(\norm{\widetilde\bH^\sT \bv}_2 > B\right) 
   &\le \P\left(\norm{\widetilde\bH^\sT \bv}_\infty > \frac{B}{\sqrt{M}} \right)\nonumber \\
   &\stackrel{(a)}{\le} \sum_{m=1}^K\P\left(\left|\bx^\sT \btheta_m\right| >\frac{B}{\sqrt{M}}\right) + \P\left(\left|\widetilde \bg^\sT \btheta_m\right| >\frac{B}{\sqrt{M}}\right) 
   \nonumber\\
   &\stackrel{(b)}{\le} C_0 M \exp\left\{-\frac{c_0 
B^2}{M}\right\}\label{eq:tail_bound_bwx}
\end{align}
for some universal constants $c_0,C_0 \in (0,\infty)$. Here in $(a)$ we used that
$|\mu_\bg^\sT\btheta_m | \le B/(2\sqrt{M})$,
and in $(b)$ we used that $\bg$ and $\bx$ are subgaussian with constant subgaussian norm
and that $\E\left[\bx^\sT\btheta_m\right] \vee \E\left[\bg^\sT\btheta_m\right] \le
B/(2\sqrt{M})$. An analogous argument then shows
\begin{equation}
\nonumber
\P\left(\norm{\widetilde\bH^\sT \bh}_2 > B \right) \le C_0 M \exp\left\{ -\frac{c_1 B^2}{M}\right\}\label{eq:tail_bound_bwg}
\end{equation}
for some $c_1>0$. 

Now fix such a $B$ arbitrary and let 
\begin{equation}
\nonumber
  u_{B}(t) := \begin{cases}
  1 & t < B \\
  B+1-t & t\in[B,B+1)\\
  0 & t \ge B+1
   \end{cases}.
\end{equation}
and define $\varphi_B(\bs) := \varphi(\bs) u_{B}\left(\norm{\bs}_2\right)$. 
Noting that $\one_{\{\norm{\bs}_2 \le B-1\}} \le u_{B}\left(\norm{\bs}_2 \right) \le \one_{\{\norm{\bs}_2 \le B\}}$ and that $h_{B}$ is Lipschitz, we see that $\varphi_B$ is bounded and Lipschitz.
To see that it is indeed Lipschitz, take $\bs,\bt$ with $\norm{\bt}_2 \le \norm{\bs}_2$,
\begin{align}
    \left|
   \varphi_{B}(\bt) - \varphi_B(\bs)
    \right|
    &\le |\varphi(\bt)|\one_{\left\{\norm{s}_2\le B+1\right\}}|u_B(\norm{\bs}_2) - u_B(\norm{\bt}_2)|\nonumber\\
    &\hspace{10mm}+ u_B(\norm{\bs}_2)\one_{\{\norm{s}_2\le B+1\}} |\varphi(\bt) - \varphi(\bs)|\nonumber\\
    &\le
    C_1(B) \norm{ \bt - \bs}_2 + C_2(B)\norm{\bt - \bs}_2
    \label{eq:soft_lip_truncation}
\end{align}
for $C_1,C_2$ depending only on $B$ since $\varphi$ is locally Lipschitz. 
We can now write
\begin{align*}
&\lim_{p\to\infty}\sup_{\bH\in\cS_p^K} \left|\E \left[\varphi\left(\widetilde\bH^\sT \bv\right) - \varphi\left(\widetilde\bH^\sT \bh\right) \right]\right|\\
&\hspace{15mm}\le
\lim_{p\to\infty}\sup_{\bH\in\cS_p^K} \left|\E \left[\left(\varphi_B\left(\widetilde\bH^\sT \bv\right) - \varphi_B\left(\widetilde\bH^\sT \bh\right) \right)
\right]\right|\\
&\hspace{30mm}+ 
\lim_{p\to\infty} \sup_{\bH\in\cS_p^K} \E\left[\left|\varphi\left(\widetilde \bH^\sT\bv\right)\left(1-u_B\left(\big\|\widetilde \bH^\sT\bv\big\|_2\right)\right)\right|\right] \\
&\hspace{45mm}+ 
\lim_{p\to\infty} \sup_{\bH\in\cS_p^K} \E\left[\left|\varphi\left(\widetilde \bH^\sT\bv\right)\left(1-u_B\left(\big\|\widetilde \bH^\sT\bv\big\|_2\right)\right)\right|\right] 
\\
&\hspace{15mm} \stackrel{(a)}\le
C_3\lim_{p\to\infty}\sup_{\bH\in\cS_p^K}
\E\left[\left|\varphi\left(\widetilde\bH^\sT\bv\right)\right|^2\right]^{1/2}\P\left(\big\|\widetilde\bH^\sT\bv\big\|_2 > B\right)^{1/2} \\
&\hspace{30mm}+
C_3\lim_{p\to\infty}\sup_{\bH\in\cS_p^K}
\E\left[\left|\varphi\left(\widetilde\bH^\sT\bh\right)\right|^2\right]^{1/2} 
\hspace{-2mm}\P\left(\big\|\widetilde\bH^\sT \bh\big\|_2 > B \right)^{1/2}
\\
&\hspace{15mm}\stackrel{(b)}\le C_4 M \exp\left\{ \frac{-c_2 B^2}{M}\right\}
\end{align*}
for some $C_3,C_4,c_2> 0$. Here, $(a)$ follows from Lemma~\ref{lemma:proj}
and that $ 0 \le 1-u_B(t) \le \one_{t > B}$,
and $(b)$ follows from the tail bounds in equations~\eqref{eq:tail_bound_bwx} 
and~\eqref{eq:tail_bound_bwg} along with the square integrability assumption of $\varphi$. Sending $B\to\infty$ completes the proof.
\end{proof}

Now, we establish Lemma~\ref{lemma:gaussian_approx}.

\begin{proof}[Proof of Lemma~\ref{lemma:gaussian_approx}] 

Recall equations~\eqref{eq:slepian} and~\eqref{eq:bd_def} 
defining $\bu_{t,1},\widetilde\bu_{t,1}$ 
and $\widehat\bd_{t,1}$, respectively, 
in terms of $\bx_1$ and $\bg_1$. Further, recall the definitions of $\widetilde \bg_1,\bw_{t,1},\widetilde \bw_{t,1},\widetilde \epsilon_1$ and $\widehat \bq_{t,1}$ in the statement of the lemma.
Define $\bH := (\bTheta^\star,\bTheta_0,\bTheta_1,\dots,\bTheta_J)$ and the function $\varphi$
\begin{align}
   &\varphi\left( 
   \bH^\sT\bx_1,
    \bH^\sT\bg_1, \bH^\sT \bmu_\bg
   \right) \\
   &\qquad\qquad:= 
   \E\left[
\widetilde \bu_{t,1}^\sT \widehat \bd_{t,1}(\bTheta_0)
\exp\left\{-\beta \sum_{l=0}^J \ell\left(\bTheta_l^\sT \bu_{t,1}; \eta\left(\bTheta^{\star\sT}\bu_{t,1},\epsilon_1 \right)\right)\right\}
\Bigg| \bx_1,\bg_1
\right],
\nonumber
\end{align}
i.e., the expectation is with respect to $\epsilon_1$. Since $\widetilde \epsilon_1$ has the same distribution as $\epsilon_1$, we have
\begin{align}
   &\varphi\left( 
   \bH^\sT\widetilde\bg_1,
    \bH^\sT\bg_1, \bH^\sT \bmu_\bg
   \right) \\
   &\qquad= 
   \E\left[
\widetilde \bw_{t,1}^\sT \widehat \bq_{t,1}(\bTheta_0)
\exp\left\{-\beta \sum_{l=0}^J \ell\left(\bTheta_l^\sT \bw_{t,1}; \eta\left(\bTheta^{\star\sT}\bw_{t,1},\widetilde\epsilon_1 \right)\right)\right\}
\Bigg| \widetilde \bg_1, \bg_1
\right].
\nonumber
\end{align}

Now note that $\varphi$ is locally Lipschitz, 
by the locally Lipschitz assumption on the derivative of $\ell$ in Assumption~\hyperref[ass:loss_labeling_prime]{5''}. 
Additionally,
\begin{equation}
\nonumber
    \sup_{\bH \in\cS_p^{\sfk^\star + (J+1)\sfk}}\E\left[\left| \varphi\left(\bH^\sT \bx_1,\bH^\sT \bg_1, \bH^\sT \bmu_\bg\right) \right|^2\right] \le 
   \sup_{\bTheta^\star\in\cS_p^{\sfk^\star}, \bTheta\in\cS_p^\sfk}\E\left[\left(
\widetilde \bu_{t,1}^\sT \widehat \bd_{t,1}(\bTheta)\right)^2 \right]
\le C_1
\end{equation}
for some $C_1>0$
by Lemma~\ref{lemma:dct_bound} and the nonnegativity of $\ell$ and $\beta$. 
Furthermore, since by the conditions on $\bmu_\bg$ and $\bSigma_\bg$ in 
Assumption~\ref{ass:X}, we have
\begin{equation}
\sup_{\btheta\in\cS_p,\norm{\btheta}_2\le 1}\norm{\widetilde \bg^\sT\btheta}_{\psi_2} \le 
\sup_{\btheta\in\cS_p,\norm{\btheta}_2\le 1}\norm{(\widetilde \bg-\bmu_\bg)^\sT\btheta}_{\psi_2} + \norm{\bmu_\bg}_2 \le \norm{\bSigma_\bg}_{\cS_p} + \norm{\bmu_\bg}_2 \le 2\sfK,
\end{equation}
the distribution $\widetilde \bg_1$
itself satisfies the statement of
Assumption~\ref{ass:X} (in place of $\bx_1$). Hence we similarly have
\begin{equation}
\nonumber
 \sup_{\bH \in\cS_p^{\sfk^\star + (J+1)\sfk}}\E\left[\left| \varphi\left(\bH^\sT \widetilde\bg_1,\bH^\sT \bg_1,\bH^\sT \bmu_\bg\right) \right|^2\right]  \le C_2
\end{equation}
for some $C_2>0$. Therefore, $\varphi$ satisfies the square integrability condition in~\eqref{eq:square_integrability_varphi} of Lemma~\ref{lemma:proj_locally_lip}.
An application of this lemma then yields the claim of Lemma~\ref{lemma:gaussian_approx}.
\end{proof}

\end{document}

%% file: new_def.tex
\usepackage{amsmath}
\usepackage{amssymb}
\usepackage{amsthm}
\usepackage{epsfig}
\usepackage{xcolor}
\usepackage{graphicx}
\usepackage{bm}
\usepackage{caption}
\usepackage{subcaption}
\usepackage{hyperref}
\usepackage[shortlabels]{enumitem}

\usepackage[top=1in, bottom=1in, left=1in, right=1in]{geometry}

\renewcommand{\P}{\mathbb{P}}
\newcommand{\E}{\mathbb{E}}
\newcommand{\Var}{\text{Var}}

\newcommand{\cN}{\mathcal{N}}

\newcommand{\Z}{\mathbb{Z}}
\newcommand{\R}{\mathbb{R}}

\renewcommand{\S}{\mathbb{S}}

\newcommand{\eps}{\varepsilon} 

\def\id{{\boldsymbol I}}

\newcommand{\<}{\langle}
\renewcommand{\>}{\rangle}
\newcommand{\sign}{\text{sign}}
\newcommand{\diag}{\text{diag}}

\newcommand{\op}{{\rm op}}

\newcommand{\grad}{\nabla}
\def\sT{{\mathsf T}}
\def\bzero{{\boldsymbol 0}}

\DeclareMathOperator*{\argmin}{arg\,min}

\def\simiid{{\stackrel{i.i.d.}{\sim}}}

\theoremstyle{plain}
\newtheorem{theorem}{Theorem}
\newtheorem*{theorem*}{Theorem}
\newtheorem{lemma}{Lemma}
\newtheorem{corollary}{Corollary}

\newtheorem{proposition}{Proposition}

\theoremstyle{remark}
\newtheorem{assumption}{Assumption}
\newtheorem{definition}{Definition}
\newtheorem{example}{Example}

\newtheoremstyle{numbered}{}{}{\itshape}{}{\rm}{.}{.5em}{#1\thmnote{ #3}}
\theoremstyle{numbered}
\newtheorem*{numassumption}{Assumption}

\newtheoremstyle{myremark} 
    {\topsep}                    
    {\topsep}                    
    {\rm}                        
    {}                           
    {\rm}                        
    {.}                          
    {.5em}                       
    {}  

\theoremstyle{myremark}
\newtheorem{remark}{Remark}[section]

\usepackage[toc,page]{appendix}
\usepackage{mathtools}
\usepackage{yfonts}

\DeclareSymbolFont{Greekletters}{OT1}{iwona}{m}{n}
\DeclareSymbolFont{greekletters}{OML}{iwona}{m}{it}
\DeclareMathSymbol{\salpha}{\mathord}{greekletters}{"0B}
\DeclareMathSymbol{\sbeta}{\mathord}{greekletters}{"0C}
\DeclareMathSymbol{\sgamma}{\mathord}{greekletters}{"0D}
\DeclareMathSymbol{\sOmega}{\mathord}{Greekletters}{"0A}
\DeclareMathSymbol{\smu}{\mathord}{greekletters}{"16}
\DeclareMathSymbol{\svarepsilon}{\mathord}{greekletters}{"22}
\DeclareMathSymbol{\svarrho}{\mathord}{greekletters}{"25}
\DeclareMathSymbol{\svarphi}{\mathord}{greekletters}{"27}

\newcommand{\vast}{\bBigg@{3}}
\newcommand{\Vast}{\bBigg@{4}}

\newcommand{\norm}[1]{\left\lVert#1\right\rVert}
\newcommand{\inner}[1]{\left\langle#1\right\rangle}

\newcommand{\up}[1]{{(#1)}}
\newcommand{\bigzero}{\mbox{\normalfont\large\bfseries 0}}

\def\sfk{\textsf{k}}
\def\sfa{\textsf{a}}
\def\sfb{\textsf{b}}

\def\sfC{\textsf{C}}

\def\sfK{\textsf{K}}
\def\sfk{\textsf{k}}
\def\sfR{\textsf{R}}

\def\Opt{\mathrm{Opt}}

\def\cA{\mathcal{A}}

\def\op{\mathrm{op}}
\def\stst{{\mbox{\tiny\rm test}}}
\def\snew{{\mbox{\tiny\rm new}}}
\def\Lip{{\mbox{\tiny\rm Lip}}}
\def\bT{{\boldsymbol T}}

\def\cC{{\mathcal C}}
\def\cS{{\mathcal S}}
\def\cD{{\mathcal D}}
\def\cG{{\mathcal G}}

\def\bi{{\boldsymbol i}}

\def\bmu{{\boldsymbol \mu}}

\def\bSigma{\boldsymbol{\Sigma}}
\def\obX{\overline{\boldsymbol{X}}}

\def\supp{{\rm supp}}

\def\bfzero{{\boldsymbol 0}}

\def\bsigma{\boldsymbol{\sigma}}
\def\btau{{\boldsymbol{\tau}}}
\def\one{\mathbf{1}}

\def\normal{{\mathcal{N}}}

\def\balpha{{\boldsymbol \alpha}}
\def\bepsilon{{\boldsymbol \epsilon}}
\def\bbeta{{\boldsymbol \beta}}
\def\bv{{\boldsymbol v}}
\def\bd{{\boldsymbol d}}
\def\bR{{\boldsymbol R}}

\def\bq{{\boldsymbol q}}
\def\bsigma{{\boldsymbol \sigma}}
\def\bZ{{\boldsymbol Z}}

\def\bM{{\boldsymbol M}}
\def\bN{{\boldsymbol N}}

\def\bB{{\boldsymbol B}}

\def\beps{{\boldsymbol \varepsilon}}

\def\bh{{\boldsymbol h}}
\def\bH{{\boldsymbol H}}
\def\bI{{\boldsymbol I}}
\def\bX{{\boldsymbol X}}

\def\bw{{\boldsymbol w}}
\def\de{\textrm{d}}
\def\bx{{\boldsymbol x}}

\def\by{{\boldsymbol y}}
\def\bW{{\boldsymbol W}}
\def\ba{{\boldsymbol a}}
\def\bv{{\boldsymbol v}}

\def\Unif{\textrm{Unif}}
\def\bU{{\boldsymbol U}}
\def\bV{{\boldsymbol V}}
\def\bz{{\boldsymbol z}}

\def\bt{{\boldsymbol t}}
\def\bs{{\boldsymbol s}}
\def\bxi{{\boldsymbol \xi}}

\def\bF{{\boldsymbol F}}

\def\be{{\boldsymbol e}}
\def\bu{{\boldsymbol u}}
\def\bg{{\boldsymbol g}}

\def\RF{\textrm{RF}}
\def\NT{\textrm{NT}}
\def\bA{{\boldsymbol A}}
\def\bG{{\boldsymbol G}}
\def\btheta{{\boldsymbol \theta}}
\def\bfeta{{\boldsymbol \eta}}
\def\balpha{{\boldsymbol \alpha}}
\def\bTheta{{\boldsymbol \Theta}}
\def\hbTheta{\widehat{\boldsymbol \Theta}}
\def\hbtheta{\widehat{\boldsymbol \theta}}
\def\bH{{\boldsymbol H}}

\def\cU{{\mathcal U}}

\def\cS{{\mathcal S}}

\def\bP{{\boldsymbol P}}
\def\diag{\textrm{diag}}
\def\bS{{\boldsymbol S}}

\def\bD{{\boldsymbol D}}

\def\cB{\mathcal{B}}

\def\obx{\overline{\boldsymbol x}}

\def\reals{{\mathbb R}}

\def\Prox{\text{Prox}}

\def\bphi{{\boldsymbol \phi}}
\def\hR{\widehat{R}}
\def\hK{\widehat{K}}
\def\hy{\widehat{y}}
\def\ERM{{\normalfont\textrm{ERM}}}

\DeclareMathOperator*{\plim}{p-lim}
\DeclareMathOperator*{\pliminf}{p-liminf}

\hypersetup{
    colorlinks,
    linkcolor={blue!80!black},
    citecolor={green!50!black},
}
\colorlet{linkequation}{blue}